\documentclass[11pt]{article}

\usepackage{fullpage}
\usepackage{amsmath}
\usepackage{amsthm}
\usepackage{amsfonts}
\usepackage{graphicx}
\usepackage[margin=5pt, format=hang]{subfig}
\usepackage{amsrefs}
\usepackage{amssymb}                                                                                                                                                                                                                                                                                                                                                                                                                                                                                                                                                                                                                                                                                                                                                                                                 
\usepackage{pinlabel}
\usepackage{enumerate}

\theoremstyle{plain} \newtheorem{thm}{Theorem}[subsection]
\theoremstyle{plain} \newtheorem{cor}[thm]{Corollary}
\theoremstyle{plain} \newtheorem{prop}[thm]{Proposition}
\theoremstyle{plain} 
\theoremstyle{plain} \newtheorem{lem}[thm]{Lemma}
\theoremstyle{definition} \newtheorem{df}[thm]{Definition}
\theoremstyle{remark} \newtheorem{rmk}[thm]{Remark}
\theoremstyle{remark} 
\newcommand{\Bn}{B_{2n}}
\newcommand{\Kn}{K_{2n}}
\newcommand{\B}[1]{B_{#1} }
\newcommand{\K}[1]{K_{#1} }
\newcommand{\gena}{\sigma_{1}}

\newcommand{\genb}{\sigma_{2}\sigma_{1}^{2}\sigma_{2}}

\newcommand{\genc}{\sigma_{2i}\sigma_{2i-1}\sigma_{2i+1}\sigma_{2i}}

\newcommand{\DBCs}[1]{\Sigma(#1)\#(S^{2}\times S^{1})}
\newcommand{\SxS}{S^{2}\times S^{1}}

\newcommand{\DBC}[1]{\Sigma(#1)}
\newcommand{\SUM}[1]{#1 \#(S^{1}\times S^{2})}
\newcommand{\Ztil}{\tld{\mathcal{Z}}}

\newcommand{\Zcal}{\mathcal{Z}}
\newcommand{\Ccal}{\mathcal{C}}
\newcommand{\G}{\mathcal{G}}

\newcommand{\T}{\mathcal{T}}

\newcommand{\Fcal}{\mathcal{F}}

\newcommand{\DD}{\mathcal{D}}

\newcommand{\h}{\mathcal{H}}
\newcommand{\del}{\partial}
\newcommand{\tld}[1]{\widetilde{#1}}
\newcommand{\delh}{\widehat{\del}}
\newcommand{\Gtil}{\tld{\mathcal{G}}}

\newcommand{\confC}{Conf^{n}(\mathbb{C})}
\newcommand{\red}[1]{\underline{#1}}
\newcommand{\OS}{Ozsv\'ath and Szab\'o }
\newcommand{\Symg}{\text{Sym}^{g}(\Sigma)}
\newcommand{\Symn}[1]{\text{Sym}^{n}(#1)}

\newcommand{\HFM}{\widehat{HF}(M)}
\newcommand{\Us}{\mathfrak{U}_{\mathfrak{s}}}

\newcommand{\HFx}[1]{\widehat{HF}(\DBCs{#1})}
\newcommand{\CFxs}[1]{\widehat{CF}(\DBCs{#1}, \mathfrak{s})}
\newcommand{\HFxs}[1]{\widehat{HF}(\DBCs{#1}, \mathfrak{s})}

\newcommand{\CF}[1]{\widehat{CF}(\DBC{K})}
\newcommand{\HF}[1]{\widehat{HF}\left(\DBC{K}\right)}

\newcommand{\Scx}[1]{\text{Spin}^{c}(\DBCs{#1})}
\newcommand{\sz}[1]{\mathfrak{s}_{z}\left( #1 \right)}

\newcommand{\Sa}{S_{\ba}}

\newcommand{\Dab}[1]{\DD^{\ba\bb}_{#1}}
\newcommand{\Dbg}[1]{\DD^{\bb\bg}_{#1}}

\newcommand{\Dbd}[1]{\DD^{\bb\bd}_{#1}}

\newcommand{\Sb}{S_{\bb}}

\newcommand{\Pab}{\Pi_{\ba\bb}}
\newcommand{\Sab}{S_{\ba\bb}}

\newcommand{\Pabg}{\Pi_{\ba\bb\bg}}
\newcommand{\Sabg}{S_{\ba\bb\bg}}
\newcommand{\Pabgd}{\Pi_{\ba\bb\bg\bd}}
\newcommand{\Sabgd}{S_{\ba\bb\bg\bd}}
\newcommand{\ah}{\widehat{\alpha}}
\newcommand{\bh}{\widehat{\beta}}
\newcommand{\ba}{\boldsymbol{\alpha}}
\newcommand{\bb}{\boldsymbol{\beta}}
\newcommand{\bg}{\boldsymbol{\gamma}}
\newcommand{\bd}{\boldsymbol{\delta}}
\newcommand{\bah}{\widehat{\ba}}
\newcommand{\bbh}{\widehat{\bb}}

\newcommand{\bat}{\boldsymbol{\tld{\alpha}}}
\newcommand{\bbt}{\boldsymbol{\tld{\beta}}}

\newcommand{\Tah}{\mathbb{T}_{\bah}}
\newcommand{\Tbh}{\mathbb{T}_{\bbh}}

\newcommand{\Ta}{\mathbb{T}_{\ba}}

\newcommand{\Tb}{\mathbb{T}_{\bb}}

\newcommand{\Tg}{\mathbb{T}_{\bg}}
\newcommand{\Td}{\mathbb{T}_{\bd}}
\newcommand{\Tap}{\tor{\ba'}}
\newcommand{\Tbp}{\tor{\bb'}}
\newcommand{\bx}{\mathbf{x}}
\newcommand{\be}{\mathbf{e}}
\newcommand{\by}{\mathbf{y}}
\newcommand{\bbu}{\mathbf{u}}
\newcommand{\bv}{\mathbf{v}}

\newcommand{\bw}{\mathbf{w}}
\newcommand{\bu}{\mathbf{u}}
\newcommand{\bz}{\mathbf{z}}
\newcommand{\bxh}{\widehat{\mathbf{x}}}
\newcommand{\byh}{\widehat{\mathbf{y}}}
\newcommand{\Yab}{Y_{\ba\bb}}
\newcommand{\Ybg}{Y_{\bb\bg}}

\newcommand{\Ybbp}{Y_{\bb\bb'}}
\newcommand{\Yapb}{Y_{\ba'\bb}}
\newcommand{\Yag}{Y_{\ba\bg}}

\newcommand{\thetabb}{\boldsymbol{\theta}_{\bb \bb'}}

\newcommand{\thet}[1]{\boldsymbol{\theta}_{#1}}

\newcommand{\thetabpb}{\boldsymbol{\theta}_{\bb' \bb}}

\newcommand{\tor}[1]{\mathbb{T}_{#1}}
\newcommand{\Tspace}{\rule{0pt}{2.6ex}}
\newcommand{\Bspace}{\rule[-1.2ex]{0pt}{0pt}}
\newcommand{\Ks}[1]{Kh_{symp}(#1)}
\newcommand{\Kst}[1]{Kh_{symp,inv}(#1)}

\newcommand{\AD}{\nabla}
\newcommand{\Zcaltwo}{\mathbb{Z}/2\mathbb{Z}}
\newcommand{\tbt}{\tld{\theta}_{\beta}}
\newcommand{\tat}{\tld{\theta}_{\alpha}}
\newcommand{\ta}{\theta_{\alpha}}
\newcommand{\tb}{\theta_{\beta}}

\begin{document}
\title{On the anti-diagonal filtration for the Heegaard Floer chain complex of a branched double-cover} 
\author{Eamonn Tweedy} 
\maketitle
\begin{abstract} 
Seidel and Smith introduced the graded fixed-point symplectic Khovanov cohomology group $Kh_{symp,inv}(K)$ for a knot $K \subset S^{3}$, as well as a spectral sequence converging to the Heegaard Floer homology group $\widehat{HF}(\DBC{K}\#(S^{2}\times S^{1}))$ with $E_{1}$-page isomorphic to a factor of $Kh_{symp,inv}(K)$ \cite{ss:R2}.  There the authors proved that $Kh_{symp,inv}$ is a knot invariant.  We show here that the higher pages of their spectral sequence are knot invariants also.
\end{abstract}

\section{Introduction}\label{sec:intro}

Heegaard Floer homology was introduced by \OS in \cite{os:disk}, and has proven to be a very useful tool in   studying manifolds of dimensions three and four.  We'll be particularly interested in the invariant $\widehat{HF}$, which assigns to a 3-manifold $M$ an abelian group $\widehat{HF}(M)$.  Given a knot $K \subset S^{3}$, the present paper will study $\widehat{HF}\left(\DBCs{K}\right)$, where $\DBC{K}$ is the two-fold cover of the sphere $S^{3}$ branched along the knot $K$.  The Heegaard Floer homology of branched double-covers was studied in \cite{os:bc}, in which \OS constructedd a spectral sequence from the reduced Khovanov homology group $\tld{Kh}\left(-K; \Zcaltwo\right)$ to the group $\widehat{HF}\left(\DBC{K}; \Zcaltwo\right)$, where $-K$ denotes the mirror of $K$.

Given a presentation of a knot $K \subset S^{3}$ as the braid closure of a braid $b$, Seidel and Smith introduced in \cite{ss:R1} the symplectic Khovanov cohomology group $\Ks{b}$, which is defined by taking the Lagrangian Floer cohomology of two Lagrangian submanifolds inside an affine variety.  Clearly there may be different braids which have isotopic braid closures.  However, Seidel and Smith proved in \cite{ss:R1} that $Kh_{symp}$ is a knot invariant.  In \cite{reza:ss}, Rezazadegan proved the existence of a spectral sequence from $Kh(L)$ to $Kh_{symp}(L)$ with $\mathbb{Z}/2\mathbb{Z}$ coefficients.  Recent work-in-progress of Abouzaid and Smith \cite{AS} indicates that in fact $\text{rk}_{\mathbb{Q}}Kh(L) = \text{rk}_{\mathbb{Q}}Kh_{symp}(L)$.

Further, by studying the fixed-point sets of an involution on the variety, Seidel and Smith further define in \cite{ss:R2} the fixed-point symplectic Khovanov cohomology group $\Kst{b}$ for a braid $b$.  Via the choice of a particular holomorphic volume form, one obtains gradings (in the sense of \cite{s:GL}) on the totally-real submanifolds $\T$ and $\T'$ used to define $\Kst{b}$; the gradings on these submanifolds induce an absolute $\mathbb{Z}$-valued Maslov grading $\tld{R}$ on the set $\T \cap \T'$.

We'll consider braids in $\Bn$, the braid group on $2n$ strands (where $n \in \mathbb{N}$), and obtain knot diagrams by taking plat closures.  Although Seidel and Smith \cite{ss:R1},\cite{ss:R2} and Manolescu \cite{cm:R} considered braid closures instead, our convention will be chosen for computational reasons (note that Waldron illustrated in \cite{jw:plat} that $Kh_{symp}$ can be defined for bridge diagrams coming from such plat closures).  We'll recall the definition for the set $\G$ of \textit{Bigelow generators}, unordered $n$-tuples of distinct intersection points in a \textit{fork diagram} obtained from the braid $b$.  Following \cite{big:jones} and \cite{cm:R}, we'll then define functions $Q,  P, T: \G \rightarrow \mathbb{Z}$ which can be computed from this diagram in an elementary fashion.

In \cite{cm:R}, Manolescu used the fork diagram to give a description of the group $\Kst{b}$, and in particular showed a one-to-one correspondence between $\G$ and a set of generators for the Seidel-Smith cochain complex.  In this context, one can view the totally real submanifolds $\T$ and $\T'$ as admissible Heegaard tori for the manifold $\DBCs{K}$.  Thus the set $\G$ is also in one-to-one correspondence with a set of generators for the chain group  $\widehat{CF}(\DBCs{K})$.  This identification provides a function $\tld{R}: \G \rightarrow \mathbb{Z}$, and following \cite{cm:R} we have that $\tld{R} = T - Q +  P$.

The function $R$ is obtained from $\tld{R}$ by a rational shift $s_{R}$ which depends on some properties of the braid $b \in \Bn$ and the knot diagram $D$ which is its plat closure.  Let $e(b)$ be the signed count of braid generators in the word $b$ and let $w(D)$ be the writhe of the diagram $D$ for $K$ given by the plat closure of $b$.  Then define
\begin{equation*}
R = \tld{R} + s_{R}(b,D), \quad \text{where} \quad s_{R}(b,D) = \frac{e(b) - w(D) -2n}{4}.
\end{equation*}
Furthermore, for $\mathfrak{s} \in \text{Spin}^{c}(M)$ torsion, \OS used surgery cobordisms to define an absolute $\mathbb{Q}$-valued grading $\tld{gr}$ on the subcomplex $\widehat{CF}(M, \mathfrak{s})$ which is an absolute lift of the relative Maslov $\mathbb{Z}$-grading.  Then for torsion $\mathfrak{s}$, we define a filtration $\rho$ on the Heegaard Floer chain complex $\CFxs{K}$ by $\rho = R - \tld{gr}$.

Two braids with isotopic plat closures can be connected via a finite sequence of Birman moves \cite{bir:moves}, which in turn induce sequences of isotopies, handleslides, and stabilizations (and associated chain homotopy equivalences on the Heegaard Floer complexes).  We will prove the following theorem about the filtration $\rho$ in Section \ref{sec:Rmoves}:

\begin{thm}\label{thm:Rthm}
Let the braids $b \in B_{2n}$ and $b' \in B_{2m}$ have plat closures which are diagrams for the knot $K$.  Let $\h$ and $\h'$ be the pointed Heegaard diagrams for $\DBCs{K}$ induced by $b$ and $b'$,respectively, in the sense of Proposition \ref{prop:DBC} below.  Let $\mathfrak{s} \in \text{Spin}^{c}(\DBCs{K})$ be torsion.  Then the $\rho$-filtered chain complexes $\widehat{CF}(\h, \mathfrak{s})$ and $\widehat{CF}(\h', \mathfrak{s})$ have the same filtered chain homotopy type.
\end{thm}

More concisely, we can state the following:

\begin{cor}\label{cor:Rcor}
For each torsion $\mathfrak{s} \in \text{Spin}^{c}(\DBCs{K})$, the $\rho$-filtered chain homotopy type of the complex $\CFxs{K}$ is an invariant of $K$.
\end{cor}

In a standard way, the filtration $\rho$ provides a spectral sequence (whose pages we'll denote by $E_{k}$) computing the group $\HFx{K}$.  Furthermore, the page $E_{1}$ is isomorphic to the subgroup of $\Kst{b}$ obtained by taking cohomology of the subcomplex whose generators correspond to generators of $\widehat{CF}$ in the torsion $\text{Spin}^{c}$ structures on $\DBCs{K}$.  This spectral sequence is the same as the one defined by Seidel and Smith in \cite{ss:R2}.  There they proved that $Kh_{symp,inv}$ is a knot invariant, and so the the factor corresponding to $E_{1}$ is also.  Because higher pages are determined by the filtered chain homotopy type of $E_{0}$, Corollary \ref{cor:Rcor} implies the following.

\begin{cor}\label{cor:Rss}
For $k \geq 1$, the page $E_{k}$ is a knot invariant.
\end{cor}

Under certain degeneracy conditions of the spectral sequence, the function $R$ in fact provides a homological grading on Heegaard Floer theory.  We say that a knot $K$ is \textit{$\rho$-degenerate} if the spectral sequence collapses at $E_{1}$ and the induced filtration $\rho$ on $E_{\infty}$ is constant on each nontrivial factor $\HFxs{K}$.  The following is an easy consequence of the definitions.

\begin{prop}\label{prop:Rdeg}
Let $K \subset S^{3}$ be a knot.  Then the following are equivalent:
\begin{enumerate}[(i)]
\item $K$ is $\rho$-degenerate. \label{Rdeg:1}
\item The filtration $R$ is a grading and lifts the relative Maslov $\mathbb{Z}$-grading on each nontrivial factor $\HFxs{K}$. \label{Rdeg:2}
\end{enumerate}
Moreover, the grading $R$ is a knot invariant when the above hold.
\end{prop}

\section{Topological preliminaries}

In \cite{os:disk}, \OS define the Heegaard Floer homology group $\widehat{HF}(M)$ associated to a connected, closed, oriented 3-manifold $M$.  A genus-g Heegaard splitting for such a manifold can be described via a \textit{pointed Heegaard diagram} $\h = \left(\Sigma; \ba; \bb; z\right)$, where $\Sigma$ is the splitting surface, $\ba$ and $\bb$ are g-tuples of attaching curves for the handlebodies, and $z \in (\Sigma - \cup \alpha_{i} - \cup \beta_{i})$.

\begin{df}
Let $\left(\Sigma; \ba; \bb; z\right)$ be a pointed Heegaard diagram, and let $D_1, \ldots, D_m$ be the connected components of $\Sigma \setminus \left( \cup \alpha_i \right) \setminus \left( \cup \beta_i \right)$, where $z \in D_m$.  Then a two-chain
$$ \mathcal{P}:= \sum_{i =1}^{m-1}n_{i} D_i \quad \text{with} \quad n_i \in \mathbb{Z}$$
is called a \textit{periodic domain} if its boundary is a sum of $\alpha$ and $\beta$ circles.
\end{df}

\begin{df}
A Heegaard diagram $\left(\Sigma; \ba; \bb; z\right)$ is called \textit{admissible} if every periodic domain has both positive and negative coefficients.
\end{df}

If $\h$ is an admissible pointed Heegaard diagram, then one can compute the chain complex $\widehat{CF}(\h)$ and its homology group $\widehat{HF}(M)$ is the Lagrangian Floer homology of the tori $\Ta$ and $\Tb$ lying inside of the symplectic manifold $\text{Sym}^g \left( \Sigma \setminus z \right)$.

More precisely, the group $\widehat{CF}(\h)$ is generated by the set of intersections $\Ta \cap \Tb \subset \Symg$, and the differential is given by
$$\widehat{\partial}(\bx) =
\displaystyle\sum_{\by \in \Ta \cap \Tb} \left( \displaystyle\sum_{\{\phi \in \pi_{2}(\bx, \by)| \mu(\phi) = 1, n_{z}(\phi) = 0\}} \left(\# \widehat{\mathcal{M}}\left(\phi\right)\right)\right)\by,
$$
where $\widehat{\mathcal{M}}(\phi)$ denotes the reduced moduli space of pseudo-holomorphic representatives for the class $\phi$, $\mu(\phi)$ denotes the Maslov index of $\phi$, and $n_z(\phi):=\text{Im}(\phi) \cap \left( \{ z \} \cap \text{Sym}^{g-1}(\Sigma)\right).$

Recall that there is a function
\begin{equation*}
\mathfrak{s}_{z}: \Ta \cap \Tb \longrightarrow \text{Spin}^{c}(M)
\end{equation*}
partitioning $\Ta \cap \Tb$ into equivalence classes $\Us$.  In fact, this function $\mathfrak{s}_{z}$ induces decompositions
\begin{equation*}
\widehat{CF}(\h)  = \displaystyle \bigoplus_{\mathfrak{s} \in \text{Spin}^{c}(M)} \widehat{CF}(\h, \mathfrak{s}) \quad \text{and} \quad
\HFM = \displaystyle \bigoplus_{\mathfrak{s} \in \text{Spin}^{c}(M)} \widehat{HF}(M, \mathfrak{s}).
\end{equation*}

For each $\mathfrak{s} \in \text{Spin}^{c}(M)$ the chain complex $\widehat{CF}(M,\mathfrak{s})$ carries a relative grading $gr$ defined via the Maslov index.  For $\mathfrak{s} \in \text{Spin}^{c}(M)$ torsion, \OS use surgery cobordisms to construct in \cite{os:tri} an absolute $\mathbb{Q}$-valued grading $\tld{gr}$ on $\Us$ which lifts the relative grading in the following sense: if $\bx, \by \in \Us$, then
\begin{equation*}
\tld{gr}(\bx) - \tld{gr}(\by) = gr(\bx, \by). 
\end{equation*}
Whenever $b_{1}(M) = 0$, all $\text{Spin}^{c}$ structures on $M$ are torsion and so the group $\widehat{HF}(M)$ can be absolutely graded via $\tld{gr}$.  In particular, this holds for $M =\DBC{K}$ for a knot $K \subset S^{3}$.  However, although $\text{Spin}^{c}(\DBCs{K})$ contains non-torsion elements, the group $\HFxs{K}$ is nontrivial only if $\mathfrak{s}$ is torsion.  

\subsection{3-gon chain maps and 4-gon homotopies}
\label{sec:triangles}
\begin{rmk}
There can be some ambiguity surrounding terms like ``triangle" and ``quadrilateral", in particular when distinguishing between the polygons in the symmetric product $\text{Sym}^{g}(\Sigma)$ and the regions which are their shadows in the surface $\Sigma$.  We'll follow Sarkar's convention in \cite{sarkar:tri} in using neither of these words.  The Whitney polygons in symmetric products will be referred to as \textit{n-gons} and regions in surfaces will be referred to as \textit{n-sided regions}.
\end{rmk}

In \cite{os:disk} and \cite{os:tri}, maps between Floer homologies are constructed by counting pseudo-holomorphic 3-gons in a certain equivalence class.  We review these ideas below.

First recall the notion of a \textit{pointed Heegaard triple-diagram} $\left(\Sigma; \ba; \bb; \bg; z\right)$, where $\Sigma$ is an oriented two-manifold of genus $g$, $\ba$, $\bb$, and $\bg$ are complete $g$-tuples of attaching circles for handlebodies $U_{\alpha}$, $U_{\beta}$, and $U_{\gamma}$, respectively, and $z \in (\Sigma \setminus \cup \alpha_{i} \setminus \cup \beta_{i} \setminus \cup \gamma_{i})$.  We then have pointed Heegaard diagrams $\h_{\ba\bb} = \left(\Sigma; \ba; \bb; z\right)$, $\h_{\bb\bg} = \left(\Sigma; \bb; \bg; z\right)$, and $\h_{\ba\bg} = \left(\Sigma; \ba; \bg; z\right)$, depicting manifolds $\Yab$, $\Ybg$, and $\Yag$, respectively.  There is an analogous notion of a \textit{pointed Heegaard quadruple-diagram} $\left(\Sigma; \ba; \bb; \bg; \bd; z\right)$.

There are notions of \textit{triply-periodic domains} in triple-diagrams and \textit{quadruply-periodic domains} in quadruple-diagrams, and the definitions are analogous to that of a periodic domain.  Multi-diagrams also have analogous notions of admissibility.

\begin{df}
A pointed Heegaard triple-diagram (resp. quadruple-diagram) is \textit{admissible} if every triply-periodic domain (resp. quadruply-periodic domain) has both positive and negative coefficients.
\end{df}

If the pointed triple-diagram $\left(\Sigma; \ba; \bb; \bg; z\right)$ is admissible, then there is a chain map
\begin{equation*}
\widehat{f}_{\alpha\beta\gamma}: \widehat{CF}(\h_{\ba\bb}) \otimes \widehat{CF}(\h_{\bb\bg}) \rightarrow \widehat{CF}(\h_{\ba\bg})
\end{equation*}
given by the formula
\begin{equation*}
\widehat{f}_{\alpha\beta\gamma}(\bx \otimes \by) =
\displaystyle\sum_{\bw \in \Ta \cap \Tg} \left( \displaystyle\sum_{\{\psi \in \pi_{2}(\bx, \by, \bw)| \mu(\psi) = 0, n_{z}(\psi) = 0\}} \left(\# \mathcal{M}\left(\psi\right)\right)\right)\bw
\end{equation*}
where $\mathcal{M}\left( \psi \right)$ is the moduli space of pseudo-holomorphic representatives for the class $\psi$.  The induced map on homology will be denoted by $\widehat{F}_{\alpha\beta\gamma}$.

If the pointed quadruple-diagram $\left(\Sigma; \ba; \bb; \bg; \bd; z\right)$ is admissible, then one can define a map
\begin{equation*}
\widehat{h}_{\alpha\beta\gamma\delta}: \widehat{CF}(\h_{\ba\bb}) \otimes \widehat{CF}(\h_{\bb\bg}) \otimes \widehat{CF}(\h_{\bg\bd}) \rightarrow \widehat{CF}(\h_{\ba\bd})
\end{equation*}
by the formula
\begin{equation*}
\widehat{h}_{\alpha\beta\gamma\delta}(\bx \otimes \by \otimes \bw) =
\displaystyle\sum_{\bz \in \Ta \cap \Td} \left( \displaystyle\sum_{\{\psi \in \pi_{2}(\bx, \by, \bw, \bz)| \mu(\psi) = -1, n_{z}(\psi) = 0\}} \left(\# \mathcal{M}\left(\psi\right)\right)\right)\bz
\end{equation*}

A 4-gon map actually provides a chain homotopy between two compositions of 3-gon maps:

\begin{thm}[\cite{os:disk}]\label{thm:assoc}
Let $\left(\Sigma; \ba; \bb; \bg; \bd; z\right)$ be an admissible pointed Heegaard quadruple-diagram.  Then for $\xi \in \widehat{CF}(\h_{\ba\bb})$, $\eta \in \widehat{CF}(\h_{\bb\bg})$, and $\zeta \in \widehat{CF}(\h_{\bg\bd})$,
\begin{equation*}
\partial \widehat{h}_{\alpha\beta\gamma\delta}(\xi \otimes \eta \otimes \zeta) + \widehat{h}_{\alpha\beta\gamma\delta}(\partial(\xi \otimes \eta \otimes \zeta))
= \widehat{f}_{\alpha\gamma\delta}(\widehat{f}_{\alpha\beta\gamma}(\xi \otimes \eta) \otimes \zeta)
- \widehat{f}_{\alpha\beta\delta}(\xi \otimes \widehat{f}_{\beta\gamma\delta}(\eta \otimes \zeta))
\end{equation*}
\end{thm}

Classes of Whitney $n$-gons can be studied by examining their `shadows' in the Heegaard surface $\Sigma$.  We recall the definition of the domain of a $2$-gon class, though there are analogous notions of domains of $n$-gon classes.

\begin{df}
Let $\left( \Sigma; \ba ; \bb; z \right)$ be a pointed Heegaard diagram, and denote by $\mathcal{D}_0, \mathcal{D}_1, \ldots, \mathcal{D}_N$ the connected components of $\Sigma \setminus \left( \cup_i \alpha_i\right) \setminus \left( \cup_i \beta_i \right),$ where $\mathcal{D}_0$ is the component containing the basepoint $z$.  Then for $0 \leq j \leq N$, choose a point $z_j$ in the interior of $\mathcal{D}_j$.  For some class $\phi \in \pi_2(\bx, \by)$ for $\bx, \by \in \Ta \cap \Tb$, the \textit{domain of $\phi$}  is the 2-chain
$$\mathcal{D}(\phi):= \sum_{j = 0}^{N} n_j \mathcal{D}_j \quad \text{where} \quad n_j:= \text{Im}(\phi) \cap \left(\left \{ z_j \right\} \times \text{Sym}^{g-1}(\Sigma) \right).$$
We'll say that $\phi$ \textit{avoids the basepoint} if $n_0 = 0$ (equivalently, $n_z(\phi) = 0$).
\end{df}

\subsubsection{Some index-zero 3-gon classes}\label{sec:ztri}

We're interested in 3-gon classes of Maslov index zero.  To calculate index, we'll follow Sarkar's work in \cite{sarkar:tri} on Whitney $n$-gons, which we'll review here.  Some labeling conventions have been modified to fit our notation, and we'll specialize to the $n=3$ case for this discussion.

Let $\left(\Sigma; \ba; \bb; \bg; z\right)$ be an admissible pointed Heegaard triple-diagram, and let $\psi$ be a 3-gon class connecting $\bx$, $\by$, and $\bw$ as defined above.  Denote by $a(\psi)$, $b(\psi)$, and $c(\psi)$ the boundaries $\partial \mathcal{D}(\psi)|_{\alpha}$, $\partial \mathcal{D}(\psi)|_{\beta}$, and $\partial \mathcal{D}(\psi)|_{\gamma}$, respectively.

Now given some 1-chains $a$ supported on $\alpha$ and $b$ supported on $\beta$, Sarkar defines the number $b.a$ as follows.  Assuming some orientation on the $\alpha$ and $\beta$ circles and on $\Sigma$, we have four well-defined directions in which we can translate $b$ so that no endpoint of $a$ lies on the translate $b'$ and no endpoint of $b'$ lies on $a$.  These can be thought of as `northeast", ``northwest", ``southeast", and ``southwest".  After a small translation in some direction, we can calculate the intersection number of $b'$ with $a$.  Then $b.a$ is defined to be the average of these numbers over the four possible translation directions.

Some element $\bx \in \Ta \cap \Tb$ is an unordered $g$-tuple $ \left\{ x_{1}, x_{2}, \ldots, x_{g} \right \}$.  Define the number $\mu_{x}(\psi) = \sum \mu_{x_{i}}(\psi)$, where $\mu_{x_{i}}(\psi)$ is the average of the local coefficients of the 2-chain $\mathcal{D}(\psi)$ over the four quadrants around $x_{i} \in \Sigma$.

The Euler measure of $\mathcal{D}(\psi)$ will be denoted by $e(\psi)$.  The Euler measure is additive, and it is enough for our purposes to know that the measure of an $n$-sided region is $(1-n/4)$.

Equipped with these concepts, we present the following formula of Sarkar:

\begin{thm}[\cite{sarkar:tri}]\label{maslovformula}
Let $\left(\Sigma; \ba; \bb; \bg; z\right)$ be a pointed Heegaard triple-diagram, and let $\psi \in \pi_2(\bx,\by,\bw)$ be a 3-gon class connecting $\bx \in \tor{\ba}\cap\tor{\bb}$, $\by \in \tor{\bb}\cap\tor{\bg}$, and $\bw \in \tor{\ba} \cap \tor{\bg}$.  Then the Maslov index $\mu (\psi)$ satisfies the formula
\begin{equation*}
\mu (\psi) = e(\psi) + \mu_{\bx}(\psi) + \mu_{\by} (\psi) + a(\psi).c(\psi) - g/2.
\label{eq:sarkar}
\end{equation*}
\end{thm}

Here we'll discuss two types of 3-gon classes in $\Symg$.

A 3-gon $\psi$ of the first type has domain $\mathcal{D}(\psi)$ given by the sum of $g$ disjoint 3-sided regions, each with coefficient $+1$.  A 3-gon $\psi$ of the second type has domain $\mathcal{D}(\psi')$ given by the sum of $(g-1)$ disjoint regions, consisting of $(g-2)$ 3-sided regions and a single 6-sided region with one angle larger than $\pi$, each with coefficient $+1$.  Components of $\ba$, $\bb$, and $\bg$ are solid, dashed, and dotted arcs, respectively.  Components of $\bx$, $\by$, and $\bw$ are dark gray, white, and light gray dots, respectively.

The reader can verify that $\mu\left(\psi\right)=0$ in either case (in the second, it will help to split the obtuse hexagonal component of the domain as seen in Figure \ref{fig:trisplit}).

\begin{figure}[h!]
\centering
\begin{minipage}[c]{.45\linewidth}
\labellist
\pinlabel* $=$ at 550 195
\pinlabel* $+$ at 1000 195
\endlabellist
\includegraphics[height = 20mm]{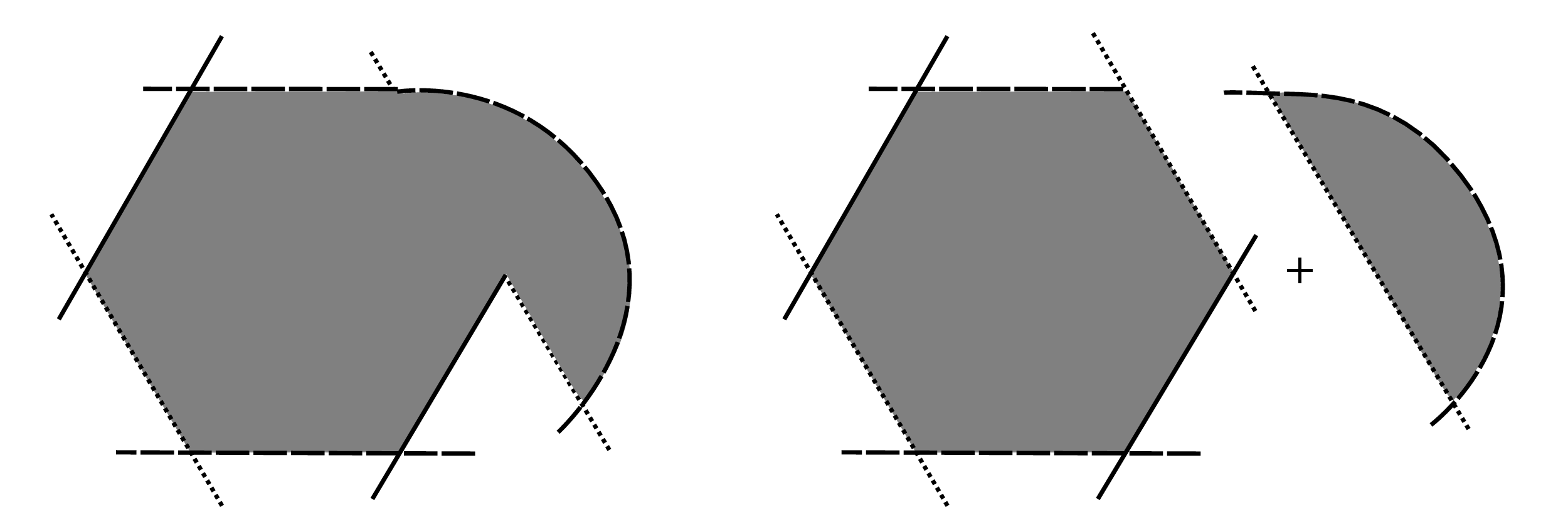}
\end{minipage}
\begin{minipage}[c]{.53\linewidth}
\caption[Type II 3-gon decomposition]{Decomposing the obtuse 6-sided component of $\mathcal{D}(\psi)$
\label{fig:trisplit}}
\end{minipage}
\end{figure}

\subsubsection{3-gons and 4-gons in Heegaard moves}\label{sec:movegons}

\begin{df}\label{def:moves}
Let $\left(\Sigma; \ba; \bb; \bb'; z\right)$ be a pointed Heegaard triple-diagram.
\begin{enumerate}[(i)]
\item Let $\beta'_j$ differ from $\beta_j$ by an isotopy (avoiding $z$) such that $\beta'_j$ intersects $\beta_j$ transversely in two canceling points and $\beta_j \cap \beta'_i = \emptyset$ when $i \neq j.$  Then we say that $\bb'$ differs from $\bb$ by a \textit{pointed isotopy}.  A pointed isotopy which preserves the set of intersection points $\tor{\ba}\cap\tor{\bb}$ in the obvious way will be called a \textit{small pointed isotopy}.
\item Instead let $\beta_{1}$, $\beta_{2}$, and $\beta_{1}'$ bound an embedded pair of pants disjoint from $z$ such that $\beta'_{1}$ intersects $\beta_{1}$ transversely in two points.  Assume also that $\beta_j \cap \beta'_i$ for $i \neq j$, and that for $i > 1$, $\beta'_{i}$ relates to $\beta_i$ as $(i)$ above.  Then we say that $\bb'$ differs from $\bb$ by a \textit{pointed handleslide}.
\end{enumerate}
\end{df}

In either of the cases above, $\left( \Sigma; \bb; \bb'; z \right)$ is an admissible pointed Heegaard diagram for $\#^{g}(S^{1} \times S^{2})$ and there is a canonical intersection point $\thetabb \in \Tb \cap \Tbp$ representing the top-degree homology class in $\widehat{HF}(\Ybbp)$.  If the triple-diagram is admissible, we have a well-defined chain map 
\begin{equation*}
\widehat{f}_{\alpha\beta\beta'}(\cdot \otimes \thetabb):
\widehat{CF}(\h_{\ba\bb}) \rightarrow \widehat{CF}(\h_{\ba\bb'}).
\end{equation*}

Note that in the original proof of invariance in \cite{os:disk}, isotopies weren't treated in terms of chain maps which count pseudo-holomorphic 3-gons.  Lipshitz proves in Proposition 11.4 from  \cite{lip:cyl} that this can be done.

Now let $\left(\Sigma; \ba; \bb; \bb'; \bbt\right)$ be an admissible pointed Heegaard quadruple-diagram, where $\bbt$ differs from $\bb$ by a small pointed isotopy, and $\bb'$ differs from $\bb$ (and necessarily from $\bbt$) by a pointed handleslide or a pointed isotopy.  We can identify $\tor{\ba}\cap\tor{\bb}$ with $\tor{\ba}\cap \tor{\tld{\bb}}$ via the canonical \textit{nearest-neighbor map} $N_{\bb\tld{\bb}}:\bx \mapsto \tld{\bx}$, and extend this linearly to a chain complex isomorphism $N_{\bb \tld{\bb}}: \widehat{CF}(\h_{\ba \bb}) \rightarrow \widehat{CF}(\h_{\ba \tld{\bb}})$ (notice that $\left( N_{\bb \tld{\bb}} \right)6{-1} = N_{\tld{\bb} \bb} $).  We then have that
\begin{equation*}
\widehat{f}_{\ba\bb\bbt} \left( \bx \otimes \widehat{f}_{\bb\bb'\bbt} \left(  \thet{\bb\bb'} \otimes \thet{\bb'\bbt} \right) \right) =
\widehat{f}_{\ba\bb\bbt} \left( \bx \otimes \thet{\bb\bbt} \right) = \tld{\bx} \quad \text{for all} \quad \bx \in \Ta \cap \Tb,
\end{equation*}
where the last equality is due to Lemma 9.28 of \cite{jt:nat} (cf. Proposition 9.8 of \cite{os:disk}).
Then by Theorem \ref{thm:assoc}, we have that
\begin{align*}
&\widehat{f}_{\ba\bb'\bbt} \left( \widehat{f}_{\ba\bb\bb'} \left( \bx \otimes \thet{\bb\bb'} \right) \otimes \thet{\bb'\bbt} \right)
-  \tld{\bx}= \\
&\delh \left( \widehat{h}_{\ba\bb\bb'\bbt} \left(\bx \otimes \thet{\bb\bb'} \otimes \thet{\bb'\bbt} \right) \right)
+ \widehat{h}_{\ba\bb\bb'\bbt}  \left( \delh \left( \bx \otimes \thet{\bb\bb'} \otimes \thet{\bb'\bbt} \right) \right).
\end{align*}


Letting $\bbt'$ differ from $\bb'$ by a pointed isotopy and studying the admissible pointed Heegaard quadruple-diagram $\left(\Sigma; \ba; \bb'; \bb; \bbt'; z\right)$, one finds that for $\bx \in \Ta \cap \Tbp$,
\begin{align*}
&\widehat{f}_{\ba\bb\bbt'} \left( \widehat{f}_{\ba\bb'\bb} \left( \bx \otimes \thet{\bb'\bb} \right) \otimes \thet{\bb\bbt'} \right)
- \tld{\bx} = \\
&\delh \left( \widehat{h}_{\ba\bb'\bb\bbt'} \left(\bx \otimes \thet{\bb'\bb} \otimes \thet{\bb\bbt'} \right) \right)
+ \widehat{h}_{\ba\bb'\bb\bbt'}  \left( \delh \left( \bx \otimes \thet{\bb'\bb} \otimes \thet{\bb\bbt'} \right) \right).
\end{align*}

Therefore, we see that when $\bb'$ differs from $\bb$ by a pointed isotopy or a pointed handleslide, the chain map $\widehat{f}_{\alpha\beta\beta'}(\cdot \otimes \thetabb)$ is a chain homotopy equivalence with homotopy inverse given by $\widehat{f}_{\alpha\beta'\beta}(\cdot \otimes \thetabpb)$.  Furthermore, the associated homotopies relating their compositions to the appropriate identity maps are given by
\begin{equation}\label{eqn:beta}
\begin{aligned}
N_{\tld{\bb}\bb} &\circ \widehat{h}_{\ba\bb\bb'\bbt} \left(\cdot \otimes \thet{\bb\bb'} \otimes \thet{\bb'\bbt} \right) \quad \text{and}\\
N_{\tld{\bb}'\bb'} &\circ \widehat{h}_{\ba\bb'\bb\bbt'} \left(\bx \otimes \thet{\bb'\bb} \otimes \thet{\bb\bbt'} \right).
\end{aligned}
\end{equation}

\begin{rmk}
Given admissible pointed Heegaard quadruple-diagrams $\left(\Sigma; \bat; \ba'; \ba; \bb; z\right)$ and $\left(\Sigma; \bat'; \ba; \ba'; \bb; z\right)$, where $\ba'$ differs from $\ba$ by a pointed isotopy or a pointed handleslide (with $\bat$ and $\bat'$ analogous to $\bbt$ and $\bbt'$), one can define the chain maps $\widehat{f}_{\ba', \ba,\bb}(\thet{\ba'\ba} \otimes \cdot)$ and $\widehat{f}_{\ba, \ba',\bb}(\thet{\ba\ba'} \otimes \cdot)$.  These two maps are chain homotopy inverses to one another, and the associated chain homotopies are
\begin{equation}\label{eqn:alpha}
\begin{aligned}
N_{\tld{\ba}\ba} &\circ  \widehat{h}_{\bat\ba'\ba\bb} \left( \thet{\bat\ba'} \otimes \thet{\ba'\ba} \otimes \cdot \right) \quad \text{and}\\
N_{\tld{\ba}' \ba'} &\circ  \widehat{h}_{\bat'\ba\ba'\bb} \left( \thet{\bat'\ba} \otimes \thet{\ba\ba'} \otimes \cdot \right).
\end{aligned}
\end{equation}
\end{rmk}

\subsection{Periodic domains}\label{sec:pd}

Recall that a periodic domain in a pointed Heegaard diagram $\left(\Sigma; \ba; \bb; z\right)$ is a domain avoiding the basepoint $z$ whose boundary is a sum of the $\ba$ and $\bb$ circles.  Denote by $\Pab \subset H_{2}\left(\Sigma; \mathbb{Z}\right)$ the group of such periodic domains and let $\Sab = \Sa + \Sb \subset H_{1}\left(\Sigma; \mathbb{Z}\right)$ be the span of the $\ba$ and $\bb$ circles.

Recall also the analogous notions of triply- and quadruply-periodic domains in Heegaard triple-diagrams and quadruple-diagrams.

In \cite{cmoz:thin}, it is shown that if $\left(\Sigma; \ba; \bb; z\right)$ is a pointed Heegaard diagram, then $\Pab$ is a free Abelian group of rank $2g - \text{rank}(\Sab)$.  It can be shown in a completely analogous way that for a triple-diagram (respectively quadruple-diagram), the group $\Pabg$ (respectively $\Pabgd$) is free Abelian of rank $3g - \text{rank}(\Sabg)$ (respectively $4g - \text{rank}(\Sabgd)$).  One should note that because we don't permit periodic domains in a pointed Heegaard diagram to intersect the basepoint, our ranks are 1 lower than those stated in \cite{cmoz:thin}.

Let $\ba$ and $\bb$ be two $g$-tuples of attaching circles on a genus-$g$ surface $\Sigma$ such that $\bb$ differs from $\ba$ by a pointed isotopy.  Then for each $i$, the circles $\alpha_{i}$ and $\beta_{i}$ are separated by two 2-sided regions, and we denote by $\Dab{i}$ the periodic domain which is their difference - these domains look like the ones shown in Figure \ref{fig:pdhs1}.

Now instead let $\ba$ and $\bb$ be two $g$-tuples of attaching circles on a genus-$g$ surface $\Sigma$ such that $\bb$ differs from $\ba$ by a pointed handleslide of $\alpha_{1}$ over $\alpha_{2}$.  For $i > 1$, the circles $\alpha_{i}$ and $\beta_{i}$ are separated by two thin 2-sided regions, and we denote by $\Dab{i}$ the periodic domain which is their difference.  The circles $\alpha_{1}$ and $\beta_{1}$ are separated by a thin 2-sided region, and we denote by $\Dab{1}$ the periodic domain which is the difference between this region and the annular region bounded by $\alpha_{1}$, $\alpha_{2}$, and $\beta_{1}$.  These domains can be seen in Figure \ref{fig:pdhs}.

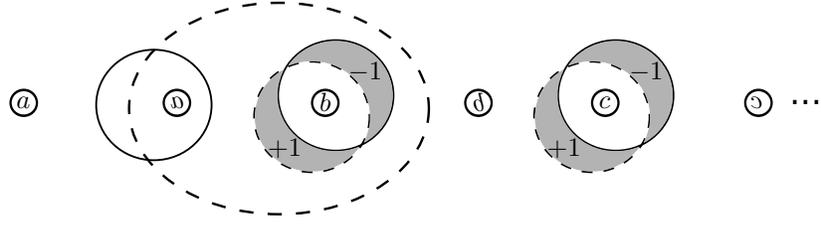
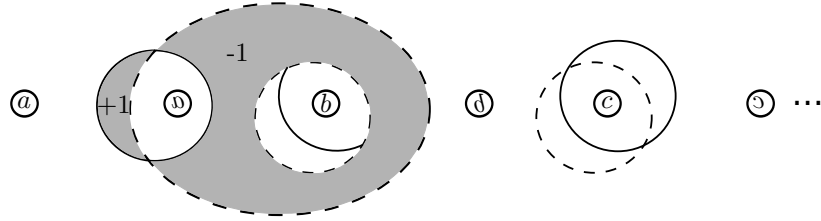
\begin{figure}
\centering
\subfloat[The handleslide domains $\Dab{i}$ for $i > 1$.]{
\labellist
\small
\pinlabel* $+1$ at 202 60
\pinlabel* $-1$ at 256 111
\pinlabel* $+1$ at 389 60
\pinlabel* $-1$ at 444 111
\pinlabel* $a$ at 27 91
\pinlabel* $b$ at 229 91
\pinlabel* $c$ at 416 91
\pinlabel* \reflectbox{$a$} at 130 91
\pinlabel* \reflectbox{$b$} at 332 91
\pinlabel* \reflectbox{$c$} at 518 91
\endlabellist 
\includegraphics[height = 35mm]{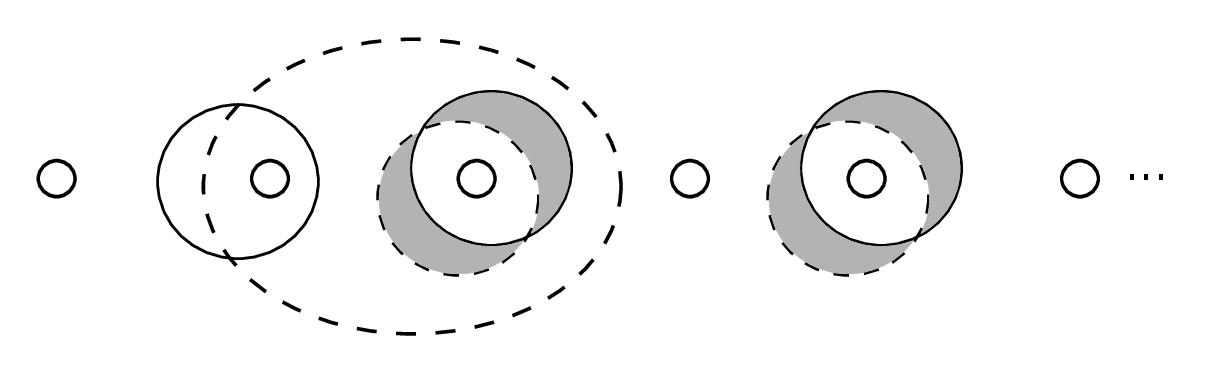}\label{fig:pdhs1}}\\
\subfloat[The handleslide domain $\Dab{1}$.]{
\labellist 
\small
\pinlabel* +1 at 87 89
\pinlabel* -1 at 170 125
\pinlabel* $a$ at 27 91
\pinlabel* $b$ at 229 91
\pinlabel* $c$ at 417 91
\pinlabel* \reflectbox{$a$} at 130 91
\pinlabel* \reflectbox{$b$} at 332 91
\pinlabel* \reflectbox{$c$} at 519 91
\endlabellist 
\includegraphics[height = 35mm]{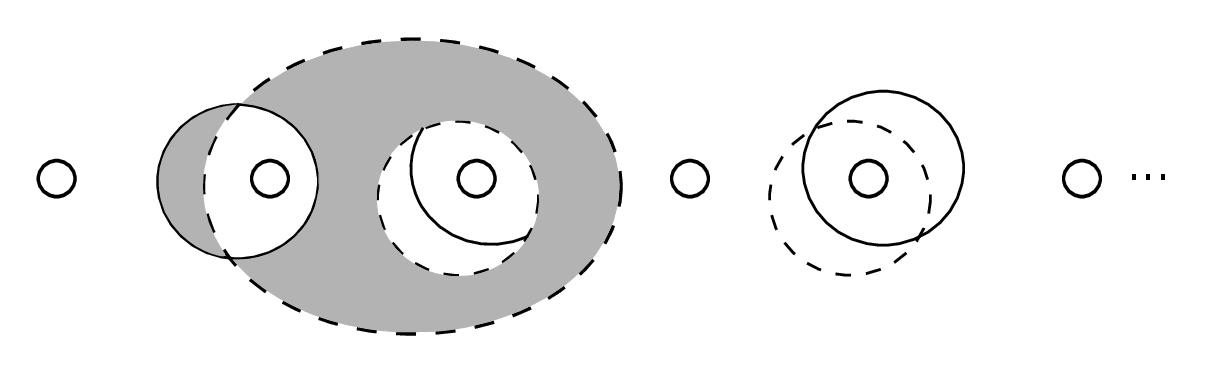}\label{fig:pdhs2}}
\caption[Periodic domains for a handleslide]{The periodic domains $\Dab{i}$ for the handleslide of $\alpha_{1}$ over $\alpha_{2}$.  The $\ba$ circles are solid and the $\bb$ circles are dashed.  The domains of interest are shaded and local coefficients are labelled.}
\label{fig:pdhs}
\end{figure}

The following facts are exercises in linear algebra:

\begin{prop}\label{prop:bi}
Let $\left(\Sigma; \ba; \bb; z\right)$ be a pointed Heegaard diagram of genus $g$ such that $\bb$ is obtained from $\ba$ via a pointed isotopy or pointed handleslide.  Then the set $\{ \Dab{1} , \ldots, \Dab{g}\}$ is a generating set for the group $\Pab$.
\end{prop}

\begin{prop}\label{prop:tri}
Let $\left(\Sigma; \ba; \bb; \bg; z\right)$ be a pointed Heegaard triple-diagram of genus $g$ such that $\bg$ is obtained from $\bb$ via a pointed isotopy or pointed handleslide.  Then the set $\{ \Dbg{1} , \ldots, \Dbg{g}\}$ is a generating set for the group $\Pabg$.
\end{prop}

\begin{prop}\label{prop:quad}
Let $\left(\Sigma; \ba; \bb; \bg; \bd; z\right)$ be a pointed Heegaard quadruple-diagram of genus $g$ such that $\bg$ is obtained from $\bb$ via a pointed isotopy or pointed handleslide, and  $\bd$  is obtained from $\bb$ via a small pointed isotopy.  Then the set $\{ \Dbg{1} , \ldots, \Dbg{g}\} \cup \{ \Dbd{1} , \ldots, \Dbd{g}\}$ is a generating set for the group $\Pabgd$.
\end{prop}

The above facts imply the following useful fact about admissibility of multi-diagrams:

\begin{prop}\label{prop:adm}
Let $\left(\Sigma; \ba; \bb; \bg; \bd; z\right)$ be a pointed Heegaard quadruple-diagram of genus $g$  such that $\bg$ is obtained from $\bb$ via a pointed isotopy or pointed handleslide, and  $\bd$  is obtained from $\bb$ via a small pointed isotopy.  Then if the six pointed diagrams formed by choosing any two tuples out of $\ba$, $\bb$, $\bg$, and $\bd$ are all admissible, so is the quadruple-diagram.  Moreover, each of the four triple-diagrams composed of three of the tuples is also admissible.
\end{prop}

\begin{proof}
Let $\mathcal{D}_0, \mathcal{D}_1, \ldots, \mathcal{D}_N$ denote the connected components of
$$ \Sigma \setminus \left( \cup_{i} \alpha_i \right)\setminus \left( \cup_{i} \beta_i \right)\setminus \left( \cup_{i} \gamma_i \right)\setminus \left( \cup_{i} \delta_i \right),$$
where $\mathcal{D}_0$ is the component containing $z$.
Consider some nontrivial quadruply-periodic domain
$$ \mathcal{P} = \sum_{j = 1}^N c_j \mathcal{D}_j.$$
Then by Proposition \ref{prop:quad}, we can write
\begin{equation}\label{eqn:pd}
\mathcal{P} = \sum_{j = 1}^g  n_j \Dbg{j} + \sum_{j=1}^{g} m_j \Dbd{j}.
\end{equation}
Now at least one of $n_1, \ldots, n_g, m_1, \ldots, m_g$ is nonzero - without loss of generality, let it be $n_1$.  Now the domain $\Dbg{1}$ is the sum of two regions, which have coefficients $+1$ and $-1$, respectively.  Neither can be cancelled by any other terms in the right side of Equation \ref{eqn:pd}, and so there are both positive and negative numbers among the $c_i$.

The argument for triple diagrams is similar, making use of Proposition \ref{prop:tri}.
\end{proof}

\subsection{Filtrations and spectral sequences}\label{sec:filt}

Let $(\Ccal_*, \del)$ be a chain complex generated by $\{ x_{i} \}_{i = 1}^{n}$ and equipped with a filtration grading $f:\{x_{i}\} \rightarrow \mathbb{Z}$.  We can view the filtration as the nested family of subcomplexes $\{ \Fcal_{k} \}_{k \in \mathbb{Z}}$, with
\begin{equation*}
\Fcal_{k} = \text{span}\{ x_{i} : f(x_{i}) \leq k\}.
\end{equation*}

\begin{df}\label{df:fmap}
Let $(\Ccal, \del)$ and $(\Ccal', \del')$ be chain complexes with filtrations $\{ \Fcal_{k} \}$ and $\{ \Fcal'_{k} \}$.
\begin{enumerate}[(a)]
\item A chain map $F: (\Ccal, \del) \rightarrow (\Ccal', \del')$ is called a \textit{filtered chain map} if for all $k$, $F(\Fcal_{k}) \subset \Fcal'_{k}.$
\item Let $H: (\Ccal, \del) \rightarrow (\Ccal', \del')$ be a chain homotopy connecting two maps $F,G:(\Ccal, \del) \rightarrow (\Ccal', \del')$.  We call $H$ a \textit{filtered chain homotopy} if for all $k$, $H(\Fcal_{k}) \subset \Fcal'_{k+1}.$
\item Let $F: (\Ccal, \del) \rightarrow (\Ccal', \del')$ be a chain homotopy equivalence with homotopy inverse map $G: (\Ccal', \del') \rightarrow (\Ccal, \del)$ and associated homotopies $H: (\Ccal, \del) \rightarrow (\Ccal, \del)$ from $G \circ F$ to $id_{\Ccal}$ and $H': (\Ccal', \del') \rightarrow (\Ccal', \del')$ from $F \circ G$ to $id_{\Ccal'}$.  We say that $F$ is a \textit{filtered chain homotopy equivalence} if both $F$ and $G$ are filtered maps and both $H$ and $H'$ are filtered chain homotopies.
\end{enumerate}
\end{df}

For each $i,k \in \mathbb{Z}$, let $\Fcal_k\Ccal_i := \Fcal_k \cap \Ccal_i$.  Now notice that the filtration on $\Ccal_*$ induces a filtration on the homology of $\Ccal_*$ given by
$$ \Fcal_k H_i \left( \Ccal_* \right) :=  \left\{ \alpha \in H_i \left( \Ccal_* \right) \big| \alpha = [x] \text{ for some } x \in \Fcal_k\Ccal_i \ \right\}.$$

One can associate to a filtered complex a spectral sequence, which is defined recursively.  First, for each $p,q \in \mathbb{Z}$, define the \textit{associated graded module} by
$$ E^{0}_{p,q} := \Fcal_p\Ccal_{p+q} / \Fcal_{p-1}\Ccal_{p+q}.$$
The differential $\del$ induces a differential $\del_0: E^0_{p,q} \rightarrow E^0_{p,q-1}$, and we refer to the chain complex $\left( E^0, \del_0  \right)$ as the \textit{$E^0$-page} of the spectral sequence.  The homology of this associated graded complex is denoted by
$$ E^1_{p,q} := H_{p+q} \left( \Fcal_p\Ccal_* / \Fcal_{p-1} \Ccal_*  \right),$$
and $\del$ induces a differential $\del_1: E^1_{p,q} \rightarrow E^1_{p-1,q}$ (yielding the \textit{$E^1$-page} $\left( E^1, \del_1 \right)$).

Continuing this process, one obtains a sequence of chain complexes $\left( E^k, \del_k \right)$ (the \textit{$E^k$-pages}), where $\del_k: E^k_{p,q} \rightarrow E^k_{p-k,q+k-1}$ and
$$ E^k_{p,q} := \frac{\text{Ker} \left( \del_k : E^k_{p,q} \rightarrow E^k_{p-k,q+k-1} \right)}{\text{Im} \left( \del_k:E^r_{p-r,q-r+1} \rightarrow E^k_{p,q} \right)}.$$

Since $\Ccal_*$ was finitely-generated, eventually these pages stabilize and are isomorphic to the homology of $\Ccal_*$.  More precisely, for $k$ sufficiently large,
$$ E^k_{p,q} \cong \Fcal_p H_{p+q} \left( \Ccal_* \right) / \Fcal_{p-1} H_{p+q} \left( \Ccal_* \right) \quad \text{and} \quad \del_k \equiv 0.$$
If $K$ denotes the smallest such $k$ such that the above holds, we say that the spectral sequence \textit{collapses at $E^K$}.

One should notice that the spectral sequence will collapse at $E^1$ if $\del$ preserves the filtration, i.e. if for each $j$,
$$ \del(x_j) = \sum_{f(x_i) = f(x_j)} a_i x_i.$$

%

\section{Braids and the Bigelow picture}

Let $\Bn$ denote the braid group on $2n$ strands.  This group is generated by $\{ \sigma_{1} , \ldots, \sigma_{2n-1}\}$, where $\sigma_{k}$ denotes a half-twist of the $k^{th}$ strand over the $(k+1)^{st}$ strand.  Given a braid $b \in \Bn$, we can obtain a diagram of a knot or a link (the \textit{plat closure} of $b$) by connecting ends of consecutive strands with segments at the top and bottom, as shown in Figure \ref{fig:explat}.

\begin{figure}[h!]
\centering
\begin{minipage}[c]{.47\linewidth}
\centering\includegraphics[height = 24mm]{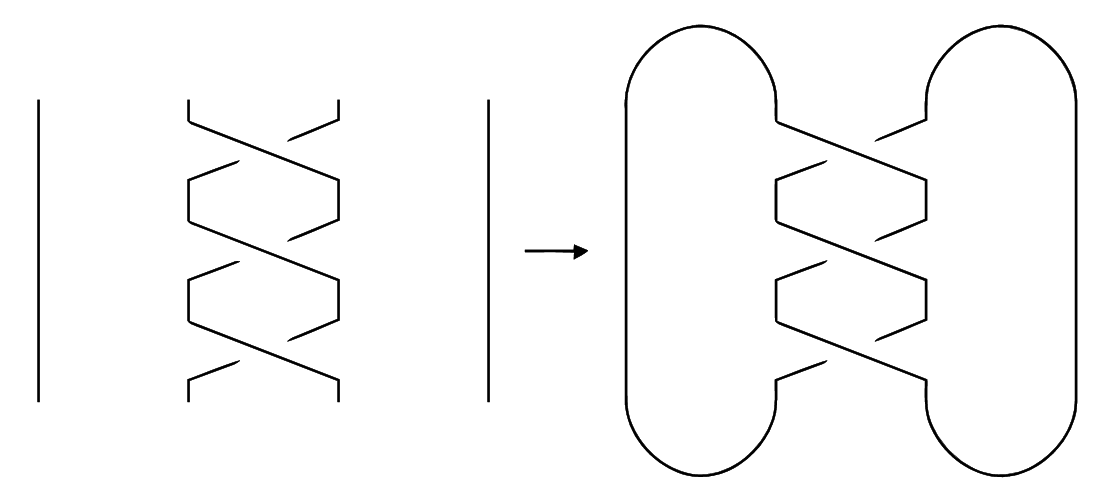}
\caption{The left-handed trefoil}\label{fig:explat}
\end{minipage}
\begin{minipage}[c]{.52\linewidth}
\centering\includegraphics[height = 24mm]{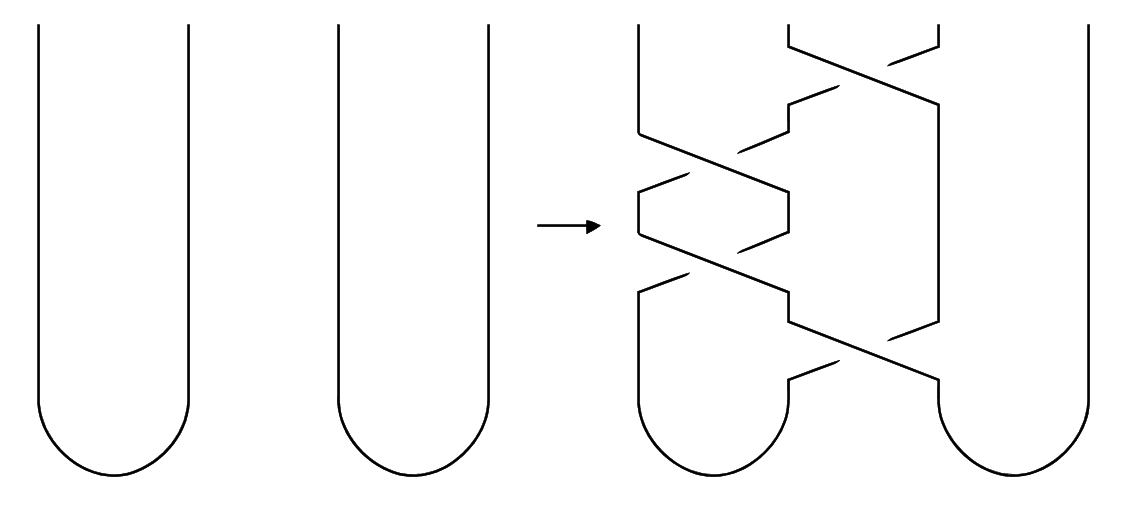}
\caption{The Birman move $b \mapsto b B$}\label{fig:birman}
\end{minipage}
\end{figure}

Any knot $K$ can be presented as the plat closure of an element in  $\Bn$.  Many distinct braid elements can have isotopic plat closures, but such braids are related.

\begin{df}
Let $\Kn$ be the subgroup of the braid group $\Bn$ generated by $A = \gena$, $B = \genb$, and $C_{i} = \genc$ for $i=1,2, \ldots ,2n$.
\end{df}

\begin{thm}[Theorem 1 from \cite{bir:moves}]\label{thm:birman}
Let $b \in \B{2n}$ and $b' \in \B{2n'}$ be two oriented braids.  The braids $b$ and $b'$ have isotopic plat closures if and only if they are related by a finite sequence of the following moves:
\begin{enumerate}[(i)]
  \item $b \mapsto g b h \text{ where } b \in \B{2n} \text{ and } g, h \in \K{2n}$
  \item $b \leftrightarrow \sigma_{2n} b \text{ where } b \in \B{2n} \text{ and } \sigma_{2n} b \in \B{2n+2}$
\end{enumerate}
\end{thm}

\subsection{The Bigelow generators}\label{sec:biggen}
Let $D_{2n} \subset \mathbb{C}$ denote the unit disk with $2n$ punctures $\mu_{1} , \ldots, \mu_{2n}$ evenly spaced along $\mathbb{R} \cap D$.  We can view the braid group $\B{2n}$ as the mapping class group of $D_{2n}$, where the generator $\sigma_{k}$ is a diffeomorphism which is the identity outside of a neighborhood of the $k^{th}$ and $(k+1)^{st}$ punctures and exchanges these two punctures by a counter-clockwise half-twist.  Any braid can be written as a word in the $\sigma_{k}'s$, and we view them as operating on $D_{2n}$ in this way, read from left to right.

Let $b \in \Bn$ be an oriented braid on $2n$ strands.  We'll establish some terminology, following Bigelow in \cite{big:jones}.
\begin{df}\label{defforks}
Let $D \subset \mathbb{C}$ be the unit disk.
\begin{enumerate}[(i)]
\item Let the \textit{standard fork diagram} in $D_{2n}$ be a collection of maps $\alpha_{1}, \ldots ,\alpha_{n}:I \rightarrow D$ and $h_{1}, \ldots , h_{n}:I \rightarrow D$ called \textit{tine edges} and \textit{handles}, respectively, such that the following hold:
\begin{enumerate}[(a)]
  \item The segments $\alpha_{i}|_{(0,1)}$ are disjoint embeddings of $(0,1)$ into $D_{2n}$ such that for each $k$, $\alpha_{k}(0) = \mu_{2k-1}$, $\alpha_{k}(1) = \mu_{2k}$, and $\alpha_{k}(t) \in \mathbb{R}$ for all $t \in I$.
\item The segments $h_{i}|_{(0,1)}$ are disjoint embeddings of $(0,1)$ into $D_{2n}$ such that that for each $k$, $h_{k}(1) = d_{k} \in \partial D$, $h_{k}(0) = m_{k}$ is the midpoint of the segment $\alpha_{k}$, and the segment $h_{k}$ is vertical.
\end{enumerate}
\item Let a \textit{fork diagram for b} be the standard fork diagram along with the compositions $b \circ \alpha_{1}, \ldots, b \circ \alpha_{n}$ and $b \circ h_{1}, \ldots, b \circ h_{n}$.  We'll let $\beta_{k} = b \circ \alpha_{k}$.
\item Let an \textit{augmented fork diagram for b} be obtained from a fork diagram by replacing each arc $\beta_{k}$ with $bE_{k}$, where $E_{k}$ is a figure-eight encircling $\mu_{2k-1}$ and $\mu_{2k}$, where $E_{k}$ is oriented such that it winds counter-clockwise about $\mu_{2k}$.
\end{enumerate}
\end{df}

\begin{figure}[h!]
\centering
\subfloat[Tine edges $\alpha_{k}$]{
\includegraphics[height=30mm]{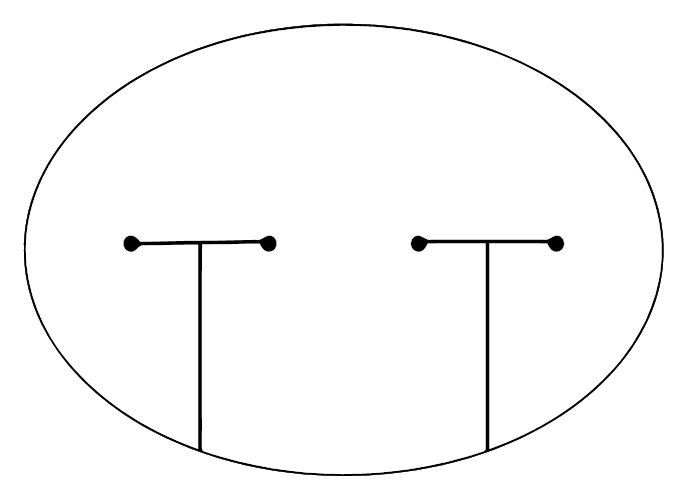}}\quad
\subfloat[Figure-eights $E_{k}$]{
\includegraphics[height=30mm]{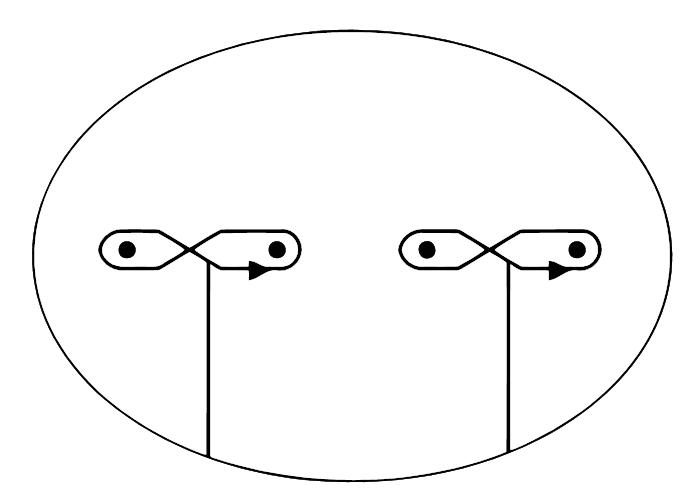}}
\caption{Structures in fork diagrams \label{fig:standardfork}}
\end{figure}

The reader should notice that by drawing a picture containing just the $\alpha$ and $\beta$ arcs and treating the $\alpha$ arcs as undercrossings at each intersection, we get a diagram of the plat closure of $b$.

We'll define some notation.  Let $Conf^{n}(\mathbb{C})$ denote the configuration space of $\mathbb{C}$, i.e. the set of unordered $n$-tuples of distinct points in $\mathbb{C}$.  Let $\Ztil$ be the set of intersections between $\alpha \text{ and } \beta$ arcs.  Then if $\tau$ denotes the set of puncture points, we see that $\tau \subset \Ztil$.  Then we construct a set $\Zcal$ by doubling the points in $\Ztil \setminus\tau$ by introducing for each $x \in \tau$ one element $e_{x} \in \Zcal$ and for each $x \in \Ztil \setminus \tau$ two elements $e_{x},e'_{x} \in \Zcal$.  The set $\Zcal$ can then be seen as the intersections points between $\alpha$ arcs and figure-eights $bE_{k}$.  We distinguish between $e_{x} \text{ and } e'_{x}$ by requiring that the loop traveling along a figure-eight from $e_{x} \text{ to } e'_{x}$ and back to $e_{x}$ along an $\alpha$ arc has winding number +1 around the puncture point.

\begin{rmk} Via an abuse of notation, we'll often refer to the points corresponding to $x \in \Ztil$ as $x \in \Zcal$ (if $x \in \tau$) or $x,x' \in \Zcal$ (if $x \in \Ztil \setminus \tau$).
\end{rmk}

We then define $\Gtil = (\alpha_{1} \times \ldots \times \alpha_{n}) \cap (\beta_{1} \times \ldots \times \beta_{n})\subset \confC$,
the set of unordered $n$-tuples of points in $\Ztil$ such that no two points are on the same $\alpha$ or  $\beta$ arcs.

Similarly, define $\G = (\alpha_{1} \times \ldots \times \alpha_{n}) \cap (bE_{1} \times \ldots \times bE_{n}) \subset \confC$, whch will be referred to as the set of \textit{Bigelow generators} for the diagram.

\begin{rmk}
From this point forward, something of the form $x\by$ will denote an element in $\G$ or $\Gtil$ such that $x \in \Zcal$ or $x \in \Ztil$ is some component of the $n$-tuple and $\by$ is the rest of the $n$-tuple.
\end{rmk}

\subsection{Gradings on the Bigelow generators}\label{sec:gradings}

We will define some gradings $Q,T, P: \G \rightarrow \mathbb{Z}$ based on loops in the configuration space of the disk.  Our definitions of $Q$ and $T$ are identical to Bigelow's in \cite{big:jones}, while our definition for $P$ is adapted from Manolescu's definition for $\tld{P}$ in \cite{cm:R} (in which he used braid closures).

For the sake of concreteness, a sample calculation will accompany the description of the gradings.  We'll study the left-handed trefoil knot depicted as the plat closure of $\sigma_{2}^{3} \in  \B{4}$, as seen in Figure \ref{fig:explat}.

\begin{figure}[h!]
\centering
\begin{minipage}[c]{.36\linewidth}
\centering\includegraphics[height = 35mm]{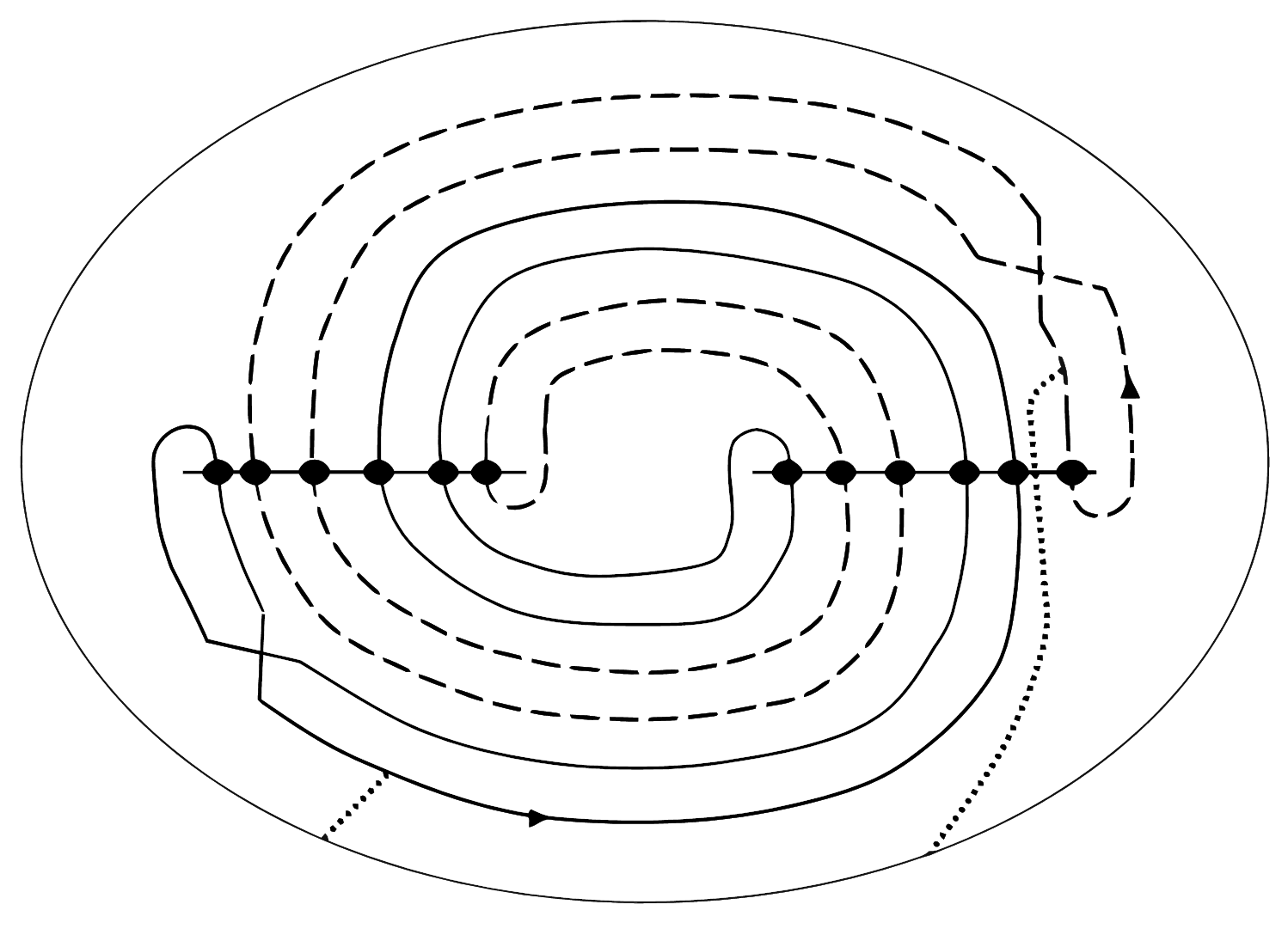}
\end{minipage}
\begin{minipage}[c]{.62\linewidth}
\centering \caption[An augmented fork diagram]{The augmented fork diagram induced by $\sigma_{2}^{3} \in  \B{4}$} \label{fig:exfork}
\end{minipage}
\end{figure}

Figure \ref{fig:exfork} depicts the augmented fork diagram for our example.  Label the elements of $\Zcal$ from left to right in the diagram as $\Zcal = \{x_{1},  s,  s',t, t', x_{2}, x_{3},  u', u, v', v, x_{4}\}$.

One can verify that the set of Bigelow generators is given by
\begin{equation*}
\G = 
\begin{Bmatrix}
 x_{1}x_{4}, & x_{1}u, & x_{1}u', & x_{2}v, & x_{2}v', & x_{2}x_{3}, & sx_{3}, & s'x_{3}, & tx_{4},\\ t'x_{4}, & sv, & s'v, &sv', & s'v', & tu, & t'u, & tu', & t'u'
 \end{Bmatrix}.
\end{equation*}

We'll turn to defining various gradings on the set $\G$.  Grading distributions for our trefoil example can be found in Table \ref{tab:QTP}.  Figure \ref{fig:exQTP} illustrates how to compute the gradings in practice.

\begin{figure}[h!]
\centering
\subfloat[$Q^{*} (x_{3}) = 2$.]{\includegraphics[height = 25mm]{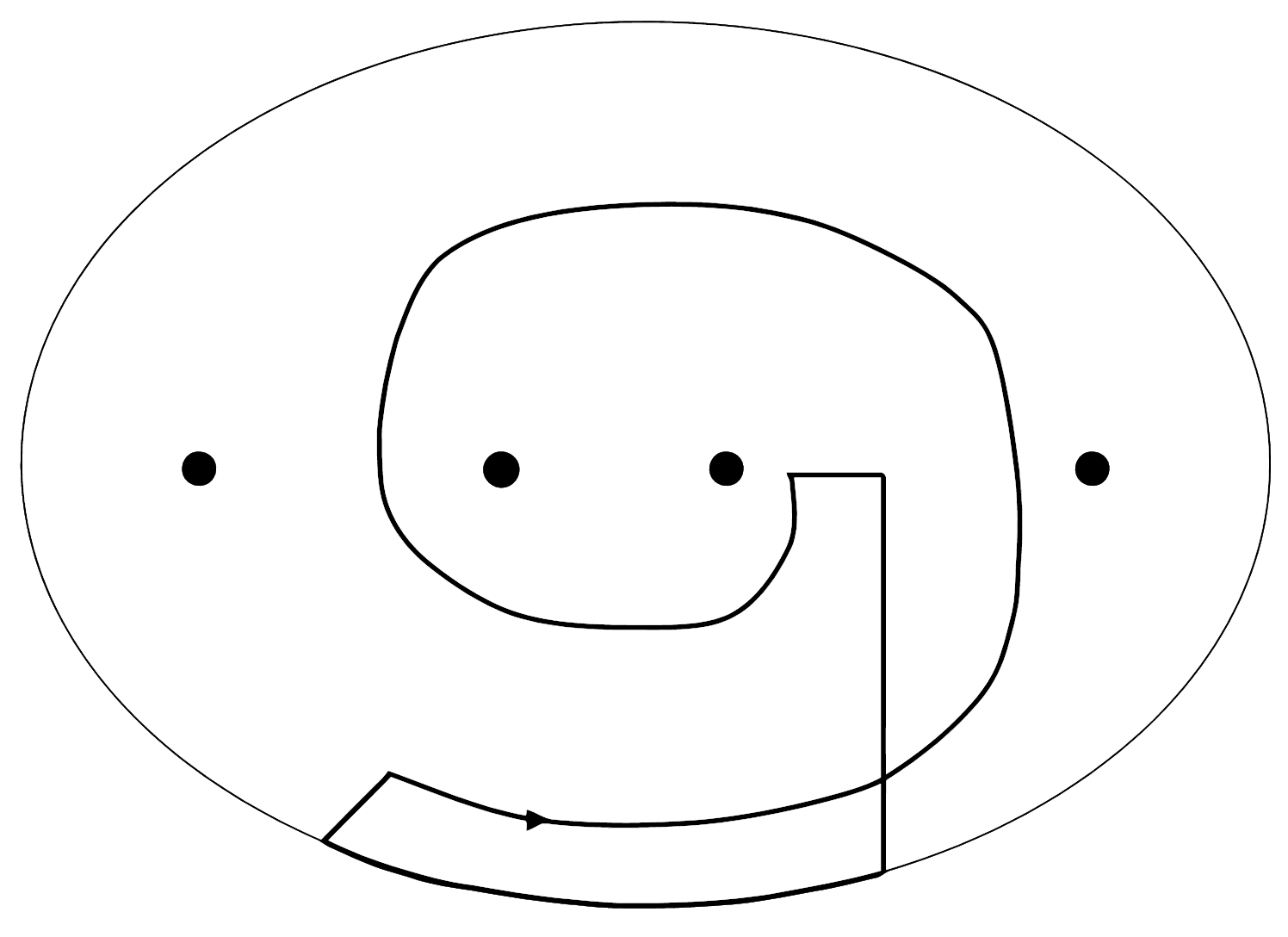}}\quad
\subfloat[$T(x_{2}x_{3}) = 3$.]{\includegraphics[height = 25mm]{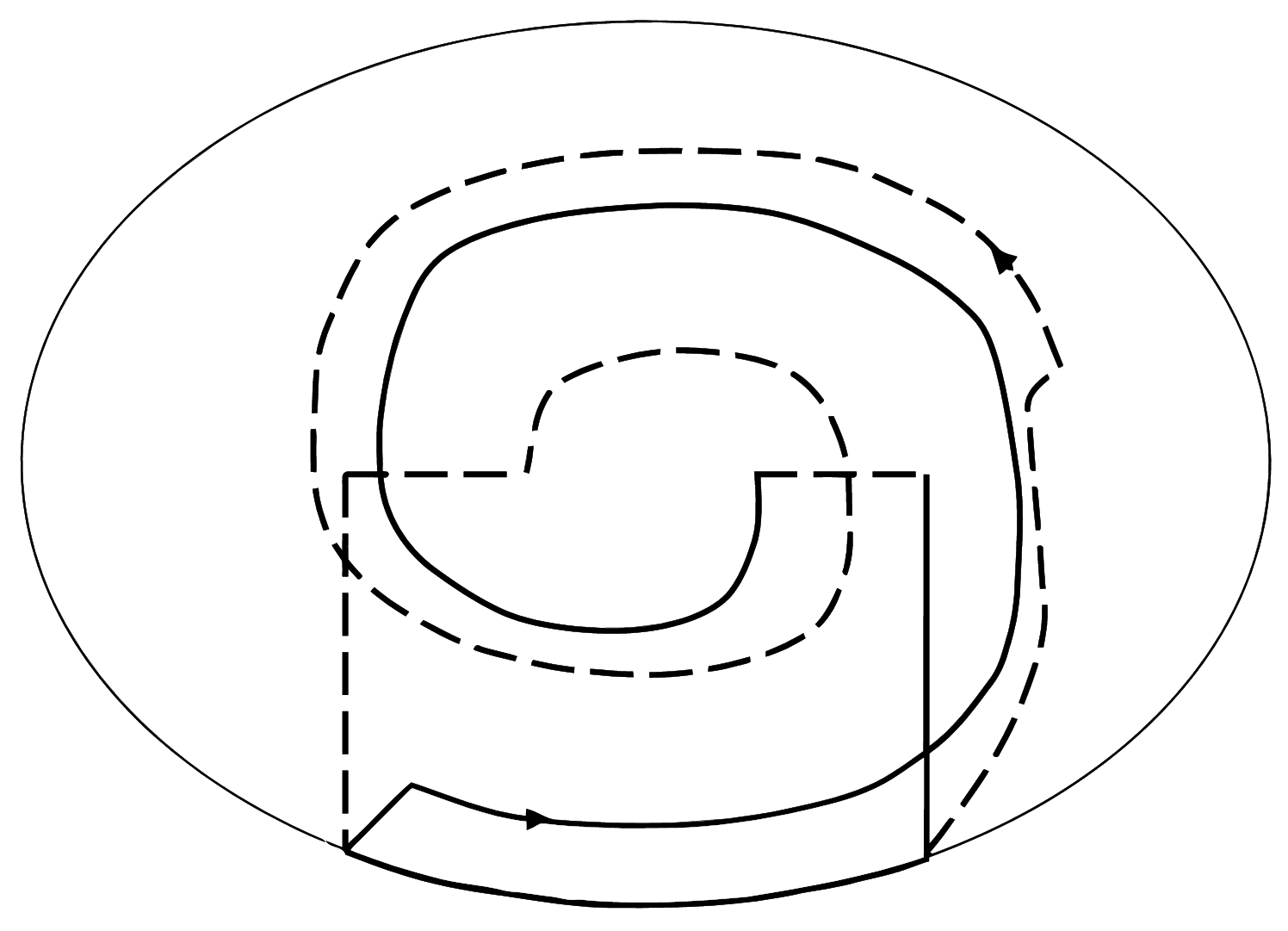}}\quad
\subfloat[$ P^{*}(x_{3}) = 3$ and $ P^{*}(x_{1}) = 0$.]{\includegraphics[height = 25mm]{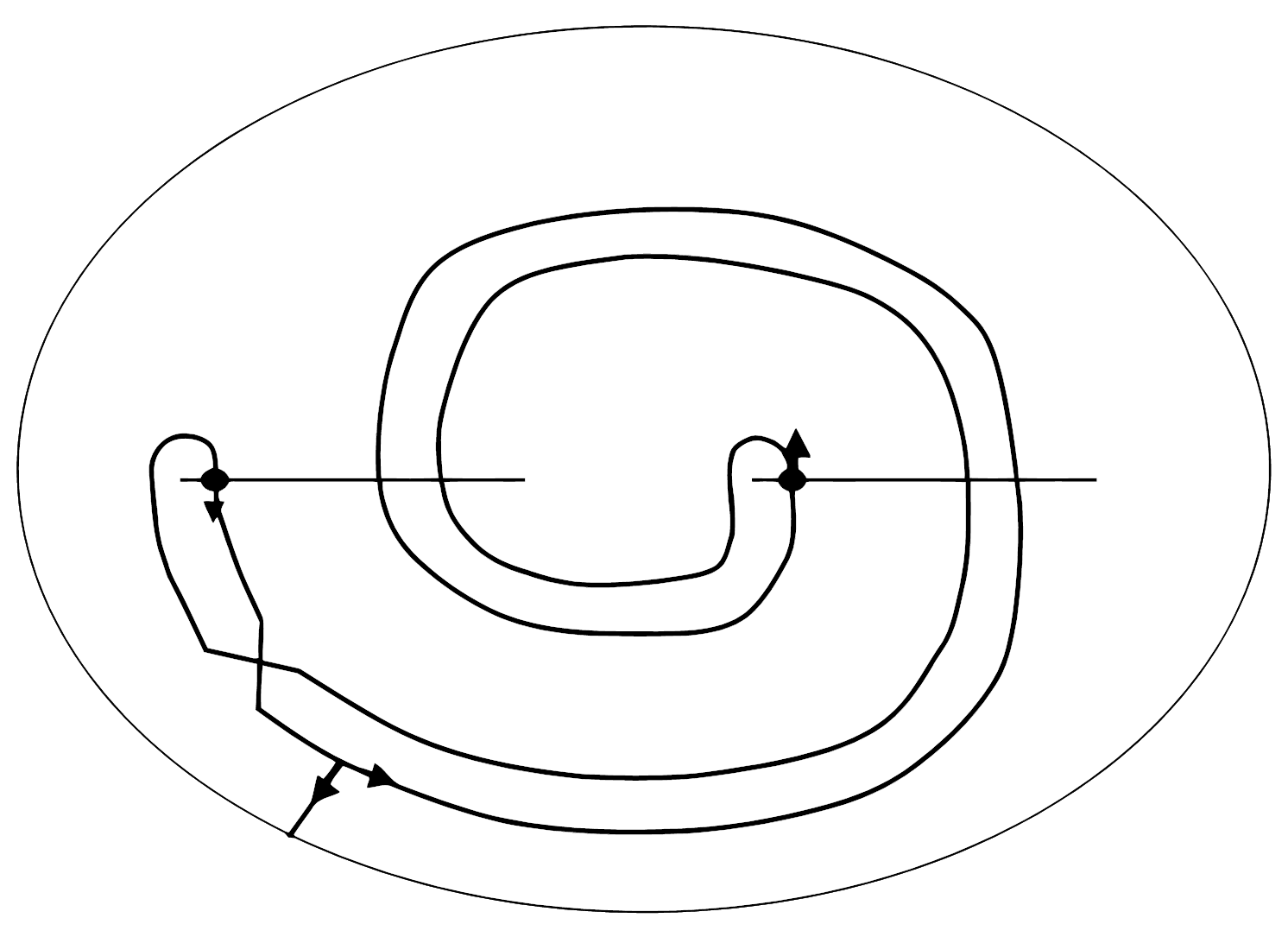}}
\caption{Computing the gradings $Q$, $T$, and $P$ \label{fig:exQTP}}
\end{figure}

\begin{table}[h!]
\centering
\subfloat{
\begin{tabular}{|c|l|}
\hline
\Tspace$T$ \Bspace& $\G$ elements\\
\hline
\hline
0 & $x_{1}x_{4}$, $x_{1}u$, $x_{1}u'$, $tx_{4}$, $t'x_{4}$\\
\hline
1 & $x_{2}v$, $x_{2}v'$, $sx_{3}$, $s'x_{3}$, $sv$, $s'v$, $sv'$, $s'v'$\\
\hline
2 & $tu$, $t'u$, $tu'$, $t'u'$\\
\hline
3 & $x_{2}x_{3}$\\
\hline
\end{tabular}
}
\subfloat{
\begin{tabular}{|c|l|}
\hline
$\Tspace P^{*}\Bspace$ & $\Zcal$ elements\\
\hline
\hline
0 & $x_{1}$, $x_{4}$\\
\hline
1 & $s$, $v$\\
\hline
2 & $s'$, $t$, $u$, $v'$\\
\hline
3 & $t'$, $x_{2}$, $x_{3}$, $u'$\\
\hline
\end{tabular}
}\\
\subfloat{
\begin{tabular}{|c|l|}
\hline
$\Tspace P\Bspace$ & $\G$ elements\\
\hline
\hline
0 & $x_{1} x_{4}$\\
\hline
2 & $x_{1}u$, $tx_{4}$, $sv$\\
\hline
3 & $x_{1}u'$, $t'x_{4}$, $s'v$, $sv'$\\
\hline
4 & $x_{2} v$, $sx_{3}$, $s'v'$, $tu$\\
\hline
5 & $x_{2} v'$, $s'x_{3}$, $t'u$, $tu'$\\
\hline
6 & $x_{2} x_{3}$, $t'u'$\\
\hline
\end{tabular}
}
\subfloat{
\begin{tabular}{|c|l|}
\hline
\Tspace$\text{ }Q^{*}$ \Bspace& $\Zcal$ elements\\
\hline
\hline
0 & $x_{1}$, $v$, $x_{4}$\\
\hline
1 & $v'$ \\
\hline
2 & $x_{3}$, $t$, $u$, $s$\\
\hline
3 & $t'$, $u'$, $s'$\\
\hline
4 & $x_{2}$\\
\hline
\end{tabular}
}
\subfloat{
\begin{tabular}{|c|l|}
\hline
\Tspace$Q$ \Bspace& $\G$ elements\\
\hline
\hline
0 & $x_{1}x_{4}$\\
\hline
2 & $x_{1}u$, $tx_{4}$, $sv$\\
\hline
3 & $x_{1}u'$, $t'x_{4}$, $s'v$, $sv'$\\
\hline
4 & $x_{2}v$, $sx_{3}$, $s'v'$, $tu$\\
\hline
5 & $x_{2}v'$, $s'x_{3}$, $t'u$, $tu'$\\
\hline
6 & $x_{3}x_{3}$, $t'u'$\\
\hline
\end{tabular}
}
\caption{Distributions of $T$, $P^{*}$, $P$, $Q^{*}$, and $Q$ for $\sigma_{2}^{3} \in \B{4}$}
\label{tab:QTP}
\end{table}

\subsubsection{The $Q$ grading}\label{sec:Qdef}

The grading $Q$ on $\G$ will be computed additively from a grading $Q^{*}:\Zcal \rightarrow \mathbb{Z}$.  Consider some $x \in \Zcal$, where $x \in \alpha_{i} \cap bE_{j}$.

Define an arc $\gamma_{x}$ in the disk by starting at $d _{j}$, traveling along $-bh_{j}$ to $bm_{j}$, traveling along $bE_{j}$ to $x$, traveling along $\alpha_{i}$ to $m_{i}$, and traveling along $h_{i}$ to $d_{i}$.  Then let $\gamma_{ij}$ be the arc traveling along the lower portion of $\partial D$ from $d_{i}$ to $d_{j}$.  Then $\gamma_{x}\gamma_{ij}$ is an arc from $d_{j}$ to itself, and we define $Q^{*}(x)$ to be the winding number of this loop around the set of punctures.

Then for each $\be =  e_{1} e_{2} \ldots e_{n}  \in \G$, define
\begin{equation*}
Q(\be) = \sum_{i = 1}^{n} Q^{*}(e_{i}).
\end{equation*}

\subsubsection{The $T$ grading}\label{sec:Tdef}
Given $\be = e_{1} e_{2} \ldots e_{n}  \in \G$, we have that for each $k$, $e_{k} = e_{x_{k}} \text{ or } e_{k}=e'_{x_{k}}$ for some $x_{k} \in \Ztil$.  Now let $\bx = x_{1} x_{2} \ldots x_{n} \in \Gtil$.

Then denote by $\tilde{\gamma}_{x_{k}}$  the arc obtained by replacing the figure-eight segments of $\gamma_{x_{k}}$ with the corresponding $\beta$ arc segments.  Then $T(\bx)$ can be computed as twice the sum of the pairwise winding of the $\tilde{\gamma}_{x_{k}}$ around each other.  In other words, if 
$\tilde{\gamma}_{x_{k}}$ and $\tilde{\gamma}_{x_{m}}$ make a half-twist counter-clockwise around each other for $k \neq m$, this contributes +1 to the value of $T(\bx)$.  Define $T:\G \rightarrow \mathbb{Z}$ by letting $T(\be) = T(\bx)$.

\subsubsection{The $P$ grading}\label{sec:Pdef}
 
This grading will be computed additively from a grading $ P^{*} : \Zcal \rightarrow \mathbb{Z}$, which measures twice the relative winding number of the tangent vectors to the figure eights $E'_{k}$ at the points in $\Zcal$.
 
For $x \in \Zcal$, where $x \in \alpha_{i} \cap bE_{j}$, we define $ P^{*}(x)$ in the following way. We view the arc $bh_{j}$ as being oriented downward at the point where it intersects $\partial D$.  Let $bE_{j}$ have the orientation induced by the orientation on $E_{j}$ in the standard fork diagram.  Then we let $ P^{*}(x)$ be twice the winding number of the tangent vector relative to the downward-pointing tangent vector at the point $h'_{j} \cap \partial D$.  In other words, if the tangent vector makes $k$ counter-clockwise half-revolutions and $m$ clockwise half-revolutions as we travel first along $bh_{j}$ from $bh_{j} (0)$ to $bh_{j} (1)$ then along $bE_{j}$ to $x$, then we set $ P^{*}(x) = m-k$.  This number is an integer because we assume that at any point $x \in \Zcal$, the figure-eight intersects the $\alpha$ arc at a right angle.

Then for $\be = e_{1} e_{2} \ldots e_{n} \in \G$, we define
\begin{equation*}
 P(\be) = \sum_{i = 1}^{n}   P^{*}(e_{i}).
\end{equation*}

\section{The anti-diagonal filtration}\label{sec:AD}

We review here how one obtains from the above picture a filtration on the Heegaard Floer complex, following Manolescu in \cite{cm:R} and Seidel and Smith in \cite{ss:R2}.

We'll first recall in Section \ref{sec:totreal} a formal construction involving graded totally-real submanifolds, as discussed by Manolescu in \cite{cm:R}.  This repeats the construction of graded Lagrangians by Seidel in  \cite{s:GL}, following the ideas of Kontsevich in \cite{kont:GL}.

Then we'll apply the formalism in Section \ref{sec:totreal} to define Seidel gradings on two particular totally real tori in the $n$-fold symmetric product of a Riemann surface $\Sigma$.  It is illustrated in \cite{cm:R} that by taking the Lagrangian Floer cohomology of these tori in the complement of a certain divisor $\AD \subset \text{Sym}^{n}(\Sigma)$, one obtains the fixed-point symplectic Khovanov homology group $\Kst{K}$.  However, Manolescu also showed that these tori can be viewed as Heegaard tori $\Tah, \Tbh$ for the manifold $\DBCs{K}$.  A holomorphic volume form on $W = \text{Sym}^{n}(\Sigma) \setminus \AD$ induces an absolute Maslov grading $\tld{R}$ on intersections of these tori when viewed inside $W$.

Further, we have an identification of the set of Bigelow generators $\G$ with a generating set for the Heegaard Floer chain groups.  This allows us to view $\tld{R}$ as a function on $\G$, and it in fact coincides with $ P - Q + T$.  However, when we view these tori inside all of $\text{Sym}^{n}(\Sigma)$, this grading is no longer a priori consistent with Maslov index calculations (but rather also records intersections of 2-gons with the factor $\AD$).

We can use $R$ (a shifted version of $\tld{R}$) to define a filtration $\rho$ on $\CFxs{K}$ for each torsion $\mathfrak{s} \in \text{Spin}^{c}(\DBCs{K})$.  The definition for $\rho$ will appear to depend heavily on the braid $b$ chosen to represent the knot $K$.  However, we'll obtain an invariance result for this filtration in Section \ref{sec:Rmoves} in the form of Theorem \ref{thm:Rthm}.

\subsection{Graded totally real submanifolds} \label{sec:totreal}

First recall the following definition:

\begin{df}
A real subspace $V \subset \mathbb{C}^{n}$ is called \textit{totally real} (with respect to the standard complex structure if $dim_{\mathbb{R}}V = n$ and $V \cap iV = 0.$  A half-dimensonal submanifold $\mathcal{T}$ of an almost complex manifold $(Y,J)$ is called \textit{totally real} if $T_{x}\mathcal{T} \cap J(T_{x}\mathcal{T}) = 0$ for all $x \in \mathcal{T}$.
\end{df}
We'll first work in the setting of a K\"ahler manifold $(Y,\Omega)$ such that $\Omega$ is exact and $c_{1}(Y) = 0$.  Furthermore, let $\mathcal{T}$ and $\mathcal{T}'$ be two totally real submanifolds of $Y$, intersecting transversely.

Under these conditions, there is a well-defined abelian group $HF^{*}(\mathcal{T}, \mathcal{T}') = H(CF^{*}(\mathcal{T}, \mathcal{T}'),d)$ with a relative $\mathbb{Z}$-grading given by a Maslov index calculation.  However, by a construction of Seidel in \cite{s:GL}, this relative grading can be improved to an absolute $\mathbb{Z}$-grading.

Let $\mathfrak{T} \rightarrow Y$ be the natural fiber bundle whose fibers $\mathfrak{T}_{x}$ are the manifolds of totally real subspaces of $T_{x}Y$.  Choosing a complex volume form $\Theta$ on $Y$ determines a square phase map $\theta: \mathfrak{T} \rightarrow \mathbb{C}^{*} / \mathbb{R}_{+} \cong S^{1}$ defined by $\theta(V) = \Theta(e_{1} \wedge \ldots \wedge e_{n})^{2}$, where $e_{1}, \ldots , e_{n}$ is any orthonormal basis for $V \subset T_{x}Y.$

Let $\tld{\mathfrak{T}} \rightarrow \mathfrak{T}$ be the infinite cyclic covering obtained by pulling back the covering $\mathbb{R} \rightarrow S^{1}$ via the map $\theta$.  Consider the canonical section $s_{\mathcal{T}}: \mathcal{T} \rightarrow \mathfrak{T}$ given by $s_{\mathcal{T}}(x) = T_{x}\mathcal{T}$.  This section induces a $S^{1}$-valued map $\theta_{\mathcal{T}} = \theta \circ s_{\mathcal{T}}$.  In some cases, the section $s_{\mathcal{T}}$ can be lifted to a section $\tld{s}_{\mathcal{T}}: \mathcal{T} \rightarrow \tld{\mathfrak{T}}$ (inducing a lift $\tld{\theta}_{\mathcal{T}} :\mathcal{T} \rightarrow \mathbb{R}$ of the map $\theta_{\mathcal{T}}$).  Let's assume such a lift exists.

\begin{df}
A \textit{grading} on $\mathcal{T}$ is a choice of lift $\tld{\theta}_{\mathcal{T}}: \mathcal{T} \rightarrow \mathbb{R}$.
\end{df}
Given such gradings on the submanifolds $\mathcal{T}$ and $\mathcal{T}'$, one can define the absolute Maslov index $\mu(x) \in \mathbb{Z}$ for each element $x \in \mathcal{T} \cap \mathcal{T}'$ \cite{s:GL}.  This index is constructed using the Maslov index of paths in $\mathfrak{T}_{x}$, which is discussed in \cite{rs:paths}.  We'll sometimes refer to the grading structure on $\mathcal{T}$ as $\tld{\mathcal{T}}$, and we'll refer to the absolutely-graded Lagrangian Floer groups for the graded Lagrangians $\tld{\mathcal{T}}$ and $\tld{\mathcal{T}}'$ as $HF^{*}(\tld{\mathcal{T}}, \tld{\mathcal{T}}').$

If $\Phi: Y \rightarrow Y$ is a symplectic automorphism, let $\Phi^{\mathfrak{T}}: \mathfrak{T} \rightarrow \mathfrak{T}$ denote the map given by  $\Phi^{\mathfrak{T}}(V) = D\Phi(V)$.  We recall the following definition:

\begin{df}
Let $\Phi: Y \rightarrow Y$ be a symplectic automorphism, and suppose that there is a $\mathbb{Z}$-equivariant diffeomorphism $\tld{\Phi}: \tld{\mathfrak{T}} \mapsto \tld{\mathfrak{T}}$ which is a lift of $\Phi^{\mathfrak{T}}$.  Then the pair $(\Phi, \tld{\Phi})$ is called a \textit{graded symplectic automorphism}.
\end{df}
A graded symplectic automorphism $(\Phi, \tld{\Phi})$ acts on a graded Lagrangian submanifold $(L, \tld{L})$ by
$$ (\Phi, \tld{\Phi})(L, \tld{L}) = ( \Phi(L),  \tld{\Phi} \circ \tld{L} \circ \Phi^{-1}).$$
\begin{rmk}
We'll often write $\tld{\Phi}$ to refer to the pair $(\Phi, \tld{\Phi})$ (and thus $\tld{\Phi}(\tld{L})$ will denote $(\Phi, \tld{\Phi})(L, \tld{L})$).
\end{rmk}

As discussed in \cite{s:GL}, many Lagrangian Floer identities can be extended to the absolutely-graded case.  For instance, as absolutely-graded complexes, $CF(\tld{\mathcal{T}}, \tld{\mathcal{T}}') \cong \left(CF(\tld{\mathcal{T}}, \tld{\mathcal{T}}')\right)^{\vee}$, where $``\vee"$ denotes the dual complex.

Further, if $\tld{\Phi}$ is a graded symplectic automorphism, then there is a natural isomorphism of absolutely-graded complexes $CF(\tld{\Phi}(\tld{\mathcal{T}}), \tld{\Phi} (\tld{\mathcal{T}}')) \cong CF(\tld{\mathcal{T}}, \tld{\mathcal{T}}')$.
\subsection{From fork diagrams to Heegaard Floer homology}\label{sec:bighf}

We summarize Manolescu's work in \cite{cm:R}, describing a connection between Bigelow's fork diagram and a Heegaard diagram for the manifold $\DBCs{K}$.

We represent a knot $K$ as the plat closure of a braid $b \in \B{2n}$, the braid group on $2n$ strands, and obtain a fork diagram for $b$ by following the action the braid on the standard fork diagram, as described in Section \ref{sec:biggen}.

Now let $P_{\mu} \in \mathbb{C}[t]$ be a polynomial with set of roots $\{ \mu_{1} , \ldots \mu_{2n}\}$, which is exactly the set of punctures in $\mathbb{C}$.  We define an affine space $\widehat{S} = \{ (u, z) \in \mathbb{C}^{2} : u^{2} + P_{\mu}(z) = 0 \} \subset \mathbb{C}^{2}$.

Also , for $k = 1, \ldots , n$, define the subspaces $\ah_{k}$ and $\bh_{k}$ of $\widehat{S}$ by
\begin{align*}
\ah_{k} &= \{ (u,z) \in \mathbb{C} : z = \alpha_{k}(t),\text{ for some } t\in [0,1]; u = \pm \sqrt{-P_{\mu}(z)}\} \text{ and}\\
\bh_{k} &= \{ (u,z) \in \mathbb{C} : z = \beta_{k}(t),\text{ for some } t\in [0,1]; u = \pm \sqrt{-P_{\mu}(z)} \}.
\end{align*}

Notice that the map $\widehat{S} \rightarrow \mathbb{C}$  defined by $(u,z) \mapsto z$ is a double branched covering with branch set equal to $\{\mu_{1}, \ldots, \mu_{2n}\} \subset \mathbb{C}$.  This means that $\widehat{S}$ can be seen as $\Sigma_{n-1} - \{ \pm \infty \}$, where $\Sigma_{n-1}$ is a Riemann surface of genus $(n-1)$.  Furthermore, the $\ah_{k}$ and $\bh_{k}$ are simple closes curves in $\widehat{S}$ which induce totally real tori $\Tah = \ah_{1} \times \ldots \times \ah_{n}, \Tbh = \bh_{1}\times\ldots\times\bh_{n} \subset \text{Sym}^{n}\left( \widehat{S} \right)$.  We want a Heegaard diagram, so we stabilize this surface as shown in Figure \ref{fig:HA} to acquire $\Sigma_{n} - \{ \pm \infty \}$.
 
\begin{figure}[h]
\centering
\labellist 
\small
\pinlabel* {$+\infty$} at 475 115
\pinlabel* {$-\infty$} at 475 50
\endlabellist 
\includegraphics[height = 25mm]{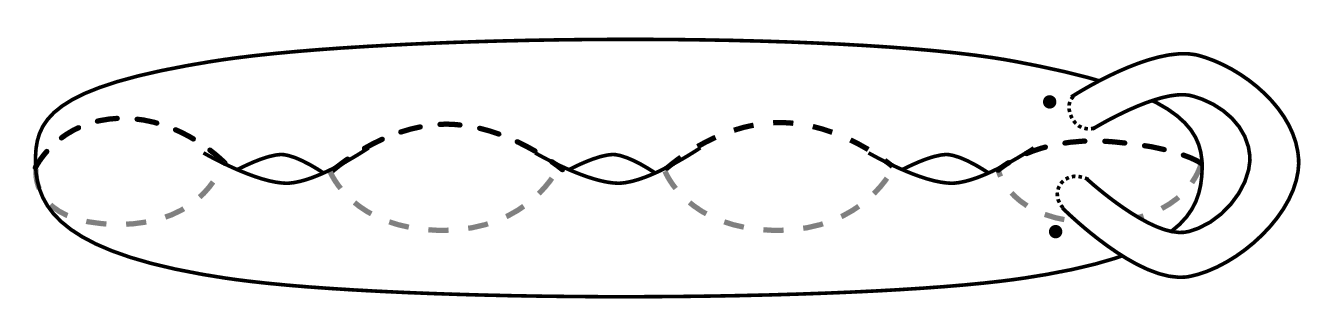}
\caption{Stabilizing near $\pm \infty$
\label{fig:HA}}
\end{figure}
 
\begin{prop}[Proposition 7.4 from \cite{cm:R}]\label{prop:DBC}
The collection of data $\h = (\Sigma_{n}; \ah_{1}, \ldots , \ah_{n} ; \bh_{1}, \ldots , \bh_{n} ; +\infty  )$ is an admissible pointed Heegaard diagram for $\DBCs{K}$.
\end{prop}

Now notice that with respect to the covering map above, each puncture $\mu_{k} \in \mathbb{C}$ has a single point as its preimage.  However, the preimage of a point $x \in \beta_{j} \cap \text{int}(\alpha_{i})$ consists of a pair of points upstairs.  This gives a bijection between the intersection $\Tah \cap \Tbh$ and the set $\G$ of Bigelow generators as defined in Section \ref{sec:biggen}.  However, this identification isn't canonical, since for some $x \in \Ztil - \tau$ it is only required that the pair $\{ e_{x}, e'_{x} \}$ is identified with the two preimages of $x$ upstairs.  In any case, the grading function $R$ defined below will satisfy $R( e_{x}) = R( e'_{x})$.

\subsection{A grading induced by a volume form}

Define a subset $W = \text{Sym}^{n}(\widehat{S}) - \nabla$, where the anti-diagonal $\nabla$ is defined by
\begin{equation*}
\nabla = \{(u_{k} , z_{k} ), k = 1, \ldots, n : u_{k}^{2} + P_{\mu}(z_{k} ) = 0 , (u_{i} , z_{i} ) = (-u_{j} , z_{j} )\text{ for some }i\neq j \}
\end{equation*}

When we restrict to $W$, the Maslov grading on $\Tah \cap \Tbh$ can be lifted to an absolute Maslov $\mathbb{Z}$-grading by endowing the tori $\Tah$, $\Tbh$ with gradings in the sense of Section \ref{sec:totreal} via the choice of a particular holomorphic volume form.

\begin{prop}[Proposition 7.5 from \cite{cm:R}]\label{prop:cmR}
There exists a complex volume form $\Theta$ on $W$ so that we can endow $\Tah$ and $\Tbh$ with gradings on the sense of Section \ref{sec:totreal}.  The resulting absolute Maslov grading (in W) on the elements of $\Tah \cap \Tbh$ is $ P - Q + T$.
\end{prop}

\begin{proof}
One can describe points in $\text{Sym}^n(\widehat{S})$ by their coordinates $\{ (u_j,z_j) \}_{j}$, where $u_j^2 + P_{\mu}(z_j) = 0$.  Following Manolescu, we let the $\mathbb{C}$-valued $n$-form $\Theta$ on $\text{Sym}^n(\widehat{S})$ be given by
$$ \Theta = \prod_{1 \leq i < j \leq n} \left( z_i - z_j \right)\cdot  \prod_{k = 1}^n \frac{dz_k}{u_k}.$$
By using a basis of symmetric functions in the $z_j$ near any point on the diagonal $\Delta \subset \text{Sym}^n(\widehat{S})$, Manolescu shows that $\Theta$ in fact gives a well-defined volume form on $W:=\text{Sym}^n(\widehat{S}) \setminus \nabla$.

As described in Section \ref{sec:totreal}, one can obtain from $\Theta$ two functions $\theta_{\alpha}:\Tah \rightarrow S^1$ and $\theta_{\beta}:\Tbh \rightarrow S^1$.  A point $\bx \in \Tbh$ has coordinates $\{ (u_j, z_j)\}_j$ where $z_j = \beta_j (t_j)$ for some $t_j \in [0,1]$, and $u_j = \pm \sqrt{-P_{\mu}(\beta_j(t_j))}$.  So, we can write
$$ \theta_{\beta}(\bx) = \prod_{1 \leq i < j \leq n} \left( \beta_i(t_i) - \beta_j(t_j) \right)^2\cdot  \prod_{k = 1}^n \frac{\beta_j'(t_j)^2}{-P_{\mu}(\beta_j(t_j))},$$
and write $\theta_{\alpha}(\bx)$ similarly for $\bx \in \Tah$.

A choice of $\mathbb{R}$-valued lifts $\tat$ and $\tbt$ of $\ta$ $\tb$ will induce an absolute Maslov grading on $\Tah \cap \Tbh \subset W$.  Notice that $\ta$ has a constant value of $1 \in S^1$; it is shown in \cite{cm:R} (by examination of the function $\tb$) that any choice of lifts $\tat$, $\tbt$ will induce a Maslov grading which agrees with $P-Q+T$ upto an overall shift, and the same argument applies here.

In \cite{cm:R}, the absolute Maslov grading is fixed to be exactly $P - Q + T$ by choosing $\tbt$ to be obtained continuously from $\tat$ by following the family of crossingless matchings (induced by the braid action) starting at $\{ \alpha_1, \ldots, \alpha_n \}$ and ending at $\{ \beta_1, \ldots, \beta_n \}$; this effectively sets $\tld{R}(\bx_0) = 0$ for a distinguished generator $\bx_0 \in \Tah \cap \Tbh$.  In our case, there is no such distinguished generator.  Instead, we set $\tat \equiv 0$ and then choose the lift $\tbt$ in a way that the induced Maslov grading $\tld{R}$ satisfies $\tld{R}(\bx) = (P - Q + T)(\bx)$ for some choice of generator $\bx \in \Tah \cap \Tbh$.  Necessarily we'll then have that $\tld{R} = P - Q + T$.


\end{proof}

We can now view $\tld{R}$ as a function both on $\Tah \cap \Tbh$ and on the set $\G$ of Bigelow generators.  Now we define a rational number $s_{R}(b,D)$, which will depends on properties of the braid $b$ and of the oriented link diagram $D$ which is its plat closure. Denote by $\epsilon(b)$ the sum of the powers (with sign) of the braid group generators making up the word $b$, and denote by $w$ the writhe of the diagram $D$.  Then let
\begin{equation*}
s_{R}(b,D) = \frac{\epsilon(b) - w(D) - 2n}{4} \in \mathbb{Q}.
\end{equation*}
Then for $\be \in \G$, define $R(\be) = \tld{R}(\be) + s_{R} =  P(\be) - Q(\be) + T(\be) + s_{R}$.

One should notice that for any $x \in \Ztil \setminus\tau$, we have that $Q^{*}(e'_{x}) = Q^{*}(e_{x}) + 1$ and $ P^{*}(e'_{x}) =  P^{*}(e_{x}) + 1$.  As a result, $R(e'_{x}\by) = R(e_{x}\by)$; following \cite{cm:R}, we say that the grading $R$ is \emph{stable}.

\subsection{Computing R for the left-handed trefoil}\label{sec:Rdef}

Here we have that $n=2$, $\epsilon = 3$, and $w = -3$, and so
\begin{equation*}
s_{R}(b,D) = \frac{(3) - (-3) - 2(2)}{4} = \frac{1}{2}.
\end{equation*}
Combining this with Table \ref{tab:QTP}, one obtains the distributions of $\tld{R}$ and $R$ seen in Table \ref{tab:R}.

\begin{table}[h!]
\centering
\subfloat{
\begin{tabular}{|  c  |  p{4.5cm}  |}
\hline
$\Tspace \tld{R} \Bspace$ & $\G$ elements\\
\hline
\hline
0 & $x_{1} x_{4}$, $x_{1} u$, $x_{1} u'$, $t x_{4}$, $t' x_{4}$\\
\hline
1 & $x_{2} v$, $x_{2} v'$, $s x_{3}$, $s' x_{3}$, $sv$, $s'v$, $sv'$, $s'v'$\\
\hline
2 & $tu$, $t'u$, $tu'$, $t'u'$\\
\hline
3 & $x_{2} x_{3}$\\
\hline
\end{tabular}
}
\qquad
\subfloat{
\begin{tabular}{| c  |  p{4.5cm} |}
\hline
$\Tspace R \Bspace$ & $\G$ elements\\
\hline
\hline
$1/2$ & $x_{1} x_{4}$, $x_{1} u$, $x_{1} u'$, $t x_{4}$, $t' x_{4}$\\
\hline
$3/2$ & $x_{2} v$, $x_{2} v'$, $s x_{3}$, $s' x_{3}$, $sv$, $s'v$, $sv'$, $s'v'$\\
\hline
$5/2$ & $tu$, $t'u$, $tu'$, $t'u'$\\
\hline
$7/2$ & $x_{2} x_{3}$\\
\hline
\end{tabular}
}
\caption{Distributions of $\tld{R}$ and $R$ for $\sigma_{2}^{3} \in \B{4}$}
\label{tab:R}
\end{table}

\subsubsection{Drawing Heegaard diagrams}\label{sec:draw}

Given a fork diagram, it is straightforward (albeit sometimes tedious) to construct the admissible pointed Heegaard diagram discussed in Proposition \ref{prop:DBC}.  Figure \ref{fig:forkcut} shows a standard fork diagram with six punctures (where the handle arcs are omitted).  Cutting along the dashed arcs produces three disks, each with two punctures.  The double cover of each such disk branched over the punctures is an annulus, as shown in Figure \ref{fig:hdcut}.  One can reglue the annuli to form a genus-two surface with two boundary components, as shown in Figure \ref{fig:hdglue}.  Capping off the boundary components and stabilizing the surface with a handle whose feet lie near $\pm \infty$ yields the required pointed Heegaard diagram.

\begin{figure}[h!]
\centering
\subfloat[A fork diagram cut along dashed arcs]{
\includegraphics[height = 30mm]{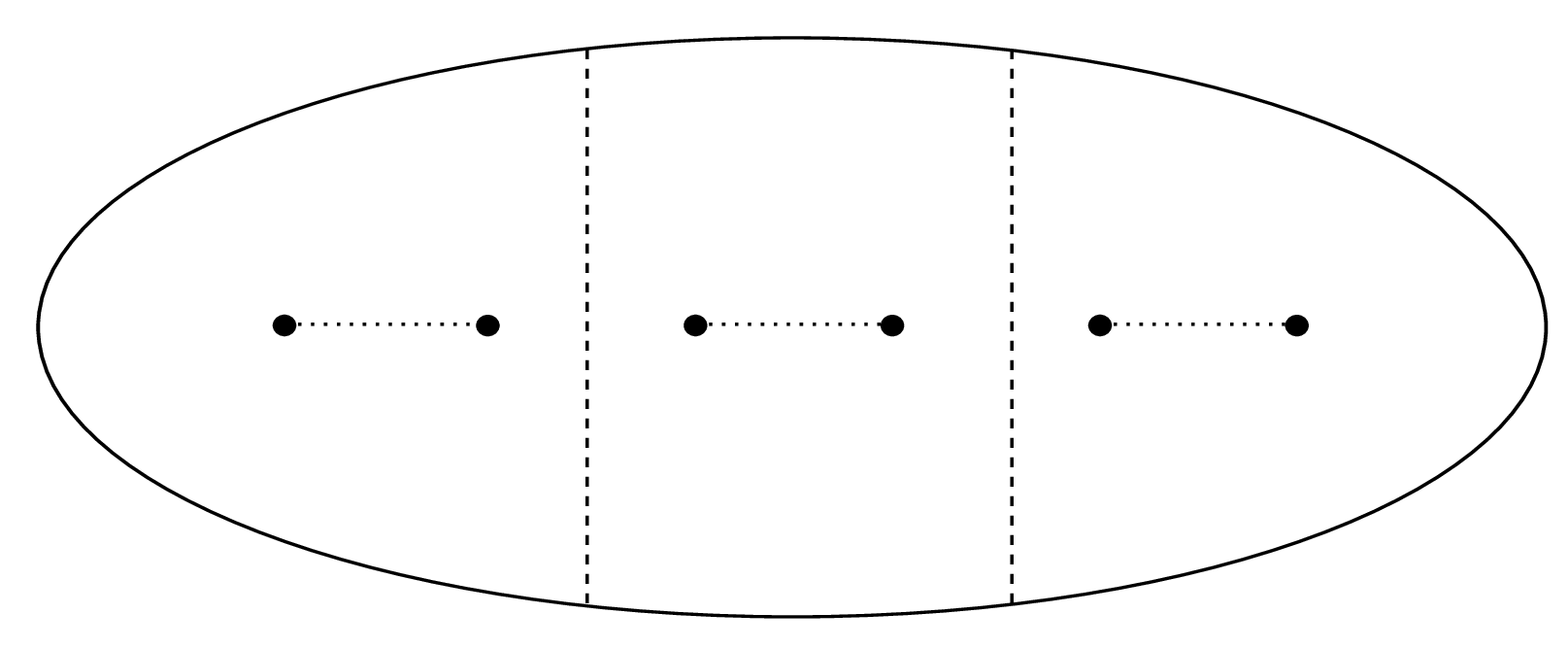}\label{fig:forkcut}}\\
\subfloat[The annuli covering the three pieces from Figure \ref{fig:forkcut}]{
\includegraphics[height = 30mm]{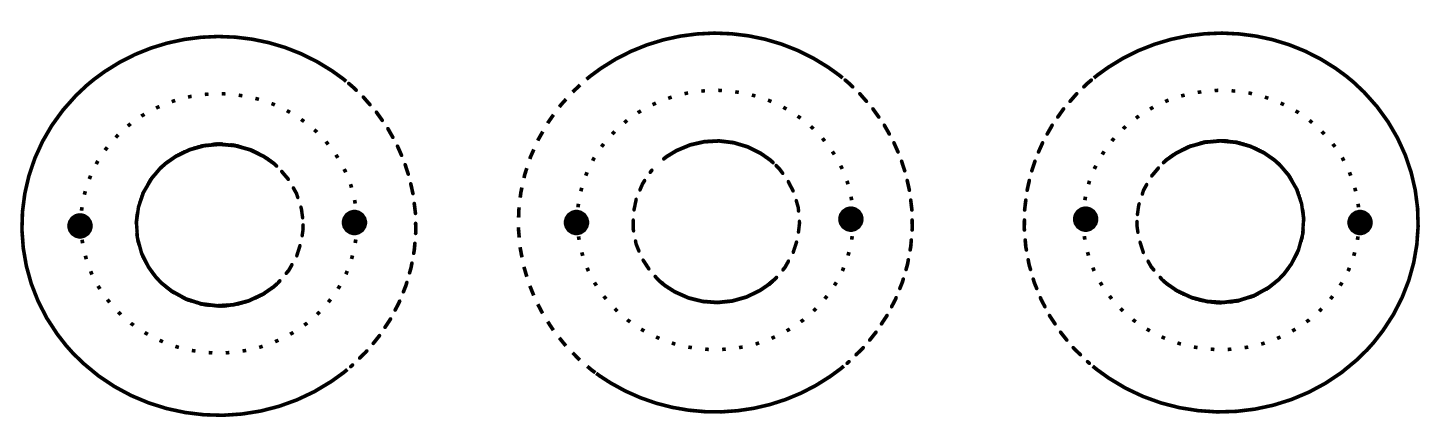}\label{fig:hdcut}}\\
\subfloat[The glued-up genus-two surface]{
\labellist 
\small
\pinlabel* $a$ at 310 220
\pinlabel* \rotatebox{180}{\reflectbox{$a$}} at 310 95
\pinlabel* $b$ at 490 220
\pinlabel* \rotatebox{180}{\reflectbox{$b$}} at 490 95
\endlabellist 
\includegraphics[height = 30mm]{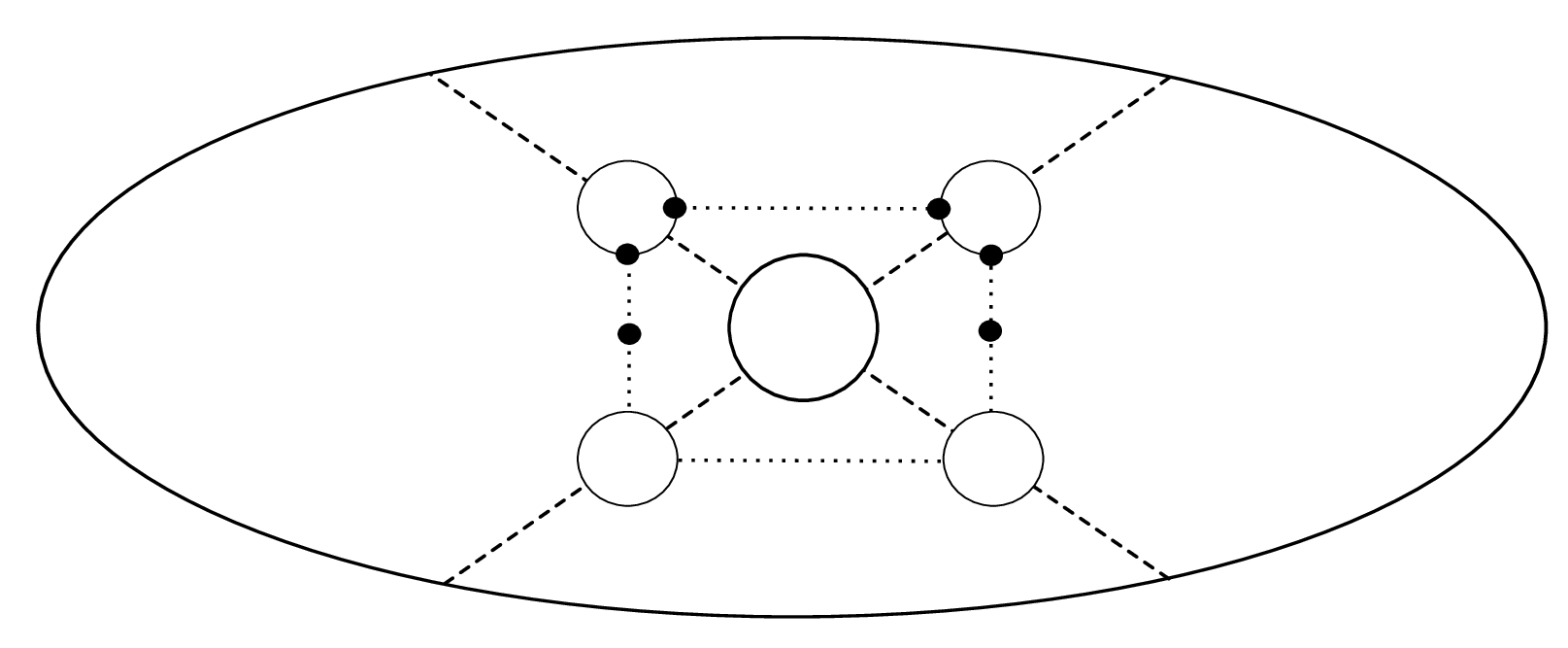}\label{fig:hdglue}}
\caption[Producing a Heegaard diagram]{Producing a pointed Heegaard diagram from a fork diagram.  The dotted arcs (resp. circles) represent the $\alpha_i$ (resp. their covers $\ah_i$) and black dots represent punctures.}
\end{figure}

\begin{figure}[h!]
\centering
\subfloat[A region of a fork diagram cut along dashed arcs]{
\includegraphics[height = 30mm]{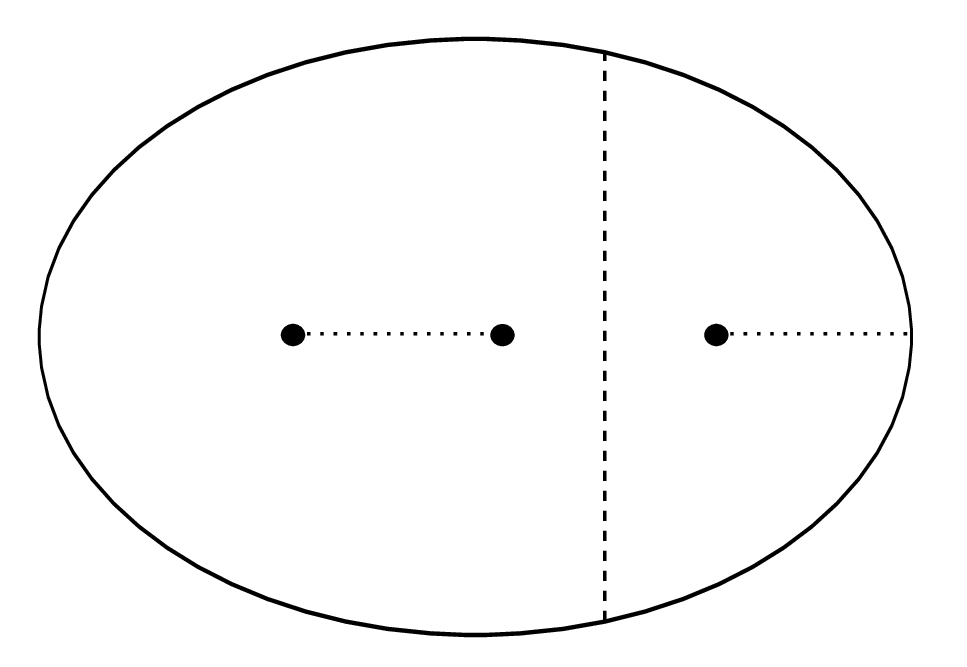}\label{fig:forkcut2}} \quad
\subfloat[The annulus and disk covering the two pieces from Figure \ref{fig:forkcut2}]{
\includegraphics[height = 30mm]{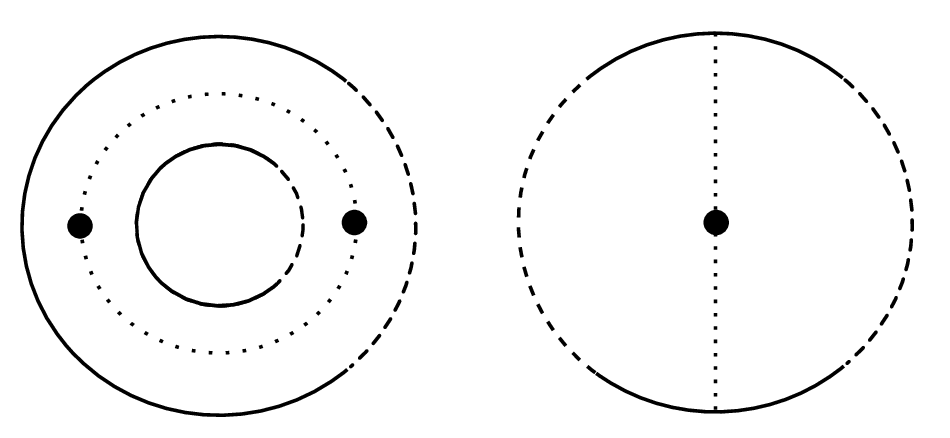}\label{fig:hdcut2}}\\
\subfloat[The glued-up genus-one surface]{
\labellist 
\small
\pinlabel* $a$ at 180 160
\pinlabel* \reflectbox{$a$} at 340 160
\endlabellist 
\includegraphics[height = 30mm]{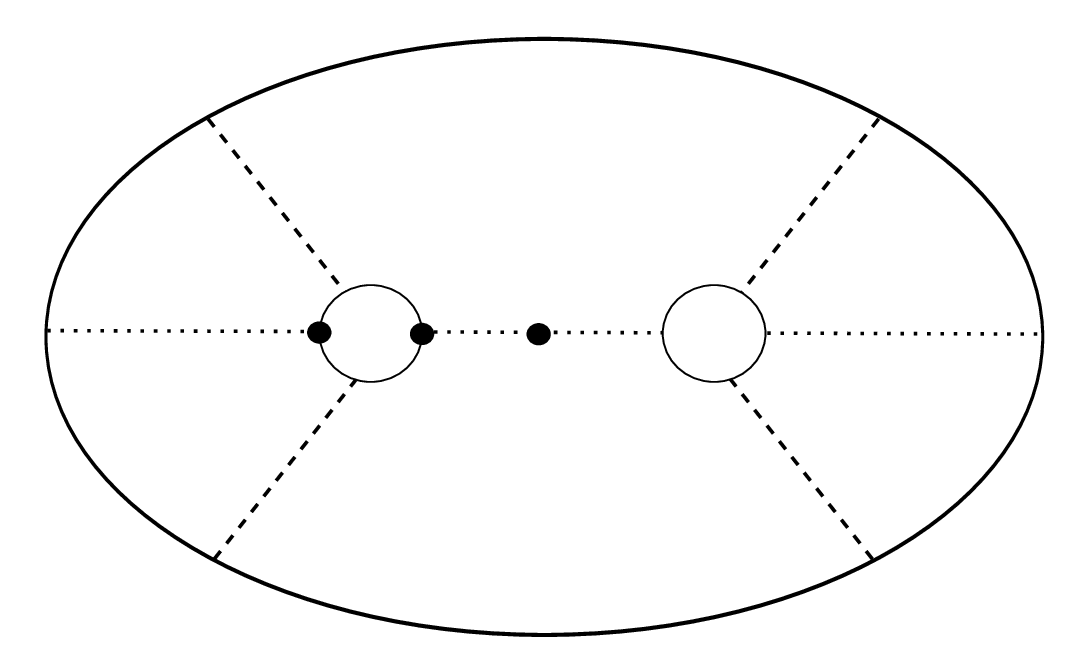}\label{fig:hdglue2}}
\caption[Producing a Heegaard diagram]{A local region of a pointed Heegaard diagram covering a thrice-punctured region of a fork diagam.}
\end{figure}

During the invariance proof, we'll exhibit local pictures of Heegaard diagrams covering local pictures of fork diagrams with three punctures.  In this case, one should cut the fork diagram into two disks (one with two punctures and one with one), as in Figure \ref{fig:forkcut2}.   The branched covers of these pieces are an annulus and a disk, respectively; gluing yields a genus-one surface with one boundary component, as seen in Figure \ref{fig:hdglue2}

\subsection{Intersections with the anti-diagonal}

However, as observed in \cite{ss:R2}, the volume form $\Theta$ has an order-one zero along the antidiagonal $\AD$.  Therefore, $R$ isn't compatible with Maslov index counts in all of $\text{Sym}^{n}(\widehat{S})$.  
%

Let $\phi \in \pi_{2}(\bx , \by)$ be counted by a term in $\widehat{\partial}(\bx)$.  If $\phi$ intersects $\AD$ with multiplicity $k \in \mathbb{Z}$ (it can be arranged that $k \geq 0$, with equality only if $\phi$ completely avoids $\AD$), then \cite{ss:R2} gives that $R(\bx) - R(\by) = 2k + 1$.

More generally, one can say that if $\phi \in \pi_{2}(\bx , \by)$ with $n_{+\infty}(\phi) = 0$, then
\begin{equation*}
R(\bx) - R(\by) = \mu(\phi) + 2\left( [\phi]\cdot[\AD] \right)= gr(\bx, \by) + 2\left( [\phi]\cdot[\AD]\right).
\end{equation*}

Now for each torsion $\mathfrak{s} \in \text{Spin}^{c}(\DBCs{K})$ define $\rho: \Us \rightarrow \mathbb{Q}$ by $\rho(\bx) = R(\bx) - \tld{gr}(\bx)$.
Then we have that if $\bx, \by \in \Us$ for $\mathfrak{s}$ torsion and $\phi \in \pi_{2}(\bx , \by)$ with $n_{+\infty}(\phi) = 0$, $\rho(\bx) - \rho(\by) = 2[\phi]\cdot[\AD]$.

Now $\rho$ provides a filtration grading on the factor $\widehat{CF}(\h, \mathfrak{s})$ for each torsion $\mathfrak{s}$.

\subsubsection{A schematic example of non-trivial intersection}

For the sake of concreteness, let's see an example of a 2-gon whose intersection number with the anti-diagonal is nonzero.  Figure \ref{fig:ADfork} shows a portion of a fork diagram induced by some braid in $\B{6}$.  Let $\bx, \by \in \G$ be the Bigelow generators whose components are indicated in Figure \ref{fig:ADfork}.

\begin{figure}[h!]
\centering
\labellist
\small
\pinlabel* $x_{1}$ at 75 125
\pinlabel* $y_{1}$ at 125 125
\pinlabel* $x_{2}=y_{2}$ at 400 110
\pinlabel* $y_{3}$ at 485 165
\pinlabel* $x_{3}$ at 535 165
\endlabellist
\includegraphics[height = 35mm]{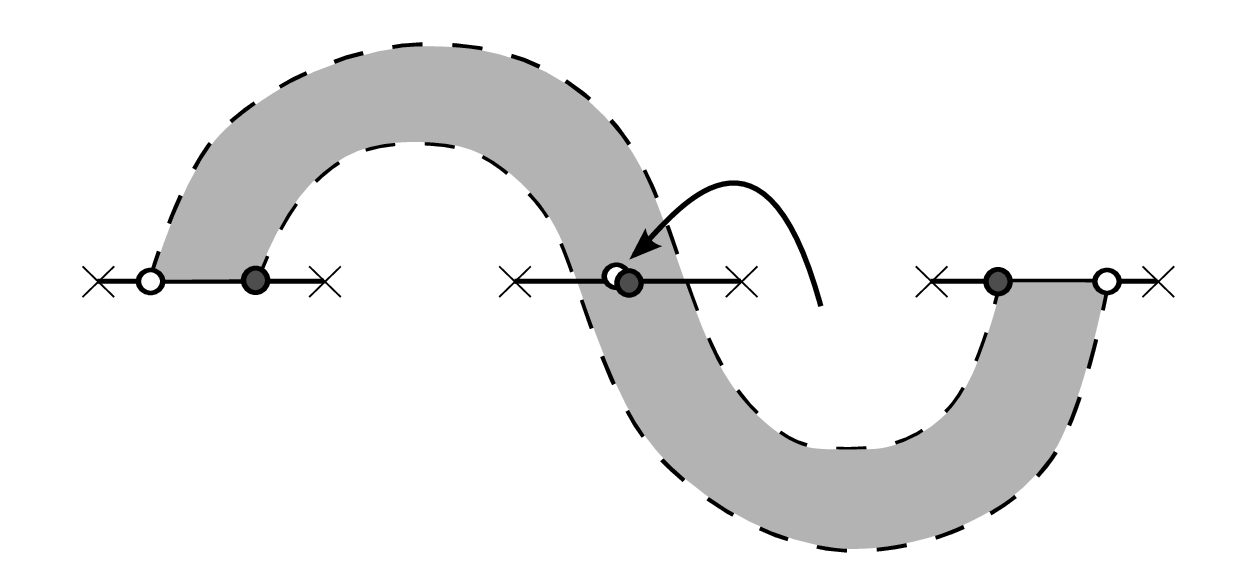}
\caption[A fork diagram depicting non-trivial intersection with $\AD$]{A portion of a fork diagram.  The domain $\pi(\mathcal{D}(\phi))$ is shaded, the $\ba$ arcs are solid, and the $\bb$ arcs are dashed.\label{fig:ADfork}}
\end{figure}
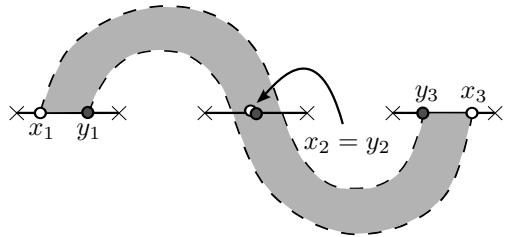

Figure \ref{fig:ADHD} shows the Heegaard diagram of genus 3 obtained from the fork diagram via Theorem \ref{prop:DBC}, and let $\pi$ denote the branched covering map.  Let $\bxh, \byh \in \Tah \cap \Tbh$ have components as indicated in the Heegaard diagram, where $\pi(\widehat{x_{i}}) = x_{i}$ and $\pi(\widehat{y_{i}}) = y_{i}$ for $i = 1, 2, 3$.  The shaded region in Figure \ref{fig:ADHD} is the domain $\mathcal{D}(\phi)$ of a 2-gon $\phi \in \pi_{2}(\bxh, \byh)$ and the shaded region in Figure \ref{fig:ADfork} is its image $\pi(\mathcal{D}(\phi))$.

\begin{figure}[h!]
\centering
\labellist 
\small
\pinlabel* $a$ at 112 144
\pinlabel* \reflectbox{$a$} at 223 144
\pinlabel* $b$ at 393 144
\pinlabel* \reflectbox{$b$} at 505 144
\pinlabel* $c$ at 673 144
\pinlabel* \reflectbox{$c$} at 785 144
\pinlabel* $\widehat{y}_{1}$ at 145 170
\pinlabel* $\widehat{x}_{1}$ at 200 120
\pinlabel* $\widehat{x}_{2}=\widehat{y}_{2}$ at  535 60
\pinlabel* $\widehat{y}_{3}$ at 420 260
\pinlabel* $\widehat{x}_{3}$ at 500 255
\endlabellist 
\includegraphics[height = 45mm]{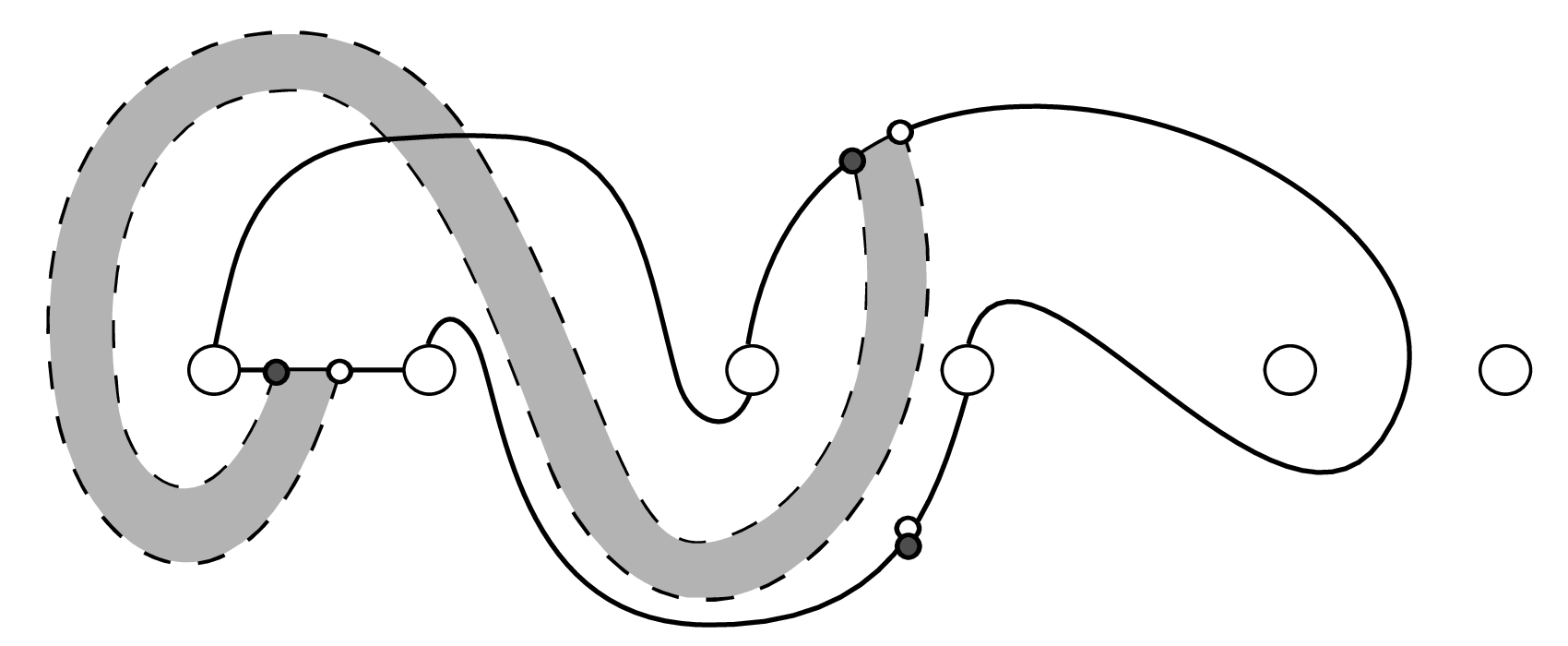}
\caption[A Heegaard diagram depicting non-trivial intersection with $\AD$]{A Heegaard diagram induced by Figure \ref{fig:ADfork}.  The domain $\mathcal{D}(\phi)$ is shaded.\label{fig:ADHD}}
\end{figure}
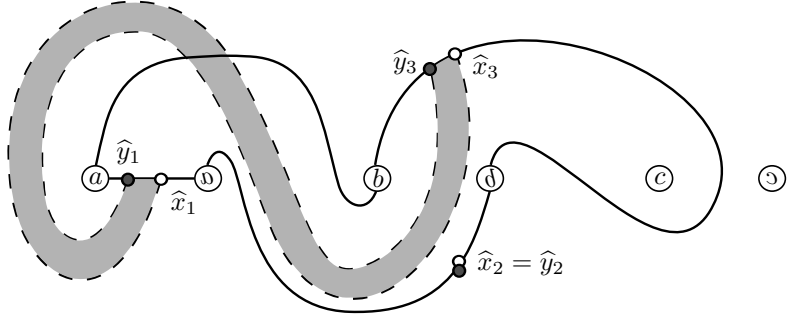

Notice that for each $i$, $\pi^{-1}(x_{i})$ contains two points; the point $\widehat{x}_{2}$ and $\widehat{y}_{2}$ are both chosen to be the preimage of $x_{2} = y_{2}$ which lies outside of the domain $\mathcal{D}(\phi)$.

One can see that $\tld{gr}(\bxh) - \tld{gr}(\byh) = \mu(\phi) = 1$.  However, $[\phi]\cdot[\AD] = 1$, and one can verify from the fork diagram that indeed $R(\bx) - R(\by) = 3.$

\subsubsection{The anti-diagonal and Heegaard multi-diagrams}\label{sec:multiad}

Throughout the rest of Section \ref{sec:AD}, we'll assume that $\Sigma$ is a genus-$n$ Heegaard surface arising as the double branched cover of $S^2$, as described in the discussion preceding Proposition \ref{prop:DBC}, with basepoint $+\infty \in \Sigma$.  Further, assume that $\ba$, $\bb$, and $\bb'$ be n-tuples of attaching curves on $\Sigma$ such that $\bb'$ differs from $\bb$ by a pointed handeslide or pointed isotopy.

We'll find in Section \ref{sec:inv} that Birman moves will induce sequences of Heegaard moves  such that only the initial and final $\alpha$ and $\beta$ circles are lifts of arcs from fork diagrams.  However, one should consider $\AD \subset \text{Sym}^{n}(\Sigma)$ as being determined by the branched covering map $\pi: \Sigma \rightarrow \mathbb{C}$ (and thus being a well-defined feature of intermediate Heegaard diagrams).

We'll analyze several types of 3-gons in the invariance proof in Section \ref{sec:inv}.  For $\bx \in \Ta \cap \Tb$ and $\by \in \Ta \cap \tor{\bb'}$, let $\psi \in \pi_{2}(\bx,\thet{\bb,\bb'},\by)$ be a 3-gon class avoiding the basepoint with $\mu(\psi) = 0$, where the domain $\mathcal{D}(\psi)$ has one of the two types discussed in Section \ref{sec:ztri}.  If $\mathcal{D}(\psi)$ is of the first type (a sum of $n$ disjoint 3-sided regions $\mathcal{D}_1, \ldots, \mathcal{D}_n$), a point in $\text{Im}(\psi) \subset \text{Sym}^n(\Sigma)$ is of the form $\bx = \{ x_1, \ldots, x_n\}$, where each $x_i \in \mathcal{D}_i$.  Further assume that at least $n-1$ of the regions $\mathcal{D}_i$ are small 3-sided regions of the type appearing in Figure \ref{fig:smalltri}.  It can easily be arranged that
\begin{equation}\label{eqn:adtri}
\mathcal{D}_i \cap \pi^{-1} \left( \pi \left( \mathcal{D}_j \right) \right) = \emptyset \quad \text{for} \quad i \neq j,
\end{equation}
and so $\bx \notin \AD.$

\begin{figure}[h!]
\centering
\begin{minipage}[c]{.38\linewidth}
\labellist
\small
\pinlabel* {$\theta_{\bb \bb'}$} at 84 94
\pinlabel* {$x_i$} at 240 50
\pinlabel* {$y_i$} at 240 140
\endlabellist
\includegraphics[height = 20mm]{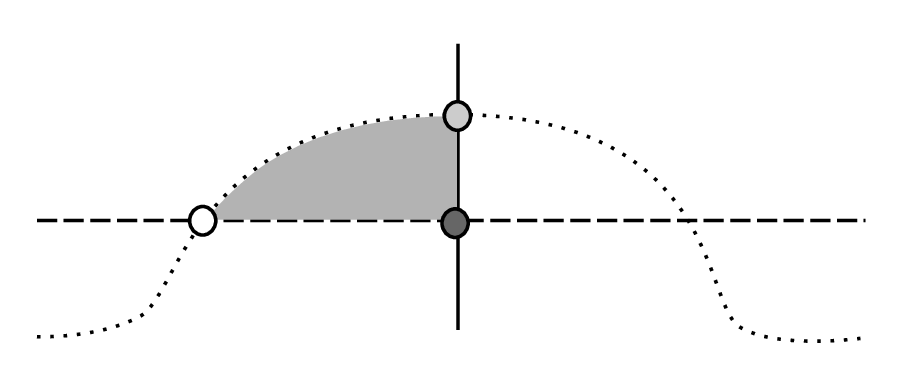}
\end{minipage}
\begin{minipage}[c]{.60\linewidth}
\caption[A small triangle]{A small 3-sided region appearing in a small isotopy}
\label{fig:smalltri}
\end{minipage}
\end{figure}

However, in Section \ref{subsec:stab}, we'll encounter a case in which $\mathcal{D}(\psi)$ is of the second type (a sum of $(n-1)$ disjoint regions $\mathcal{D}_1, \ldots, \mathcal{D}_{(n-1)}$, where the first $(n-2)$ are 3-sided regions and the last is a 6-sided region with one obtuse angle).  Additionally, assume that $\mathcal{D}_1, \ldots, \mathcal{D}_{(n-2)}$ are as shown in Figure \ref{fig:smalltri}.  In this case, a point in $\text{Im}(\psi) \subset \text{Sym}^n(\Sigma)$ is of the form $\bx = \{ x_1, \ldots, x_n\}$, where $x_i \in \mathcal{D}_i$ for $1 \leq i \leq n-2$ and $x_{(n-1)}, x_n \in \mathcal{D}_{(n-1)}$.  An analog of Equation \ref{eqn:adtri} can also be achieved here; as a result, $\bx \notin \AD$ as long as it isn't the case that $x_{g} \neq x_{(n-1)}$ and $\pi(x_{n}) = \pi(x_{(n-1)})$.  We'll show that this is impossible by arranging that $\pi^{-1} \left( \pi \left( \mathcal{D}_{(n-1)} \right) \right)$ has two connected components (one of which is $\mathcal{D}_{(n-1)}$ itself).

\subsubsection{The anti-diagonal and periodic domains}
Let $\alpha_i \cap \beta_i = \{ x_i, y_i\}$, as labelled in Figure \ref{fig:pdhs}, and let $\bx = \{ x_1, \ldots, x_n \}$ and $\by = \{ y_1, \dots, y_n\}$.  Assume without loss of generality that such a handleslide is of $\alpha_1$ over $\alpha_2$.  Then let the domains $\mathcal{D}_i^{\ba\bb}$ be the ones defined in Section \ref{sec:pd}.  In this context, we have the following fact:

\begin{lem}\label{lem:pdad}
Let $\left( \Sigma; \ba; \bb; +\infty \right)$ be the pointed Heegaard diagram mentioned above.  Let $\phi \in \pi_2\left( \bv, \bv \right)$, where $\bv = \bx$ or $\bv = \by.$  Then $\phi$ avoids the anti-diagonal $\nabla$.
\end{lem}

\begin{proof}
Without loss of generality, let $\bv = \bx$.  Letting $\mathcal{D}_i^{\ba\bb}$ be as in Section \ref{sec:pd}, one can see that for each $i$, there is a class $\phi_i \in \pi_2(\bx, \bx)$ with $\mathcal{D}(\phi_i) = \mathcal{D}_i^{\ba\bb}$.

First assume that $\bb$ differs from $\ba$ by a pointed isotopy.  Then every point in $\text{Im}(\phi_1) \subset \text{Sym}^n(\Sigma)$ is of the form $\{x, x_2, \ldots, x_n\}$, where $x \in \mathcal{D}_1^{\ba}$.  It can be arranged that $\mathcal{D}_1^{\ba\bb}$ avoids the branched covering pre-image ``twin'' of $x_j$ for each $j \geq 2$, and so $\phi_1$ avoids $\nabla$.

Instead assume that $\bb$ differs from $\ba$ by a pointed handleslide of $\alpha_1$ over $\alpha_2$, as in Figure \ref{fig:pdhs}.  The argument for the isotopy case above implies that $\phi_i$ avoids $\nabla$ for $i \geq 2$.  Notice that we can write $\phi_1 = \phi_{1,1} + \phi_{1,2}$, where $\phi_{1,1} \in \pi_2(\bx, \by)$ has domain given by the annular component of $\Dab{1}$ (with local coefficient $+1$) and $\phi_{1,2,} \in \pi_2(\by, \bx)$ has domain given by the small two-sided component of $\Dab{1}$ (with local coefficient $-1$).  By an argument analogous to that in the isotopy case, $\phi_{1,2}$ avoids $\nabla$.  Notice that any point in $\text{Im}(\phi_{1,1})$ is of the form $\{ x, x_2, \ldots, x_n\}$, where $x$ lies in the annular component.  It can be arranged that this annular component avoids the branched covering pre-image ``twin'' of $x_j$ for each $j \geq 2$, and so $\phi_{1,1}$ avoids $\nabla$ as well.

In either case, $\phi$ can be written as a sum (concatenation) of the classes $\phi_i$, and so it can be arranged that $\phi$ avoids $\nabla$. 
\end{proof}

\subsubsection{The anti-diagonal and the homotopy for a small pointed isotopy}

\section{Invariance of the filtration}\label{sec:inv}

Here we'll prove a few facts that we'll use in our invariance proofs.
\begin{rmk}\label{rmk:hats}
From now on, we'll suppress the hat when discussing a lift $\ah$ of an arc $\alpha$ unless the distinction isn't obvious from the context.
\end{rmk}

Now let $\left(\Sigma; \ba'; \ba; \bb; z\right)$ be a pointed Heegaard triple-diagram, where the set of attaching circles $\ba'$ is obtained from $\ba$ a pointed handleslide or pointed isotopy.  Then $\left(\Sigma; \ba; \bb; z\right)$ and $\left(\Sigma; \ba'; \bb; z\right)$ are two pointed Heegaard diagrams for the same 3-manifold $M$.  Recall that in fact $\left(\Sigma; \ba'; \ba; z\right)$ is an admissible pointed Heegaard diagram for $\#_{n}(S^{1} \times S^{2})$.  There is a natural choice of top-degree generator $\thet{\ba'\ba} \in \tor{\ba'} \cap \tor{\ba}$.    We also assume that $\h_{\ba \bb} := \left( \Sigma; \ba; \bb; z \right)$ and $\h_{\ba' \bb} := \left( \Sigma; \ba'; \bb; z \right)$ are admissible - Proposition \ref{prop:adm} thus implies that the triple-diagram is admissible also.  Recall that there is a 3-gon counting chain homotopy equivalence
$$\widehat{f}_{\ba', \ba,\bb}(\thet{\ba'\ba} \otimes \cdot): \widehat{CF} \left( \h_{\ba \ba} \right) \rightarrow\widehat{CF} \left( \h_{\ba' \ba} \right).$$

Now for any $\bx \in \Tap \cap \Ta$, $\by \in \Ta \cap \Tb$, and $\bz \in \Tap \cap \Tb$, there is a well-defined map
\begin{equation*}
\mathfrak{s}_{z}: \pi_{2}\left( \bx, \by, \bz \right) \rightarrow \text{Spin}^{c}(X),
\end{equation*}
where $X$ is the cobordism induced by the Heegaard move.  Since $X$ is induced by a handleslide or any isotopy, the cobordism is in fact a cylinder.  Therefore, if $\thet{\ba'\ba} \in \Tap \cap \Ta$ represents the top-degree generator of $\widehat{CF}(\#^{n}(\SxS))$, then for some $\psi \in \pi_{2}\left( \thet{\ba'\ba}, \by, \bz \right)$, $\mathfrak{s}_{z}(\psi)$ is completely determined by either restriction $\sz{\by}$ or $\sz{\bz}$.

We'll need the following fact about the absolute grading $\tld{gr}$:

\begin{lem}\label{lem:gr}
Let $\left( \Sigma; \ba'; \ba; \bb; z \right)$ be an admissible pointed Heegaard triple-diagram such that $\ba'$ differs from $\ba$ by a pointed isotopy or pointed handleslide.  Then if $\bx \in \Us \subset \tor{\ba} \cap \tor{\bb}$ for $\mathfrak{s} \in \text{Spin}^{c} \left( \Yab \right) $ torsion and $\by \in \tor{\ba'} \cap \tor{\bb}$ is a generator appearing with nonzero coefficient in the expansion of $\widehat{f}_{\ba', \ba,\bb}(\thet{\ba'\ba} \otimes \bx)$, $\tld{gr}(\by) = \tld{gr}(\bx).$
\end{lem}

\begin{proof}
Let $X$ be the cobordism induced by the Heegaard move and choose some $\mathfrak{t} \in \text{Spin}^{c}(X)$ restricting to $\mathfrak{s}$ on $\Yab$ and $\mathfrak{s}'$ on $\Yapb$, where $\by \in \mathfrak{U}_{\mathfrak{s}'} \subset  \tor{\ba'} \cap \tor{\bb}$.  The absolute grading $\tld{gr}$ is uniquely characterized in \cite{os:abs} by several properties, one of which implies that
$$ \tld{gr}(\by) - \tld{gr}(\bx) = \frac{c_1^2(\mathfrak{t}) - 2 \xi(X) - 3 \sigma(X)}{4}.$$
Since $X$ is in fact a cylinder, the right hand side vanishes.
\end{proof}

Letting $\tld{\ba}$ differ from $\ba$ by a small pointed isotopy and working in the pointed Heegaard quadruple-diagram $\left(\Sigma; \tld{\ba}; \ba'; \ba; \bb; z\right)$, one can make an analogous observation regarding $\text{Spin}^{c}$ structures associated to 4-gons.

Recall that the filtration $\rho$ is only well-defined on the summands $\CFxs{K}$ with $\mathfrak{s}\in \Scx{K}$ torsion.  However, due to the observations above, everything in sight will be $\text{Spin}^{c}$-equivalent and we will suppress $\text{Spin}^{c}$ structures in notation when proving many of the lemmas in this section.

Below we discuss Heegaard diagrams obtained from braids.  The Birman stabilization move on braids induces in the Heegaard diagram a Heegaard stabilization followed by two handleslides.  For a Heegaard diagram $\h$, stabilization amounts to taking a connected sum with $\h_{0}$, the standard genus-one pointed Heegaard diagram for $S^{3}$ with $\ba \cap \bb = \{ x \}$, where the connected sum is performed near the respective basepoints of the diagrams.  \OS showed in \cite{os:disk} that as chain complexes, $\widehat{CF}(\h) \cong \widehat{CF}(\h \# \h_{0})$.  If $\h$ is a Heegaard diagram for $\DBCs{K}$ obtained from a braid $b$, we extend $R$ and $\rho$ to $\widehat{CF}(\h \# \h_{0}, \mathfrak{s} \# \mathfrak{s}_{0})$ for each torsion $\mathfrak{s}$ by setting $R(x\by) = R(\by)$.

Let's first establish some terminology that will be used in Lemma \ref{lem:inv1} to follow.
\begin{df}\label{df:trimap}
Let $\left(\Sigma; \ba; \bb; z\right)$ and $\left(\Sigma; \ba'; \bb'; z\right)$ be two admissible Heegaard diagrams of genus $n$ appearing in some sequence of Heegaard moves connecting two diagrams covering fork diagrams, and let $\AD \subset \text{Sym}^n(\Sigma)$ denote the anti-diagonal.
\begin{enumerate}[(i)]
\item If $\bb' = \bb$ and $\ba'$ differs from $\ba$ by a pointed isotopy or pointed handleslide, then a $\ba$\textit{-triangle injection} is a function $g: \Ta \cap \Tb \hookrightarrow \tor{\ba'} \cap \Tb$ such that the following hold:
\begin{enumerate}[(a)]
\item There is a Heegaard triple-diagram $\left(\Sigma; \ba^{+}; \ba; \bb; z\right)$ (where for each $k$, $\alpha_{k}^{+}$ is isotopic to $\alpha'_{k}$ and intersects $\alpha_{k}$ transversely in two points) such that for each $\bx \in \Ta \cap \Tb$, there is a 3-gon class $\psi_{g}^{+} \in \pi_{2}(\thet{\ba^{+}\ba}, \bx, \by^{+})$ with $\mu(\psi_{g}^{+}) = 0$, $\psi_{g}^{+} \cap \AD = \emptyset$, and $n_{z}(\psi_{g}^{+}) = 0$, where $\by^{+} \in \Ta \cap \tor{\ba^{+}}$ is the nearest neighbor to $g(\bx)$.
\item There is a Heegaard triple-diagram $\left(\Sigma; \ba; \ba^{-}; \bb; z\right)$ (where for each $k$, $\alpha_{k}^{-}$ is isotopic to $\alpha'_{k}$ and intersects $\alpha_{k}$ transversely in two points)  such that for each $\bx \in \Ta' \cap \Tb$, there is a 3-gon class $\psi_{g}^{-} \in \pi_{2}(\thet{\ba\ba^{-}}, \by^{-}, \bx)$ with $\mu(\psi_{g}^{-}) = 0$, $\psi_{g}^{-} \cap \AD = \emptyset$, and $n_{z}(\psi_{g}^{-}) = 0$, where $\by^{-} \in \Ta \cap \tor{\ba^{-}}$ is the nearest neighbor to $g(\bx)$.
\end{enumerate}
\item If $\ba' = \ba$ and $\bb'$ differs from $\bb$ by a pointed isotopy or pointed handleslide, then a $\bb$\textit{-triangle injection} is a function $g: \Ta \cap \Tb \hookrightarrow \Ta \cap \tor{\bb'}$ such that the following hold:
\begin{enumerate}[(a)]
\item There is a Heegaard triple-diagram $\left(\Sigma; \ba; \bb; \bb^{+}; z\right)$ (where for each $k$, $\beta_{k}^{+}$ is isotopic to $\beta'_{k}$ and intersects $\beta_{k}$ transversely in two points) such that for each $\bx \in \Ta \cap \Tb$, there is a 3-gon class $\psi_{g}^{+} \in \pi_{2}(\bx, \thet{\bb\bb^{+}}, \by^{+})$ with $\mu(\psi_{g}^{+}) = 0$, $\psi_{g}^{+} \cap \AD = \emptyset$, and $n_{z}(\psi_{g}^{+}) = 0$, where $\by^{+} \in \Tb \cap \tor{\bb^{+}}$ is the nearest neighbor to $g(\bx)$.
\item There is a Heegaard triple-diagram $\left(\Sigma; \ba; \bb^{-}; \bb; z\right)$ (where for each $k$, $\beta_{k}^{-}$ is isotopic to $\beta'_{k}$ and intersects $\beta_{k}$ transversely in two points) such that for each $\bx \in \Ta \cap \Tb'$, there is a 3-gon class $\psi_{g}^{-} \in \pi_{2}(\by^{-}, \thet{\bb^{-}\bb}, \bx)$ with $\mu(\psi_{g}^{-}) = 0$, $\psi_{g}^{-} \cap \AD = \emptyset$, and $n_{z}(\psi_{g}^{-}) = 0$, where $\by^{-} \in \Tb \cap \tor{\bb^{-}}$ is the nearest neighbor to $g(\bx)$.
\end{enumerate}
\end{enumerate}
\end{df}


%

\begin{rmk}\label{rmk:nn}
Recall that when constructing chain homotopies associated to triangle-counting chain homotopy equivalences in Section \ref{sec:movegons}, we composed them with nearest neighbor maps so that compositions were honestly chain-homotopic to identity maps.

For instance, if $\left( \Sigma; \ba; \bb; z \right)$ and $\left( \Sigma; \ba'; \bb; z \right)$ are two admissible pointed Heegaard diagrams such that $\ba'$ is obtained from $\ba$ via a pointed isotopy or handle slide, then
\begin{align*}
\widehat{f}_{\ba\ba'\bb} \left( \thet{\ba\ba'} \otimes \widehat{f}_{\ba'\ba\bb} \left(  \thet{\ba'\ba}  \otimes \cdot \right) \right)  - id_{\widehat{CF}(\h_{\ba\bb})} &= \delh H + H\delh \quad \text{and}\\
\widehat{f}_{\ba'\ba\bb} \left( \thet{\ba' \ba} \otimes \widehat{f}_{\ba\ba'\bb} \left( \thet{\ba\ba'} \otimes \cdot \right) \right) - id_{\widehat{CF}(\h_{\ba'\bb})} &= \delh G + G\delh,
\end{align*}
where $H$ and $G$ are given by the expressions in Equation \ref{eqn:alpha}.
\end{rmk}

\begin{lem}\label{lem:inv1}
Let $\h = \left(\Sigma; \ba; \bb; z\right)$ and $\h' = \left(\Sigma; \ba'; \bb; z\right)$ be two admissible pointed Heegaard diagrams for the manifold $\DBCs{K}$ which are obtained from braids $b$ and $b'$ (possibly after Heegaard stabilization).  Assume that there is a sequence of pointed isotopies or handeslides
$$ \h = \h^0 \rightarrow \h^1 \rightarrow \ldots \rightarrow \h^n = \h'$$
where each pointed Heegaard diagram $\h^k := \left( \Sigma; \ba^k; \bb^k; z\right)$ is admissible.

Also, let $g = g^{n} \circ \ldots \circ g^{1}$ be a composition of triangle injections
\begin{equation*}
g^{k}: \tor{\ba^{k-1}} \cap \tor{\bb^{(k-1)}} \rightarrow \tor{\ba^{k}} \cap \tor{\bb^k}
\end{equation*}
and assume that $R(g(\bx)) = R(\bx)$ for each $\bx \in \tor{\ba}\cap\tor{\bb}$.
Now for $1 \leq k \leq n$, let $f^{k}: \widehat{CF}(\h^{(k-1)}) \rightarrow \widehat{CF}(\h^k)$ denote the 3-gon-counting chain homotopy equivalence induced by the $k^{th}$ Heegaard move in the sequence, let $h^{k}: \widehat{CF}(\h^k) \rightarrow \widehat{CF}(\h^{(k-1)})$ denote its homotopy inverse, and let $H^{k}:\widehat{CF}(\h^{(k-1)}) \rightarrow \widehat{CF}(\h^{(k-1)})$ and $G^{k}:\widehat{CF}(\h^k) \rightarrow \widehat{CF}(\h^k)$ be associated homotopies as described in Remark \ref{rmk:nn}.

Now let $H: \widehat{CF}(\h) \rightarrow \widehat{CF}(\h)$ and $G: \widehat{CF}(\h') \rightarrow \widehat{CF}(\h')$ be given by
\begin{equation}\label{eqn:hom}
\begin{aligned}
H:= H^{1} + \sum_{i=1}^{n-1} \left( h^{1} \circ \ldots \circ h^{i}\right) \circ H^{i+1} \circ \left( f^{i} \circ \ldots \circ f^{1}\right) \quad \text{and}\\
G:= G^{n} + \sum_{i=1}^{n-1} \left( f^{n} \circ \ldots \circ f^{(i + 1)}\right) \circ G^{i} \circ \left( h^{(i + 1)} \circ \ldots \circ h^{n}\right),
\end{aligned}
\end{equation}
so that
\begin{equation*}
hf - id_{\widehat{CF}(\h)} = \delh H + H\delh \quad \text{and} \quad fh - id_{\widehat{CF}(\h')} = \delh G + G \delh.
\end{equation*}
Then for each torsion $\mathfrak{s} \in \Scx{K}$, the following hold:
\begin{enumerate}[(i)]
\item
If $\by \in \tor{\ba'} \cap \tor{\bb'}$ is a term in the sum $f(\bx)$ for some $\bx \in \Us \subset \Ta \cap \Tb $ and if $\bw\in \Ta \cap \Tb$ is a term in the sum $h(\bz)$ for some $\bz \in \Us' \subset \tor{\ba'} \cap \tor{\bb'} $, then $\rho(\by) \leq \rho(\bx)$ and $\rho(\bw) \leq \rho(\bz)$.
\item 
If $\by \in \Ta \cap \Tb$ is a term in the sum $H(\bx)$ for some $\bx \in \Us \subset \Ta \cap \Tb$ and if $\bw \in \tor{\ba'} \cap \tor{\bb'}$ is a term in the sum $G(\bz)$ for some $\bz \in \Us' \subset \tor{\ba'} \cap \tor{\bb'}$, then $\rho(\by) \leq \rho(\bx)$ and $\rho(\bw) \leq \rho(\bz)$.
\end{enumerate}
\end{lem}

\begin{proof}[Proof of part (i)]
Let $\bx \in \Us \subset \Ta \cap \Tb$ and let $\by \in \tor{\ba'} \cap \tor{\bb'}$ be a term in the sum $f(\bx)$.  Then there are two sequences
$$ \bx = \by^0, \by^1, \ldots, \by^n = \by \quad \text{and} \quad \bx = \bx^0, \bx^1, \ldots, \bx^n = g(\bx)$$
such that $\by^j, \bx^j \in \tor{\ba^j \cap \bb^j}$, $\by^j$ is a term in the sum $f^j(\by^{(j-1)})$, and $\bx^j = g^j(\bx^{(j-1)})$ for each $j$.

We proceed by induction.  Without loss of generality, assume that the $j^{th}$ Heegaard move in the sequence is one among the $\ba$ curves. i.e. that $\bb^j = \bb^{(j-1)}$ and $g^j$ is a $\ba$-triangle injection.

Recall that we have a class $\psi_{f^{j}} \in \pi_{2}(\thet{\ba^{j} \ba^{(j-1)}}, \by^{(j-1)}, \by^{j})$ with a pseudo-holomorphic representative such that $\mu(\psi_{f^{j}}) = n_z(\psi_{f^{j}}) =  0$.  Furthermore, the triangle injection $g^{j}$ provides a class $\psi_{g^{j}} \in \pi_{2}(\thet{\ba^{j} \ba^{(j-1)}}, \bx^{(j-1)}, \bx^j)$ avoiding $\AD$ such that $\mu(\psi_{g^{j}}) =  n_z(\psi_{g^{j}})  = 0$.

Assume that we've already obtained a 2-gon class $\eta_{\ba^{(j-1)}\bb} \in \pi_{2}(\bx^{(j-1)}, \by^{(j-1)})$ with $\mu(\eta_{\ba^{(j-1)}\bb}) = n_z(\eta_{\ba^{(j-1)}\bb})= [\eta_{\ba^{(j-1)}\bb}]\cdot [\AD] = 0$.  For the base case, we notice that $\bx^0 = \by^0 = \bx$ and we let $\eta_{\ba^{0}\bb} \in \pi_2(\bx, \bx)$ be the class with trivial domain.  Then the concatenation $\tld{\psi}_j := ( \eta_{\ba^{(j-1)}\ba} + \psi_{f^{j}})$ is an element of $\pi_{2}(\thet{\ba^{j} \ba^{(j-1)}}, \bx^{(j-1)},  \by^{j})$.  The classes $\tld{\psi}_j$ and $\psi_{g^{j}}$ are $\text{Spin}^{c}$-equivalent, and thus there are 2-gons $\eta_{(j-1)} \in \pi_{2}(\bx^{(j-1)}, \bx^{(j-1)})$, $\eta_{\ba^{j}\bb} \in \pi_{2}(\bx^j,  \by^{j})$, and $\eta_{\ba^{j}\ba^{(j-1)}} \in \pi_{2}(\thet{\ba^{j}\ba^{(j-1)}},\thet{\ba^{j}\ba^{(j-1)}})$ such that
$$\tld{\psi}_{j} = \psi_{g^{j}} + \eta_{(j-1)} + \eta_{\ba^{j}\bb} + \eta_{\ba^{j}\ba^{(j-1)}}.$$

Now $\DD(\eta_{(j-1)}) + \DD(\eta_{\ba^{j}\ba^{(j-1)}})$ can be viewed as a triply-periodic element of  $\Pi_{\ba^{(j-1)}\bb^j\ba^{j}}$, and so by Proposition \ref{prop:tri} can be written as a sum of the doubly-periodic domains $\DD^{\ba^{(j-1)}\ba^{j}}_{k} \in \Pi_{\ba^{(j-1)}\ba^{j}}$.  For each $k$, let $\eta_{k}^{j} \in \pi_{2}(\thet{\ba^{j}\ba^{(j-1)}}, \thet{\ba^{j}\ba^{(j-1)}})$ denote the 2-gon whose domain is $\DD^{\ba^{(j-1)}\ba^{j}}_{k}$ - by Lemma \ref{lem:pdad}, $\eta_{k}^j$ avoids the anti-diagonal $\AD$ (and of course the basepoint).  The class $\xi_j := \eta_{(j-1)} + \eta_{\ba^{j}\ba^{(j-1)}}$ can be written as a sum of elements of classes in the set $\{  \eta_{1}^{j}, \ldots, \eta_{g}^{j} \}$, and thus $[\xi_j]\cdot [\AD] = n_{z}(\xi_j)= 0$.  Thus,
\begin{align*}
\mu(\eta_{\ba^{j}\bb}) &=\mu(\psi_{f^{j}} + \eta_{\ba^{(j-1)}\ba}) - \mu(\psi_{g^{j}}) - \mu(\xi_j) =0,\\
[\eta_{\ba^{j}\bb}]\cdot [\AD] &= [\psi_{f^{j}} + \eta_{\ba^{(j-1)}\ba}]\cdot [\AD] - [\psi_{g^{j}}]\cdot [\AD] - [\xi_j]\cdot [\AD]\\
&= [\psi_{f^{j}}]\cdot [\AD] \geq 0, \quad \text{and} \quad n_z(\eta_{\ba^{j}\bb}) = 0.
\end{align*}
where the last inequality follows from the fact that $\psi_{f^{j}}$ has a pseudo-holomorphic representative.  After $n$ steps, we obtain the required class $\eta_{\ba^n\bb} \in \pi_2 (g(\bx), \by)$.

Now since $R(\bx) = R(g(\bx)) \geq R(\by)$, $n$ iterations of Lemma \ref{lem:gr} give that
\begin{equation*}
\rho(\bx) = R(\bx) - \tld{gr}(\bx) \geq R(\by) - \tld{gr}(\bx) = R(\by) - \tld{gr}(\by) = \rho(\by).
\end{equation*}

On the other hand, let $\bz \in \Us' \subset \tor{\ba'} \cap \tor{\bb'}$ and let $\bw$ be a term in the sum $(h^1 \circ \ldots \circ h^{n})(\bz)$.  A similar induction argument provides a 2-gon class $\eta \in \pi_2 \left( g(\bw), \bz) \right)$ such that
$$ \mu(\eta) = n_z(\eta) = 0 \quad \text{and} \quad [\eta] \cdot [\AD] \leq 0.$$
Now we have that $R(\bw) = R(g(\bw)) \leq R(\bz)$, and so
$$ \rho(\bw) = R(\bw) - \tld{gr}(\bw) = R(\bz) - \tld{gr}(\bw) \leq R(\bz) - \tld{gr}(\bz) =\rho(\bz)$$
\end{proof}

\begin{proof}[Proof of part (ii)]
Without loss of generality, assume that the $k^{th}$ Heegaard move is among the $\ba$ curves (so that $\bb^{k} = \bb^{(k-1)}$ and $g^k$ is a $\ba$-triangle injection).  We work in the pointed Heegaard quadruple-diagram $\left(\Sigma;  \tld{\ba}^{(k-1)}; \ba^{k}; \ba^{(k-1)}; \bb^k; z\right)$, where $\bat^{(k-1)}$ is a set of attaching circles obtained from $\ba^{(k-1)}$ by a small admissible isotopy.  Suppose that $\by \in \Ta \cap \Tb$ appears as a term in the sum
$$\left( h^{1} \circ \ldots \circ h^{k}\right) \circ H^{(k+1)} \circ \left( f^{k} \circ \ldots \circ f^{1}\right)(\bx),$$
then $\rho(\by) \leq \rho(\bx).$

Then there are $\bz^{j}, \bbu^{j} \in \tor{\ba^{j}} \cap \tor{\bb^j}$ for $j = 0, \ldots, k$ with $\bz^{0} = \bx$, $\bbu^{0} = \by$, $\bz^{i}$ a term in $f^{i}(\bz^{(i-1)})$, $\bbu^{k}$ a term in $H^{(k+1)}(\bz^{k})$, and $\bbu^{(i-1)}$ a term in $h^{i}(\bbu^{i})$ for $i = 1, \ldots, k$.

Recall that
$$H^{(k+1)} := N_{\tld{\ba}^{k}\ba^{k}} \circ \tld{H}^{(k+1)},$$
where $N_{\tld{\ba}^{k}\ba^{k}}: \tld{\bv} \rightarrow \bv$ is the nearest neighbor isomorphism and
$$ \tld{H}^{(k+1)}:=\widehat{h}_{\tld{\ba}^{k}\ba^{(k+1)}\ba^{k}\bb} \left( \thet{\tld{\ba}^{k}\ba^{(k+1)}} \otimes \thet{\ba^{(k+1)}\ba^{k}} \otimes \cdot \right)$$
is a map counting pseudo-holomorphic representatives of 4-gon classes.  There is then such a 4-gon class $\sigma \in \pi_{2}(\thet{\bat^{k}\ba^{(k+1)}}, \thet{\ba^{(k+1)}\ba^{k}}, \bz^{k}, \tld{\bbu}^{k})$ such that $\mu(\sigma) = -1$ and $[\sigma] \cdot [\AD] \geq 0$. 

Now for $0 \leq j \leq k$, $\mathfrak{s}_{z}(\bbu^{j}) = \mathfrak{s}_{z}(\bz^{j})$ and by Lemma \ref{lem:gr},
$$\tld{gr}(\bz^{j}) - \tld{gr}(\bbu^{j}) = \tld{gr}(\bz^{k}) - \tld{gr}(\bbu^{k}) = \tld{gr}(\bz^{k}) - \tld{gr}(\tld{\bbu^{k}}) = \mu(\sigma) = -1$$

As a result, there are 2-gon classes $\zeta^{j} \in \pi_{2}( \bz^{j}, \bbu^{j})$ such that $\mu(\zeta^{j}) = -1$.  
 
There are also index-zero 3-gon classes $\tld{\psi} \in \pi_{2}(\thet{\tld{\ba}^{k}\ba^{k}}, \bbu^{k}, \tld{\bbu}^{k})$ and $\psi_{\theta} \in \pi_{2}(\thet{\tld{\ba}^{k}\ba^{(k+1)}}, \thet{\ba^{(k+1)}\ba^{k}}, \thet{\tld{\ba}^{k}\ba^{k}})$; it can be arranged that these classes have small domains, so that $\tld{\psi} \cap \AD = \psi_{\theta} \cap \AD = \emptyset$.

Notice that $\tld{\psi} + \psi_{\theta} + \zeta^{k} \in  \pi_{2}(\thet{\bat^{k}\ba^{(k+1)}}, \thet{\ba^{(k+1)}\ba^{k}}, \bz^{k}, \tld{\bbu}^{k})$, and so there is some 4-gon $\eta$ with quadruply-periodic domain such that $\tld{\psi} + \psi_{\theta} + \zeta^{k} = \sigma + \eta.$

But recall that $\eta$ can be written as the concatenation of 2-gons which avoid the anti-diagonal $\AD$ and the basepoint $z$, and so
\begin{equation*}
[\zeta^{k}]\cdot[\AD] = [\sigma]\cdot[\AD] + [\eta]\cdot[\AD] - [\tld{\psi}]\cdot[\AD] - [\psi_{\theta}]\cdot[\AD] = [\sigma]\cdot[\AD] + 0 - 0 - 0 \geq 0.
\end{equation*}

Consider the sequence $\bx^0:=\bx = \bz^0, \bx^1, \ldots, \bx^k$ with $\bx^j = g^j(\bx^{(j-1)})$ for each $j$ (here $\bx^j \in \tor{\ba^j} \cap \tor{\bb^j}$).  Recall that the proof of the first part of this lemma provided 2-gon classes $\phi^{i} \in \pi_{2}(\bx^{i}, \bz^{i})$ with $\mu(\phi^{i}) = 0$ and $[\phi^{i}] \cdot [\AD] \geq 0$ for $i = 1, \ldots, k.$  We'll show that for each $i$ with $1 \leq i \leq k$
\begin{equation}\label{eqn:hineq}
[\zeta^{(i-1)}]\cdot [\AD] + [\phi^{(i-1)}]\cdot [\AD] \geq [\zeta^i]\cdot [\AD] + [\phi^i]\cdot [\AD],
\end{equation}
where $\phi^0$ denotes the class with trivial domain connecting $\bx^0 = \bz^0 = \bx$ to itself.

Fix $i$ with $1 \leq i \leq k$ and assume (without loss of generality) that the $i^{th}$ Heegaard move is among the $\alpha$ curves.  The triangle injection $g^i$ provides a class $\psi_{g^{i}} \in \pi_{2}(\thet{\ba^{(i-1)}\ba^i}, \bx^{i}, \bx^{(i-1)})$ such that $\mu(\psi_{g^{i}}) =n_z(\psi_{g^i})=  0$ and $[\psi_{g^{i}}] \cdot [\AD] = 0$.  Additionally, there is a class $\psi_{h^{i}} \in \pi_{2}(\thet{\ba^{(i-1)}\ba^{i}}, \bu^{i}, \bu^{(i-1)})$ with pseudo-holomorphic representative such that $\mu(\psi_{h^{i}}) = n_z(\psi_{h^{i}}) =  0$.

Notice that $\psi_{g^{i}} + \zeta^{(i-1)} + \phi^{(i-1)}, \phi^{i} + \zeta^{i} + \psi_{h^{i}} \in \pi_{2}(\thet{\ba^{(i-1)}\ba^{i}}, \bx^{i}, \bu^{(i-1)})$.  Then there is some 3-gon $\eta^{i}$ with triply-periodic domain such that 
$$\psi_{g^{i}} + \zeta^{(i-1)} + \phi^{(i-1)} = \phi^{i} + \zeta^{i} + \psi_{h^{i}} + \eta^{i}.$$

Since $[\psi_{g^{i}}]\cdot[\AD] = [\eta^{i}]\cdot[\AD] = 0$ and $[\psi_{h^{i}}]\cdot[\AD]\geq 0$, Equation \ref{eqn:hineq} holds.  Therefore,
$$ [\zeta^0]\cdot [\AD] \geq [\zeta^1]\cdot [\AD] + [\phi^1]\cdot [\AD] \geq \ldots \geq [\zeta^k]\cdot [\AD] + [\phi^k]\cdot [\AD] \geq 0.$$

On the other hand, given some $\bz, \bw \in \tor{\ba'} \cap \tor{\bb'}$ such that $\bw$ is a term in $G(\bz)$, a similar argument produces a 2-gon class $\xi^{0} \in \pi_{2}(\bz, \bw)$ such that $[\xi^{0}]\cdot[\AD] \geq 0$ and $\mu(\xi^{0}) = -1.$
\end{proof}

We formulate the following restatement of Lemma \ref{lem:inv1}.

\begin{cor}\label{cor:tri3}
Let $\h = \left(\Sigma; \ba; \bb; z\right)$ and $\h' = \left(\Sigma; \ba'; \bb'; z\right)$ be two admissible pointed Heegaard diagrams for the manifold $\DBCs{K}$ which are obtained from braids $b$ and $b'$ (possibly after Heegaard stabilization), and related by handleslides and isotopies in the sense of Lemma \ref{lem:inv1}, where the intermediate pointed Heegaard diagrams are all admissible.  Assume also that there are  triangle injections corresponding to each of these Heegaard moves such that their composition $g: \tor{\ba} \cap \tor{\bb} \rightarrow \tor{\ba'} \cap \tor{\bb'}$ satisfies $R(g(\bx)) = R(\bx)$ for all $\bx \in \tor{\ba} \cap \tor{\bb}$.
Then for each $\mathfrak{s} \in \Scx{K}$ torsion, the following hold:
\begin{enumerate}[(i)]
\item The composition $f: \widehat{CF}(\h, \mathfrak{s}) \rightarrow \widehat{CF}(\h', \mathfrak{s})$ of chain homotopy equivalences induced by the moves and its homotopy inverse $h: \widehat{CF}(\h', \mathfrak{s}) \rightarrow \widehat{CF}(\h, \mathfrak{s})$ are $\rho$-filtered chain maps.
\item The homotopies $H$ from $g \circ f$ to $id_{\widehat{CF}(\h, \mathfrak{s})}$ and $G$ from $f \circ g$ to $id_{\widehat{CF}(\h', \mathfrak{s})}$ are $\rho$-filtered chain homotopies.
\end{enumerate}
In particular, the $\rho$-filtered complexes $\widehat{CF}(\h, \mathfrak{s})$ and $\widehat{CF}(\h', \mathfrak{s})$ have the same filtered chain homotopy type.
\end{cor}


We turn to a few lemmas which will later allow us to restrict our attention to multiplication of braids in $\B{2n}$ by elements of $\K{2n}$ on the right side only.  Notice first that the symplectic automorphism induced by the braid $b \in \Bn$ on the punctured disk induces a symplectic automorphism $f_{b}:\text{Sym}^{n}(\Sigma) \rightarrow \text{Sym}^{n}(\Sigma).$

One can see that there is an induced graded symplectic automorphism $\tld{f}_{b}$ with respect to gradings provided by the volume form.

\begin{lem}\label{lem:AD}
Let $f_{b}: \Symn{\Sigma} \rightarrow \Symn{\Sigma}$ be the automorphism discussed above and let $\AD \subset \Symn{\Sigma}$ denote the anti-diagonal.  Then $f_{b}(\AD) = \AD$.
\end{lem}

\begin{proof}
Let $\bx \in \AD$.  Then $\bx$ contains components $(u_{1}, z_{1})$ and $(u_{2}, z_{2})$ such that $z_{2} = z_{1}$ and $u_{2} = -u_{1}$.  Suppose that $u_{2} \neq 0$.  Then $f_{b}(\bx) = (v_{k},w_{k})$ contains components $(v_{1}, w_{1}) = f_{b}(u_{1}, z_{1})$ and $(v_{2}, w_{2}) = f_{b}(u_{2}, z_{2})$.  Since $f$ is induced by a map on the punctured disk, we have that $w_{2} = w_{1}$.  Therefore, $(v_{2})^{2} = (v_{1})^{2}$ and so $v_{2} = \pm v_{1}$.

Now let $\bx' = (u_{k}', z_{k}') \in \Delta$ be such that $u_{2}' = - u_{2}$, $u_{j}' = u_{j}$ for $j \neq 2$, and $z_{j}' = z_{j}$ for $j = 1, \ldots , n$.  Then if $f_{b}(\bx') = (v_{k}',w_{k}')$, we then have that $v_{j}' = v_{j}$ for $j \neq 2$ and $w_{j}' = w_{j}$ for $j = 1, \ldots , n$.  Further, $v_{2}' = v_{1}' = v_{1} = \pm v_{2}$.  But $\bx \neq \bx'$, so $v_{2}' \neq v_{2}$ and thus $v_{2} = -v_{1}$.

Now suppose that $u_{2} = u_{1} = 0$.  Then $z_{2}$ is a puncture point.  However, a braid element diffeomorphism on the punctured disk fixes the set of punctures, and so $v_{2} = v_{1} = 0$ also.  So, $f_{b}(\bx) \in \AD$ in this case also.

One can similarly show that $f_{b}^{-1}(\AD) \subset \AD$.
\end{proof}

As shorthand, let $\tld{f}_{b}\left(\Ta\right)$ be denoted by $b\Ta$ from now on.  Since $R$ provides an absolute grading on $CF_{*}\left(\Ta, b\Ta \right)$, when computed inside $\text{Sym}^{n}(\Sigma) - \AD$, then one can define a grading $R^{*}$ on $CF^{-*}\left(\Ta, b\Ta \right)$ by letting $R^{*}\left(\bx^{*}\right) = -R(\bx)$ for each $\bx \in \Ta \cap b\Ta$.

\begin{lem}\label{lem:Rdual1}
Let $b \in \B{2n}$ be a braid.  Then when computed inside $\text{Sym}^{n}(\Sigma) - \AD$, the complexes $CF_{*}\left(\Ta, b\Ta \right)$ and $CF^{-*}\left(\Ta, b^{-1} \Ta \right)$ are isomorphic as absolutely graded chain complexes equipped with the gradings $R$ and $R^{*}$, respectively.
\end{lem}

\begin{proof}
The grading $\tld{R}$ arises as an absolute grading induced by gradings on totally-real submanifolds, and so as $\tld{R}$-graded complexes,
\begin{equation*}
CF_{*}\left(\Ta, b\Ta \right) \cong CF^{n-*}\left(b\Ta, \Ta \right)
\cong CF^{n-*}\left(\Ta, b^{-1}\Ta \right).
\end{equation*}
But since $s_{R}(b^{-1}) + s_{R}(b) = -n$, the result follows.
\end{proof}

When one computes these complexes inside all of $\text{Sym}^{n}(\Sigma)$, recall that $CF^{-*}\left(\Ta, b \Ta \right) = \widehat{CF}\left( \h_{b} \right),$ where $\h_{b}$ is the admissible Heegaard diagram for $\DBCs{K}$ provided by Proposition \ref{prop:DBC} and $K$ is the closure of $b$.  It is clear that $R^{*}$ provides a filtration on this complex.

\begin{lem}\label{lem:Rdual2}
Let $b\in \B{2n}$ be a braid and denote by $\h_{b}$ and $\h_{b^{-1}}$ the admissible Heegaard diagrams induced by the braid $b$ and $b^{-1}$, respectively.  Then the complexes $\widehat{CF}_{*}\left( \h_{b^{-1}}\right)$ and $\widehat{CF}^{-*}\left( \h_{b}  \right)$ are isomorphic as filtered chain complexes equipped with the filtrations $R$ and $R^{*}$, respectively.
\end{lem}

\begin{proof}
When one extends the computation of the Floer complexes to all of $\text{Sym}^{n}(\Sigma)$, the differentials may have additional terms which count classes of 2-gons intersecting $\AD$.  By Lemma \ref{lem:AD}, the chain isomorphisms in the proof of Lemma \ref{lem:Rdual1} induce identifications between such classes which preserve intersection counts with $\AD$.
\end{proof}

Now given a braid word $b = \sigma_{i_{1}}^{k_{1}} \ldots \sigma_{i_{m}}^{k_{m}} \in \B{2n}$, let $-b$ denote the braid word $\sigma_{2n-i_{1}}^{-k_{1}} \ldots \sigma_{2n-i_{m}}^{-k_{m}} \in \B{2n}$.

\begin{lem}\label{lem:Rdual3}
Let $b\in \B{2n}$ be a braid whose closure is the knot $K$, let $\mathfrak{s} \in \text{Spin}^{c}(\DBCs{K})$ be torsion,  and denote by $\h_{b}$ and $\h_{-b}$ the admissible Heegaard diagrams induced by the braids $b$ and $-b$, respectively.  Then the complexes $\widehat{CF}_{*}\left( \h_{b}, \mathfrak{s} \right)$ and $\widehat{CF}^{-*}\left( \h_{-b}, \mathfrak{s}\right)$ are isomorphic as filtered chain complexes equipped with the filtrations $R$ and $R^{*}$, respectively.
\end{lem}

\begin{proof}
Let $\Fcal^{+}$ and $\Fcal^{-}$ be the fork diagrams induced by $b$ and $-b$, respectively.  Notice that if the closure of $b$ is the knot $K$, then the closure of $-b$ is $-K$, the mirror image of $K$.  Therefore, $\h_{\pm b}$ is a Heegaard diagram for $\pm \DBCs{K}$.  Let $\iota:\mathbb{C}\rightarrow\mathbb{C}$ be the map given by $z \mapsto -\overline{z}$, and let $\tld{\Fcal}$ denote the ``fork-like" diagram obtained by applying $\iota$ to $\Fcal^{+}$.  Notice that if one ignores the handles, then $\tld{\Fcal}$ is isotopic to $\Fcal^{-}$.  Figure \ref{fig:mirror} compares local pictures of these diagrams.

This map induces a diffeomorphism $\widehat{\iota}: \Sigma \rightarrow -\Sigma$, where $\h_{b} = \left( \Sigma, \ba, \bb, +\infty \right)$ and $\h_{-b} = \left( -\Sigma, \ba, \bb, +\infty \right)$.  Recall that for any closed, connected, oriented 3-manifold $Y$, $\text{Spin}^{c}(Y) \cong \text{Spin}^{c}(-Y)$.  For each torsion $\mathfrak{s} \in \text{Spin}^{c}(\DBCs{K})$, \OS described in \cite{os:disk2} a natural chain isomorphism
\begin{equation*}
\Phi:\widehat{CF}_{*}(\DBCs{-K},\mathfrak{s}) \rightarrow  \widehat{CF}^{-*}(\DBCs{K},\mathfrak{s}),
\end{equation*}
which in our case is realized by $\Phi(\bx) = \left(\{ \widehat{\iota}^{-1}(x_{i}) \} \right)^{*}$  for each generator $\bx$ for $\widehat{CF}_{*}(\DBCs{-K},\mathfrak{s})$.

\begin{figure}[h!]
\centering
\subfloat[The diagram $\Fcal^{+}$]{
\includegraphics[height = 20mm]{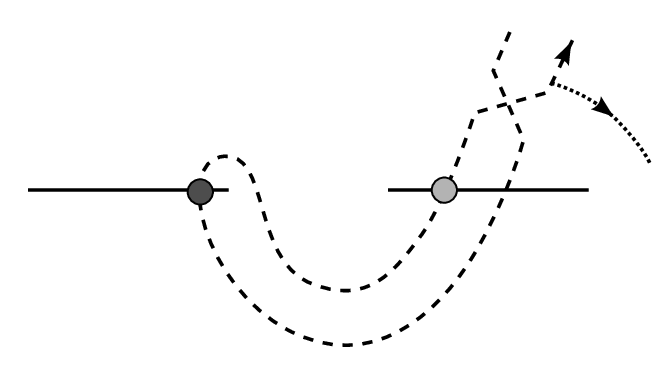}}
\quad
\subfloat[$\tld{\Fcal},$ the mirror of $\Fcal^{+}$]{
\includegraphics[height = 20mm]{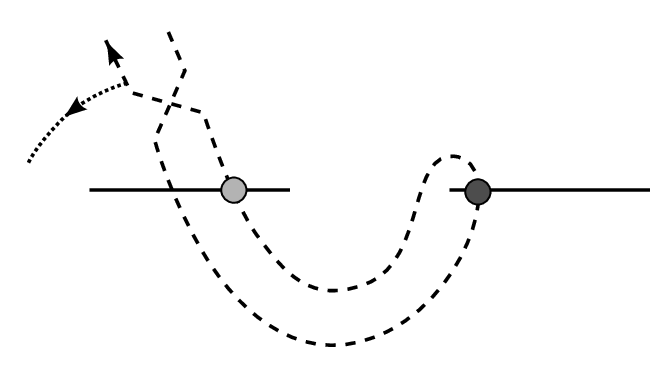}}
\quad
\subfloat[The diagram $\Fcal^{-}$]{
\includegraphics[height = 20mm]{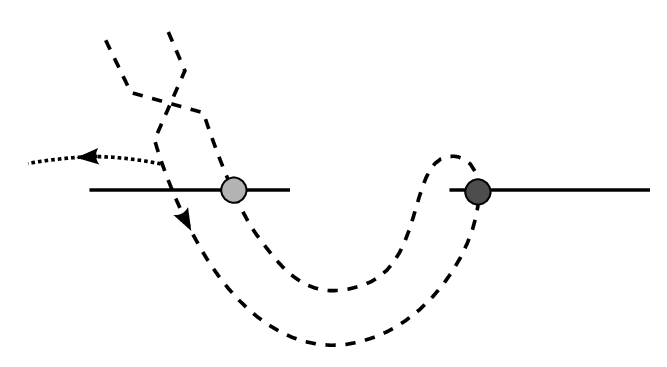}}
\caption[Fork diagrams related to a knot and its mirror]{Local pictures of fork diagrams, where corresponding points are marked with matching dots}
\label{fig:mirror}
\end{figure}

Figure \ref{fig:mirror} displays a suitably general local picture of the fork diagrams involved.  Let $\bz^{+}$ be a generator in $\Fcal^{+}$, let $\tld{\bz}$ be the tuple in $\tld{\Fcal}$ such that $\tld{z}_{j} = \iota(z_{j}^{+})$, and let $\bz^{-}$ be the corresponding generator in $\Fcal^{-}$.  Define the functions $Q$, $P$, $T$, and $\tld{R}$ on $\tld{\Fcal}$ in the obvious way.  One can verify that
\begin{align*}
\left(P^{*} - Q^{*}\right)(z^{-}_{j}) = 
&\left(P^{*} - Q^{*}\right)(\tld{z}_{j}) + 1 =
- \left(P^{*} - Q^{*}\right)(z^{+}_{j}) + 1,\\
\text{and so} \quad
& \left( P - Q \right)(\bz^{-}) = - \left( P - Q \right)(\bz^{+}) + n.
\end{align*}
Furthermore, $T(\bz^{-}) = -T(\bz^{+})$, and
\begin{align*}
s_{R}(b^{-}, D^{-})&= 
\frac{e(b^{-}) - w(D^{-}) - 2n}{4}
=\frac{-e(b^{+}) + w(D^{+}) - 2n}{4}\\
&=- \frac{e(b^{+}) - w(D^{+}) - 2n}{4} - n.
\end{align*}
Therefore, we have that $R(\bz^{-}) = -R(\bz^{+})$, and so both $\Phi$ and $\Phi^{-1}$ are filtered.
\end{proof}

\subsection{Fork diagram isotopy}\label{sec:Riso}

The identification of a braid $b$ with its associated fork diagram is only defined up to isotopy of the fork diagram.  We should verify the following:
\begin{prop}\label{prop:iso}
Let $\mathfrak{s} \in \Scx{K}$ be torsion.  Then the $\rho$-filtered chain homotopy type of $\widehat{CF}(\DBCs{K}, \mathfrak{s})$ is an invariant of the braid $b$.
\end{prop}

\begin{rmk}
The reader should note that in the original proof in \cite{os:disk} of the invariance of the group $\widehat{HF}(M)$ under isotopies of the Heegaard diagram for $M$, pseudo-holomorphic 3-gons were not used.  Lipshitz observed in \cite{lip:cyl} that the induced chain map could be defined in terms of counting 3-gons.
\end{rmk}

\begin{proof}
We omit explicit analysis of isotopies which don't introduce or annihilate intersection points between $\alpha$ and $\beta$ arcs (i.e. preserve $\Zcal$ and $\G$); these just induce intersection-preserving isotopies on the Heegaard diagram.  However, we should verify invariance under the type of isotopy shown in Figure \ref{fig:isotopy}.

\begin{figure}[h!]
\centering
\subfloat[Before isotopy]{
\labellist 
\small
\pinlabel* {$x$} at 400 300
\endlabellist 
\includegraphics[height = 30mm]{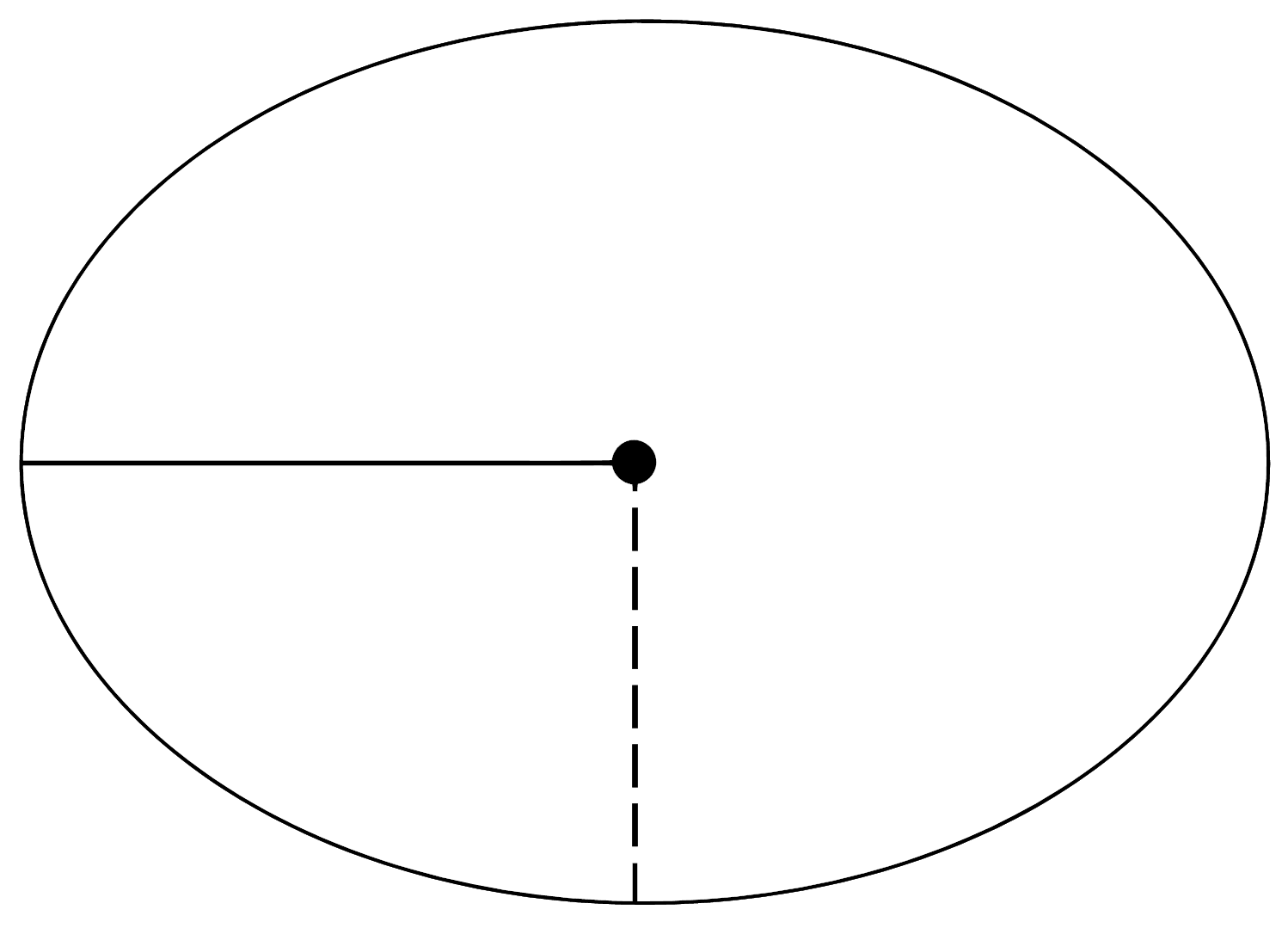}}\qquad
\subfloat[After isotopy]{
\labellist 
\small
\pinlabel* {$u$, $u'$} at 200 300
\pinlabel* {$y$} at 400 300
\endlabellist 
\includegraphics[height = 30mm]{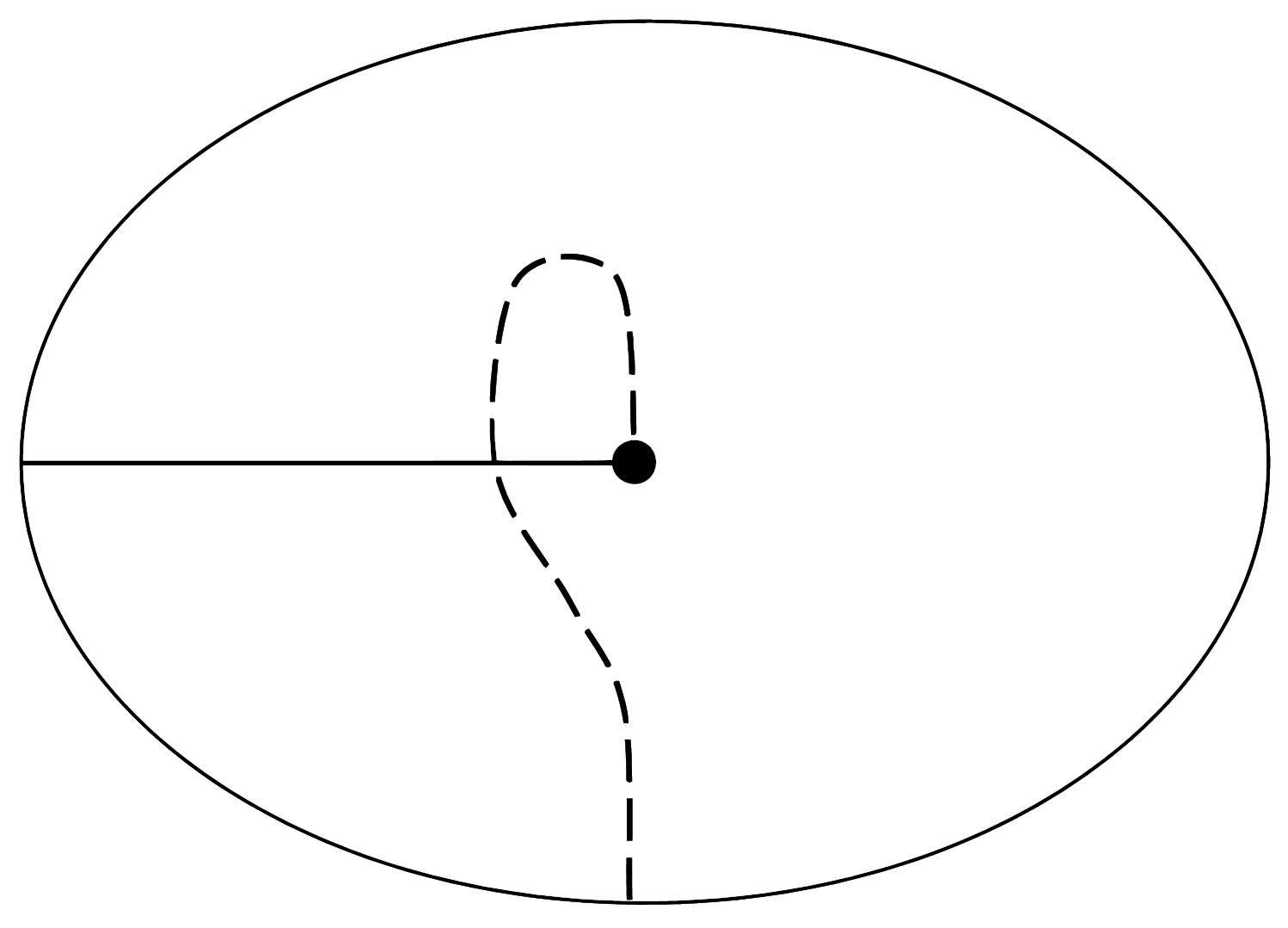}}
\caption{Introducing new intersections via an isotopy
\label{fig:isotopy}}
\end{figure}

\begin{figure}[h!]
\centering
\subfloat[Before isotopy]{
\labellist 
\small
\pinlabel* {$x$} at 390 300
\endlabellist 
\includegraphics[height = 30mm]{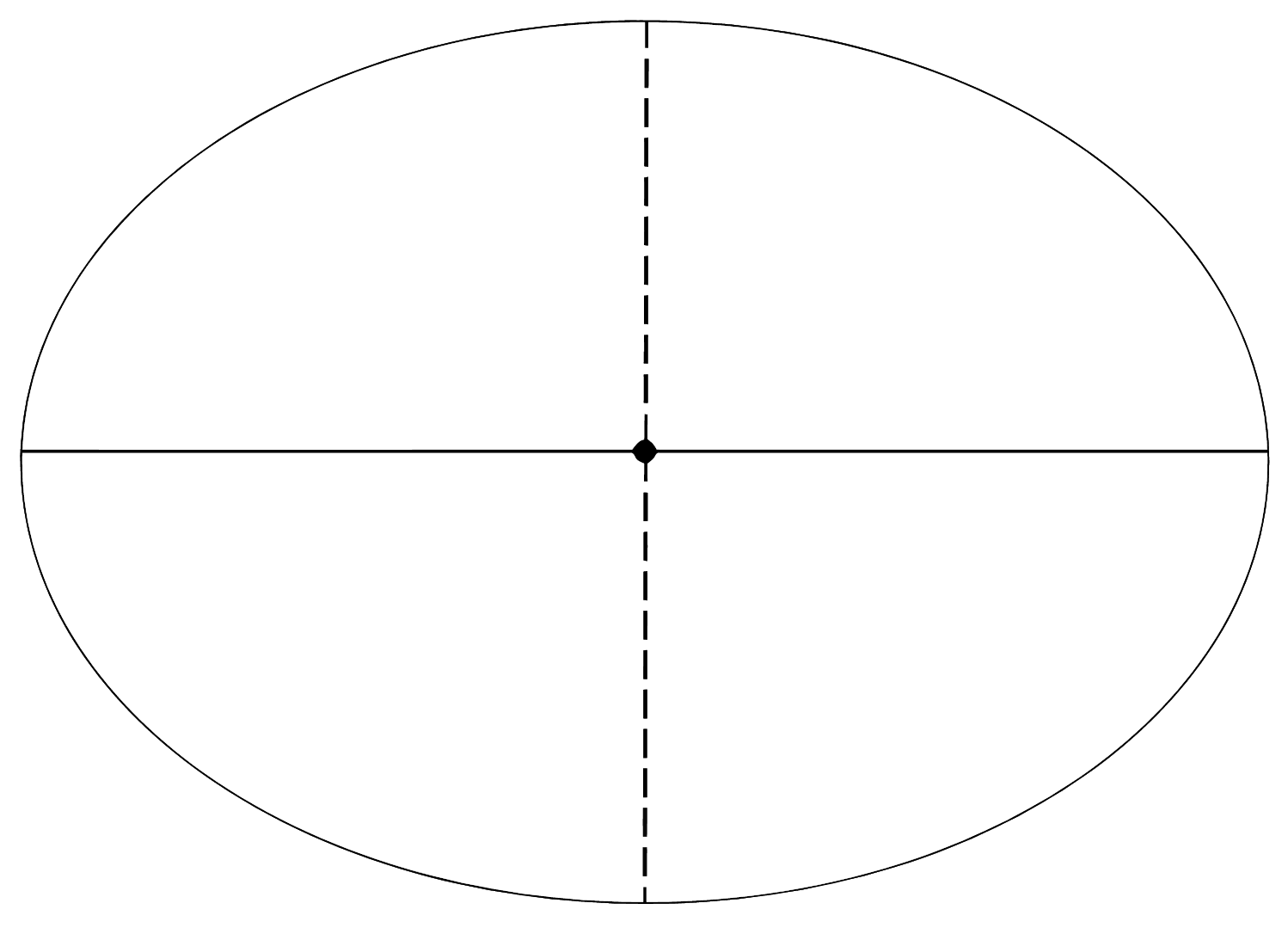}}\qquad
\subfloat[After isotopy]{
\labellist 
\small
\pinlabel* {$u'$} at 170 305
\pinlabel* {$y$} at 400 300
\pinlabel* {$u$} at 560 300
\endlabellist 
\includegraphics[height = 30mm]{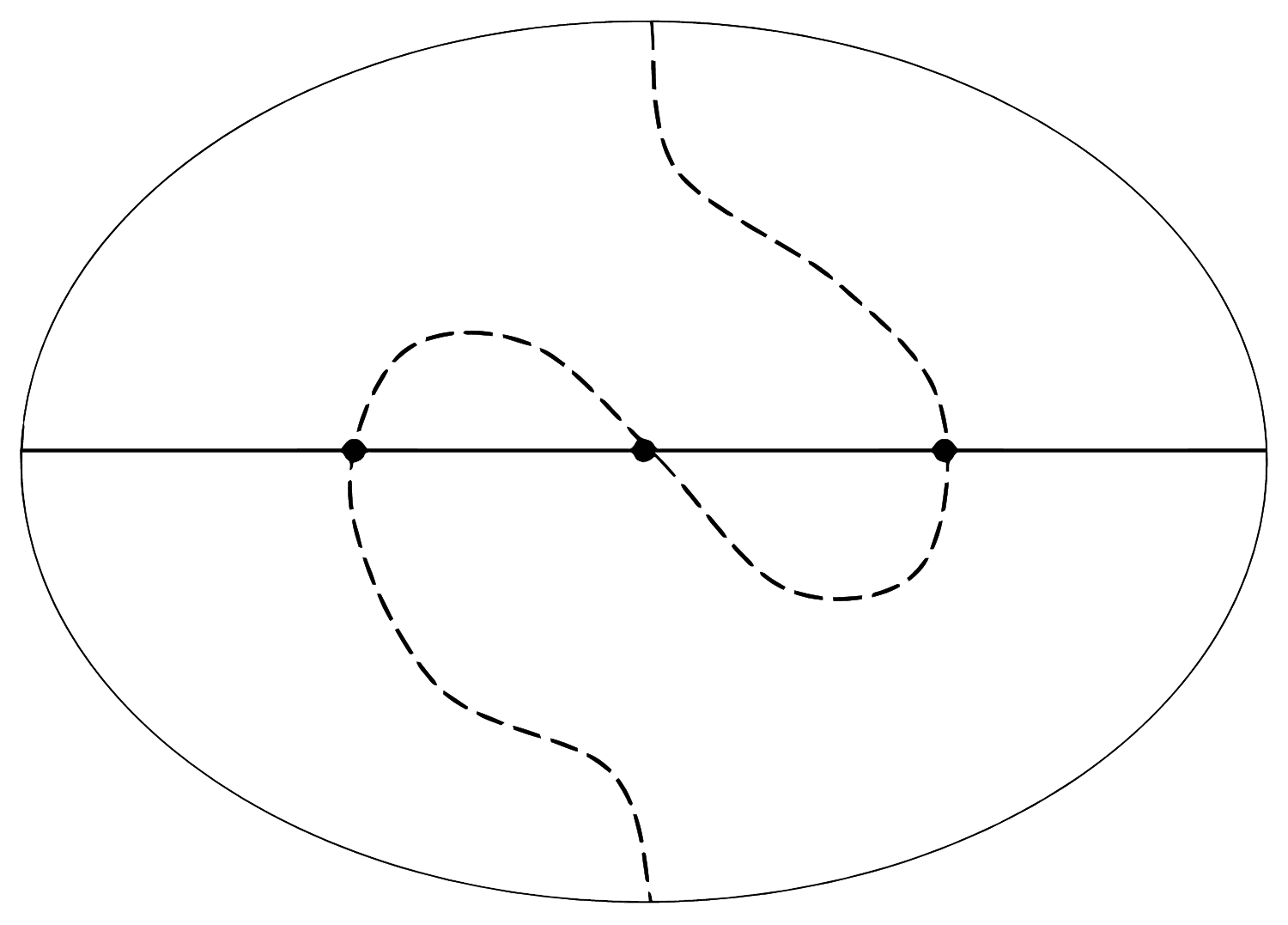}}
\caption[Heegaard diagram isotopy induced by fork diagram isotopy]{Isotopic Heegaard diagrams for $\DBCs{K}$ covering the fork diagrams in Figure \ref{fig:isotopy} \label{fig:isotopyhd}}
\end{figure}

Taking the two-fold cover of this local fork diagram branched on the one puncture gives the local Heegaard diagrams for $\DBCs{K}$ shown in Figure \ref{fig:isotopyhd}.  The isotopy on the fork diagram amounts to an isotopy on the Heegaard diagram, and by Lemma \ref{lem:tri3} it is sufficient to construct a $\bb$-triangle injection $g_{iso}$ and check that it preserves $R$.

We see from the 3-sided region in Figure \ref{fig:isotopytri} that we can define the injection $g_{iso}$ such that $g_{iso}(x\bz) = u\bz$.  Further, the loops in Figure \ref{fig:isotopyloops} show that $Q^{*}(u) = Q^{*}(x)$, $ P^{*}(u) =  P^{*}(x)$, and $T(u\bz) = T(x\bz)$.  Therefore, $R(u\bz) = R(x\bz)$.
\end{proof}

\begin{figure}[h]
\centering
\begin{minipage}[c]{.35\linewidth}
\labellist
\small
\pinlabel* {$\theta_{\bb \bb'}$} at 490 420
\pinlabel* {$\theta_{\bb' \bb}$} at 490 150
\pinlabel* {$x$} at 390 290
\pinlabel* {$u$} at 570 290
\endlabellist
\includegraphics[height = 35mm]{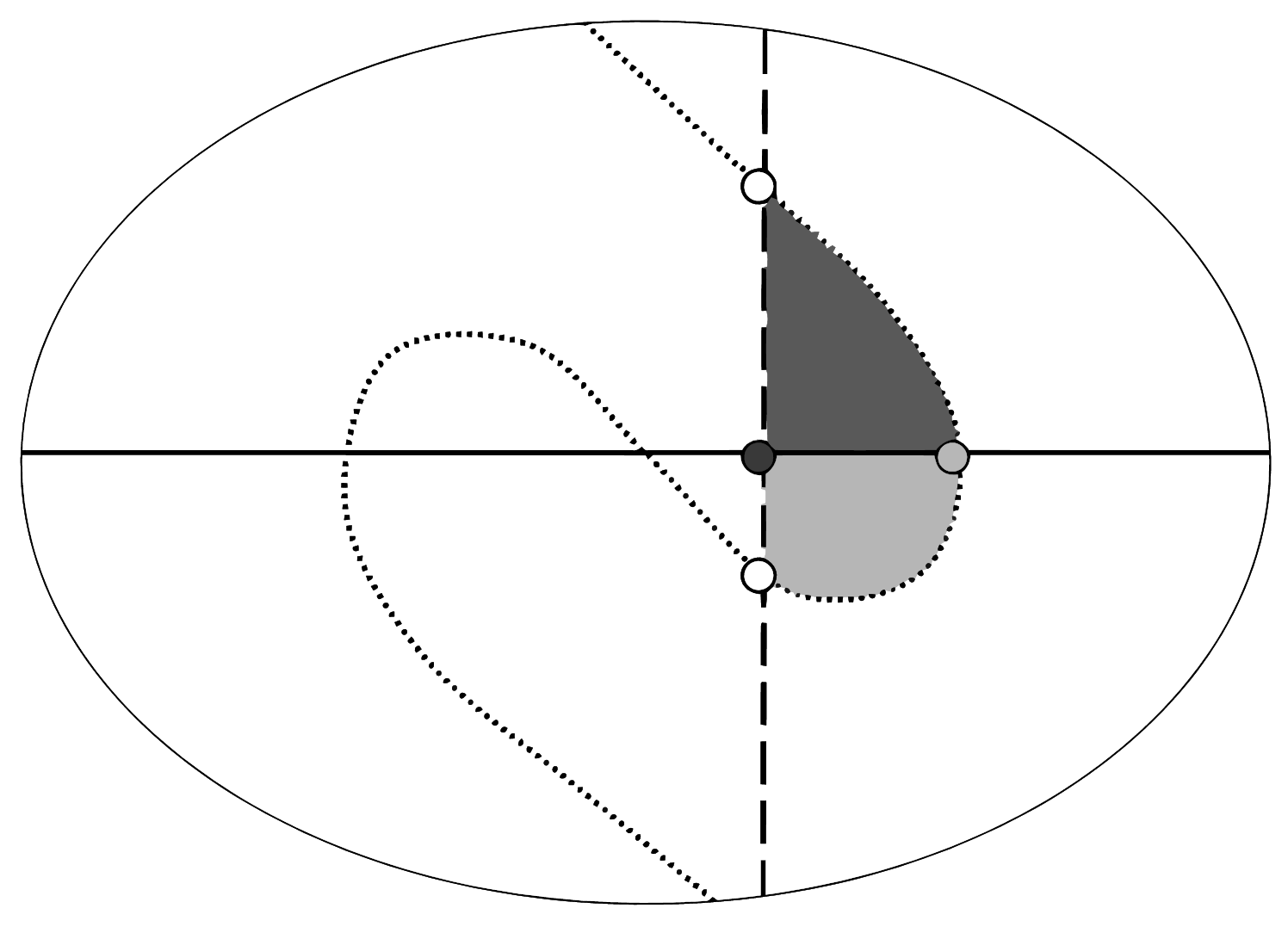} \end{minipage}
\begin{minipage}[c]{.63\linewidth}
\caption[Local components of 3-gon domains associated to isotopy]{Local regions in domains of the 3-gons $\psi^{+}_{g_{iso}}$ (dark gray) and $\psi^{-}_{g_{iso}}$ (light gray)}
\label{fig:isotopytri}
\end{minipage}
\end{figure}

\begin{figure}[h]
\centering
\subfloat[$x$]{\includegraphics[height = 30mm]{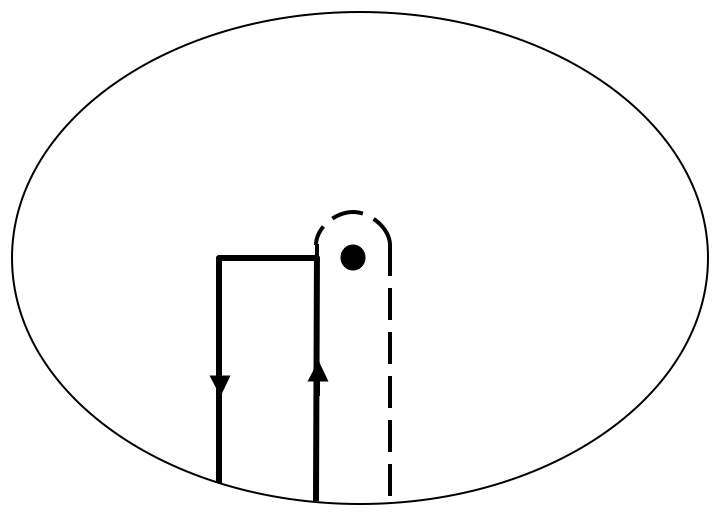}}\qquad
\subfloat[$u$]{\includegraphics[height = 30mm]{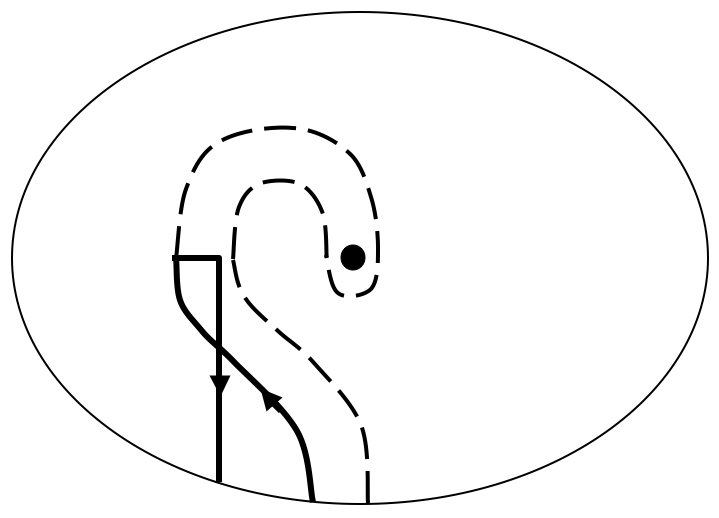}}
\caption[Grading loops associated to isotopy]{Loops used to compute gradings for $x$ and $u$ \label{fig:isotopyloops}}
\end{figure}

\subsection{Invariance under choice of braid}\label{sec:Rmoves}

\begin{proof}[Proof of Theorem \ref{thm:Rthm}]
It suffices to verify that if $b$ and $b'$ (inducing Heegaard diagrams $\h$ and $\h'$ for $\DBCs{K}$) are related by a Birman move, then the $\rho$-filtered complexes $\widehat{CF}(\h, \mathfrak{s})$ and $\widehat{CF}(\h', \mathfrak{s})$ have the same filtered chain homotopy type for each torsion $\mathfrak{s} \in \Scx{K}$.

We first claim that we don't need to explicitly examine Birman moves of the form $b \mapsto gb$, where $b \in \B{2n}$ and $g \in \K{2n}$.  Well, by Lemmas \ref{lem:Rdual2} and \ref{lem:Rdual3}, the $R$-filtered complexes$\widehat{CF}_{*}(\h_{gb}, \mathfrak{s})$ and $\widehat{CF}_{*}(\h_{bf}, \mathfrak{s})$ are filtered chain isomorphic, where $f = -g^{-1}$.  We won't explicitly analyze moves of the form $b \mapsto b(-g)$ for generators $g$ of $\K{2n}$, but the local pictures would be mirror images of those for $b \mapsto bg$ and the arguments would be completely analogous.

We'll see that each Birman move induces either a diffeomorphism of the Heegaard surface or a sequence of isotopies and handleslides relating Heegaard diagrams for $\DBCs{K}$ induced by the fork diagrams before and after the move (also preceded by a Heegaard stabilization in the Birman stabilization move).

For the case of a surface diffeomorphism, one obtains a chain complex isomorphism which will be shown to  preserve $R$.  By Lemma \ref{lem:AD}, such an isomorphism also preserves the filtration $\rho$.

We saw in Section \ref{sec:triangles} that isotopies and handleslides on Heegaard diagrams induce chain homotopy equivalences on $\widehat{CF}(\DBCs{K})$ which count pseudo-holomorphic 3-gon classes of index zero.  For each isotopy or handleslide taking $\ba$ and $\bb$ to $\ba'$ and $\bb'$, we'll define a triangle injection $g:\tor{\ba} \cap \tor{\bb} \hookrightarrow \tor{\ba'} \cap \tor{\bb'}$.  By Lemma \ref{lem:tri3},  it suffices to construct these injections $g$ and verify that $R(g(\bx)) = R(\bx)$ for each $\bx \in \G$.  Notice that since the moves induce local changes only, we'll only demonstrate local regions in the domains of the 3-gons lying in neighborhoods of the moves.  We exhibit such domains and check gradings in Section \ref{sec:moves}.
\end{proof}
\subsection{Local effects of Birman moves}\label{sec:moves}
\subsubsection{$b \mapsto b A^{\pm1} = b \sigma_{1}^{\pm1}$}\label{subsec:move1}

\begin{figure}[h]
\centering
\subfloat[Local diagram for $b$]{
\labellist 
\small
\pinlabel* {$x_{1}$} at 220 230
\pinlabel* {$u$, $u'$} at 415 235
\pinlabel* {$x_{2}$} at 490 290
\endlabellist 
\includegraphics[height = 35mm]{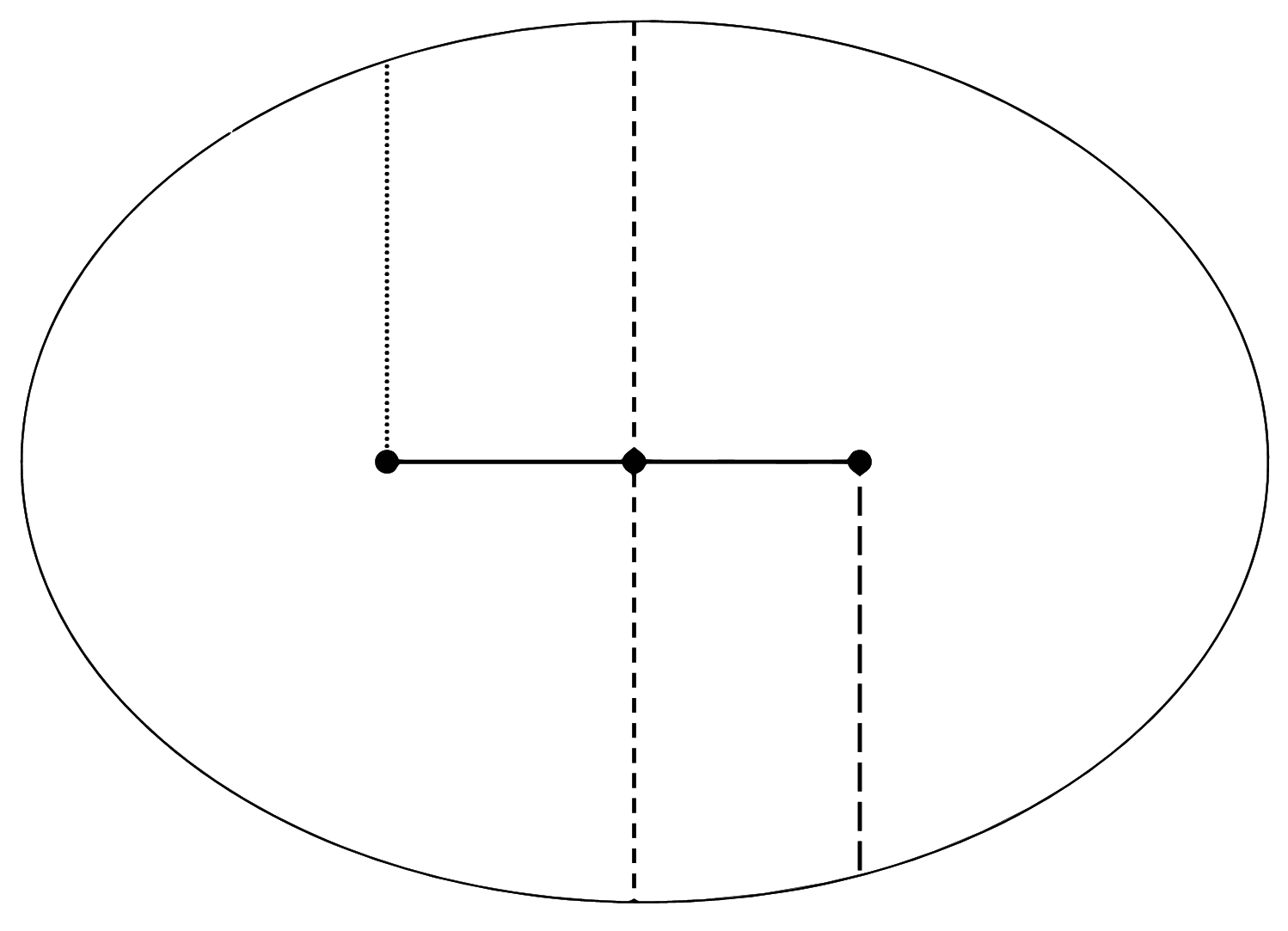}}\qquad
\subfloat[Local diagram for $b A$]{
\labellist
\small
\pinlabel* {$y_{1}$} at 220 230
\pinlabel* {$v$, $v'$} at 390 230
\pinlabel* {$y_{2}$} at 490 290
\endlabellist 
\includegraphics[height = 35mm]{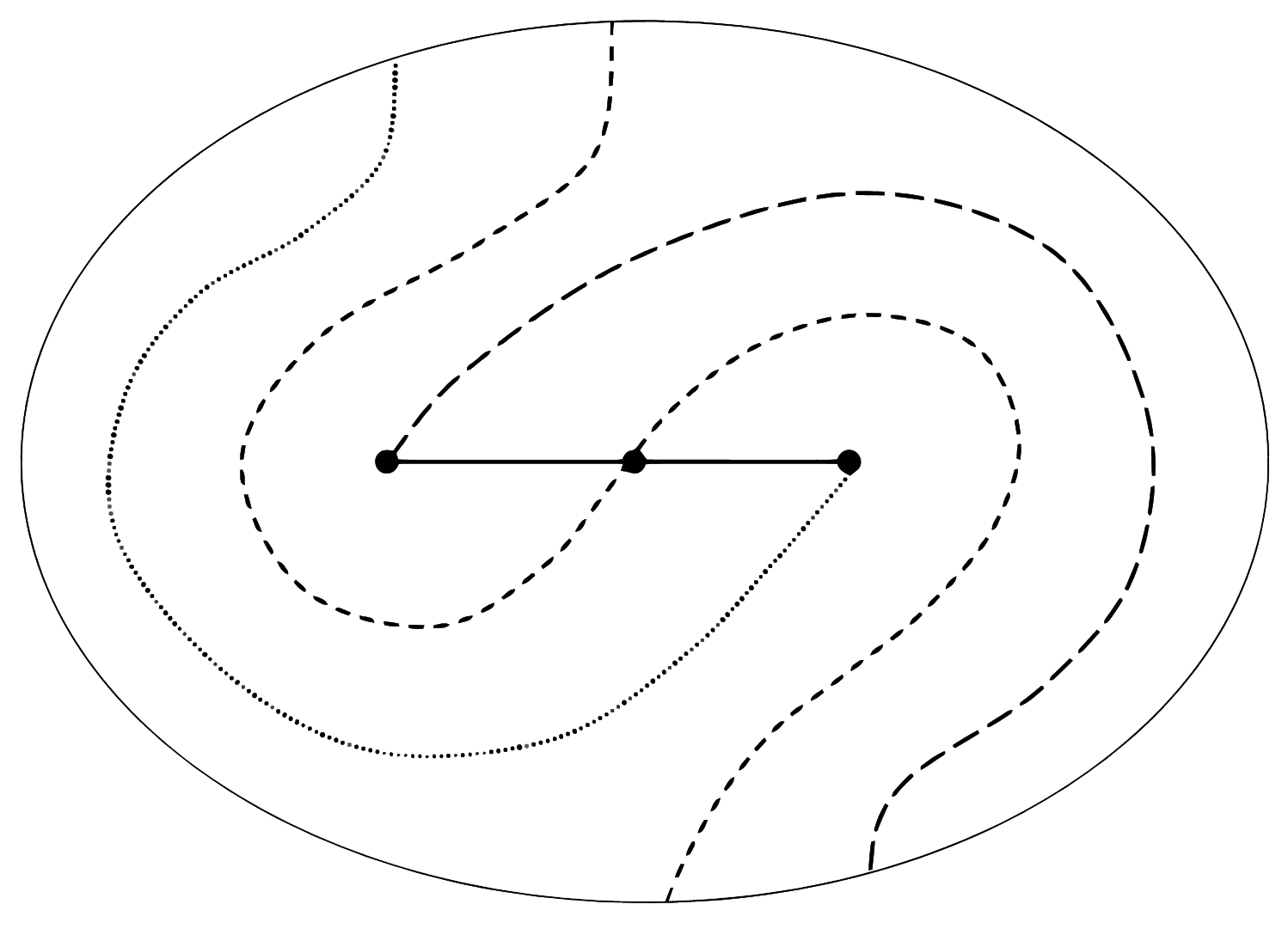}}
\caption{Fork diagrams associated to the move $b \mapsto b A $ \label{fig:move1}}
\end{figure}
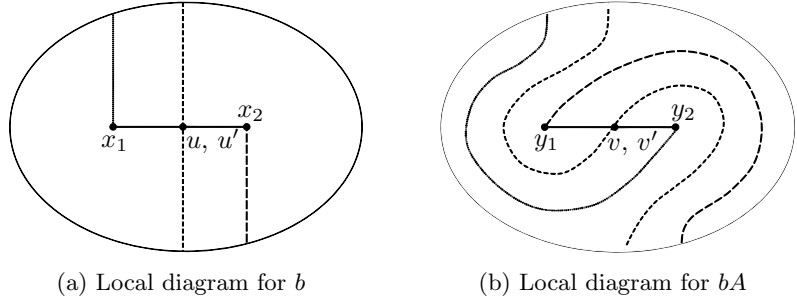

The fork diagrams for $b \mapsto b A $ can be seen in Figure \ref{fig:move1}.  This move induces a diffeomorphism on the Heegaard surface for $\DBCs{K}$ (in fact, a single Dehn twist) which sends $x_{i} \mapsto y_{i}$ for $i = 1,2$ and sends $\{ u, u'\} \hookrightarrow \{ v,v' \}$.  One thus obtains a chain isomorphism $g_{A}$ on Heegaard Floer complexes, and we verify that for each $\bw \in \G$, $R(g_{A}(\bw)) = R(\bw)$.

Since $R$ is stable, we only need to examine one $\Zcal$ representative for each element of $\Ztil \setminus \tau$.  One can verify that $Q(g_{A}(\bw)) = Q(\bw)+1$ and $P(g_{A}(\bw)) =  P(\bw) + 1$ and all $\bw \in \G$.  Since moves are local and at most one component is modified, $T$ is preserved.  The number of strands $n$ is also preserved, but $\epsilon$ and $w$ each increase by 1.  So, $s_{R}(b A) = s_{R}(b)$, and thus $R(g_{A}(\bw)) = R(\bw)$ for all $\bw \in \G$.  The details for the move $b \mapsto b A^{-1}$ are analogous.

\begin{figure}[h!]
\centering
\subfloat[$x_{1}$]{\includegraphics[height = 20mm]{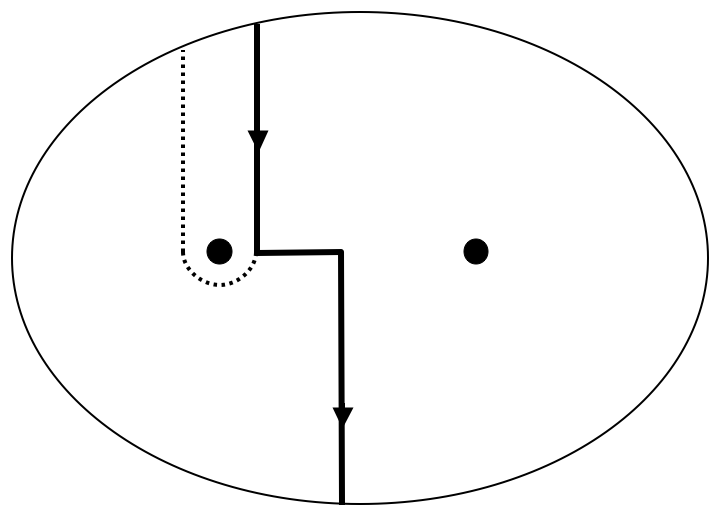}}\quad
\subfloat[$y_{2}$]{\includegraphics[height = 20mm]{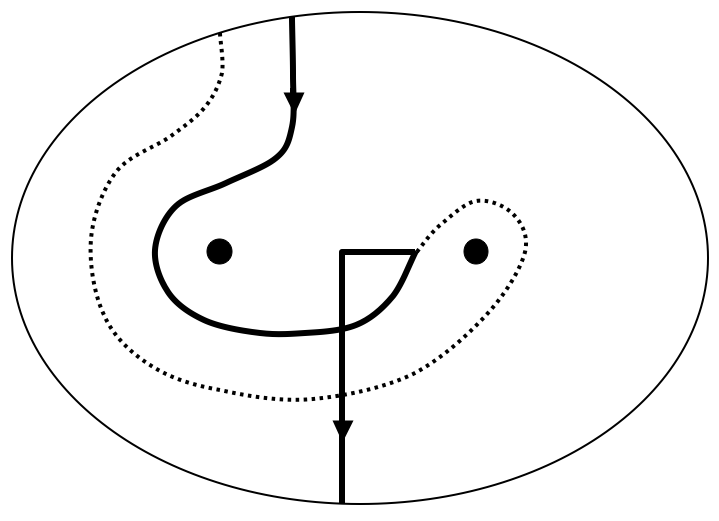}}\quad
\subfloat[$u$]{\includegraphics[height = 20mm]{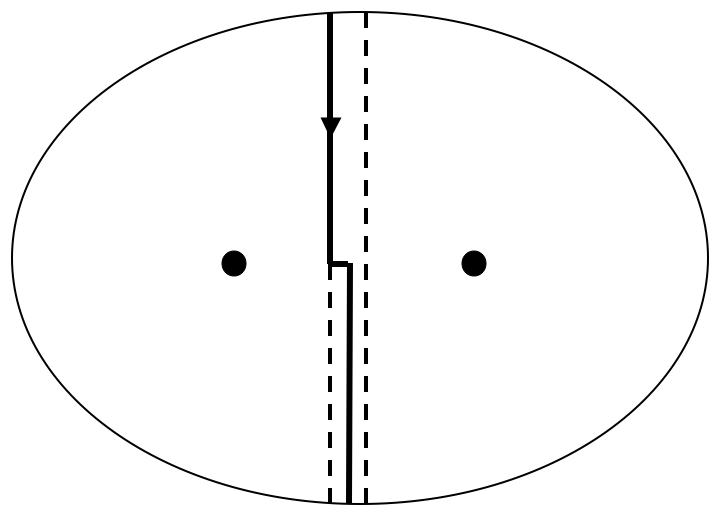}}\quad
\subfloat[$v$]{\includegraphics[height = 20mm]{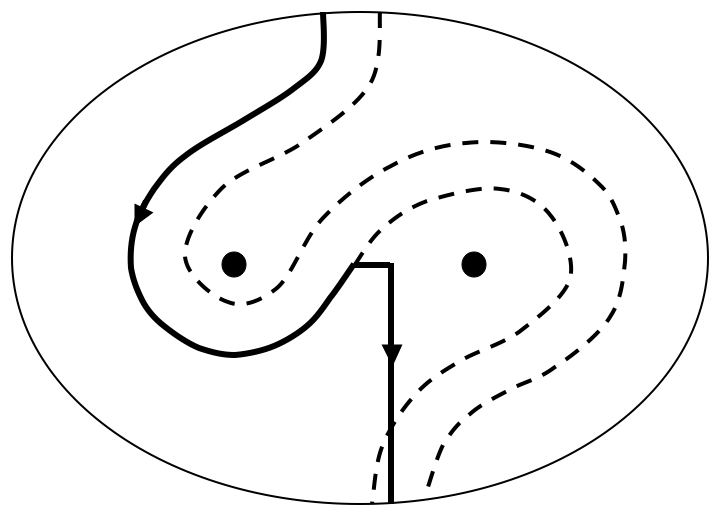}}\quad
\subfloat[$x_{2}$]{\includegraphics[height = 20mm]{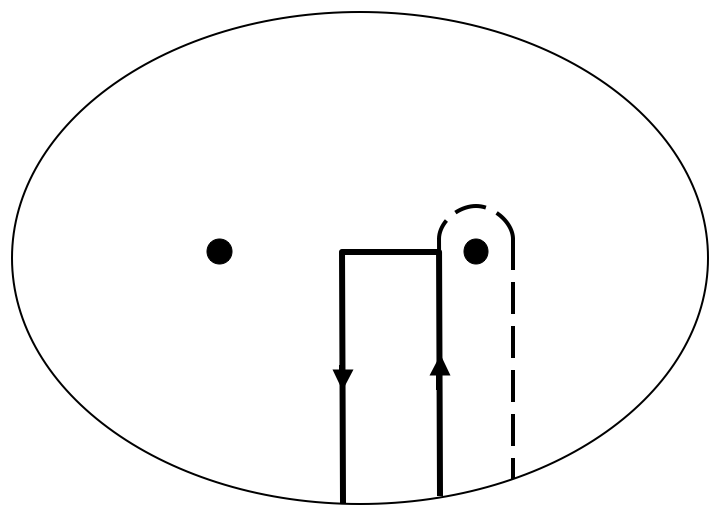}}\quad
\subfloat[$y_{1}$]{\includegraphics[height = 20mm]{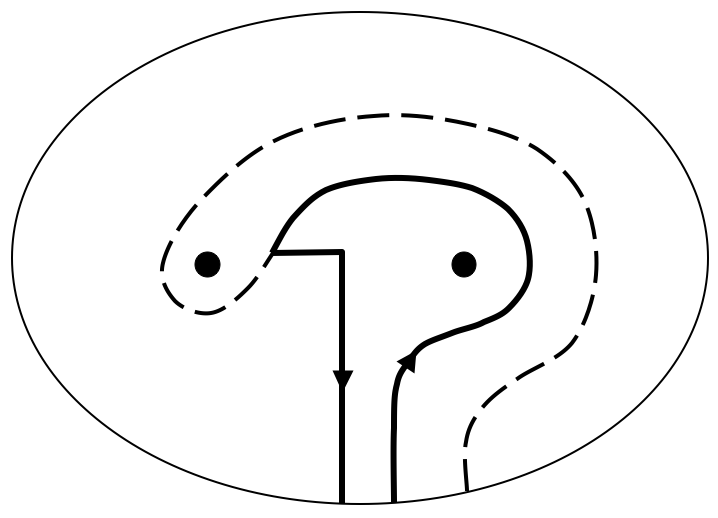}}
\caption[Loops associated to $b \mapsto b A$]{Loops associated to elements of $\Zcal$ affected by $b \mapsto b A$\label{fig:move1loops}}
\end{figure}

\subsubsection{$b \mapsto b C_{i}^{\pm 1} = 
b \left( \sigma_{2i}\sigma_{2i-1}\sigma_{2i+1}\sigma_{2i} \right)^{\pm1}$}

\begin{figure}[h]
\centering
\subfloat[Local diagram for $b$.]{
\labellist 
\small
\pinlabel* {$x_{1}$} at 95 240
\pinlabel* {$s,s'$} at 220 230
\pinlabel* {$x_{2}$} at 280 280
\pinlabel* {$x_{3}$} at 450 240
\pinlabel* {$u,u'$} at 590 290
\pinlabel* {$x_{4}$} at 640 240
\endlabellist 
\includegraphics[height = 40mm]{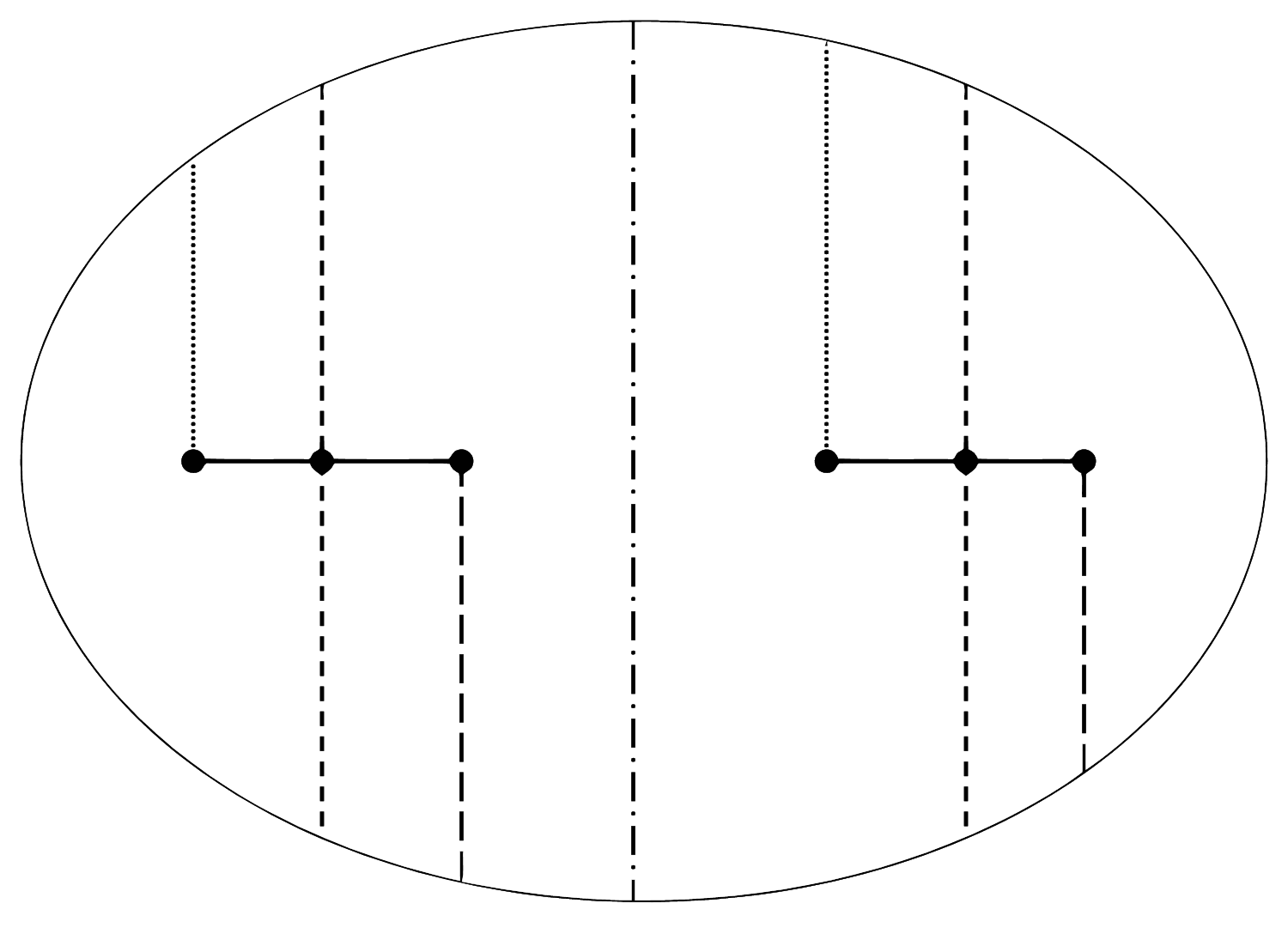}}
\subfloat[Local diagram for $b C_{i}$]{
\labellist
\small
\pinlabel* {$y_{1}$} at 115 235
\pinlabel* {$t,t'$} at 202 285
\pinlabel* {$y_{2}$} at 260 235
\pinlabel* {$y_{3}$} at 470 295
\pinlabel* {$v,v'$} at 530 235
\pinlabel* {$y_{4}$} at 595 295
\endlabellist 
\includegraphics[height = 40mm]{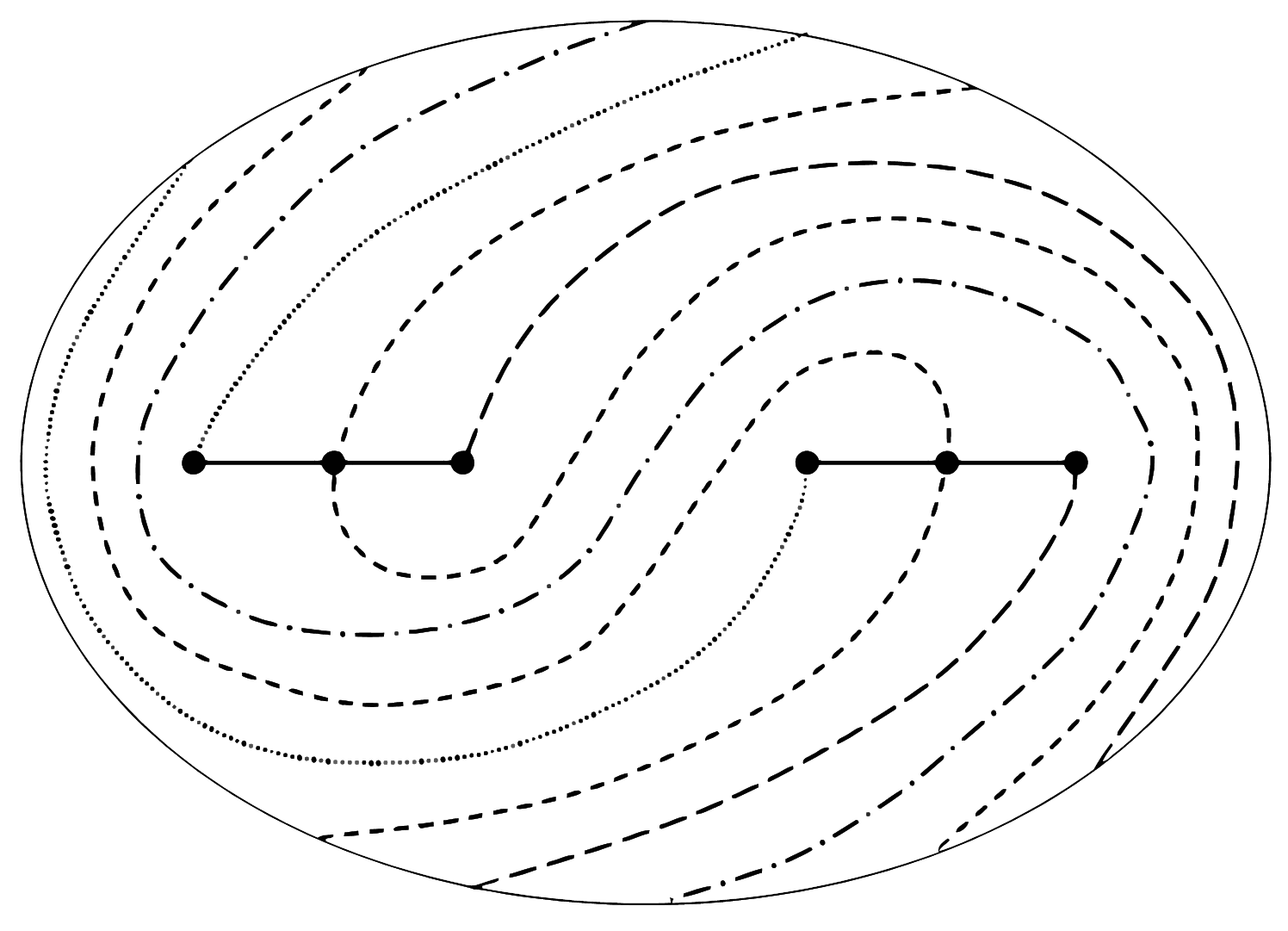}}
\caption{Fork diagrams associated to $b \mapsto b C_{i}$ \label{fig:move2}}
\end{figure}

This move induces a sequence of two Dehn twists, mapping intersections via $x_{1} \mapsto y_{3}$, $x_{2} \mapsto y_{4}$, $x_{3} \mapsto y_{1}$, $x_{4} \mapsto y_{2}$, $\{ s, s' \} \hookrightarrow \{ v, v'\}$, and $\{ u, u' \} \hookrightarrow \{ t, t'\}$.  One can verify that for each $\bw \in \G$, $Q(g_{C_{i}}(\bw)) = Q(\bw) + 2$, $P(g_{C_{i}}(\bw)) =  P(\bw)$, and $T(g_{C_{i}}(\bw)) = T(\bw) + 1$.  Here $n$ and $w$ are unchanged, but $\epsilon$ increases by 4.  So, $s_{R}(b C_{i}) = s_{R}(b) + 1$, and thus $R(g_{C_{i}}(\bw)) = R(\bw)$ for all $\bw \in \G$.  The proof associated to the move $b \mapsto b C_{i}^{-1}$ is analogous.

\begin{figure}[h!]
\centering
\subfloat[$x_{1}$]{\includegraphics[height = 20mm]{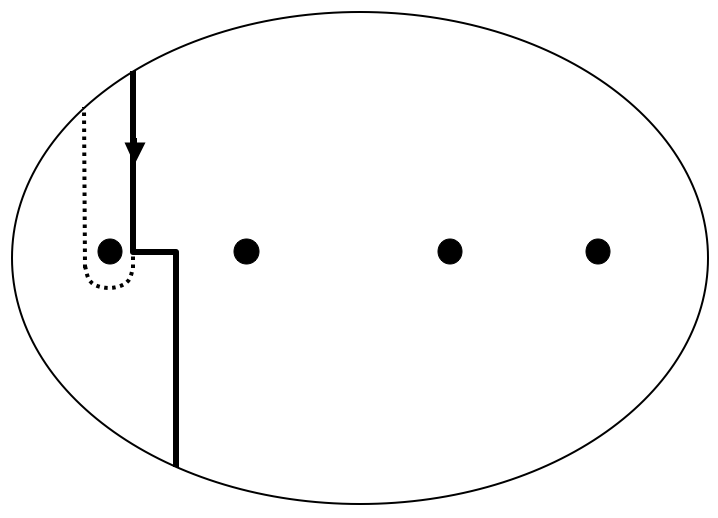}}\quad
\subfloat[$y_{3}$]{\includegraphics[height = 20mm]{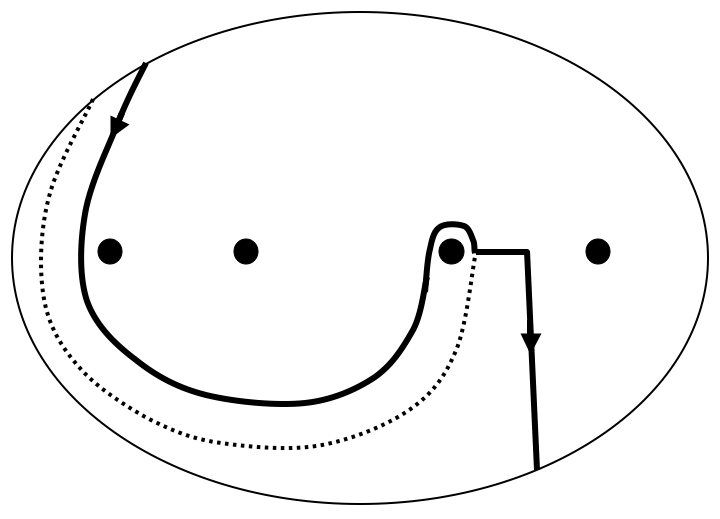}}\quad
\subfloat[$s$]{\includegraphics[height = 20mm]{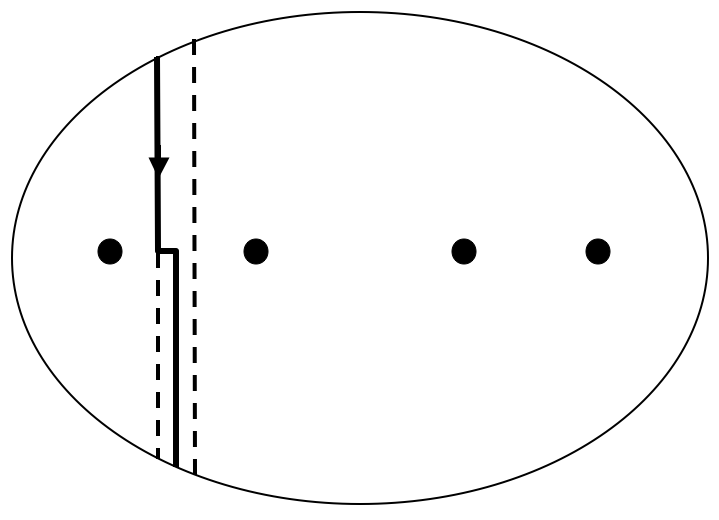}}\quad
\subfloat[$v$]{\includegraphics[height = 20mm]{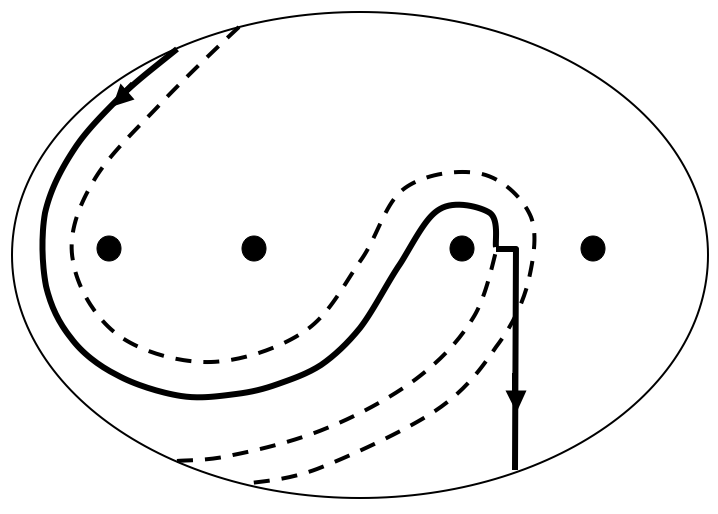}}\\
\subfloat[$x_{2}$]{\includegraphics[height = 20mm]{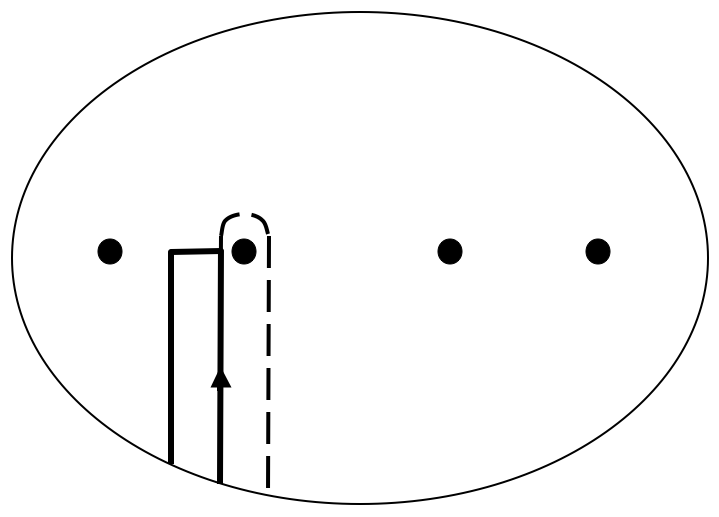}}\quad
\subfloat[$y_{4}$]{\includegraphics[height = 20mm]{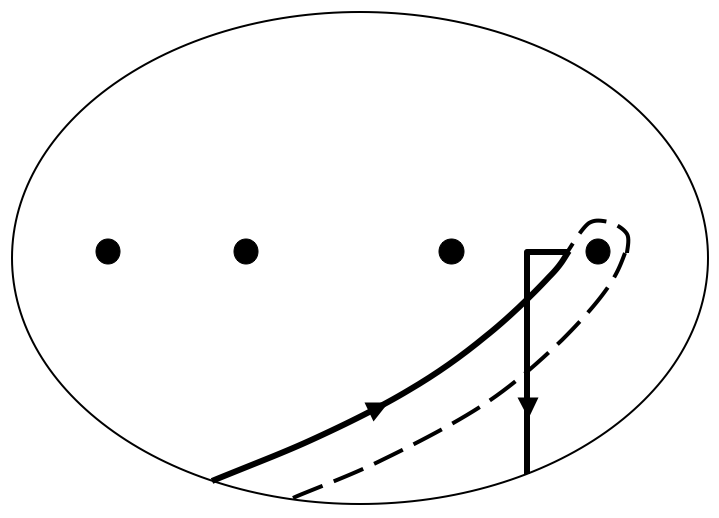}}\quad
\subfloat[$x_{3}$]{\includegraphics[height = 20mm]{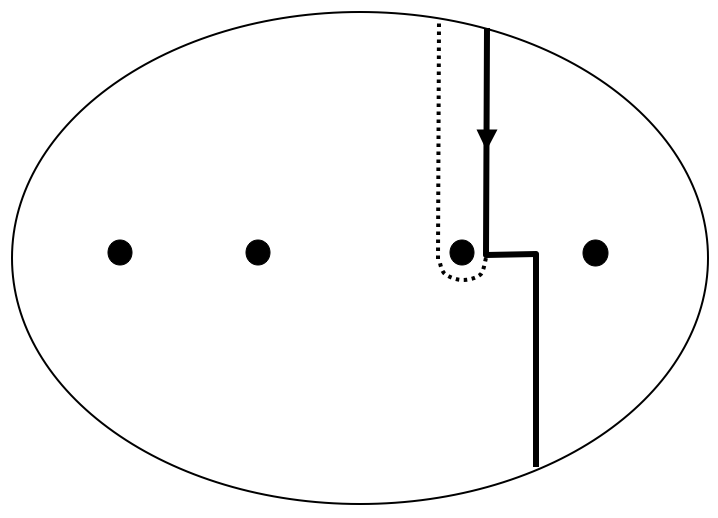}}\quad
\subfloat[$y_{1}$]{\includegraphics[height = 20mm]{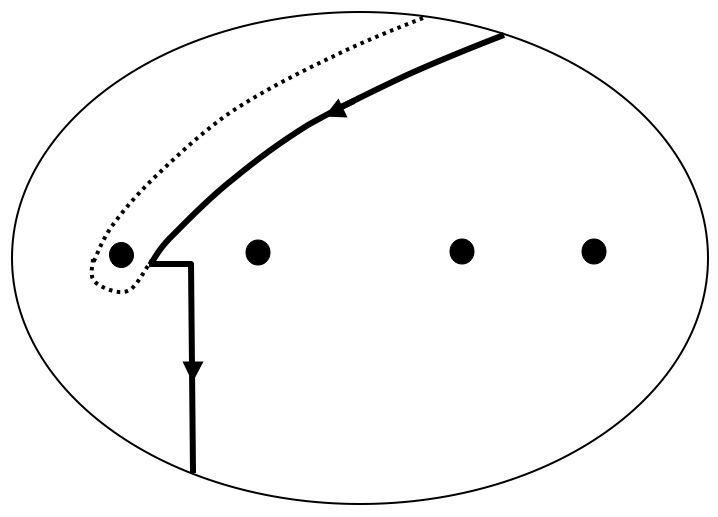}}\\
\subfloat[$u$]{\includegraphics[height = 20mm]{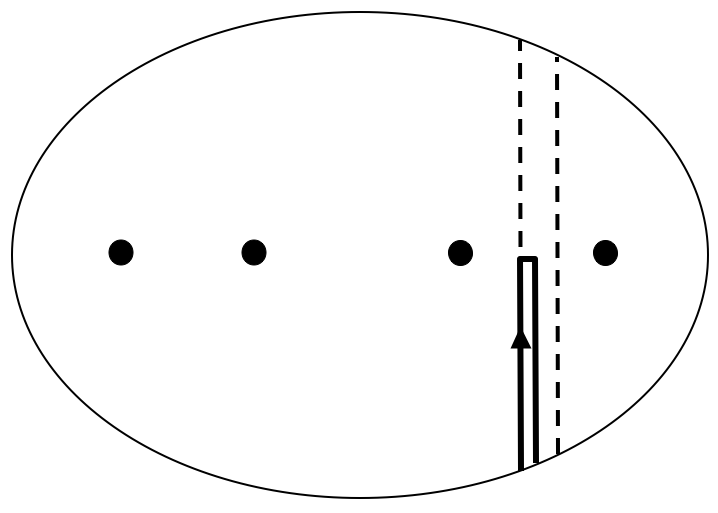}}\quad
\subfloat[$t$]{\includegraphics[height = 20mm]{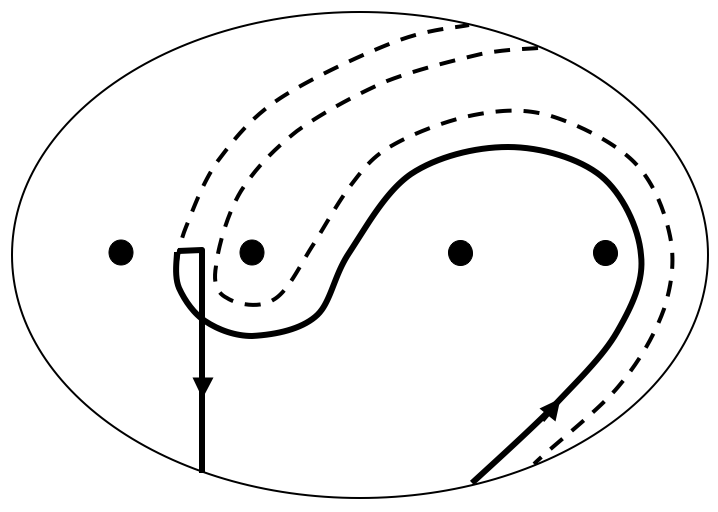}}\quad
\subfloat[$x_{4}$]{\includegraphics[height = 20mm]{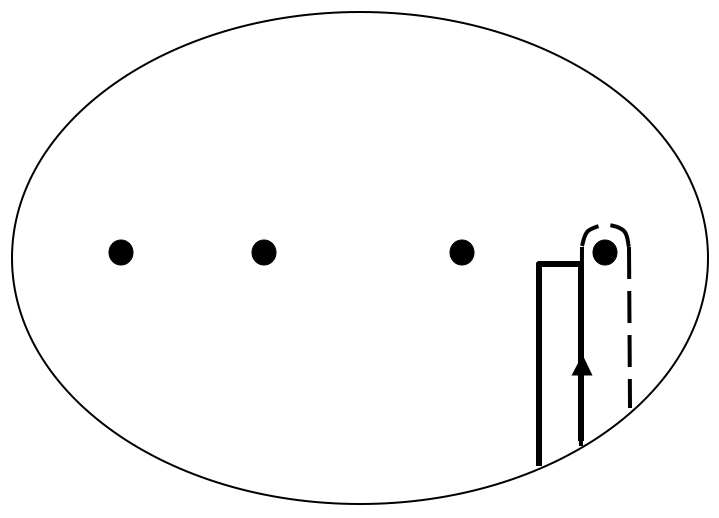}}\quad
\subfloat[$y_{2}$]{\includegraphics[height = 20mm]{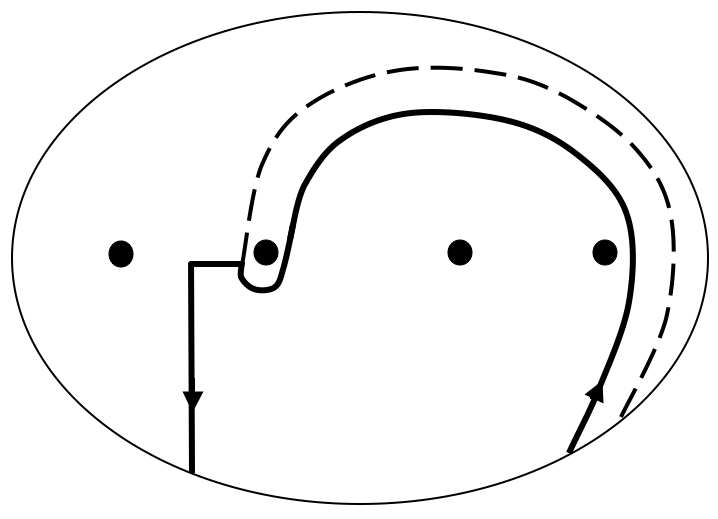}}
\caption[Loops associated to $b \mapsto b C_{i}$]{Loops associated to elements of $\Zcal$ affected by $b \mapsto b C_{i}$\label{fig:move2loops}}
\end{figure}

\subsubsection{$b \mapsto bB^{\pm 1} = b\left(\sigma_{2}\sigma_{1}^{2}\sigma_{2}\right)^{\pm1}$}\label{subsec:move3}

The fork diagrams before and after this move can be seen in Figure \ref{fig:move3}.  To better understand the fork diagram for $b B$, we perform the isotopy resulting in Figure \ref{fig:move3iso}.

\begin{figure}[h!]
\centering
\subfloat[Local diagram for $b$]{
\labellist 
\small
\pinlabel* {$x_{1}$} at 170 235
\pinlabel* {$u,u'$} at 315 230
\pinlabel* {$x_{2}$} at 365 285
\pinlabel* {$x_{3}$} at 540 235
\endlabellist 
\includegraphics[height = 40mm]{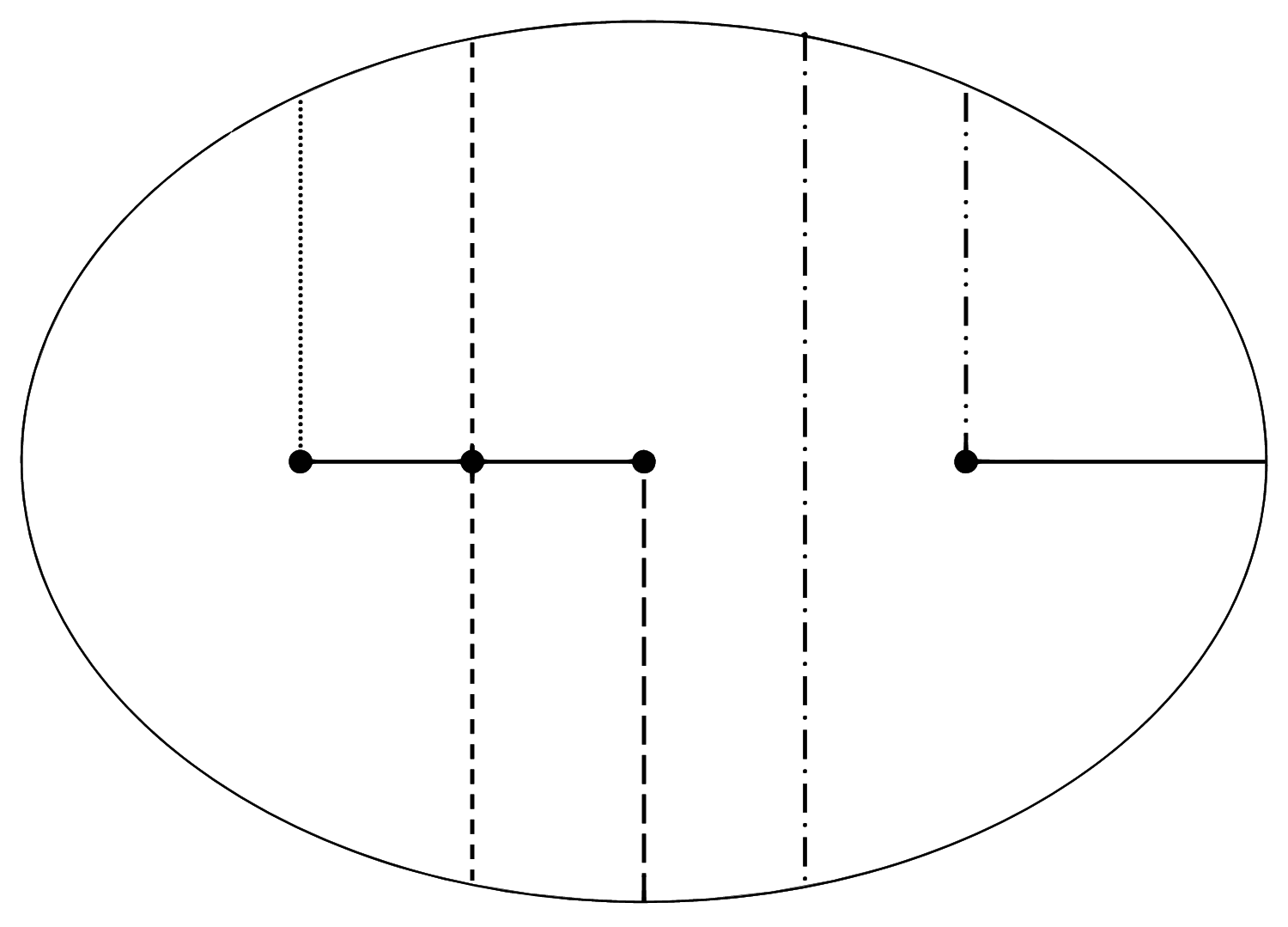}\label{fig:move3a}}
\subfloat[Local diagram for $b B$]{
\labellist
\small
\pinlabel* {$y_{1}$} at 220 330
\pinlabel* {$v,v'$} at 300 290
\pinlabel* {$y_{2}$} at 350 350
\pinlabel* {$y_{3}$} at 545 285
\endlabellist 
\includegraphics[height = 40mm]{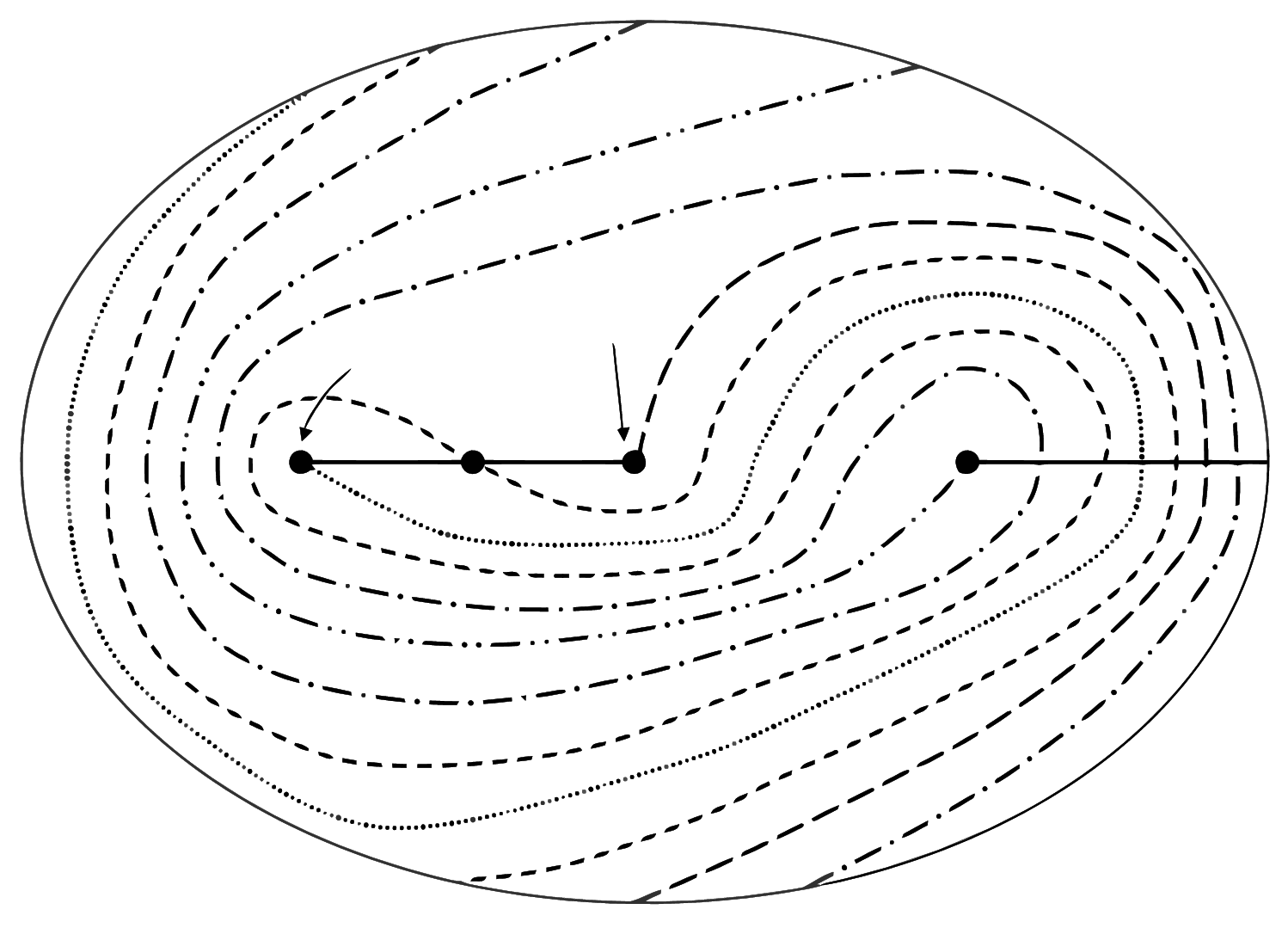}\label{fig:move3b}}\\
\subfloat[Isotopic to Figure \ref{fig:move3b}]{
\includegraphics[height = 40mm]{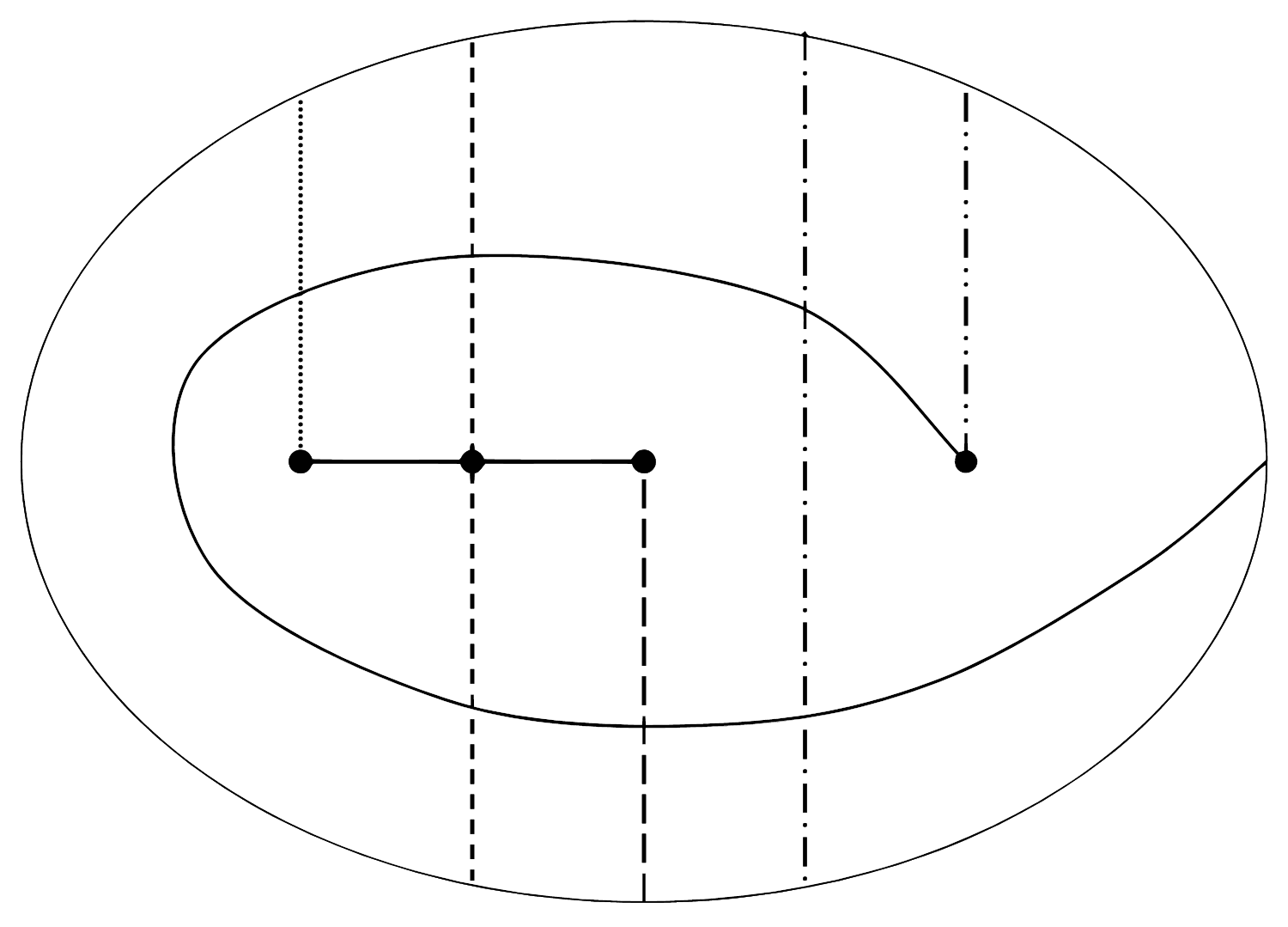} \label{fig:move3iso}}
\caption{Fork diagrams associated to $b \mapsto b B$ \label{fig:move3}}
\end{figure}
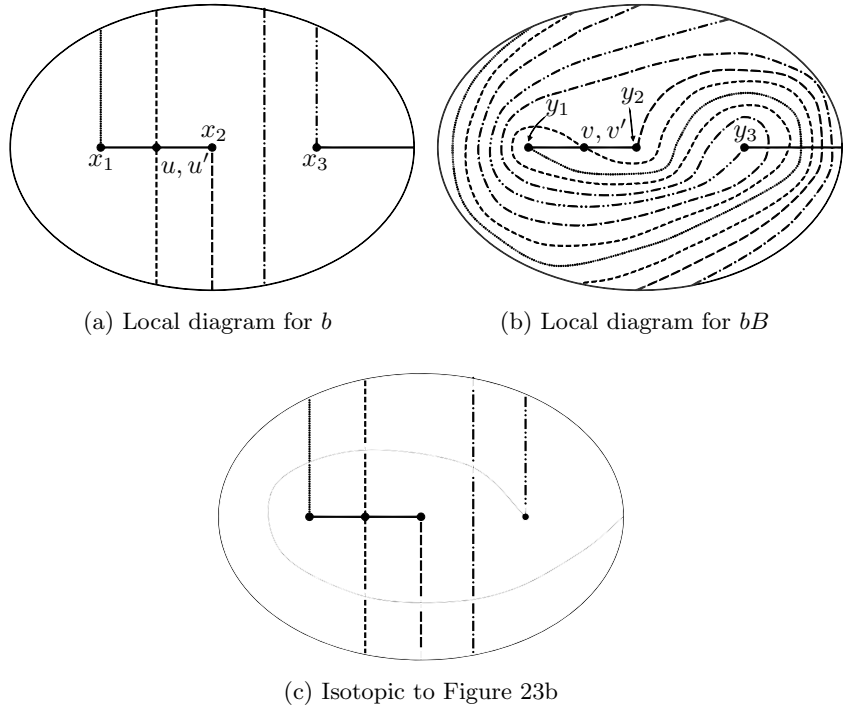

\begin{figure}[h!]
\centering
\subfloat[From $b$]{
\labellist 
\small
\pinlabel* {$x_{1}$} at 340 190
\pinlabel* {$u$} at 285 193
\pinlabel* {$u'$} at 385 200
\pinlabel* {$x_{2}$} at 450 385
\pinlabel* {$x_{3}$} at 610 240
\pinlabel* {$a$} at 248 263
\pinlabel* {\reflectbox{$a$}} at 470 263
\endlabellist 
\includegraphics[height = 40mm]{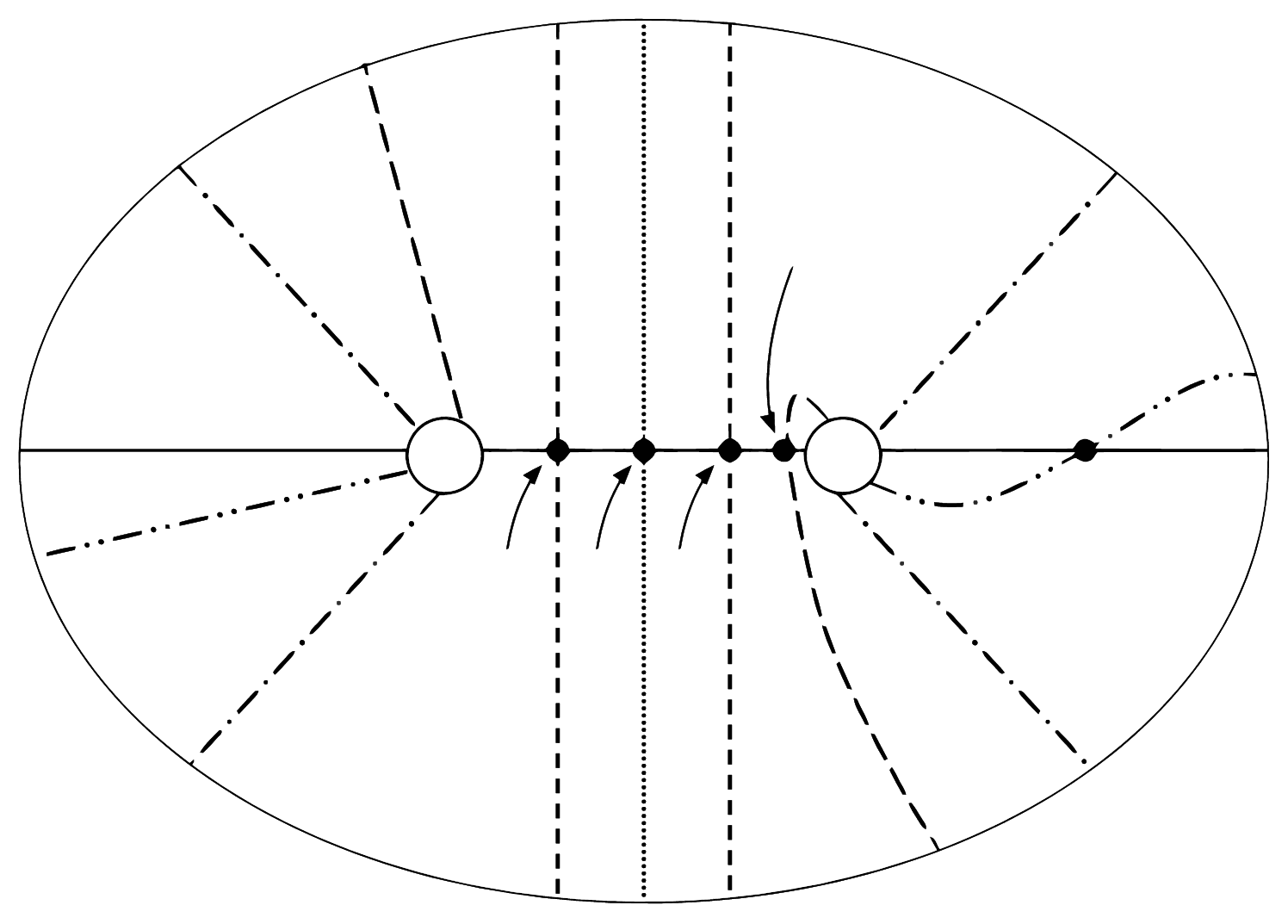}}\qquad
\subfloat[From $b B$]{
\labellist 
\small
\pinlabel* {$y_{1}$} at 340 340
\pinlabel* {$v$} at 280 340
\pinlabel* {$v'$} at 385 350
\pinlabel* {$y_{2}$} at 455 345
\pinlabel* {$y_{3}$} at 625 240
\pinlabel* {$a$} at 248 263
\pinlabel* {\reflectbox{$a$}} at 470 263
\endlabellist 
\includegraphics[height = 40mm]{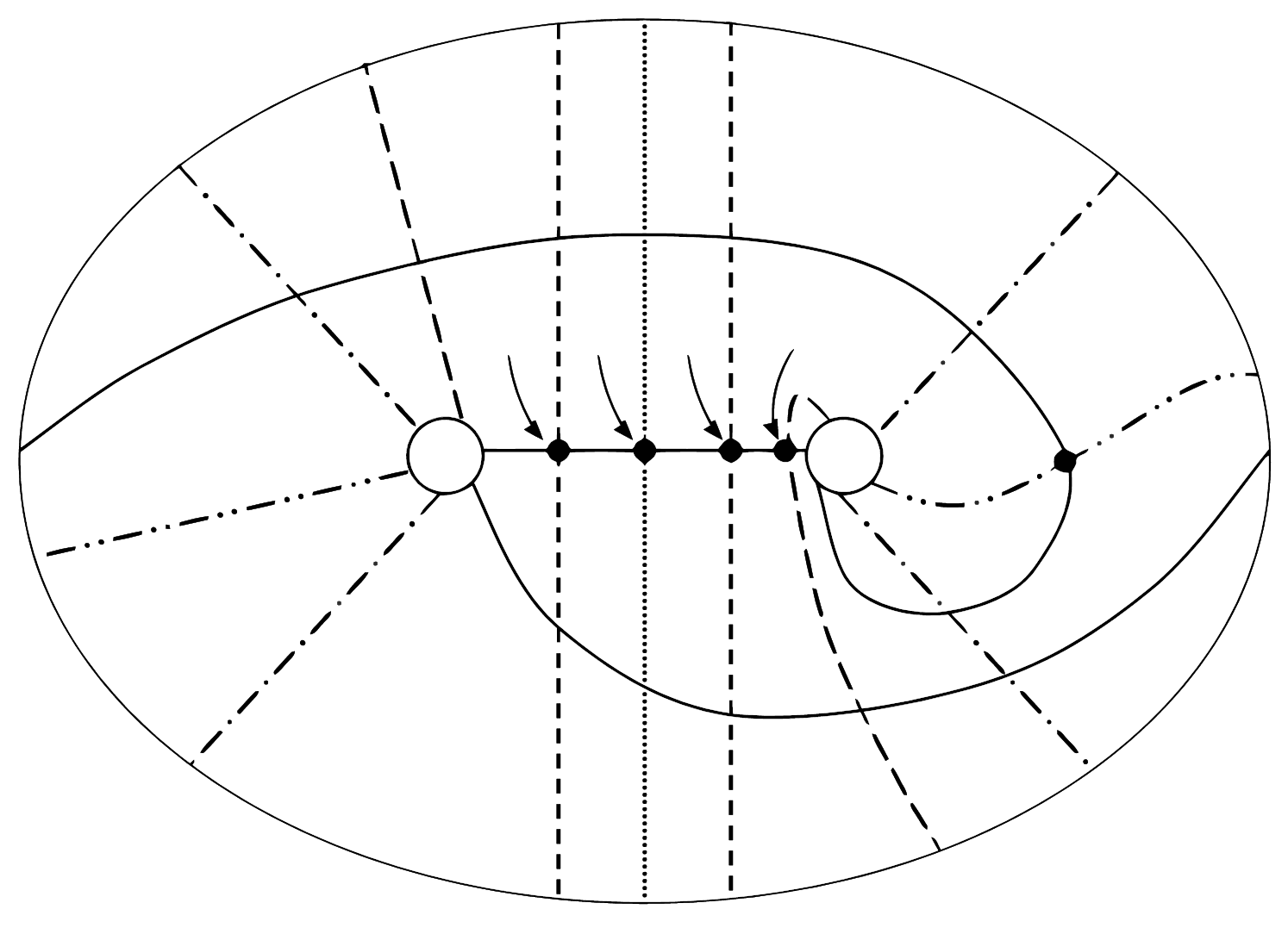}}
\caption[Heegaard diagrams covering Figure \ref{fig:move3}]{Heegaard diagrams for  $\DBCs{K}$ covering the fork diagrams in Figure \ref{fig:move3}}
\label{fig:move3hd}
\end{figure}

It is clear that only $\alpha_{2}$ is altered by the move.  To be more precise, we take a look at the Heegaard diagrams for $\DBCs{K}$ which are the branched double-covers of the fork diagrams for $b$ and $b B$.  These can be seen in Figure \ref{fig:move3hd}.

To get from the left diagram to the right, we can perform a sequence of two handleslides as shown in Figure \ref{fig:move3hs}.

\begin{figure}[h!]
\centering
\subfloat[$\ba = \ba^1\mapsto \ba^2$]{
\includegraphics[height = 40mm]{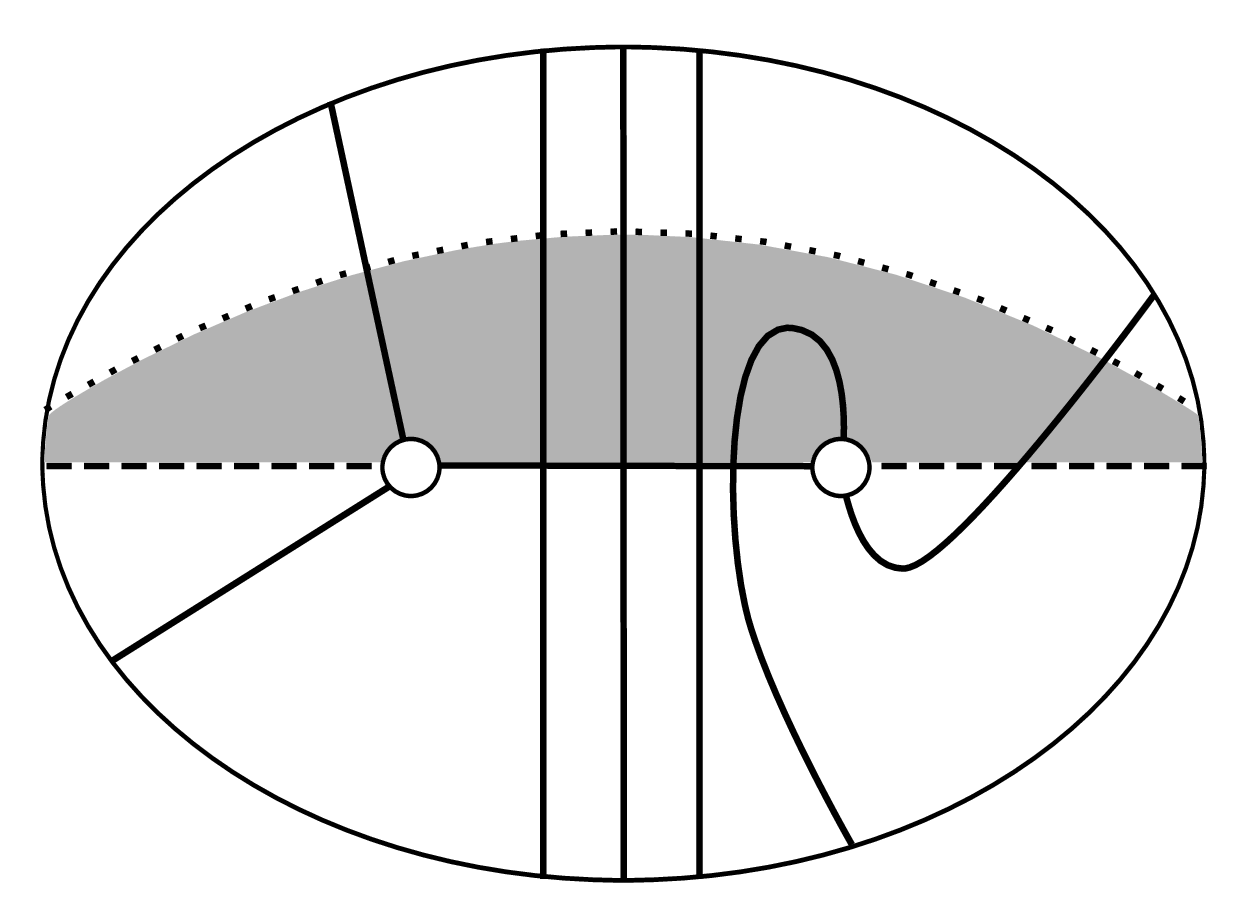}}\qquad
\subfloat[$\ba^2 \mapsto \ba^3=\ba'$]{
\includegraphics[height = 40mm]{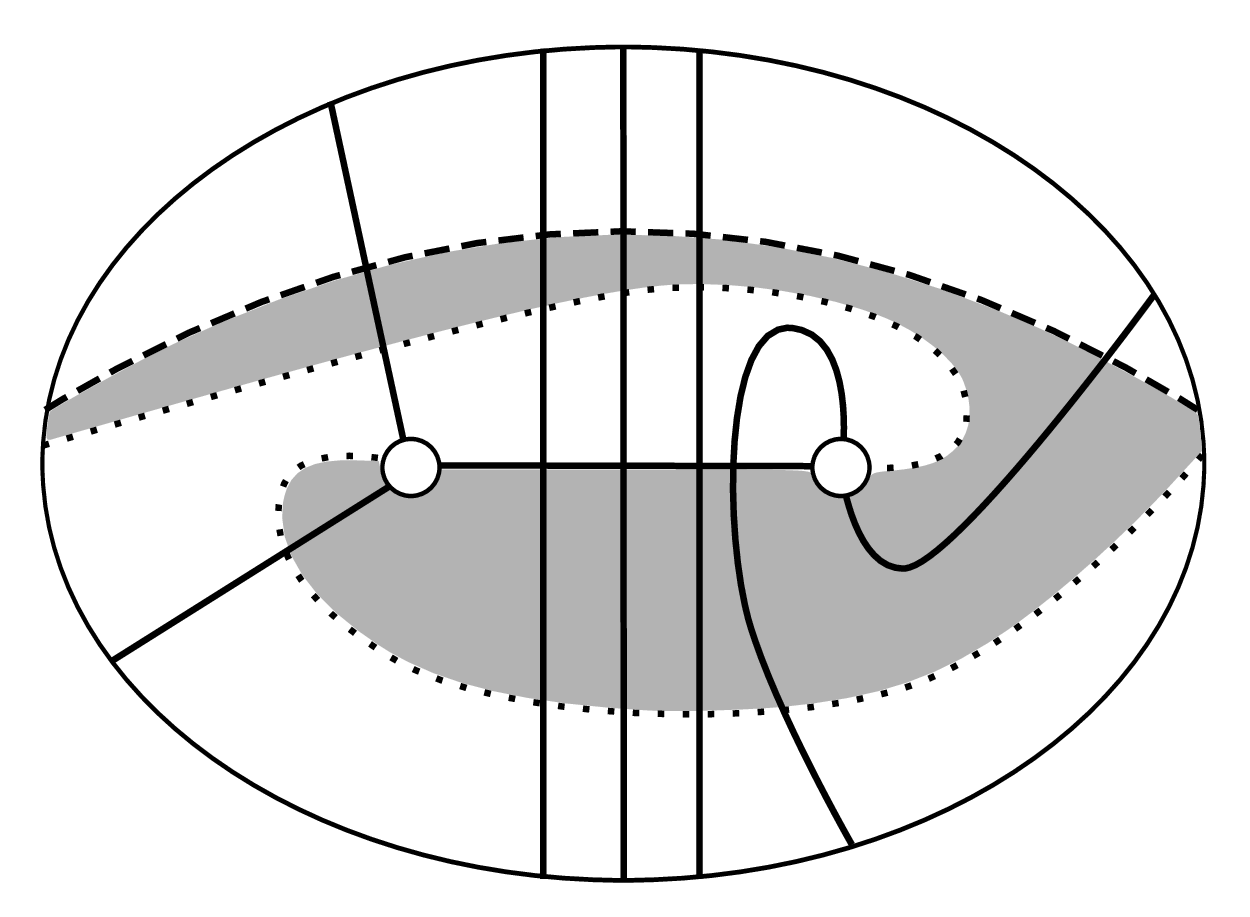}}
\caption[The handleslides induced by $b \mapsto b B$]{Two handleslides connecting Heegaard diagrams for $b$ and $b B$\label{fig:move3hs}.  In each picture, the pair of pants is shaded, the new circle is dotted, and the old one is dashed.}
\end{figure}

\begin{figure}[h!]
\centering
\subfloat[$u \mapsto w$ ]{\includegraphics[height = 24mm]{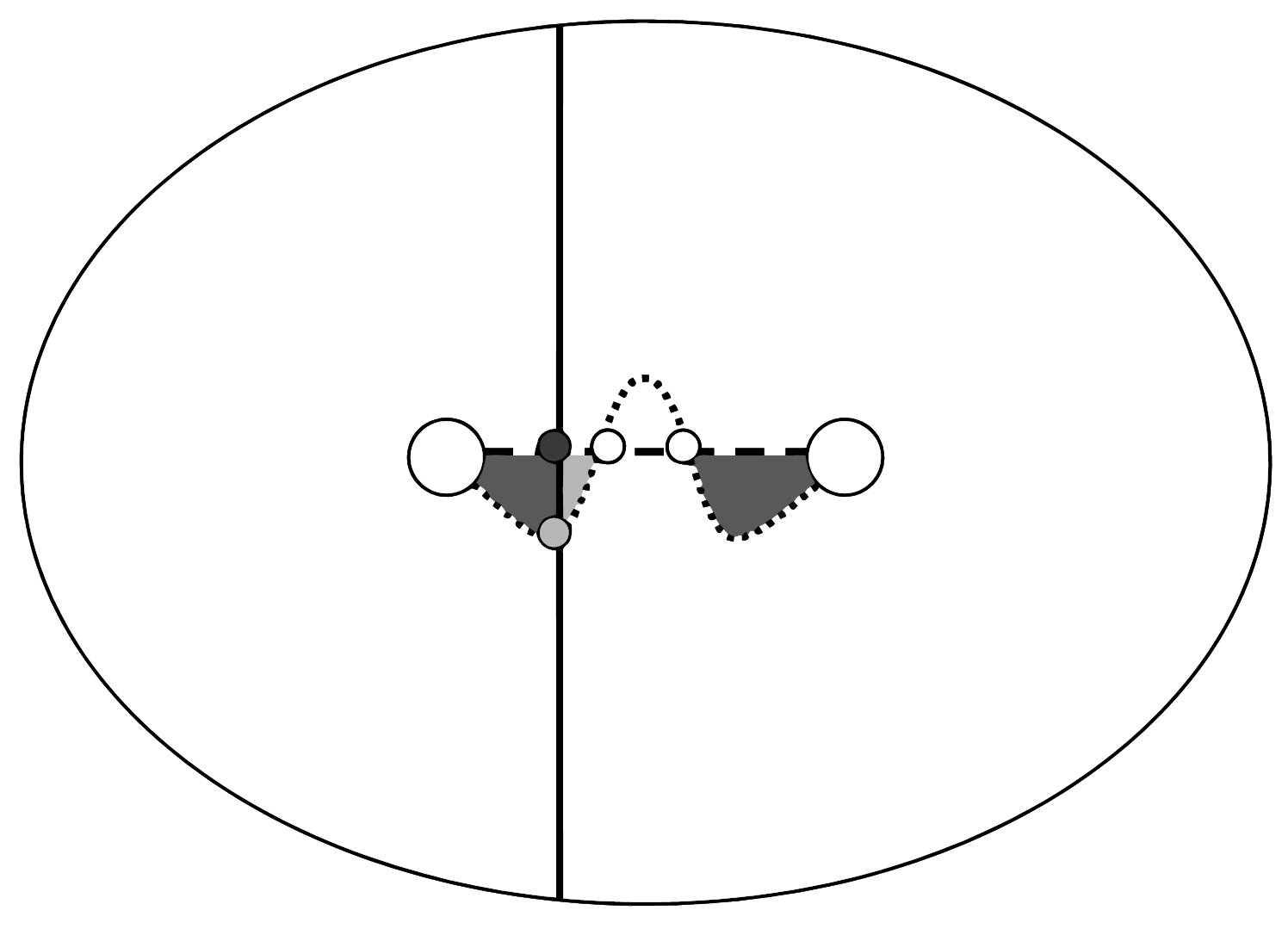}}
\subfloat[$x_{1} \mapsto s_{1}$]{\includegraphics[height = 24mm]{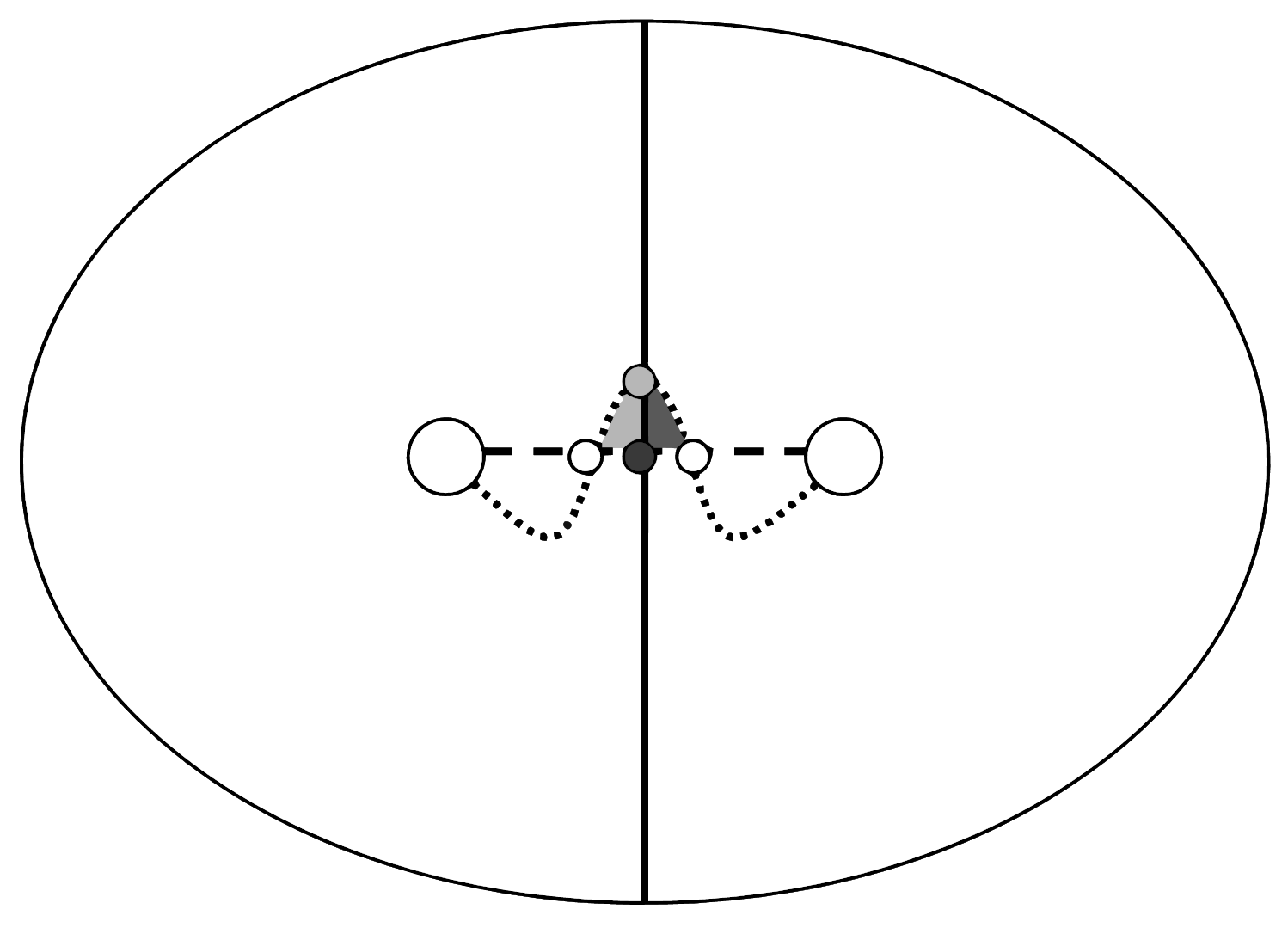}}
\subfloat[$u' \mapsto w'$]{\includegraphics[height = 24mm]{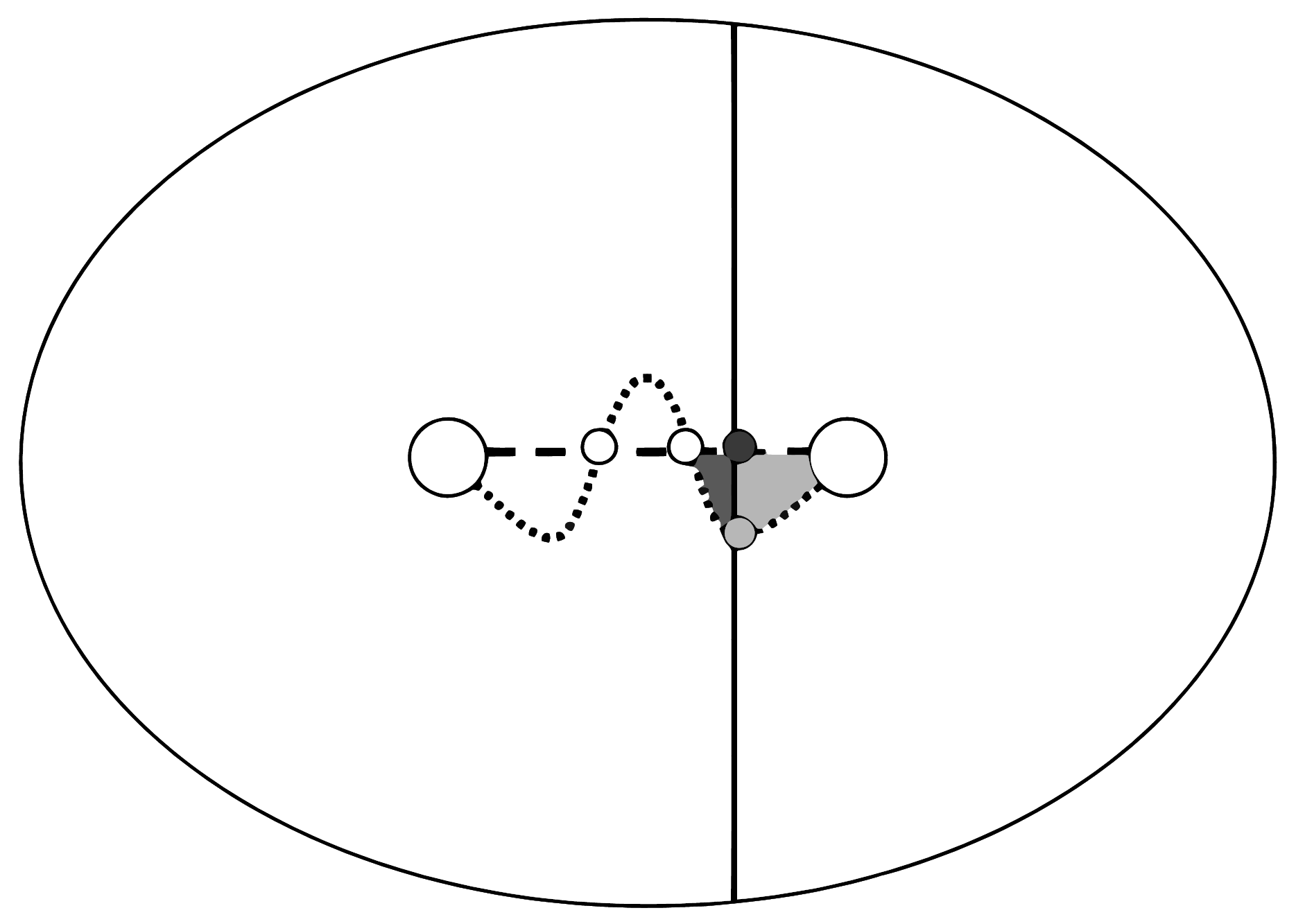}}\\
\subfloat[$x_{2} \mapsto s_{2}$]{\includegraphics[height = 24mm]{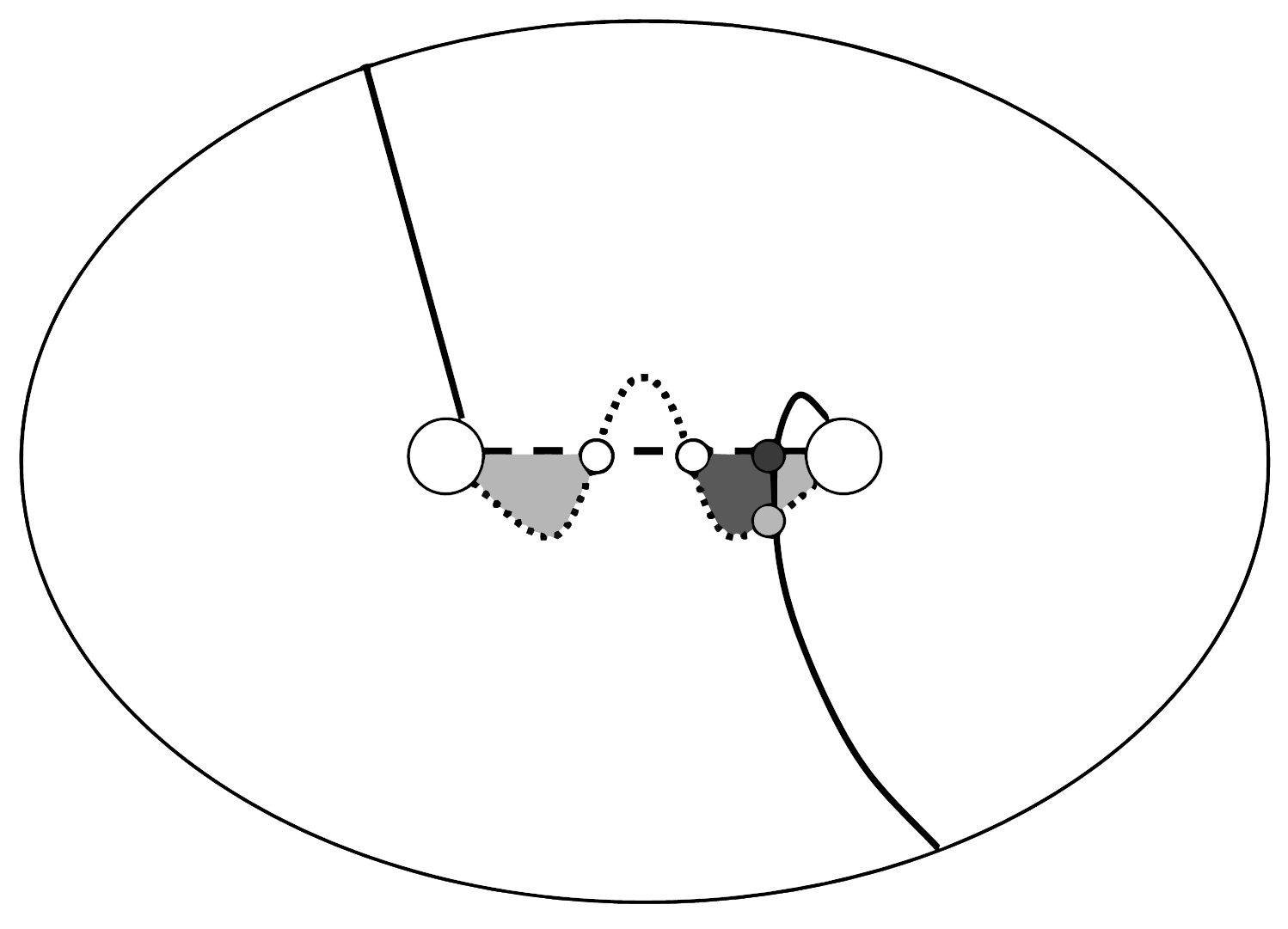}}
\subfloat[$x_{3} \mapsto s_{3}$]{\includegraphics[height = 24mm]{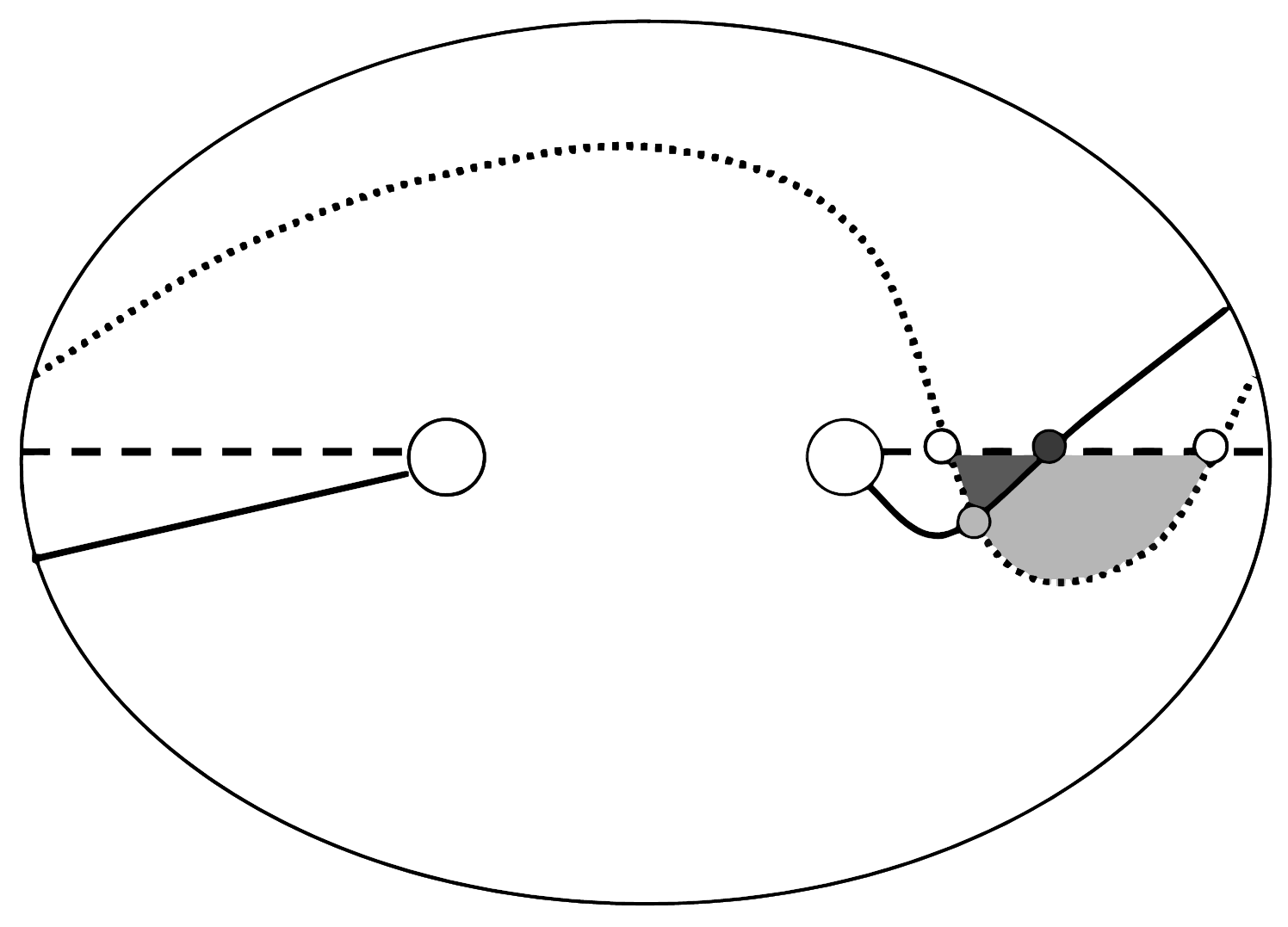}}
\caption[Local components of 3-gon domains associated to the first handleslide for  $b \mapsto b B$]{Local regions in domains of 3-gons $\psi^{+}$ (dark gray) and $\psi^{-}$ (light gray) for $g_{B}^{a}$.  White dots are components of $\thet{\ba''\ba},\thet{\ba\ba"} \in \tor{\ba} \cap \tor{\ba''}$. \label{fig:move3hs1}}
\end{figure}

We first address admissibility of the intermediate pointed Heegaard diagram $\left( \Sigma; \ba^2; \bb; +\infty \right)$.

\begin{figure}[h!]
\centering
\subfloat[$\Sigma \setminus \left( \cup_{i} \alpha_{i}^{1} \right) \setminus \left( \cup_{i} \beta_{i} \right)$]{
\labellist 
\small
\pinlabel* {$\mathcal{D}_{10}$} at 70 175
\pinlabel* {$\mathcal{D}_1$} at 120 330
\pinlabel* {$\mathcal{D}_2$} at 195 360
\pinlabel* {$\mathcal{D}_3$} at 255 370
\pinlabel* {$\mathcal{D}_4$} at 320 370
\pinlabel* {$\mathcal{D}_5$} at 535 255
\pinlabel* {$\mathcal{D}_9$} at 165 120
\pinlabel* {$\mathcal{D}_8$} at 255 120
\pinlabel* {$\mathcal{D}_{7}$} at 315 120
\pinlabel* {$\mathcal{D}_{6}$} at 390 60
\endlabellist 
\includegraphics[height = 44mm]{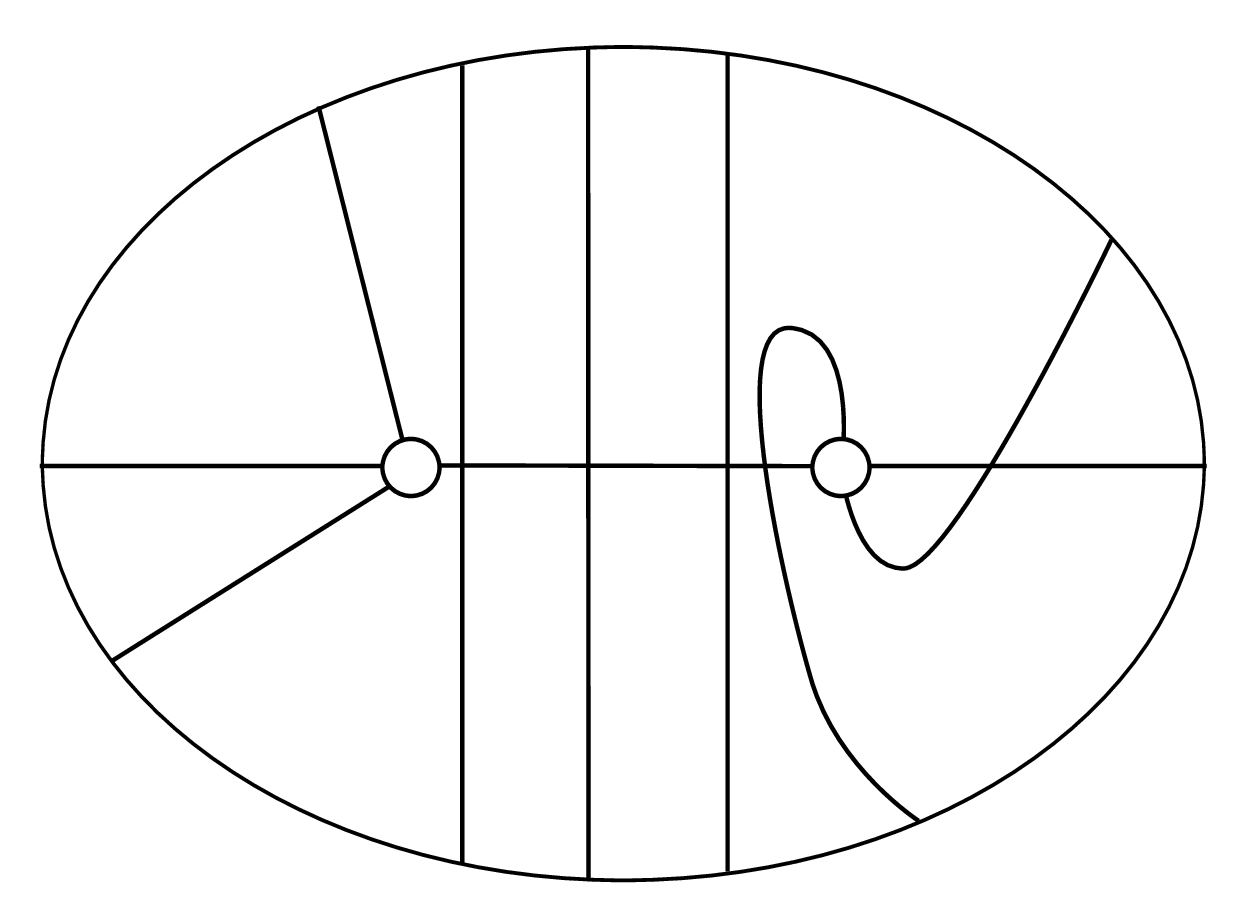}\label{fig:move3ad1}}
\subfloat[$\Sigma \setminus \left( \cup_{i} \alpha_{i}^{2} \right) \setminus \left( \cup_{i} \beta_{i} \right)$]{
\labellist 
\small
\pinlabel* {$\mathcal{D}'_{10}$} at 115 220
\pinlabel* {$\mathcal{D}'_1$} at 120 330
\pinlabel* {$\mathcal{D}'_2$} at 195 360
\pinlabel* {$\mathcal{D}'_3$} at 255 370
\pinlabel* {$\mathcal{D}'_4$} at 320 370
\pinlabel* {$\tld{\mathcal{D}}_1$} at 415 350
\pinlabel* {$\mathcal{D}'_5$} at 520 285
\pinlabel* {$\tld{\mathcal{D}}_2$} at 390 260
\pinlabel* {$\tld{\mathcal{D}}_3$} at 255 275
\pinlabel* {$\tld{\mathcal{D}}_4$} at 320 275
\pinlabel* {$\mathcal{D}'_9$} at 165 120
\pinlabel* {$\mathcal{D}'_8$} at 255 120
\pinlabel* {$\mathcal{D}'_{7}$} at 315 120
\pinlabel* {$\mathcal{D}'_{6}$} at 380 60
\endlabellist 
\includegraphics[height = 44mm]{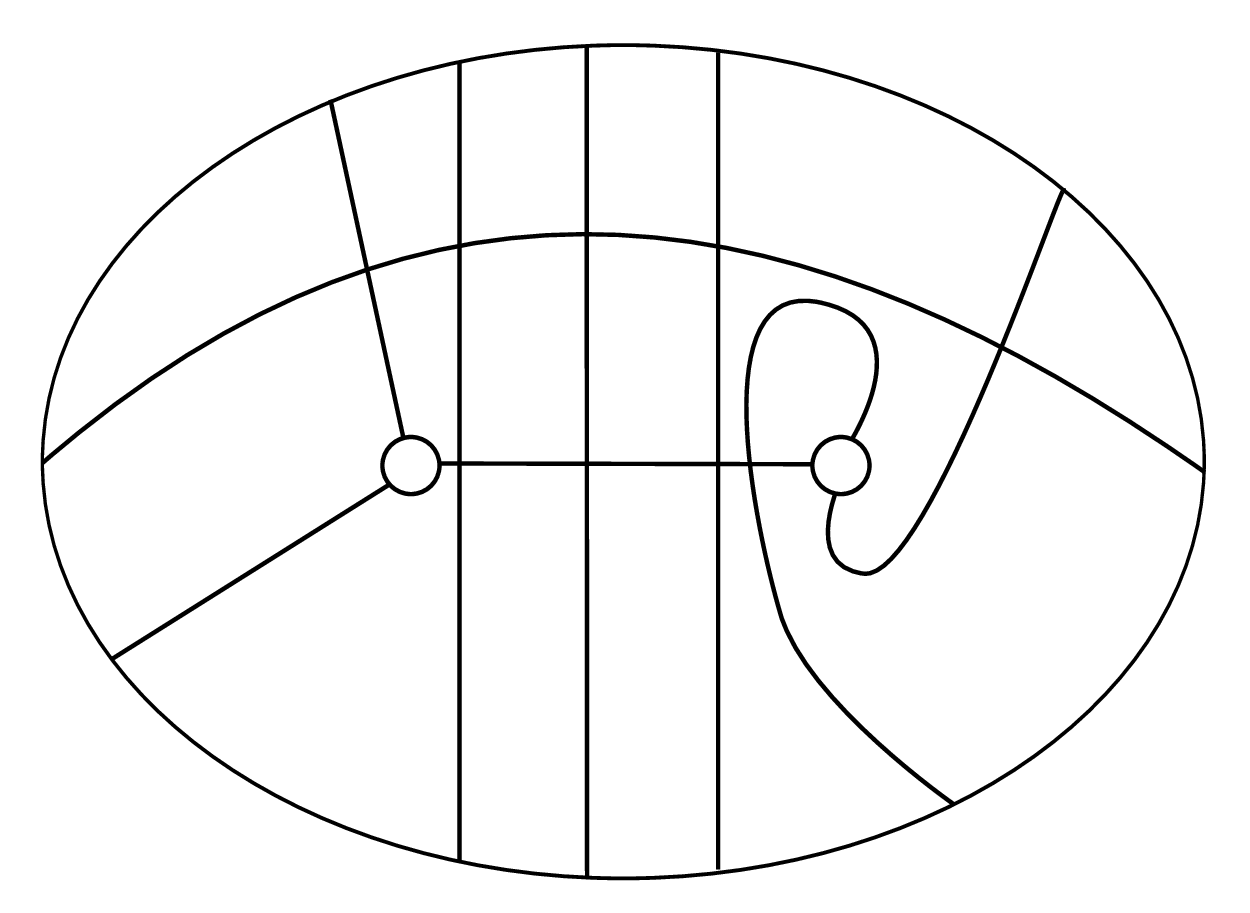}\label{fig:move3ad2}}
\caption[Domains before and after a handeslide]{Domains of on Heegaard surfaces before and after the handleslide $\ba^{1} \rightarrow \ba^2$ induced by the Birman move $b \mapsto bB$. \label{fig:move3ad}}
\end{figure}

Label the $m$ domains of $\left( \Sigma; \ba^1; \bb; +\infty \right)$ as indicated in Figure $\ref{fig:move3ad1}$, where $\mathcal{D}_{k}$ lies entirely outside of the picture for $k \geq 11$; label the $(m + 4)$ domains of $\left( \Sigma; \ba^2; \bb; +\infty \right)$ as indicated $\ref{fig:move3ad2}$, where $\mathcal{D}'_k$ corresponds to $\mathcal{D}_k$ for $k \geq 11.$  Now consider some periodic domain in $\left( \Sigma; \ba^2; \bb; +\infty \right)$
$$\mathcal{P}' = \tld{c}_1\tld{\mathcal{D}}_1 + \tld{c}_2\tld{\mathcal{D}}_2 + \tld{c}_1\tld{\mathcal{D}}_3 + \tld{c}_1\tld{\mathcal{D}}_4 + \sum_{i=1}^{m} c_i\mathcal{D}'_i.$$
Notice that
\begin{align*}
c_1 - c_{10} &= \tld{c}_1 - c_{10} = c_5 - c_9 = \tld{c}_{i} - c_i \quad \text{for} \quad 1 \leq i \leq 4 \quad \text{and}\\
\tld{c}_2 - c_9 &= \tld{c}_3 - c_8 =\tld{c}_4 - c_7 =c_{10} - c_6.
\end{align*}
As a result,
$$ \tld{c}_1 = c_1 \quad \text{and} \quad c_2 - c_9 = c_3 - c_8 = c_4 - c_7 = c_1 - c_6.$$
So, we can conclude that there is a periodic domain in $\left( \Sigma; \ba^1; \bb; +\infty \right)$ of the form
$$\mathcal{P} = \sum_{i=1}^{m} c_i\mathcal{D}_i.$$
But since $\left( \Sigma; \ba^1; \bb; +\infty \right)$ is admissible, there are both positive and negative integers among the $c_i$, and thus $\left( \Sigma; \ba^2; \bb; +\infty \right)$ is also admissible.

Let the injection $g_{B}$ act as $x_{1} x_{3} \by \mapsto y_{1}  y_{3} \by$, $x_{2} x_{3} \by \mapsto y_{2}  y_{3} \by$, $u  x_{3} \by \mapsto v  y_{3} \by$, $u' x_{3} \by \mapsto v'  y_{3} \by$, $x_{1} \bz \mapsto y_{1} \bz$, $x_{2} \bz \mapsto y_{2} \bz$, $u \bz \mapsto v \bz$, and $u' \bz \mapsto v' \bz$.  Note that $\bz $ is an $(n-1)$-tuple whose component on the $\alpha_{2}$ arc is not shown in the local fork diagram.

When viewed as a function on $\Ta \cap \Tb$, $g_{B}$ is a composition $g_{B}^{\ba,2} \circ g_{B}^{\ba,1}$ of triangle injections corresponding to the two handleslides.  Local regions in domains of 3-gons for $g_{B}^{\ba,1}$ (resp. $g_{B}^{\ba,2}$) can be seen in Figure \ref{fig:move3hs1} (resp. \ref{fig:move3hs2}).  One can verify that if two of these regions appear in the domain of a 3-gon associated with $g_{B}^{\ba,1}$, then there are neighborhoods of these regions which map to disjoint neighborhoods in the fork diagram downstairs (and likewise for those associated with $g_{B}^{\ba,2}$).  As a result, all 3-gons presented here avoid the anti-diagonal $\AD$.

\begin{figure}[h!]
\centering
\subfloat[$w \mapsto v$]{\includegraphics[height = 24mm]{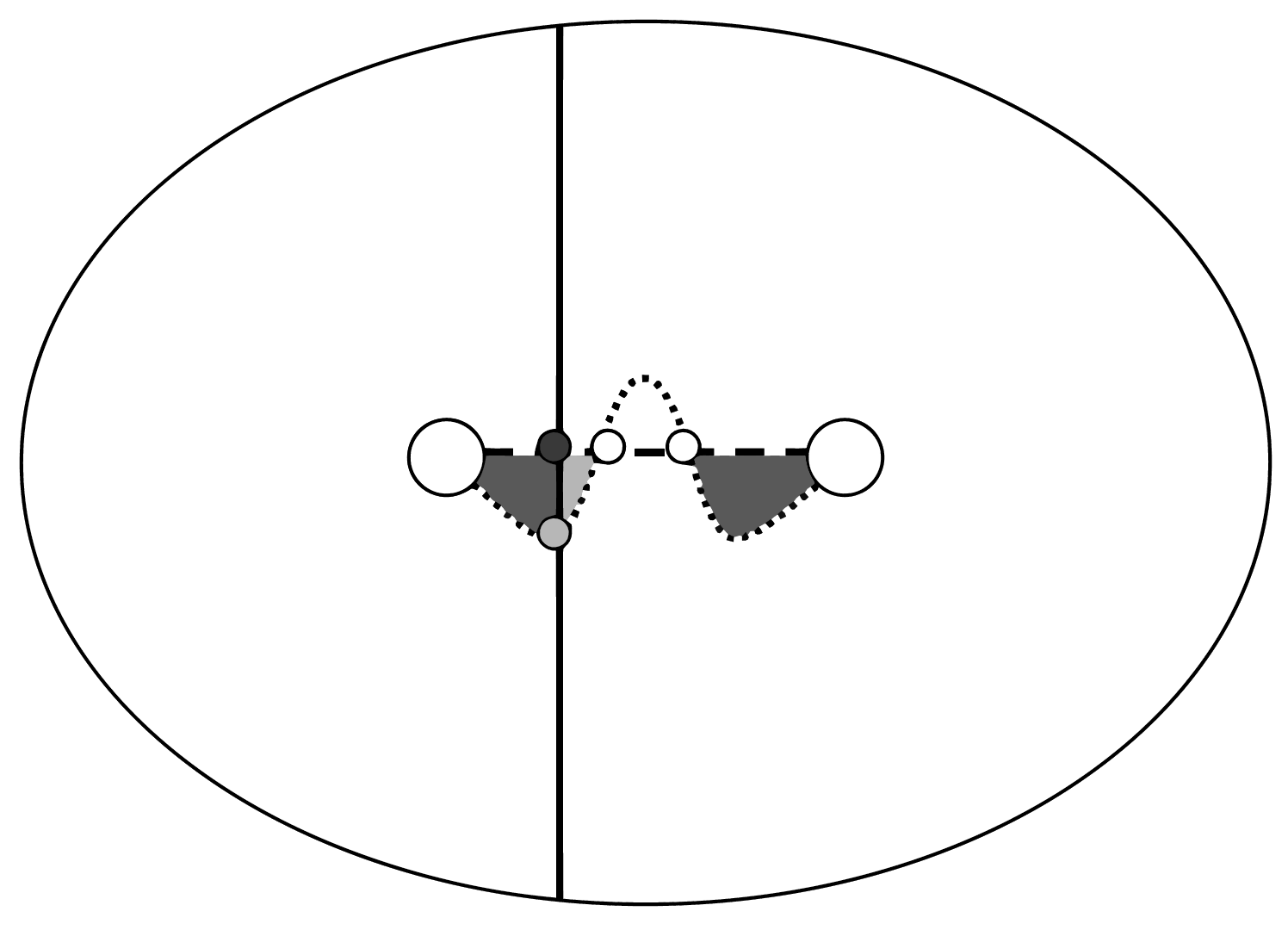}}
\subfloat[$s_{1} \mapsto y_{1}$]{\includegraphics[height = 24mm]{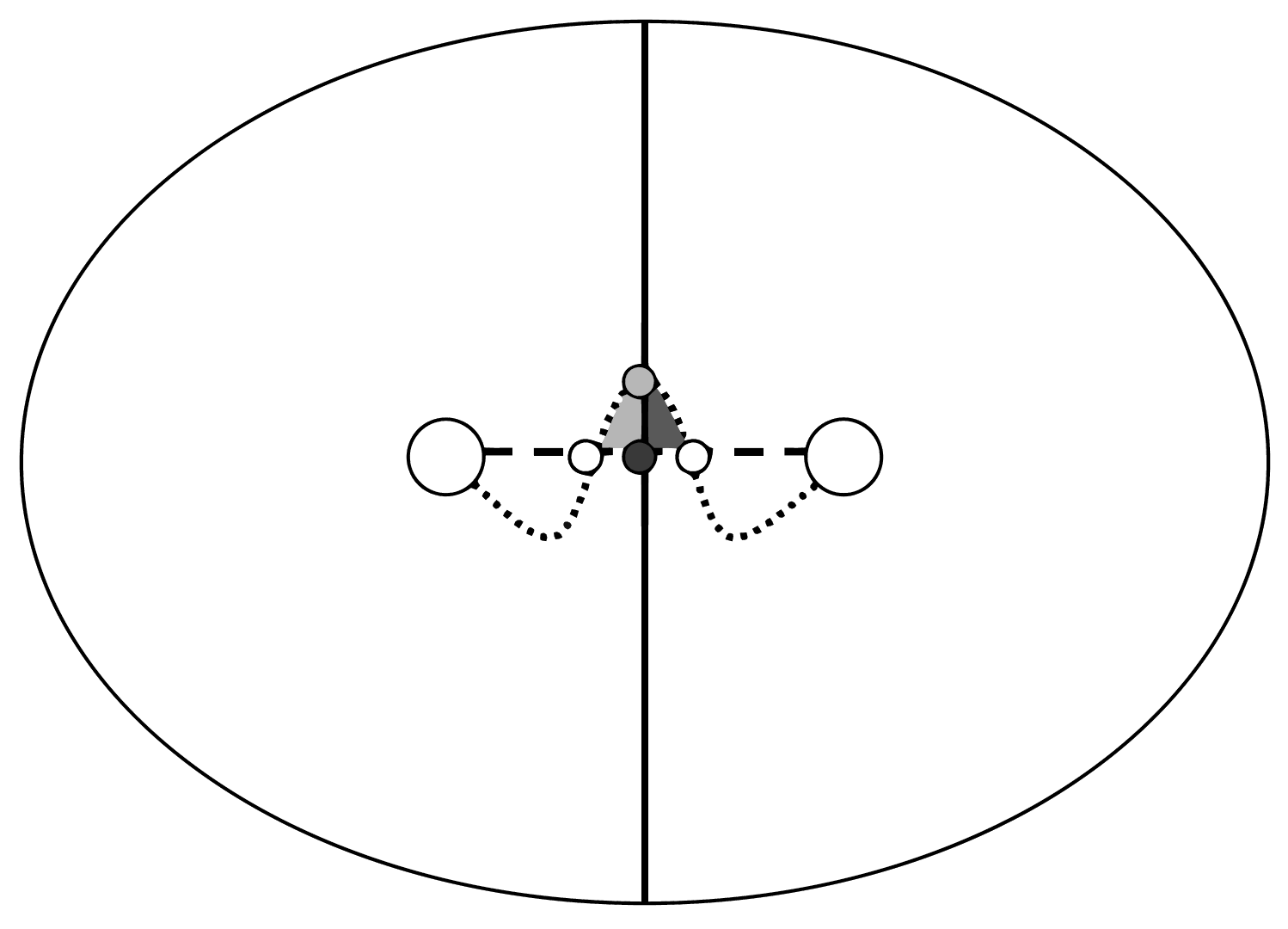}}
\subfloat[$w' \mapsto v'$]{\includegraphics[height = 24mm]{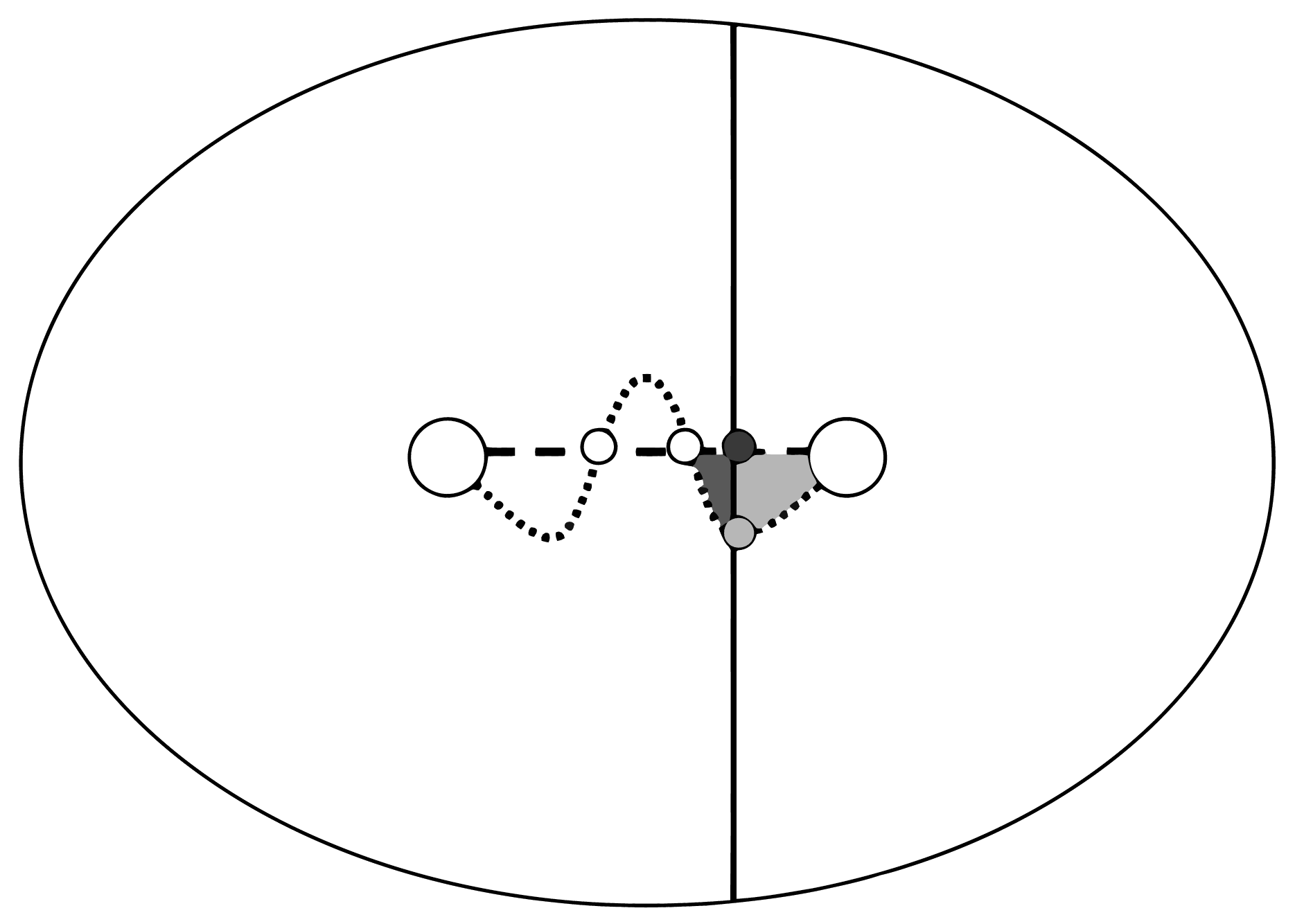}}\\
\subfloat[$s_{2} \mapsto y_{2}$]{\includegraphics[height = 24mm]{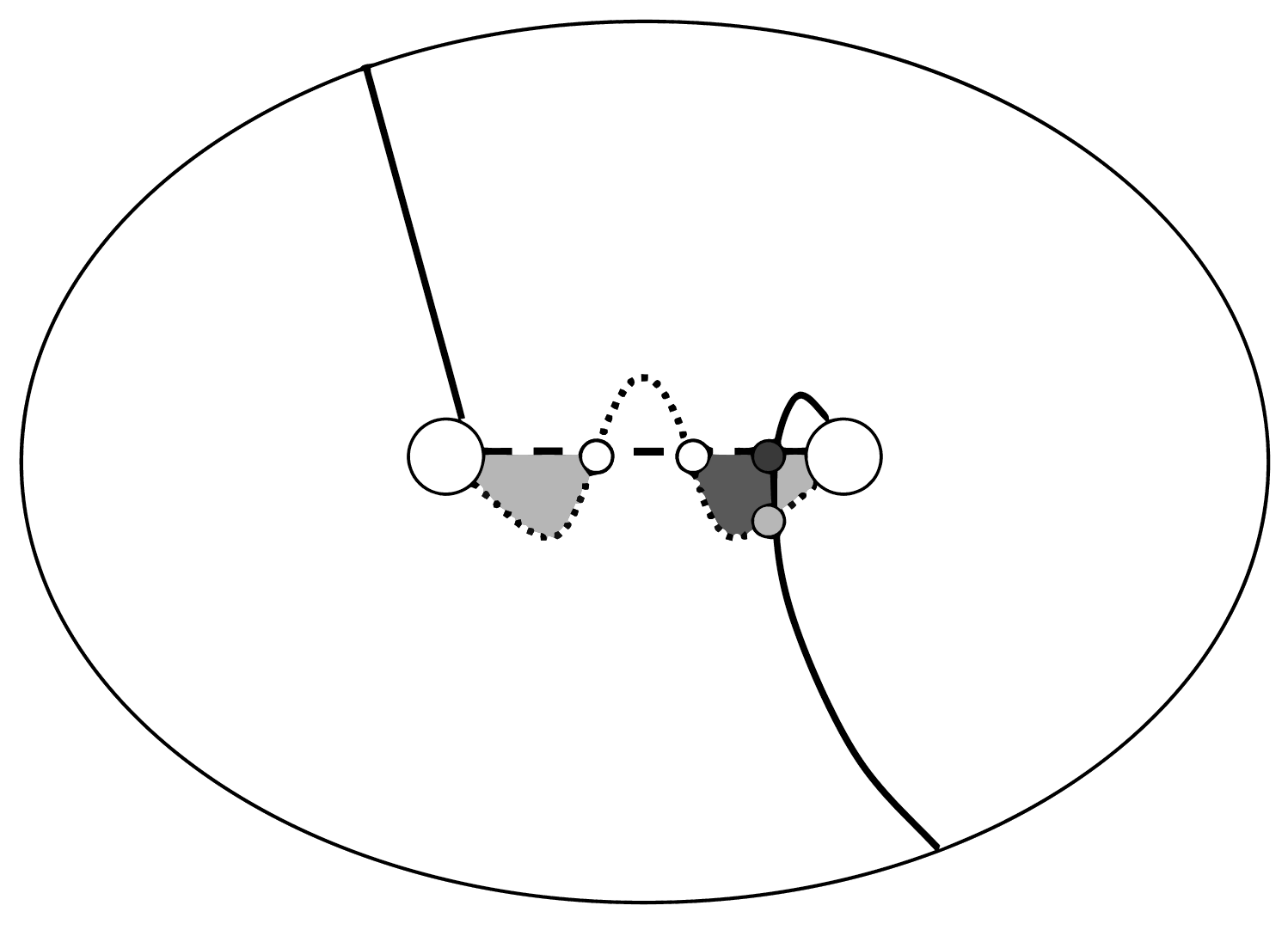}}
\subfloat[$s_{3} \mapsto y_{3}$]{\includegraphics[height = 24mm]{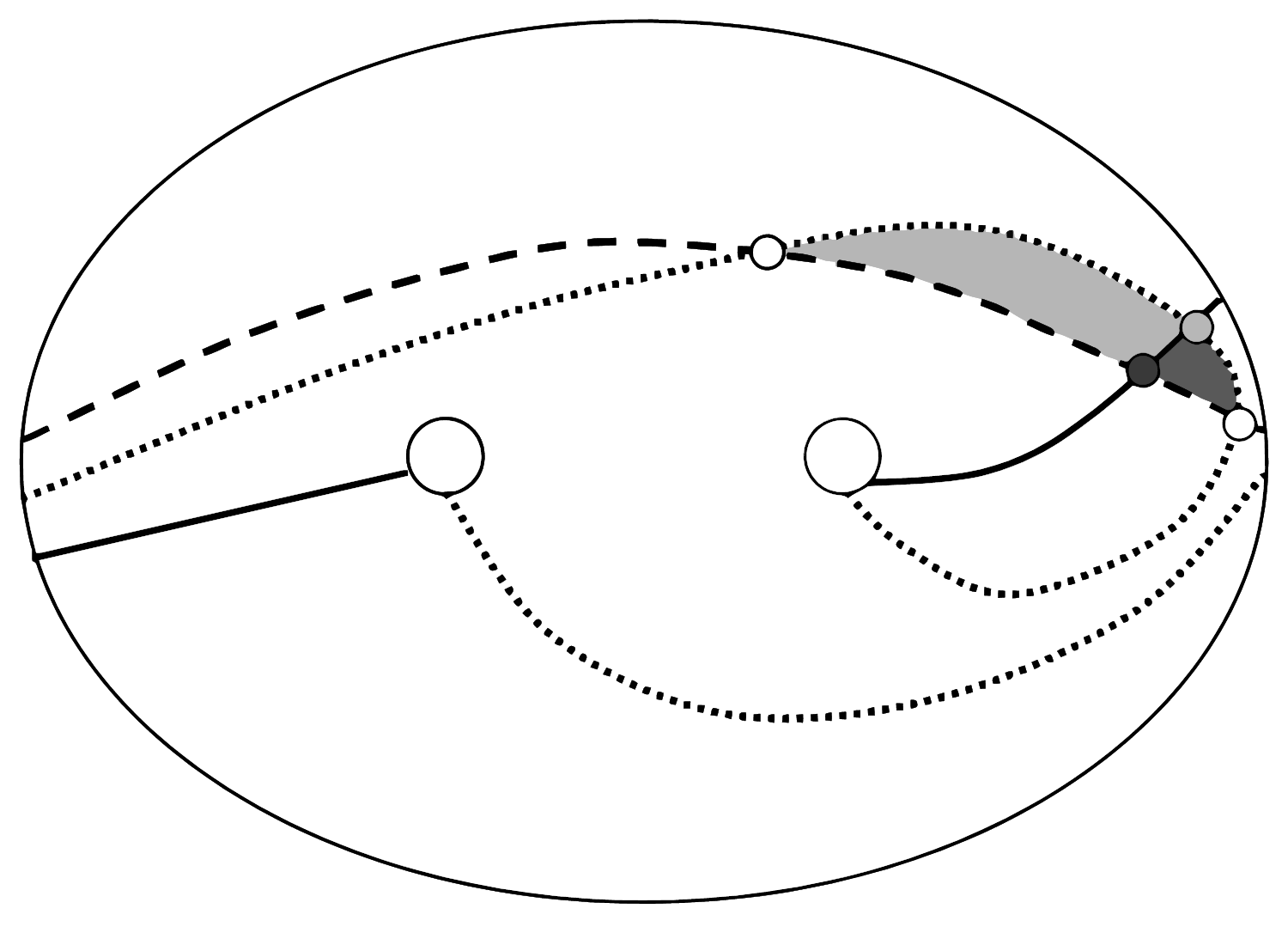}}
\caption[Local components of 3-gon domains associated to the second handleslide for  $b \mapsto b B$]{Local regions in domains of 3-gons $\psi^{+}$ (dark gray) and $\psi^{-}$ (light gray) for $g_{B}^{b}$.  White dots are components of $\thet{\ba'\ba''},\thet{\ba"\ba'} \in \tor{\ba''} \cap \tor{\ba'}$.\label{fig:move3hs2}}
\end{figure}

\begin{figure}[h!]
\centering
\subfloat[$x_{1}$]{\includegraphics[height = 20mm]{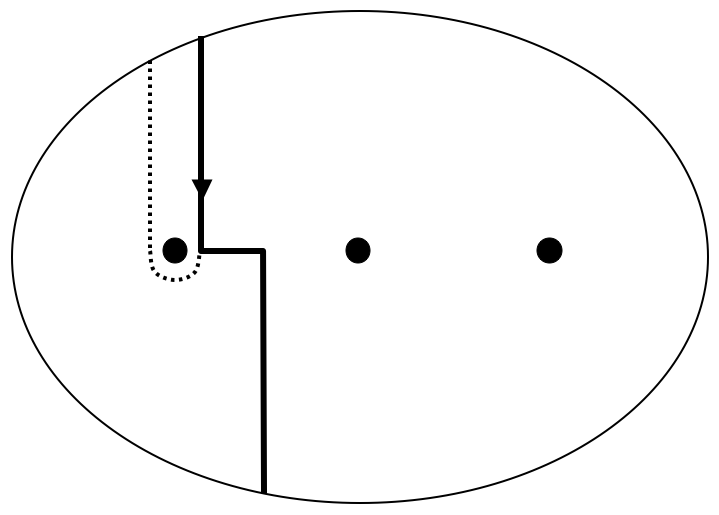}}\quad
\subfloat[$y_{1}$]{\includegraphics[height = 20mm]{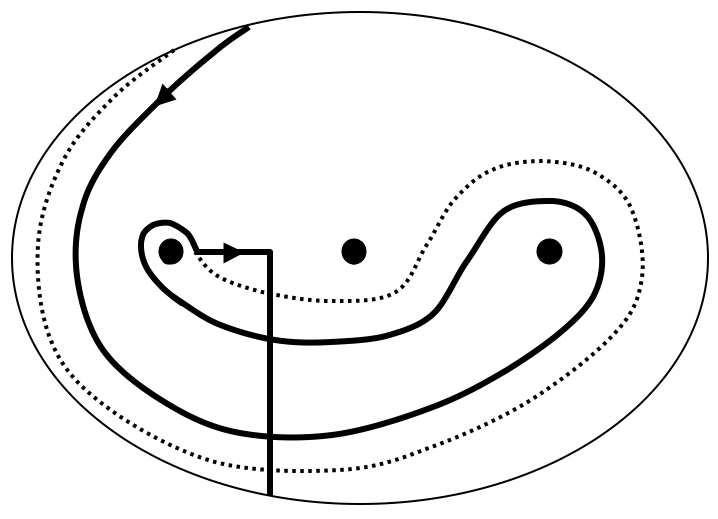}}\quad
\subfloat[$u$]{\includegraphics[height = 20mm]{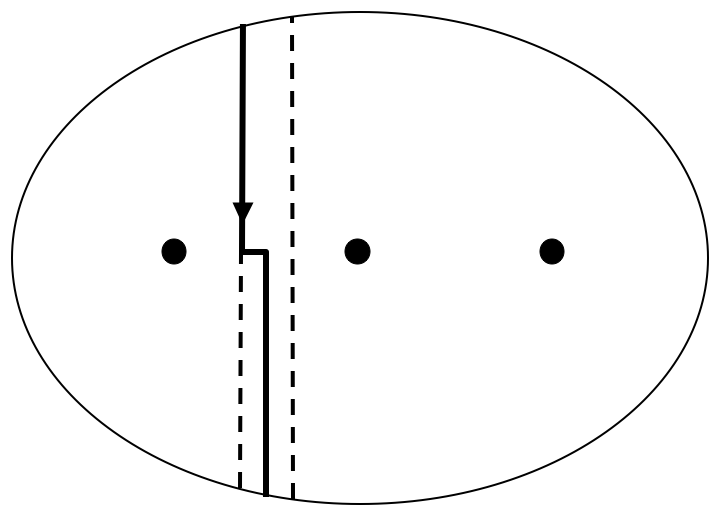}}\quad
\subfloat[$v$]{\includegraphics[height = 20mm]{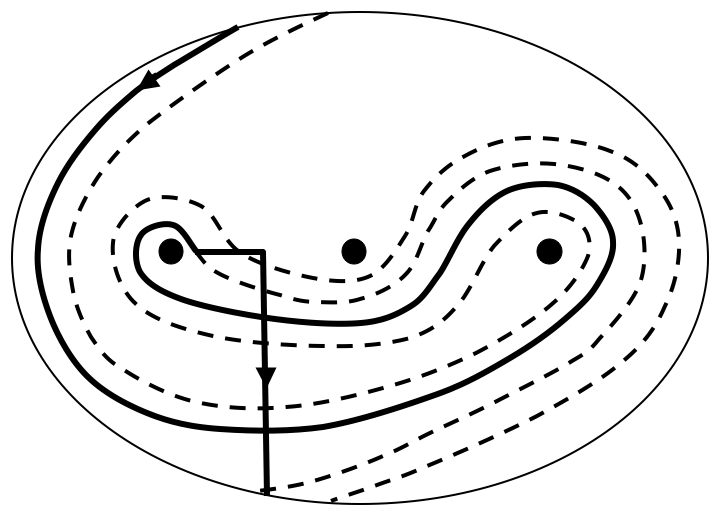}}\\
\subfloat[$x_{2}$]{\includegraphics[height = 20mm]{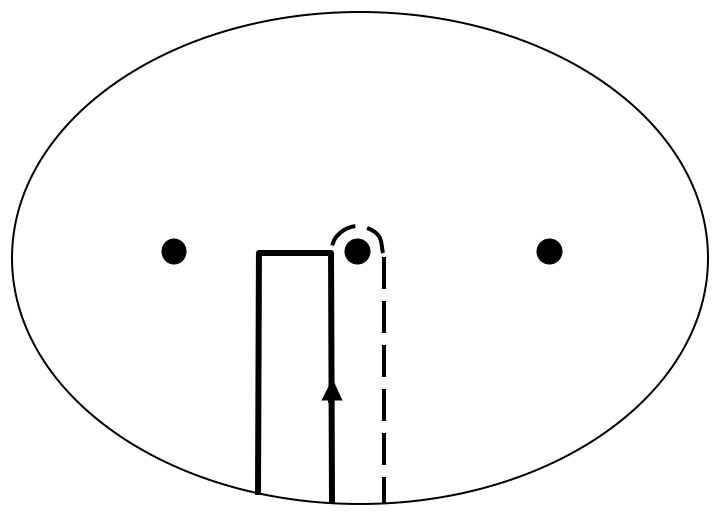}}\quad
\subfloat[$y_{2}$]{\includegraphics[height = 20mm]{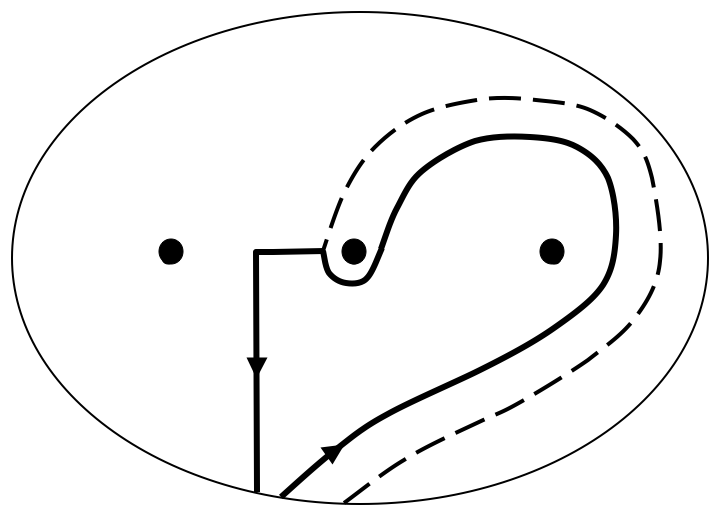}}\quad
\subfloat[$x_{3}$]{\includegraphics[height = 20mm]{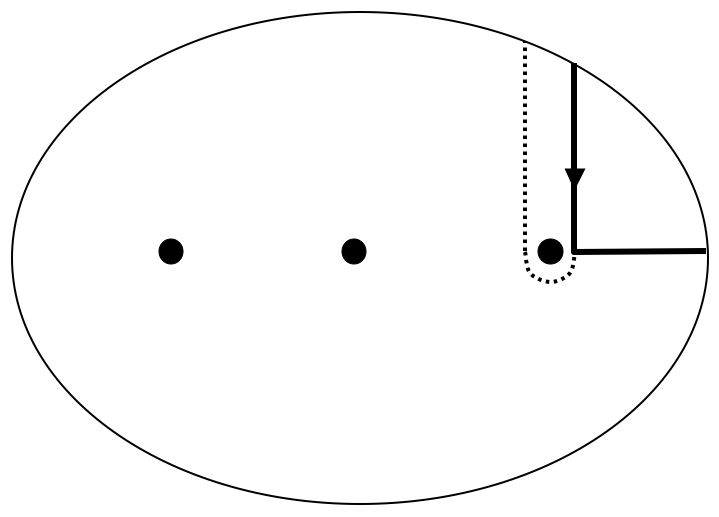}}\quad
\subfloat[$y_{3}$]{\includegraphics[height = 20mm]{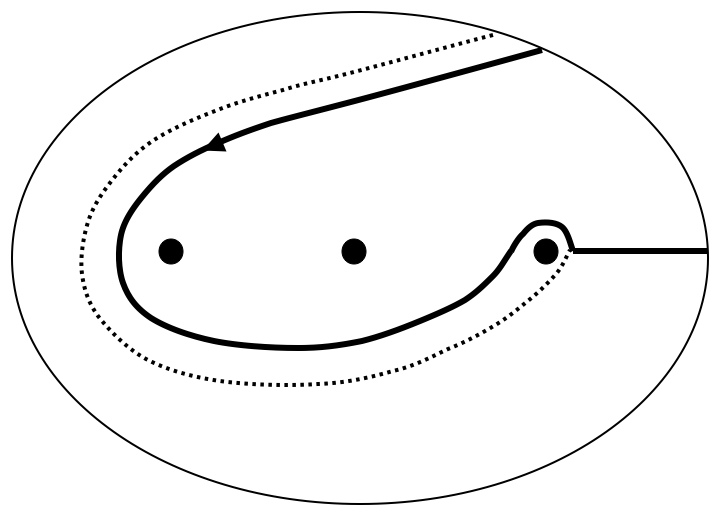}}
\caption[Grading loops associated to $b \mapsto b B$]{Loops associated to elements of $\Zcal$ affected by $b \mapsto b B$\label{fig:move3loops}}
\end{figure}

We see from Figure \ref{fig:move3loops} that $Q( y_{1} y_{3} \by ) = Q( x_{1} x_{3} \by  ) + 3$, $Q( v y_{3} \by ) = Q( u x_{3} \by  ) + 3$, $Q( y_{1} \bz ) = Q( x_{1} \bz  ) + 1$, $Q( y_{2} y_{3} \by ) = Q( x_{2} x_{3} \by  ) + 3$, $Q( y_{2} \bz ) = Q( x_{2} \bz  ) + 1$, $Q( v \bz ) = Q( u \bz  ) + 1$, and $P(g_{B}(\bw)) =  P(\bw)$ for all $\bw \in \G.$

Furthermore, one can check that $T( y_{1} y_{3} \by ) = T( x_{1} x_{3} \by  ) + 2$, $T( v y_{3} \by ) = T( u x_{3} \by  ) + 2$, $T( y_{1} \bz ) = T( x_{1} \bz  )$, $T( y_{2} y_{3} \by ) = T( x_{2} x_{3} \by  ) + 2$, $T( y_{2} \bz ) = T( x_{2} \bz  )$, and $T( v \bz ) = T( u \bz  )$.  This move preserves $n$ and $w$, but increases $\epsilon$ by 4.  Thus $s_{R}(b B) = s_{R}(b) + 1$ and $R(g_{B}(\bw)) = R(\bw)$ for all $\bw \in \G$.  The proof associated to the move $b \mapsto b B^{-1}$ is analogous.

\subsubsection{$b \in B_{2n} \leftrightarrow b \sigma_{2n} \in B_{2n + 2}$}\label{subsec:stab}
Fork diagrams before and after stabilization can be seem in Figure \ref{fig:movestab}, and the induced Heegaard diagrams in Figure \ref{fig:movestabhd}.  

\begin{figure}[h]
\centering
\subfloat[Local diagram for $b \in \B{2n}$]{
\labellist 
\small
\pinlabel* {$x_{2n}$} at 210 300
\endlabellist 
\includegraphics[height = 40mm]{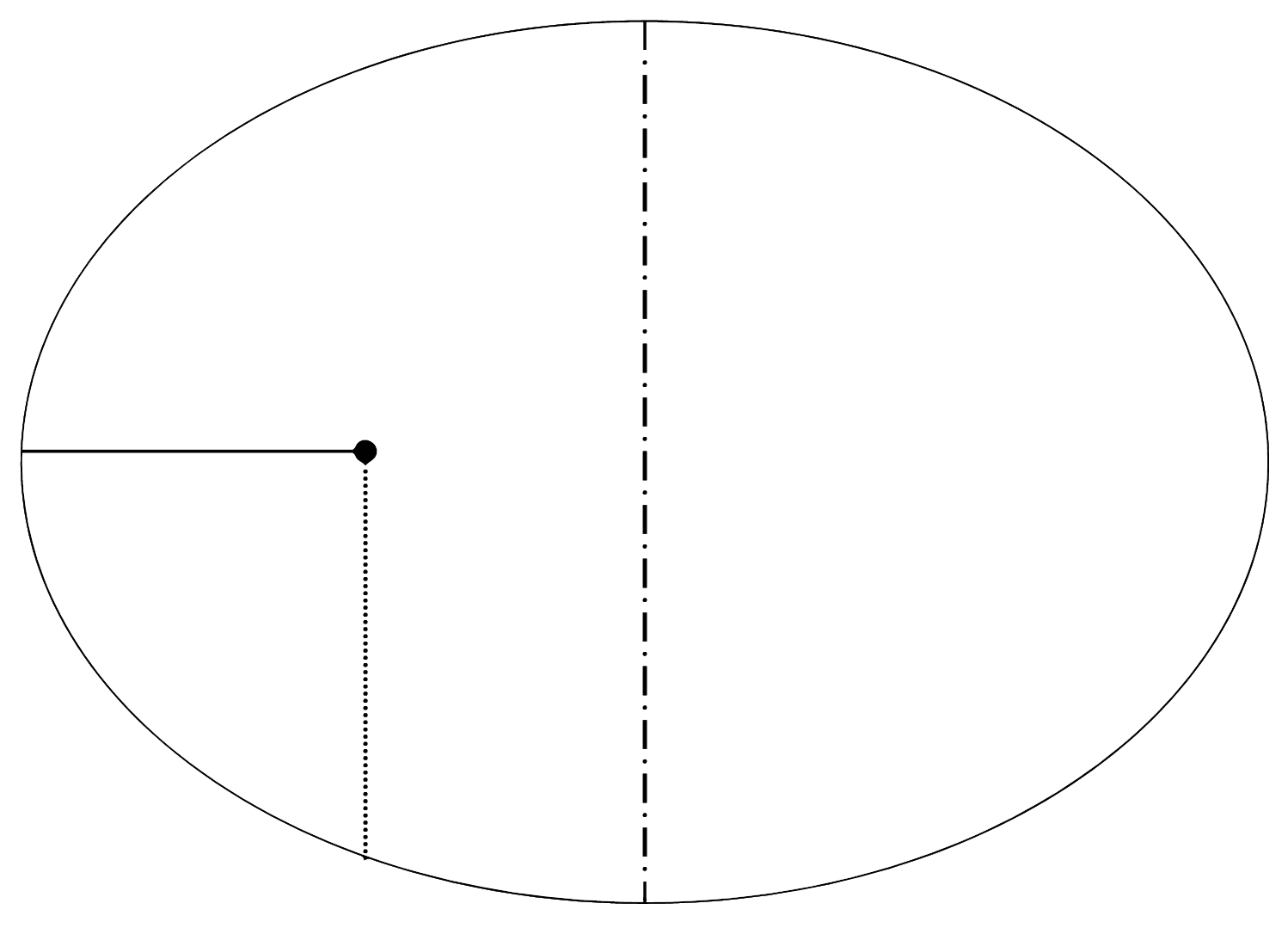}}\qquad
\subfloat[Local diagram for $b \sigma_{2n} \in \B{2n + 2}$]{
\labellist
\small
\pinlabel* {$y_{2n}$} at 215 230
\pinlabel* {$y_{2n+1}$} at 450 300
\pinlabel* {$y_{2n+2}$} at 640 230
\endlabellist 
\includegraphics[height = 40mm]{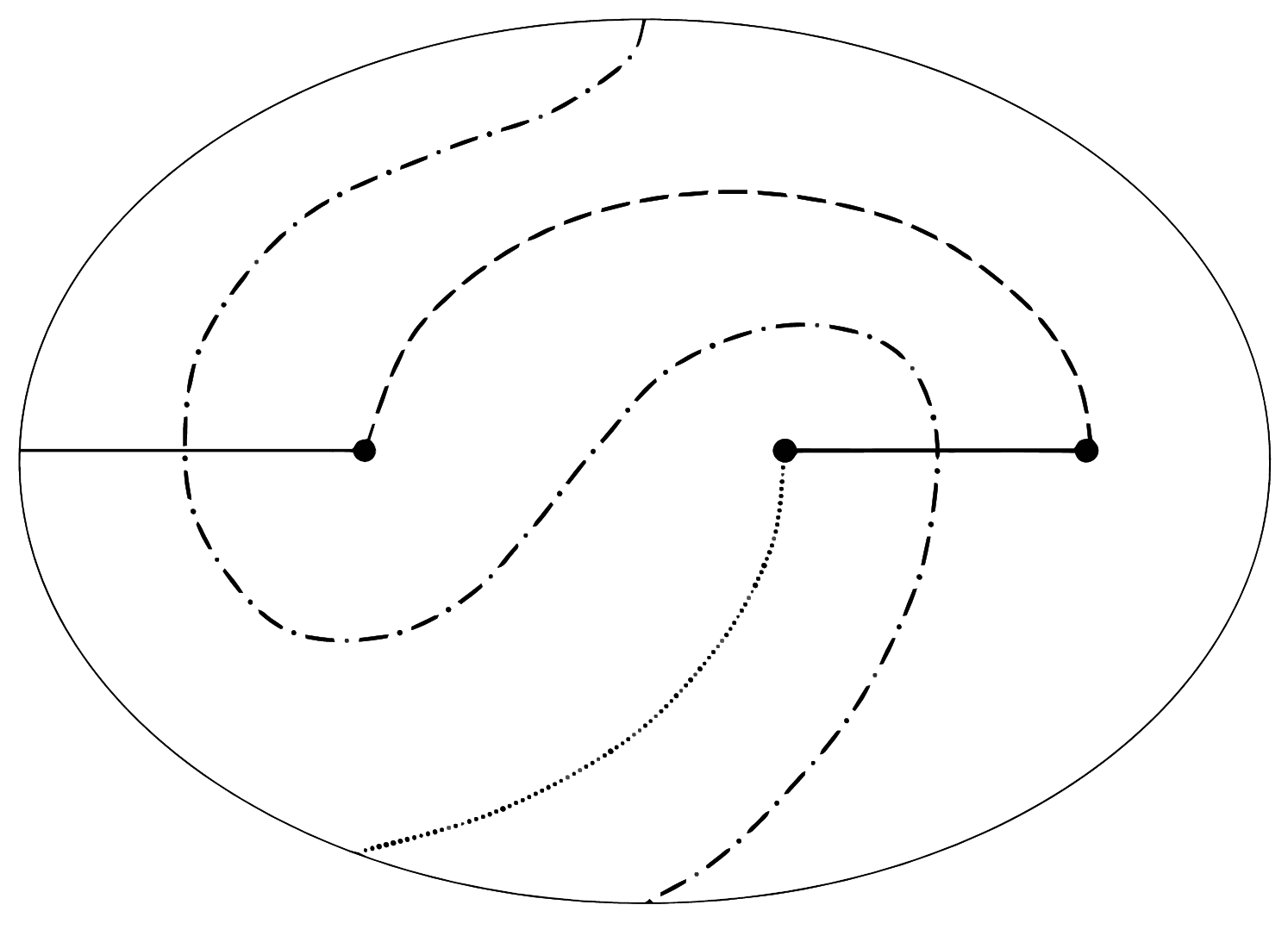}}
\caption{Fork diagrams associated to $b \in \B{2n} \mapsto b \sigma_{2n} \in \B{2n + 2}$ \label{fig:movestab}}
\end{figure}

\begin{figure}[h]
\centering
\subfloat[From $b \in \B{2n}$]{
\labellist 
\small
\pinlabel* {$x_{2n}$} at 405 290
\endlabellist 
\includegraphics[height = 40mm]{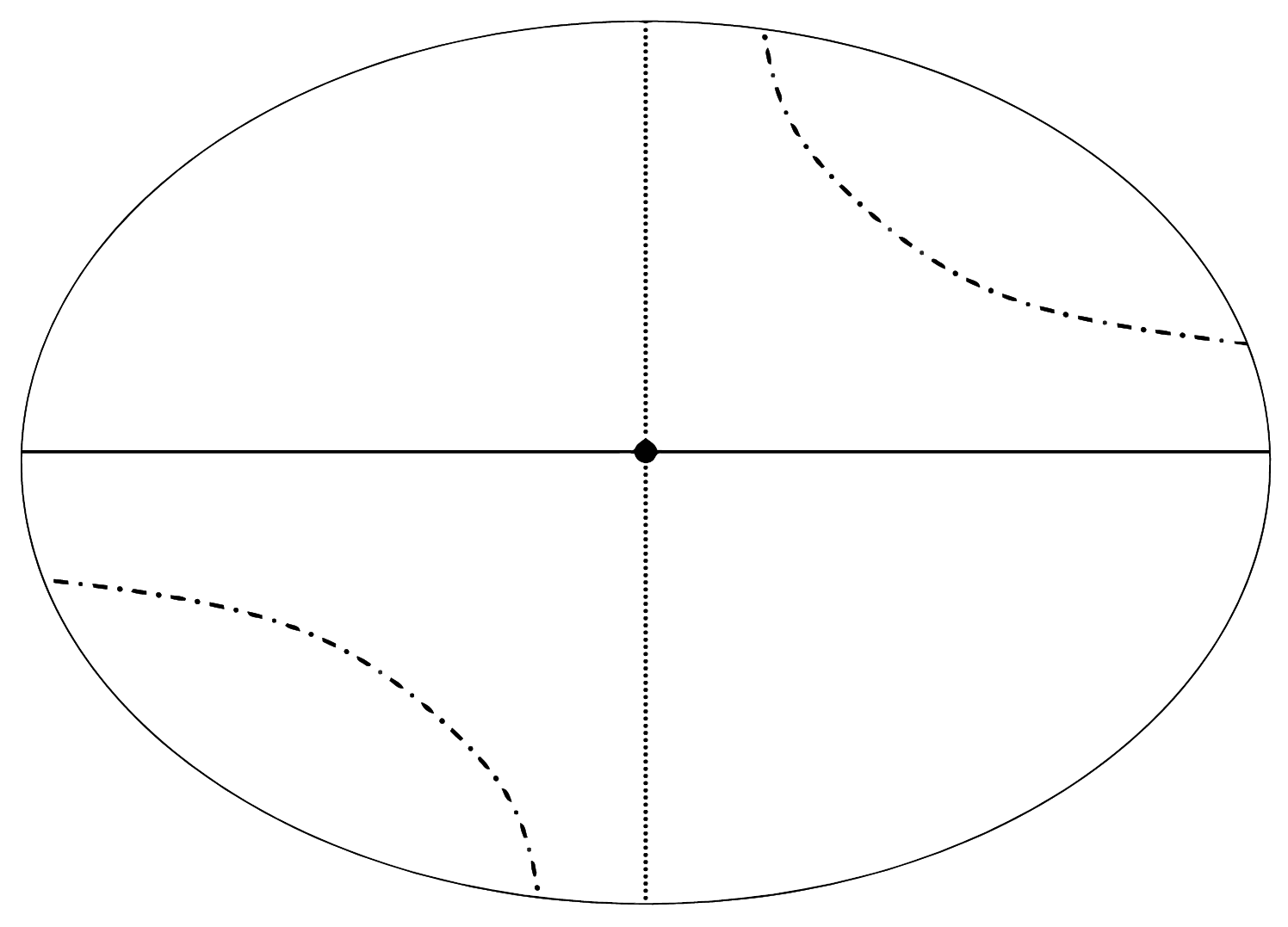}}\qquad
\subfloat[From $b \sigma_{2n} \in \B{2n + 2}$]{
\labellist 
\small
\pinlabel* {$y_{2n}$} at 565 290
\pinlabel* {$y_{2n+2}$} at 470 350
\pinlabel* {$y_{2n+1}$} at 300 355
\pinlabel* {$a$} at 250 263
\pinlabel* {\reflectbox{$a$}} at 470 263
\endlabellist 
\includegraphics[height = 40mm]{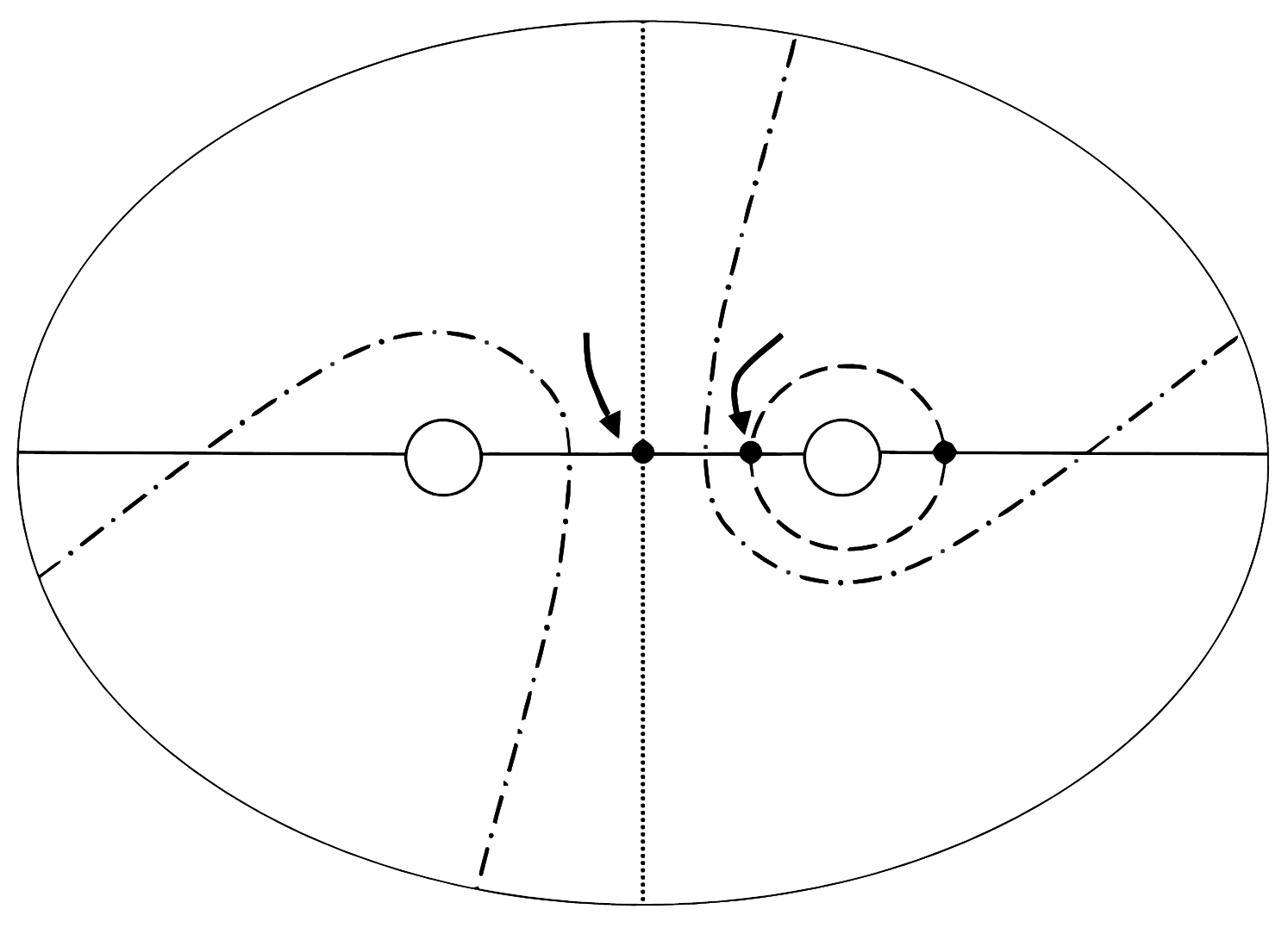}}
\caption[Heegaard diagrams covering Figure \ref{fig:movestab}]{Heegaard diagrams for  $\DBCs{K}$ covering the fork diagrams in Figure \ref{fig:movestab}}
\label{fig:movestabhd}
\end{figure}

The stabilization braid move corresponds to a Heegaard diagram stabilization followed by several handleslides, which can be seen in Figure \ref{fig:movestabhs}.  The destabilization braid move induces the inverse of this sequence of Heegaard moves.

\begin{figure}[h]
\centering
\subfloat[$\bb = \bb^1 \mapsto \bb^2$]{
\includegraphics[height = 24mm]{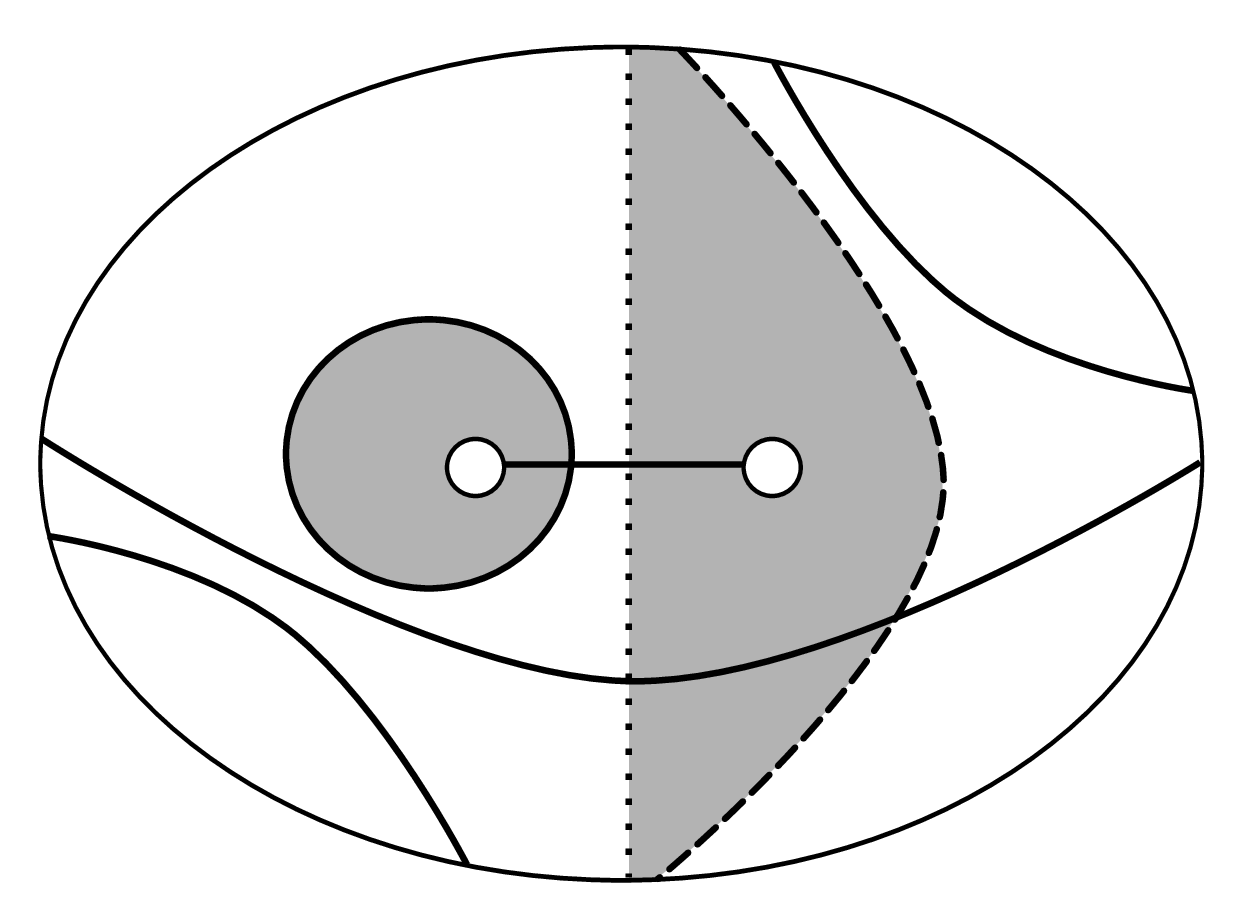}}
\subfloat[$\ba \mapsto \ba'$]{
\includegraphics[height = 24mm]{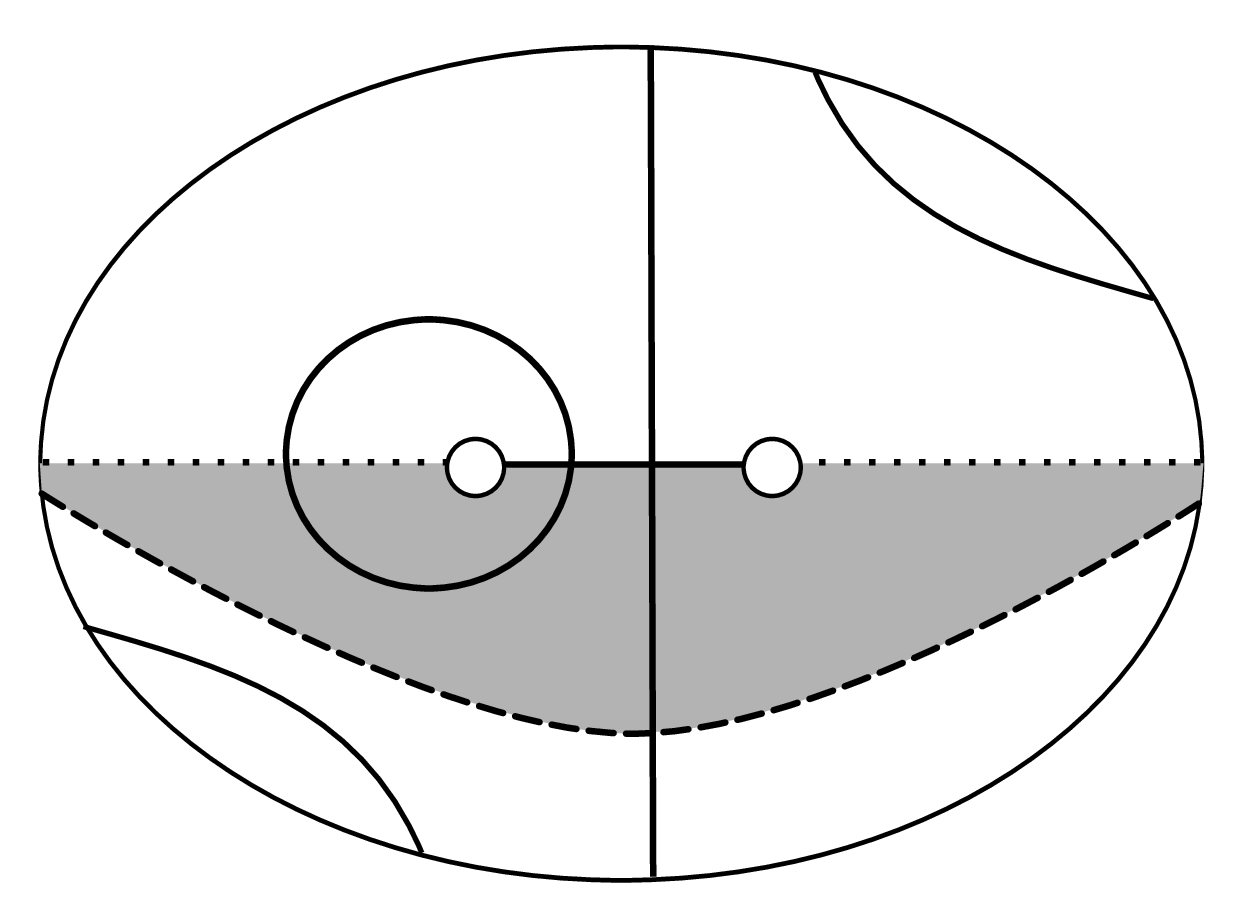}}
\subfloat[$\bb^2 \mapsto \bb^3$]{
\includegraphics[height = 24mm]{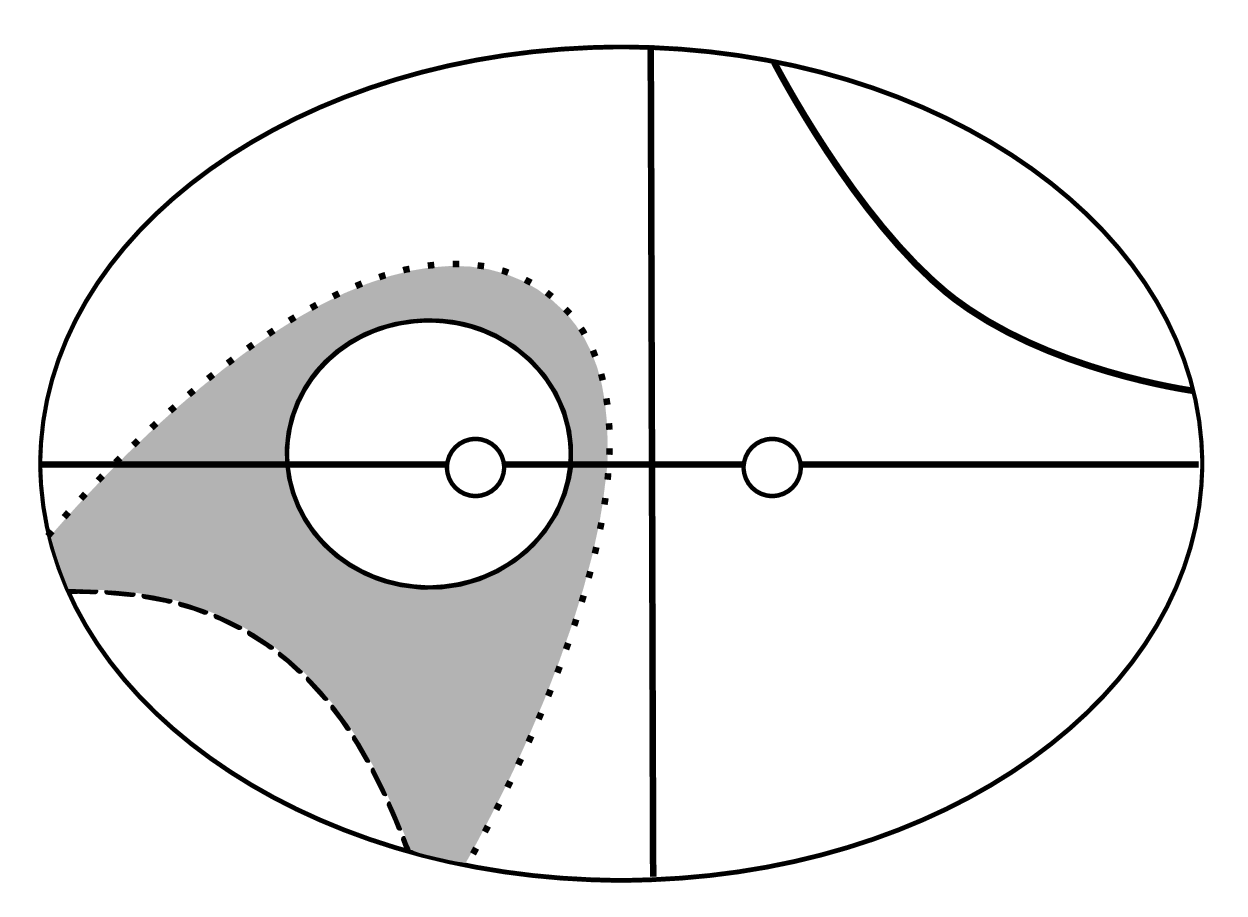}}
\subfloat[$\bb^3 \mapsto \bb^4 = \bb'$]{
\includegraphics[height = 24mm]{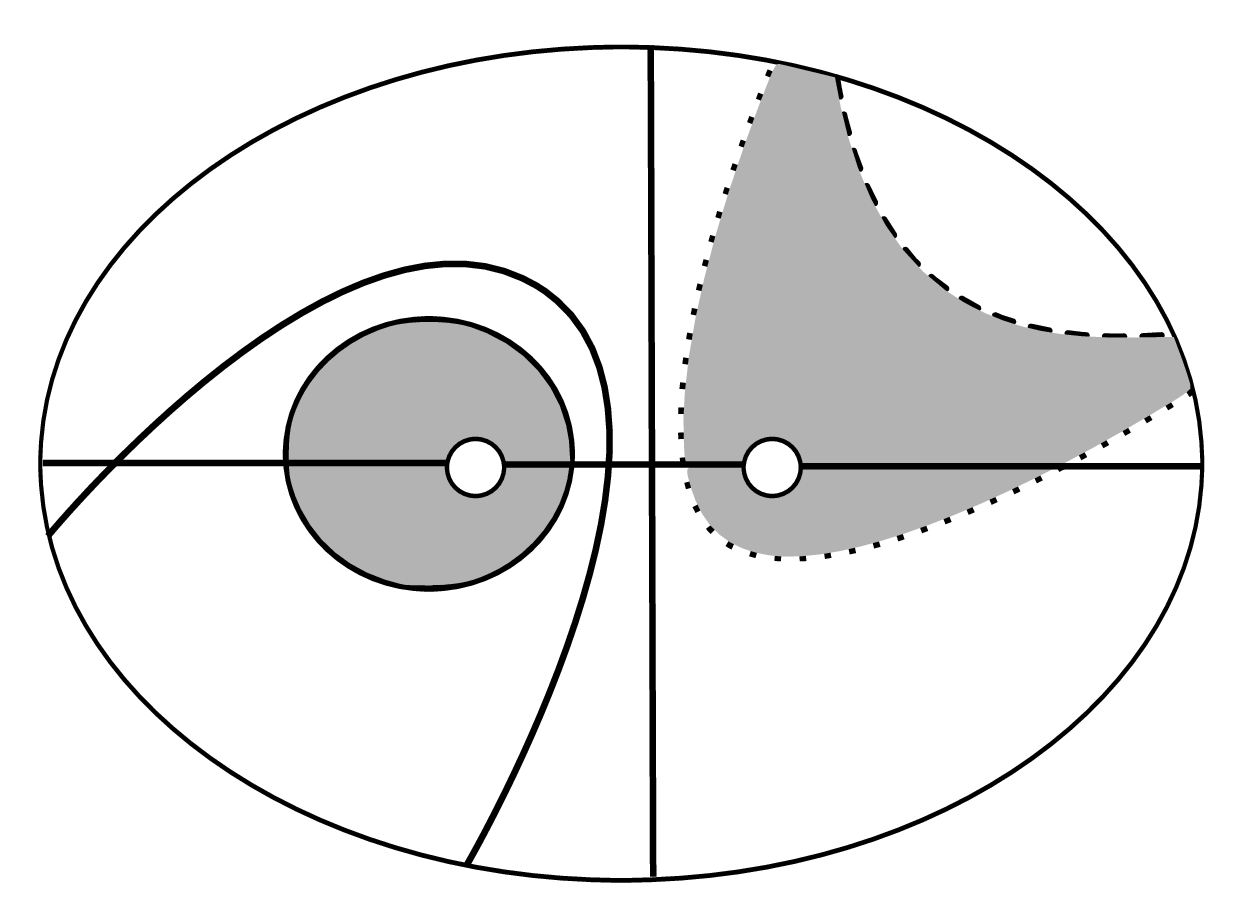}}
\caption[Handeslides associated to $b \mapsto b \sigma_{2n}$]{Four handleslides connecting Heegaard diagrams for $b$ and $b \sigma_{2n}$.  In each picture, the pair of pants is shaded, the new circle is dotted, and the old one is dashed.}
\label{fig:movestabhs}
\end{figure}

We first address admissibility.  Clearly stabilizing or destabilizing an admissible Heegaard diagram yields another admissible one.  By analyzing domains, it isn't hard to verify that the handleslides in Figure \ref{fig:movestabhs} preserve admissibility; we'll demonstrate this explicitly for third handleslide ($\bb^2 \mapsto \bb^3$), and leave the rest as an exercise to the reader.

Assume that $\left( \Sigma; \ba; \bb^2; +\infty \right)$ is admissible.  Label the $m$ domains of $\left( \Sigma; \ba; \bb^2; +\infty \right)$ as indicated in Figure $\ref{fig:movestabad1}$, where $\mathcal{D}_{k}$ lies entirely outside of the picture for $k \geq 7$; label the $(m + 2)$ domains of $\left( \Sigma; \ba; \bb^3; +\infty \right)$ as indicated in Figure \ref{fig:movestabad2}, where $\mathcal{D}'_k$ corresponds to $\mathcal{D}_k$ for $k \geq 7.$  Now consider some periodic domain in $\left( \Sigma; \ba; \bb^3; +\infty \right)$
$$\mathcal{P}' = \tld{c}_1\tld{\mathcal{D}}_1 + \tld{c}_2\tld{\mathcal{D}}_2 + \sum_{i=1}^{m} c_i\mathcal{D}'_i.$$

\begin{figure}[h!]
\centering
\subfloat[$\Sigma \setminus \left( \cup_{i} \alpha_{i} \right) \setminus \left( \cup_{i} \beta_{i}^{2} \right)$]{
\labellist 
\small
\pinlabel* {$\mathcal{D}_3$} at 230 330
\pinlabel* {$\mathcal{D}_2$} at 250 70
\pinlabel* {$\mathcal{D}_1$} at 100 115
\pinlabel* {$\mathcal{D}_4$} at 370 290
\pinlabel* {$\mathcal{D}_5$} at 390 80
\pinlabel* {$\mathcal{D}_6$} at 470 330
\endlabellist 
\includegraphics[height = 40mm]{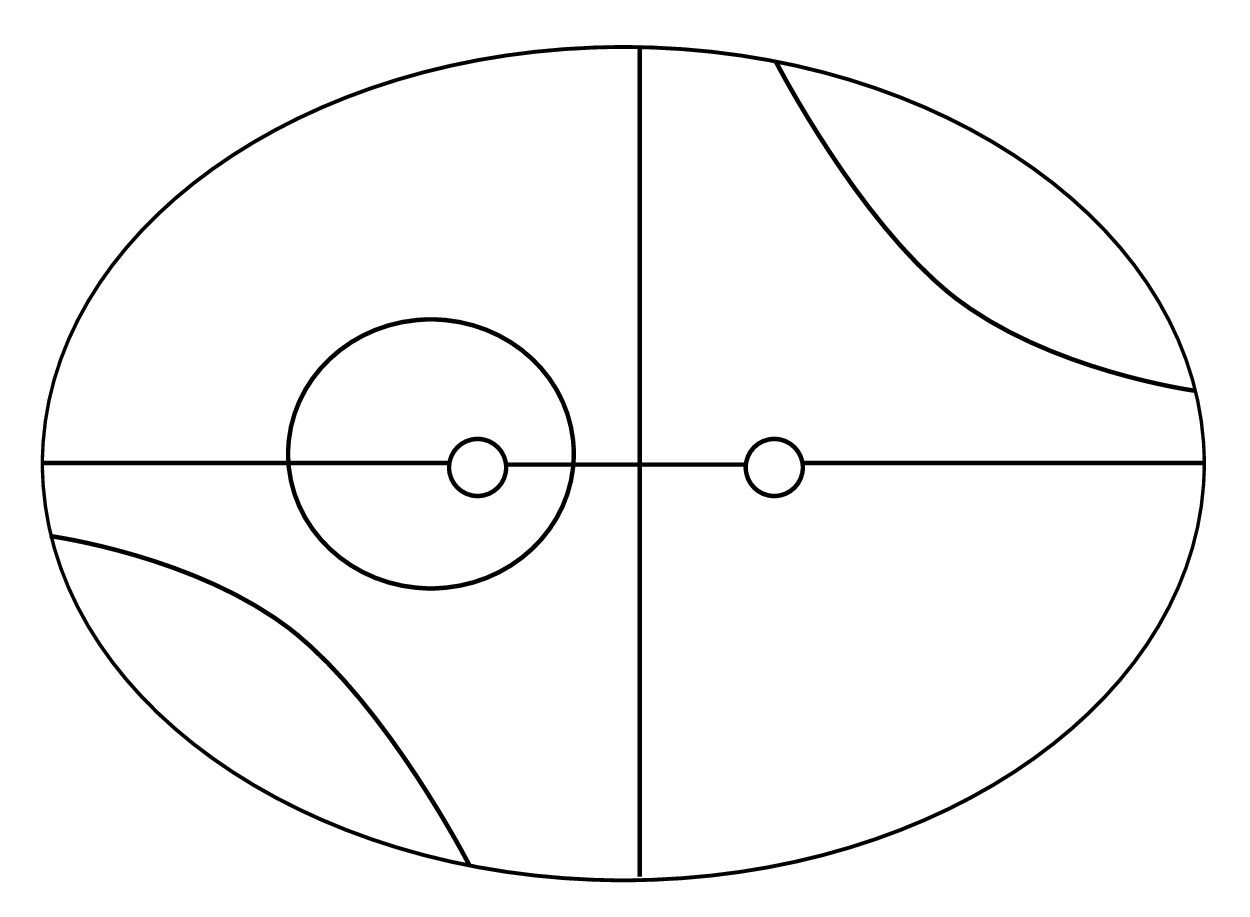}\label{fig:movestabad1}}\quad
\subfloat[$\Sigma \setminus \left( \cup_{i} \alpha_{i} \right) \setminus \left( \cup_{i} \beta_{i}^{3} \right)$]{
\labellist 
\small
\pinlabel* {$\mathcal{D}'_3$} at 245 360
\pinlabel* {$\mathcal{D}'_2$} at 265 65
\pinlabel* {$\mathcal{D}'_1$} at 140 110
\pinlabel* {$\tld{\mathcal{D}}_1$} at 45 390
\pinlabel* {$\tld{\mathcal{D}}_2$} at 45 50
\pinlabel* {$\tld{\mathcal{D}}_3$} at 555 30
\pinlabel* {$\mathcal{D}'_4$} at 370 290
\pinlabel* {$\mathcal{D}'_5$} at 390 80
\pinlabel* {$\mathcal{D}'_6$} at 470 330
\endlabellist 
\includegraphics[height = 40mm]{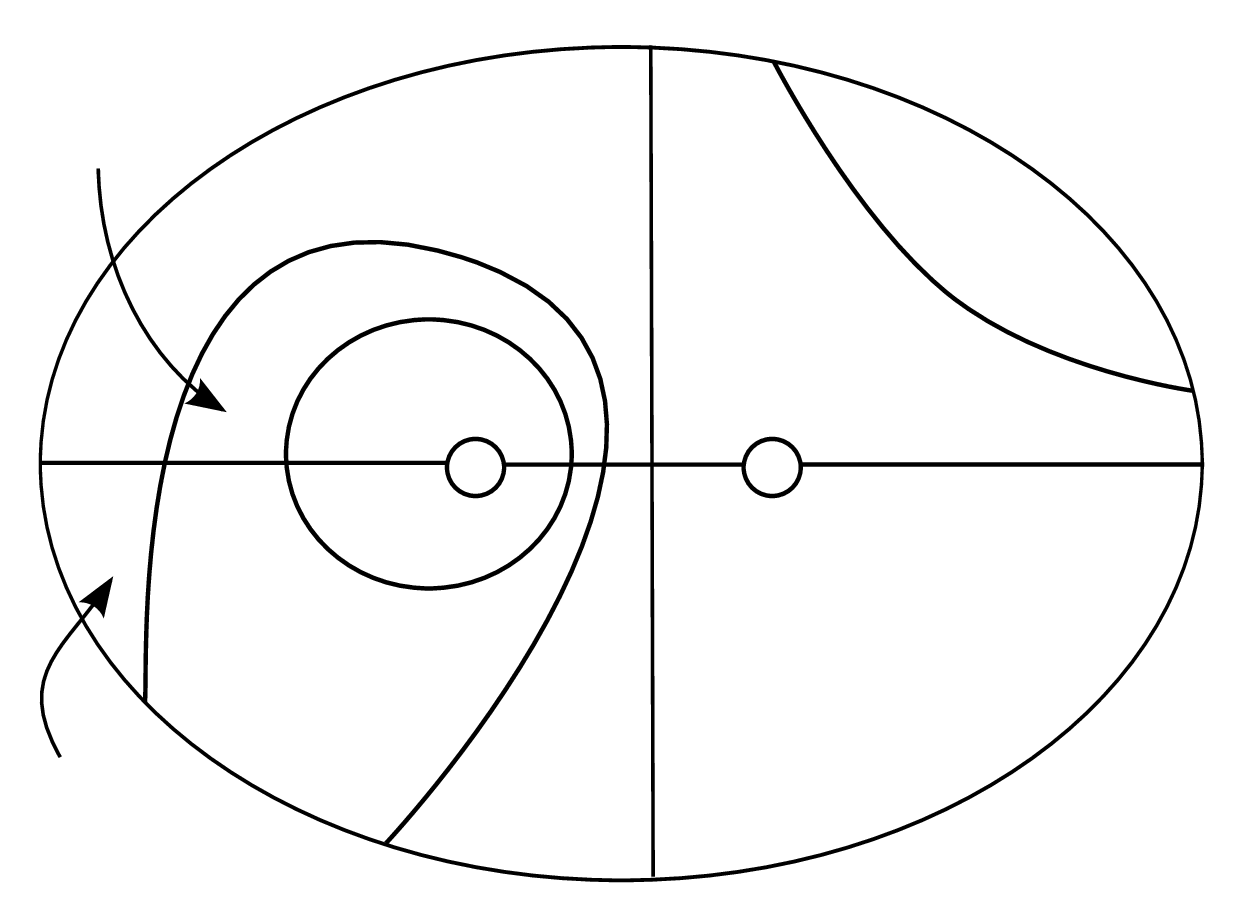}\label{fig:movestabad2}}
\caption[Domains before and after a handleslide]{Domains on Heegaard surfaces before and after the handleslide $\bb^2 \mapsto \bb^3$ induced by Birman stabilization.}
\label{fig:movestabad}
\end{figure}

Now because $\mathcal{P}'$ is periodic, we have that
$$ \tld{c}_2 - c_1 = c_2 - c_1, \quad \text{and so} \quad \tld{c}_2 = c_2.$$
Therefore, there is a periodic domain in $\left( \Sigma; \ba; \bb^2; +\infty \right)$ given by
$$\mathcal{P} = \sum_{i=1}^{m} c_i\mathcal{D}_i.$$
Because $\left( \Sigma; \ba; \bb^2; +\infty \right)$ is admissible, there is at least one positive $c_i$ and at least one negative $c_i$.  So, $\mathcal{P}'$ has positive and negative coefficients, and thus $\left( \Sigma; \ba; \bb^3; +\infty \right)$ is also admissible.


Let $x_{2n+1}$ denote the additional intersection obtained via Heegaard stabilization.  We define the injection $g_{stab}$ as $x_{2n}x_{2n+1} \bv \mapsto y_{2n+1} y_{2n} \bv$ and $x_{2n+1}\bz \mapsto y_{2n+2} \bz$, 
where $\bv $ is an $(n-1)$-tuple not contained in the local picture, and $\bz $ is an $n$-tuple not contained within the local picture. In fact, $g_{stab}$ can be written as a composition $g_{stab}^{\bb,3}\circ g_{stab}^{\bb,2}\circ g_{stab}^{\ba} \circ g_{stab}^{\bb,1}$ of triangle injections corresponding to the four handleslides in Figure \ref{fig:movestabhs}.   All of the 3-gons required for $g_{stab}^{\bb,2}$ and $g_{stab}^{\bb,3}$ are 3-gons with ``small'' domains, and we won't exhibit these in figures.  Local regions in domains of 3-gons for $g_{stab}^{\bb,1}$ and $g_{stab}^{\ba}$ are exhibited in Figures \ref{fig:movestabhs1} and \ref{fig:movestabhs2}, respectively.  Notice that the 3-gon appearing in Figure \ref{fig:movestabhs2b} is of the second type (cf. Figure \ref{fig:trisplit}).

\begin{figure}[h!]
\centering
\subfloat[$x_{2n+1} \mapsto w_{2n+1}$]{
\labellist
\small
\pinlabel* $a$ at 250 355
\pinlabel* \reflectbox{$a$} at 470 355
\endlabellist
\includegraphics[height = 40mm]{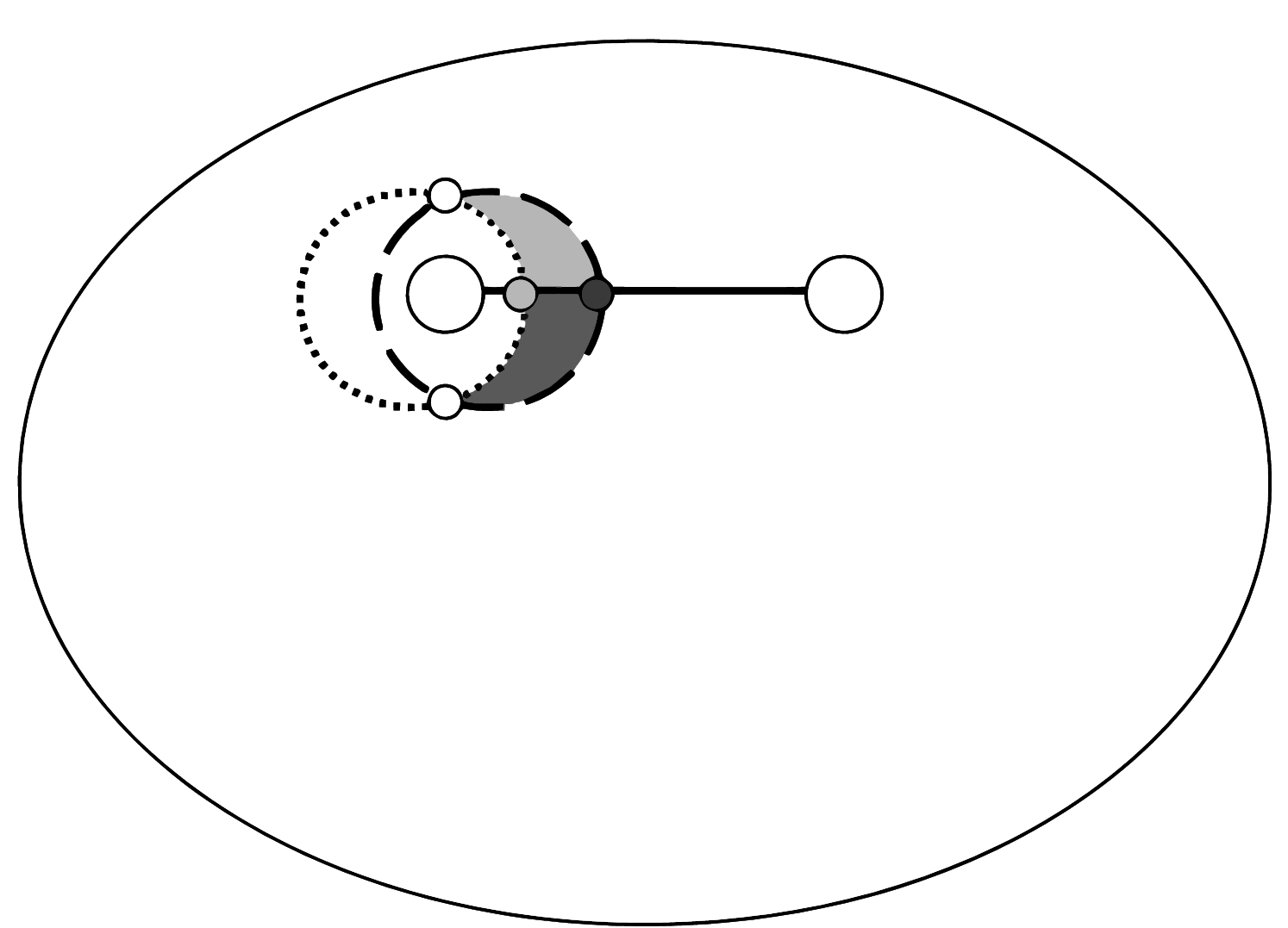}}\qquad
\subfloat[$x_{2n}x_{2n+1} \mapsto w_{2n}w_{2n+1}$]{
\labellist
\small
\pinlabel* $a$ at 250 368
\pinlabel* \reflectbox{$a$} at 470 368
\endlabellist
\includegraphics[height = 40mm]{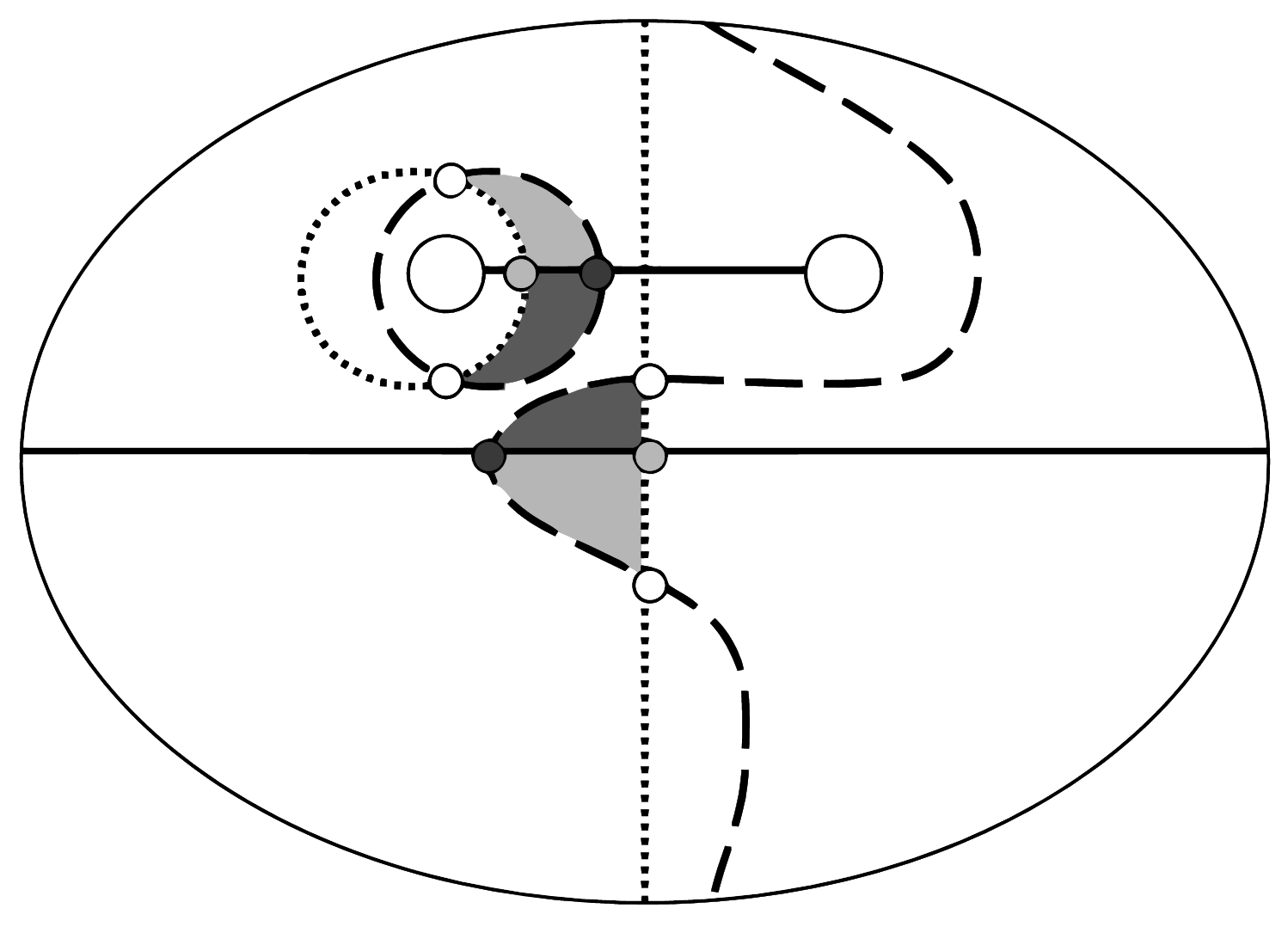}}
\caption[Local components of 3-gon domains associated to the first handleslide for $b \mapsto b \sigma_{2n}$]{Local regions in domains of the 3-gons $\psi^{+}$ (dark grey) and $\psi^{-}$ (light grey) for $g_{stab}^{\bb,1}$.  White dots are components of $\thet{\bb \bb^2},\thet{\bb^2\bb} \in \tor{\bb} \cap \tor{\bb^2}$.\label{fig:movestabhs1}}
\end{figure}

\begin{figure}[h!]
\centering
\subfloat[$w_{2n+1} \mapsto y_{2n+2}$]{
\labellist
\small
\pinlabel* $a$ at 250 263
\pinlabel* \reflectbox{$a$} at 470 263
\endlabellist
\includegraphics[height = 40mm]{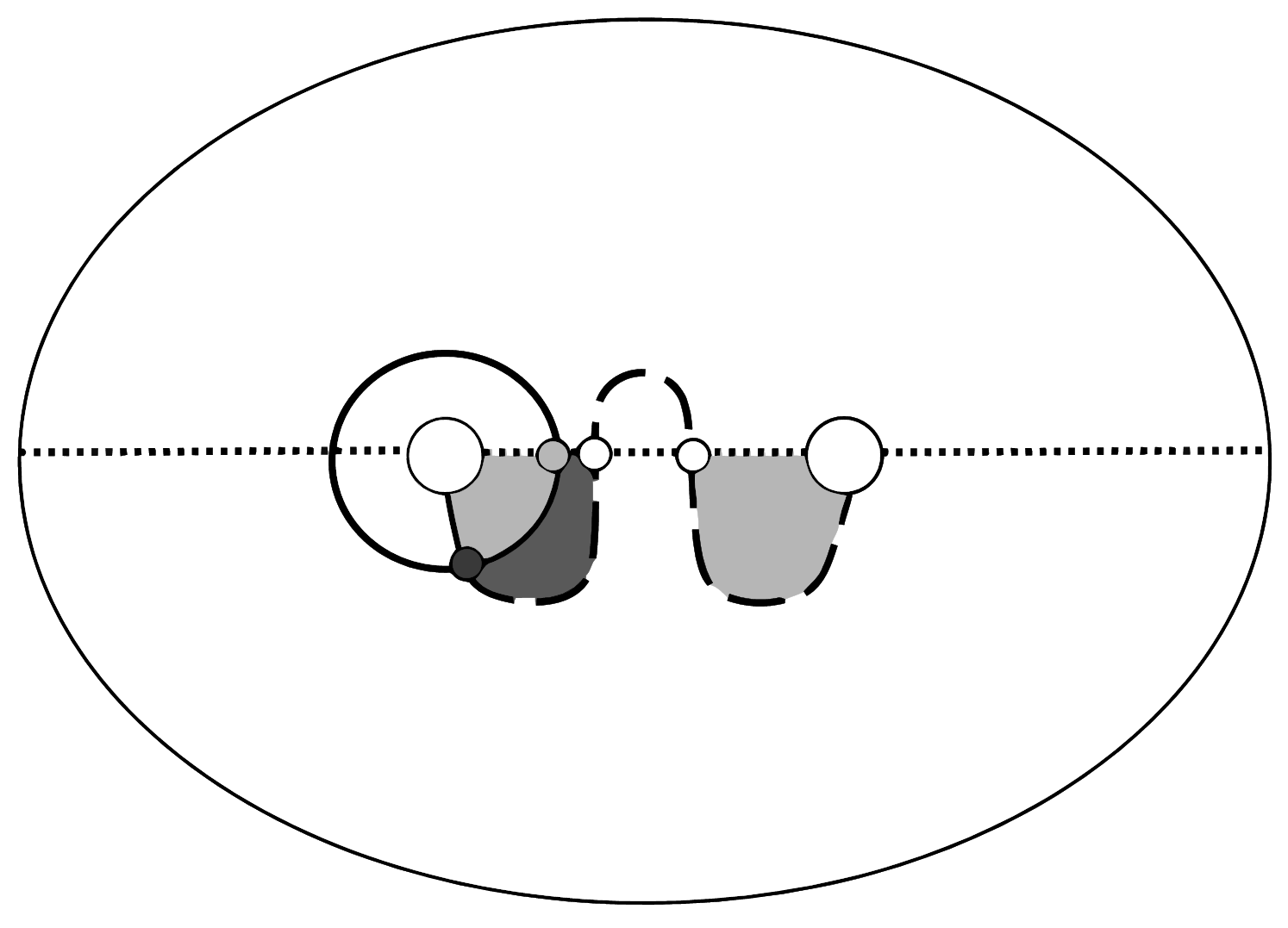}}\qquad
\subfloat[$w_{2n}w_{2n+1} \mapsto y_{2n+1} y_{2n}$]{
\labellist
\small
\pinlabel* $a$ at 250 263
\pinlabel* \reflectbox{$a$} at 470 263
\endlabellist
\includegraphics[height = 40mm]{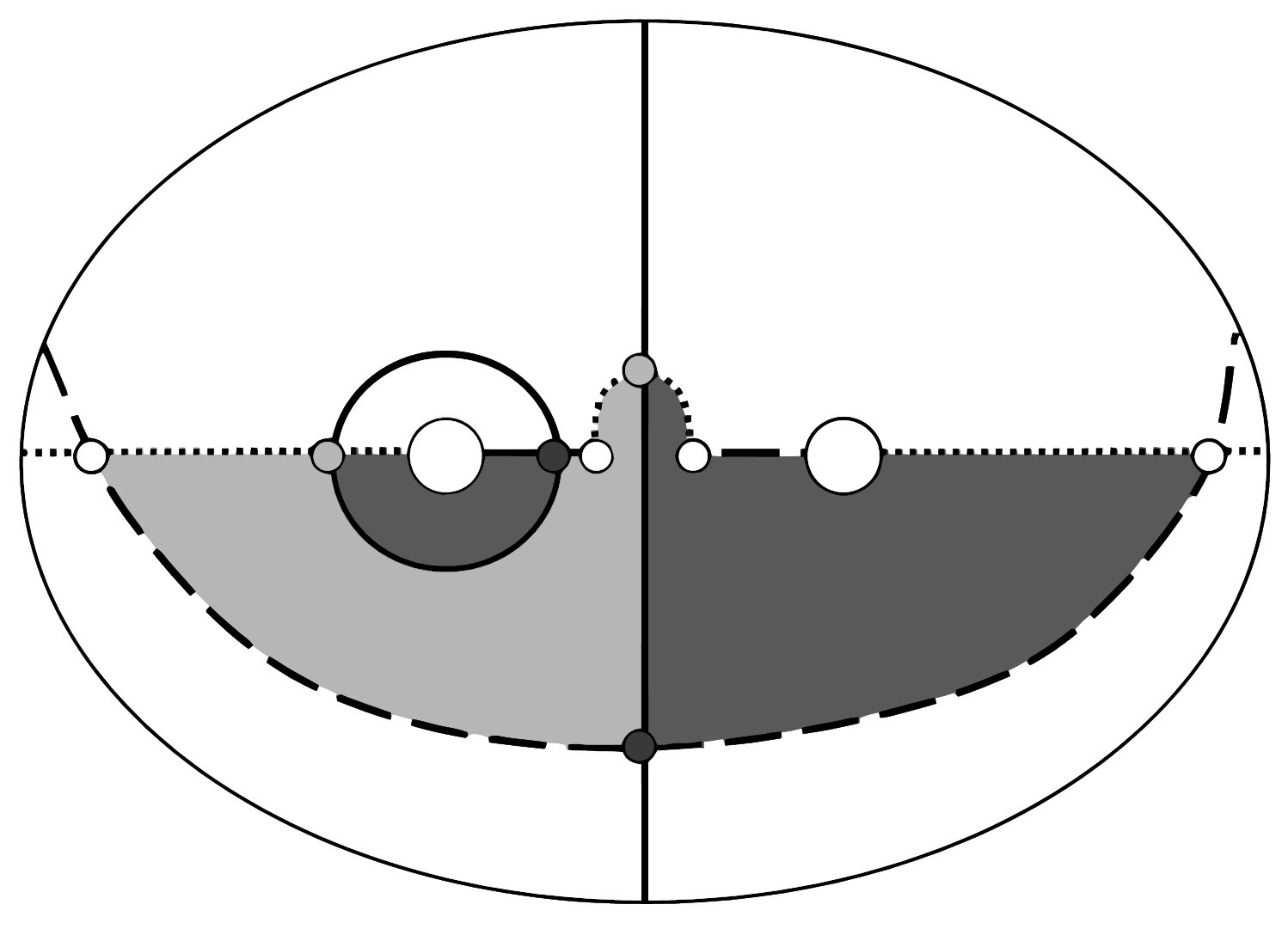}\label{fig:movestabhs2b}}
\caption[Local components of 3-gon domains associated to the second handleslide for $b \mapsto b \sigma_{2n}$]{Local regions in domains of the 3-gons $\psi^+$ (dark grey) and $\psi^-$ (light grey) for $g_{stab}^{\ba}$. White dots are components of $\thet{\ba \ba'},\thet{\ba'\ba} \in \tor{\ba} \cap \tor{\ba'}$ \label{fig:movestabhs2}}
\end{figure}

\begin{figure}[h!]
\centering
\includegraphics[height = 40mm]{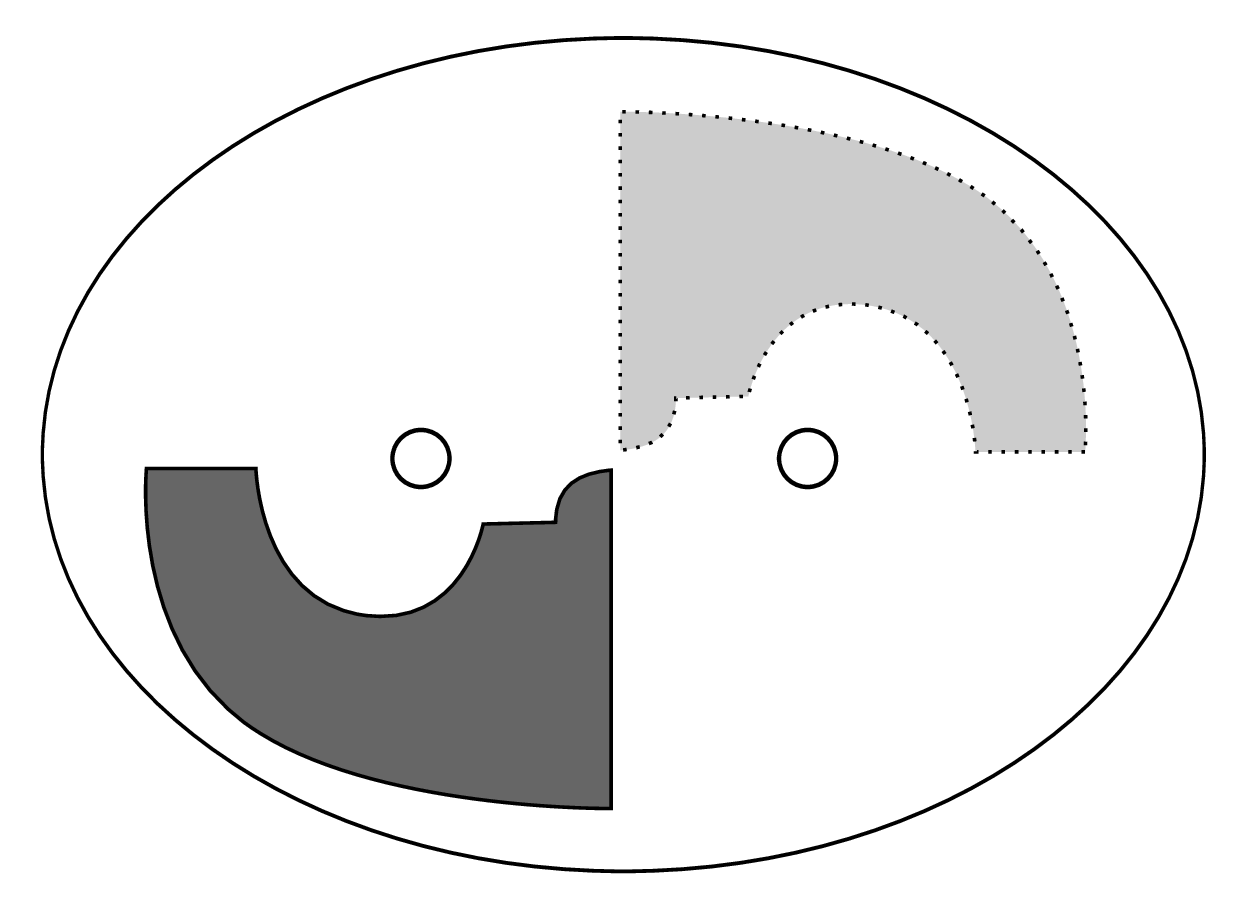}
\caption[Avoiding the anti-diagonal]{Arranging that a 3-gon avoids the anti-diagonal.  The dark grey region is the domain $\mathcal{D}$, and the light grey region is the other connected component of $\pi(\pi^{-1}(\mathcal{D}))$. \label{fig:stabtriad}}
\end{figure}

It can be arranged that these 3-gons avoid the anti-diagonal $\AD$ (making use of the discussion in Section \ref{sec:multiad}).  In particular, if $\mathcal{D}$ denotes one of the six-sided regions appearing in Figure \ref{fig:movestabhs2b} and $\pi:\Sigma \rightarrow S^2$ is the branched covering map, it suffices to show that $\pi(\pi^{-1}(\mathcal{D}))$ has two connected components.  This is illustrated in Figure \ref{fig:stabtriad}.

\begin{figure}[h]
\centering
\subfloat[$x_{2n}$]{\includegraphics[height = 20mm]{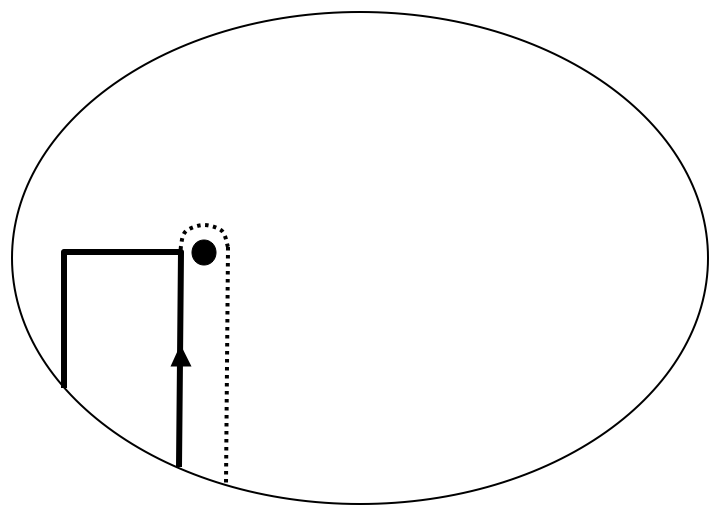}}\quad
\subfloat[$y_{2n+1}$]{\includegraphics[height = 20mm]{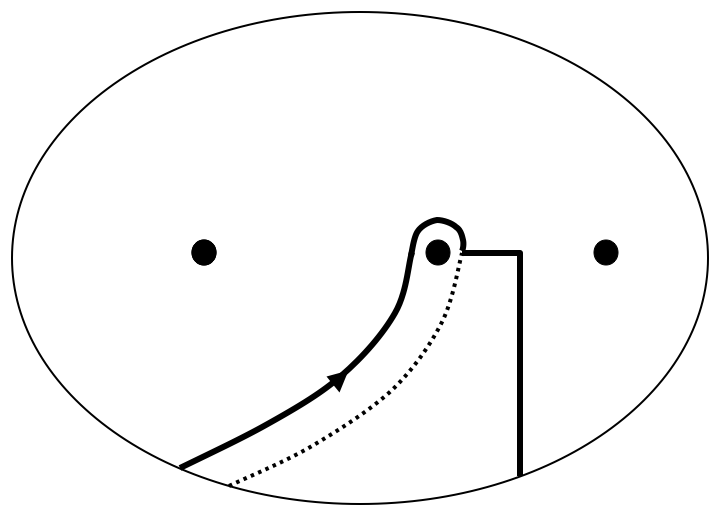}}\quad
\subfloat[$y_{2n}$]{\includegraphics[height = 20mm]{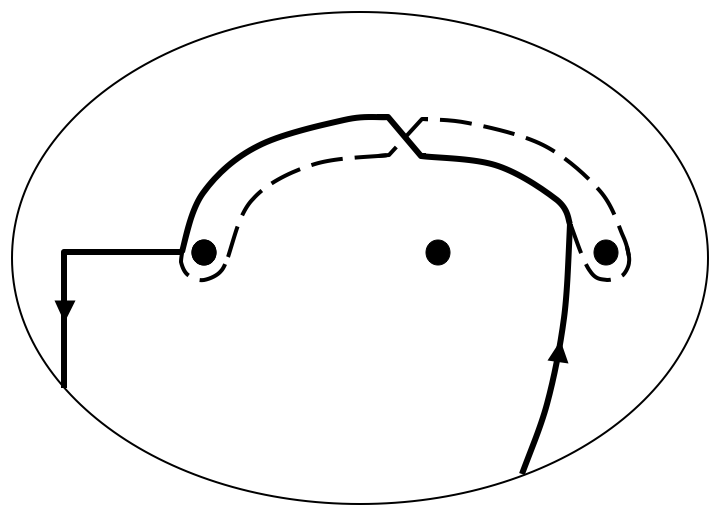}}\quad
\subfloat[$y_{2n+2}$]{\includegraphics[height = 20mm]{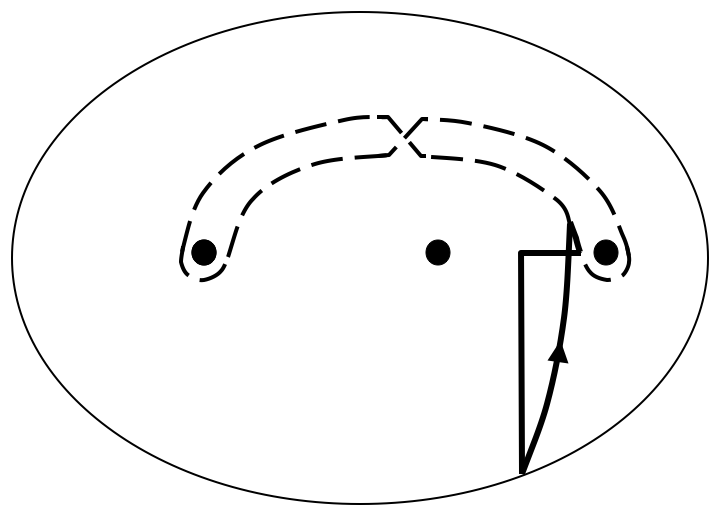}}
\caption[Grading loops associated to $b \mapsto b \sigma_{2n}$]{Loops for elements of $\Zcal$ affected by $b \in \B{2n} \leftrightarrow b \sigma_{2n} \in \B{2n + 2}$\label{fig:movestabloops}}
\end{figure}

One can verify that $Q(y_{2n+1} y_{2n} \bv )  = Q(x_{2n} \bv ) + 1$, $Q(y_{2n+2} \bz )  = Q(\bz )$, $P(y_{2n+1} y_{2n} \bv )  =  P(x_{2n} \bv )$, $P(y_{2n+2} \bz )  =  P(\bz )$, $T(y_{2n+1} y_{2n} \bv )  = T(x_{2n} \bv ) + 1$, and $T(y_{2n+2} \bz )  = T(\bz )$.  Stabilization adds two strands, increases $\epsilon$ by 1, and decreases $w$ by 1.  Thus $s_{R}(b \sigma_{2n}) = s_{R}(b)$ and so $R(y_{2n+1} y_{2n} \bv )  = R (x_{2n} \bv )$ and $R(y_{2n+2} \bz )  = R(\bz )$.

\begin{rmk}
We won't explicitly construct a triangle injection associated to the Birman destabilization move.  This move induces the inverse of the sequence of Heegaard moves depicted in Figure \ref{fig:movestabhd}, and we've already seen that intermediate diagrams are admissible.   To obtain 3-gons associated to this triangle injection, simply reverse the roles of  3-gons $\psi^{\pm}$ associated to the Birman stabilization move.
\end{rmk}

\subsection{Generality of local pictures}\label{sec:gen}
It remains to justify that our local pictures above were sufficiently general:
\begin{enumerate}[(i)]
\item There could be some $u \in \Ztil$ such that  $u\bx \in \Zcal$ and one of the loops used to compute gradings for $u\bx$ contain vertical arcs that pass through the local diagram, as seen in Figure \ref{fig:gen0}.  One can verify that the grading contributions of such arcs are preserved by Birman moves.

\begin{figure}[h]
\centering
\subfloat[Grading loop for $u\bx$]{
\includegraphics[height = 25mm]{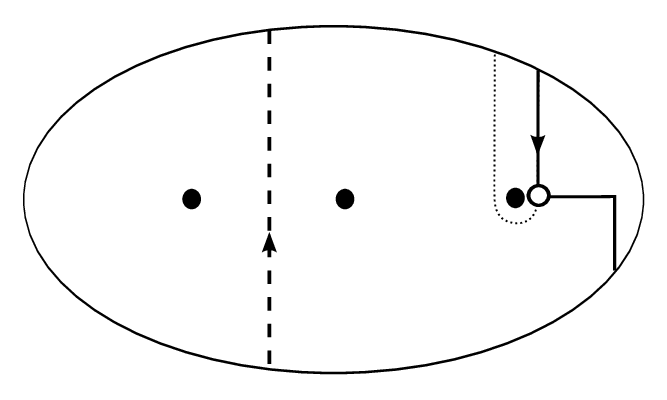}}\qquad
\subfloat[Grading loop for $u'\bx$]{
\includegraphics[height = 25mm]{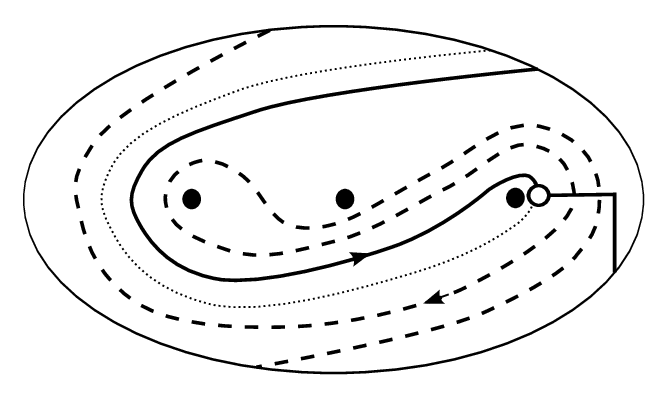}}
\caption[The effect of a Birman move on a pass-through arc]{The effect of the move $b \mapsto b B$ on a pass-through arc appearing in a grading loop}
\label{fig:gen0}
\end{figure}

\item \label{item:gen1}Several $\beta$ arcs intersecting the interior of the same $\alpha$ arc can be isotoped to be very close to one another and thus behave identically under moves.
\item We assumed above that all $\beta$ arcs shown belong to distinct $\beta_{i}$.  If two $\beta$ arcs share the same $\beta_{i}$, then fewer Bigelow generators are allowed.

\item The local pictures in Section \ref{sec:Rmoves} never have handles passing through them.  A handle $bh_{i}$ contributes a vertical arc as in (\ref{item:gen1}) to each loop associated to a point on $\beta_{i}$.

\item  For arcs terminating at punctures near the boundary of the local picture, we can modify the entrance trajectory (i.e. from above or from below) by applying an isotopy to the $\beta$ arc and shrinking the scope of the picture (see Figure \ref{fig:gen1}).

\begin{figure}[h]
\centering
\subfloat[Dashed arc enters from above]{
\includegraphics[height = 25mm]{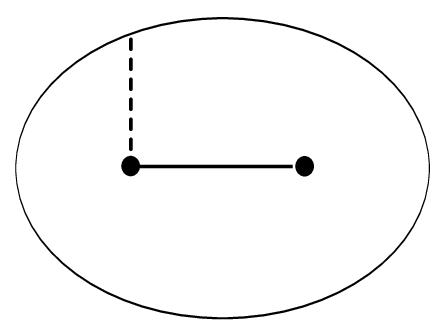}}\qquad
\subfloat[Now from below]{
\includegraphics[height = 25mm]{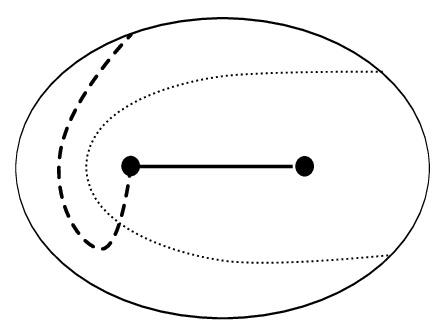}}
\caption[The entry trajectory of a $\beta$ arc terminating near the boundary]{Modifying the entry trajectory of a $\beta$ arc terminating near the boundary}
\label{fig:gen1}
\end{figure}

Entrance trajectories of $\beta$ arcs terminating far from the boundary can be modified in the same way, but at the expense of adding a pass-through arc to the local picture (see Figure \ref{fig:gen2}).
\begin{figure}[h]
\centering
\subfloat[Dashed arc enters from above]{
\includegraphics[height = 25mm]{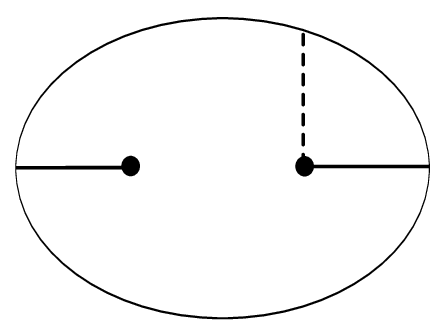}}\qquad
\subfloat[Now from below]{
\includegraphics[height = 25mm]{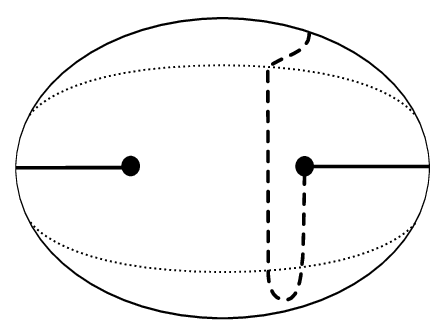}}
\caption[The entry trajectory of a $\beta$ arc terminating away from the boundary]{Modifying the entry trajectory of a $\beta$ arc terminating far from the boundary}
\label{fig:gen2}
\end{figure}

\end{enumerate}

\section{The left-handed trefoil and the lens space $L(3,1)$}

The fork diagram for the left-handed trefoil obtained from $\sigma_{2}^{3} \in \B{4}$ induces the admissible Heegaard diagram $\big( \Sigma_{2}, \{ \ah_{1}, \ah_{2} \},  \{ \bh_{1}, \bh_{2} \}, +\infty \big)$ for $\SUM{L(3,1)}$ shown in Figure \ref{fig:exhd}.  Label the intersections on $\ah_{1}$ from left to right as $s',t',x_{2},t,s,\text{ and }x_{1}$, and label those on $\ah_{2}$ from bottom top as $x_{4},v', u', x_{3}, u,\text{ and }v$.

\begin{figure}[h!]
\centering
\labellist 
\small
\pinlabel* {$+\infty$} at 660 370
\pinlabel* {$\ah_{1}$} at 405 245
\pinlabel* {$\ah_{2}$} at 355 335
\pinlabel* {$\bh_{1}$} at 405 88
\pinlabel* {$\bh_{2}$} at 55 390
\pinlabel* {$a$} at 385 186
\pinlabel* {$a$} at 236 186
\pinlabel* {$b$} at 578 186
\pinlabel* {$b$} at 731 186
\endlabellist 
\includegraphics[height = 50mm]{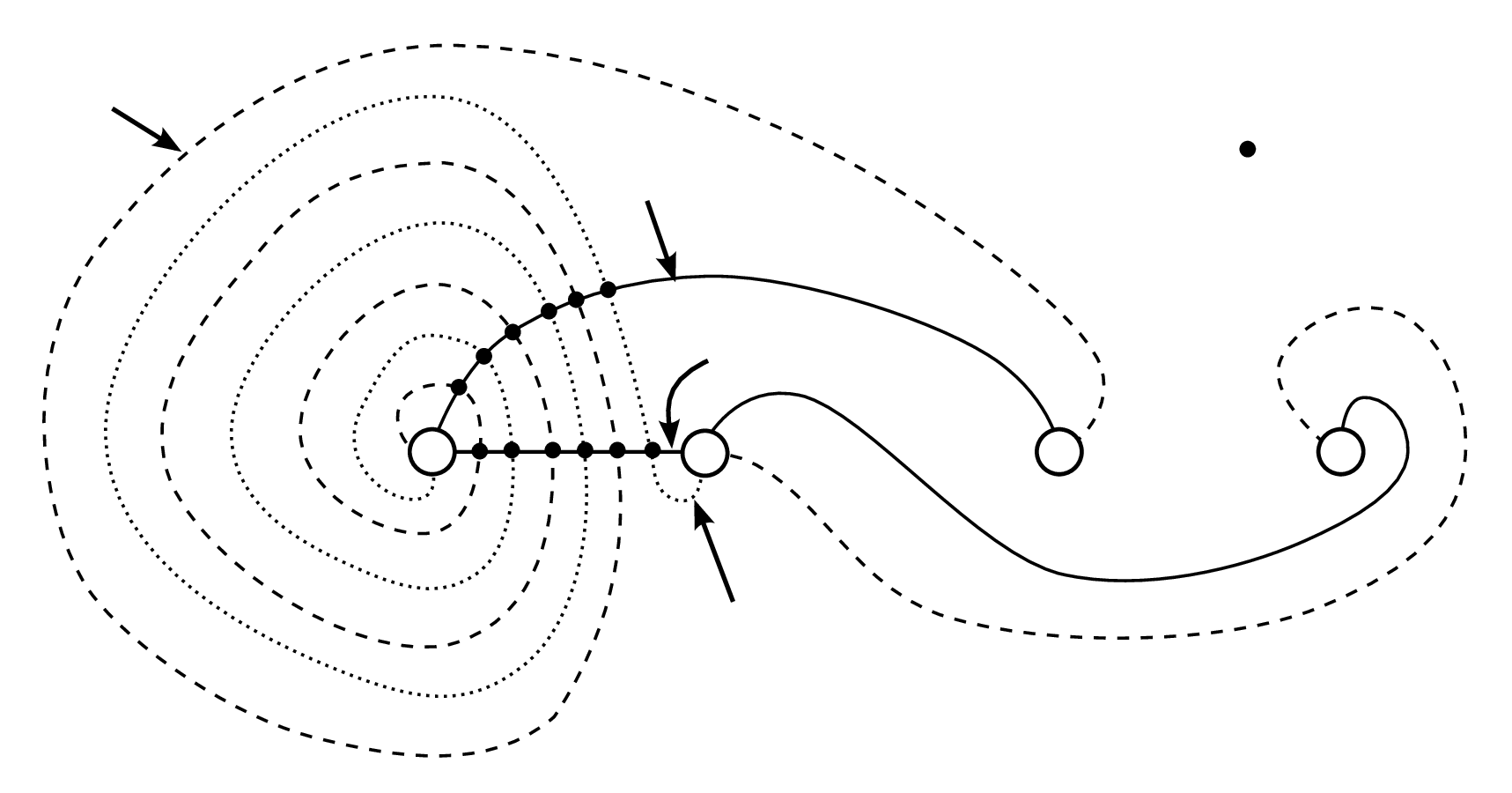}
\caption{The Heegaard diagram for $\SUM{L(3,1)}$ obtained from $\sigma_{2}^{3} \in \B{4}$}
\label{fig:exhd}
\end{figure}
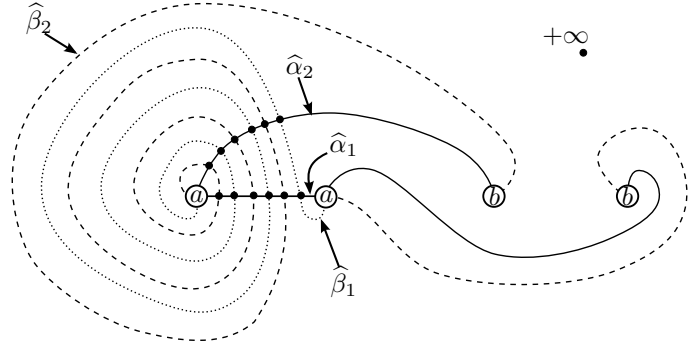

Let's perform the calculation with $\Zcaltwo$ coefficients; this can be done combinatorially, since the diagram in Figure \ref{fig:exhd} is \textit{nice} in the sense of \cite{sw:CHF}.

The differential is thus
\begin{gather*}
\widehat{\partial}(x_{2}x_{3}) = ut' + u't, \quad
\widehat{\partial}(ut) = sx_{3} + vx_{2}, \quad
\widehat{\partial}(ut') =  \widehat{\partial}(u't) = sv' + s'v,\\
\widehat{\partial}(u't') = s'x_{3} + v'x_{2}, \quad
\widehat{\partial}(s'v') = \widehat{\partial}(sx_{3}) = \widehat{\partial}(vx_{2}) = u'x_{1} + t'x_{4},\\\widehat{\partial}(s'v) = \widehat{\partial}(sv') = x_{1}x_{4} + x_{1}x_{4} = 0, \quad
\widehat{\partial}(v'x_{2}) = \widehat{\partial}(s'x_{3}) = \widehat{\partial}(sv) = tx_{4} + ux_{1},\\
\text{and} \quad \widehat{\partial}(u'x_{1}) = \widehat{\partial}(t'x_{4}) = \widehat{\partial}(x_{1}x_{4}) = \widehat{\partial}(tx_{4}) = \widehat{\partial}(ux_{1}) = 0.
\end{gather*}
Now recall that $L(3,1)$, has three $\text{Spin}^{c}$ structures $\mathfrak{s}_{i}$, $i = 0, 1, 2$.  These induce three $\text{Spin}^{c}$ structures on $\SUM{L(3,1)}$ given by $\mathfrak{s_{i}} \# \mathfrak{s}$, where $\mathfrak{s}$ is the unique torsion $\text{Spin}^{c}$-structure on $S^{2} \times S^{1}$.  One should observe that the diagram in Figure \ref{fig:exhd} can be related via handleslides to one which is the disjoint union of a diagram for $\DBC{K}$ and the usual admissible genus-one diagram for $\SxS$.  As a result, all generators $\bx$ in Figure \ref{fig:exhd} have $\mathfrak{s}_{+\infty}(\bx) = \mathfrak{s_{i}} \# \mathfrak{s}$.  They partition this set of generators as
\begin{align*}
\mathfrak{U}_{0} &= \{ x_{2}x_{3}, ut', u't, s'v, sv', x_{1}x_{4} \}, \\
\mathfrak{U}_{1} &= \{ ut, s'v', sx_{3}, vx_{2}. u'x_{1}, t'x_{4} \}, \\
\mathfrak{U}_{2} &= \{ u't', v'x_{2}, s'x_{3}, sv, tx_{4}, ux_{1} \}.
\end{align*}

Notice that the differential always lowers the $R$-grading by 1 in this case, and thus the left-handed trefoil is evidently $\rho$-degenerate.  The $R$-grading then provides an absolute Maslov grading on the group $\widehat{HF}(\SUM{L(3,1)}; \Zcaltwo)$.

One can see that homology group decomposes with respect to the $R$-grading as
\begin{equation*}
\widehat{HF}(\SUM{L(3,1)}; \Zcaltwo) = \left[ \left( \Zcaltwo \right)^{\oplus 3} \right] _{R = 3/2} \oplus\left[ \left( \Zcaltwo \right)^{\oplus 3} \right] _{R = 1/2}.
\end{equation*}

\section{Reduced theory}

The sequel to the present paper \cite{et:R2} outlines a reduced theory which provides a filtration on the Heegaard Floer chain complex for $\DBC{K}$.  The reduced version is easier to compute, and it is shown in \cite{et:R2} that the spectral sequence discussed in the present paper is completely determined by an analogous reduced spectral sequence.  The reduced theory can be shown to have some nice formal properties with respect to taking connected sums and mirrors of knots, and can be used to show that all two-bridge knots are $\rho$-degenerate.  It would be tempting to speculate that all alternating knots are $\rho$-degenerate, but this is not known.

\section{Future directions}
\subsection{Relationship with $Kh_{symp,inv}$}

Given a pointed Heegaard diagram $\h$ for $\DBCs{K}$ coming from a braid, we saw that the filtration $\rho$ can only be defined on generators in torsion $\text{Spin}^{c}$ structures.  It would be interesting to investigate whether Heegaard diagrams encountered in this context actually contain generators in non-torsion $\text{Spin}^{c}$ structures; if not, then $E_{1}$ is in fact all of $\Kst{K}$.  In particular, we would obtain that when $K$ is $\rho$-degenerate,
\begin{equation*}
\HFx{K} \cong \Kst{K}.
\end{equation*}

\subsection{The Khovanov-Heegard Floer spectral sequence}

Ozsv\'ath and Szab\'o showed in \cite{os:bc} that the groups $\widehat{HF}(\DBC{L})$, $\widehat{HF}(\DBC{L_{0}})$, and $\widehat{HF}(\DBC{L_{1}})$ fit into a long exact sequence:
\begin{equation*}\label{eq:skein}
\ldots \longrightarrow \widehat{HF}(\DBC{L_{0}} \longrightarrow \widehat{HF}(\DBC{L_{1}} \longrightarrow \widehat{HF}(\DBC{L} \longrightarrow \ldots
\end{equation*}
where the diagrams for $L_{0}$ and $L_{1}$ exhibit the two smooth resolutions of some crossing $c$ in $L$ and coincide with $L$ away from $c$.  The existence of this sequence is a consequence of the surgery exact sequence for $\widehat{HF}$, and Ozsv\'ath and Szab\'o use it to construct a spectral sequence whose $E^{2}$ term is isomorphic to the reduced Khovanov homology $\tld{Kh}(\overline{L}; \Zcaltwo)$ of the mirror of $L$ and which converges to the Heegaard Floer homology group $\widehat{HF}(\DBC{L};\Zcaltwo)$.  Let  $\delta =(j - i)$ denote the quantum grading, the collapse of the bigrading on the group
\begin{equation*}
\tld{Kh}(L) = \displaystyle\bigoplus_{i,j} \tld{Kh}^{i,j}(L).
\end{equation*}

It was shown in \cite{cmoz:thin} that the class of quasi-alternating links is Khovanov-thin, with $\tld{Kh}^{i,j} (L)\neq 0$ only if $\delta = (j-i) = -\sigma(L)/2 = \sigma(\overline{L})/2$.  Notice that when $L$ is a two-bridge link, this is exactly the $\red{R}$-level supporting $\widehat{HF}(\DBC{\overline{L}}).$

Baldwin \cite{baldwin:ss} conjectured the existence of an induced $\delta$-grading on higher pages in the spectral sequence, and Greene \cite{jg:tree} conjectured that a term arising in his spanning tree model could provide a quantum grading on $\HF{K}$.  If the gradings conjectured by Greene and Baldwin indeed exist, it would be interesting to compare them to the $R$-grading for $\rho$-degenerate knots.

Furthermore, Szab\'o \cite{szabo:kh} constructed a geometric spectral sequence in $\mathbb{Z}/2\mathbb{Z}$ Khovanov homology.  Although this spectral sequence is not known to abut to the Heegaard Floer homology, the construction is similar to that for the spectral sequence in \cite{os:bc}.  Szab\'o's spectral sequence preserves the Khovanov $\delta$-grading, and it would be interesting to compare the induced grading on the $E_{\infty}$-page with the reduced function $\red{R}$ appearing in the sequel of the current article, \cite{et:R2}.

\section*{Acknowledgments}
It is my pleasure to thank Ciprian Manolescu for suggesting this problem to me and for his invaluable guidance as an advisor.  I would also like to thank Liam Watson and Tye Lidman for some instructive discussions, Stephen Bigelow for some helpful email correspondence related to his paper \cite{big:jones}, Yi Ni for some useful comments regarding relative Maslov gradings.  This paper has been rewritten from a previous version to account for an update to the paper \cite{ss:R2}.  I am indebted to Ivan Smith for pointing out this change.

I would also like to thank the acknowledge the anonymous Referee, whose corrections and suggestions led to countless improvements in this article.
\clearpage
\bibliography{references}

\begin{bibdiv}
\begin{biblist}

\bib{AS}{article}{
      author={Abouzaid, M.},
      author={Smith, I.},
     journal={work in progress},
}

\bib{baldwin:ss}{article}{
      author={Baldwin, John~A.},
       title={On the spectral sequence from {K}hovanov homology to {H}eegaard
  {F}loer homology},
        date={2011},
     journal={Int. Math. Res. Not. IMRN},
      number={15},
       pages={3426\ndash 3470},
}

\bib{big:jones}{incollection}{
      author={Bigelow, S.},
       title={A homological definition of the {J}ones polynomial},
        date={2002},
   booktitle={Invariants of knots and 3-manifolds ({K}yoto, 2001)},
      series={Geom. Topol. Monogr.},
      volume={4},
   publisher={Geom. Topol. Publ., Coventry},
       pages={29\ndash 41 (electronic)},
}

\bib{bir:moves}{article}{
      author={Birman, J.},
       title={On the stable equivalence of plat representations of knots and
  links},
        date={1976},
     journal={Canad. J. Math.},
      volume={28},
       pages={264\ndash 290},
}

\bib{jg:tree}{article}{
      author={Greene, J.},
       title={A spanning tree model for the {H}eegaard {F}loer homology of a
  branched double-cover},
        date={2008},
     journal={Electronic pre-print, arXiv:math/0805.1381v1},
}

\bib{jt:nat}{article}{
      author={Juh\'asz, A.},
      author={Thurston, D.},
       title={Naturality and mapping class groups in {H}eegaard {F}loer
  homology},
        date={2012},
     journal={Electronic pre-print, arXiv:1210.4996v1},
}

\bib{kont:GL}{inproceedings}{
      author={Kontsevich, M.},
       title={Homological algebra of mirror symmetry},
        date={1978},
   booktitle={Proceedings of the {I}nternational {C}ongress of {M}athematicians
  ({Z}\"urich, 1994)},
   publisher={Birkh\"auser},
       pages={120\ndash 139},
}

\bib{lip:cyl}{article}{
      author={Lipshitz, R.},
       title={A cylindrical reformulation of {H}eegaard {F}loer homlogy},
        date={2006},
     journal={{G}eom. {T}opol.},
      volume={10},
       pages={955\ndash 1097},
}

\bib{cm:R}{article}{
      author={Manolescu, C.},
       title={Nilpotent slices, {H}ilbert schemes, and the {J}ones polynomial},
        date={2006},
     journal={Duke Math. J.},
      volume={132},
       pages={311\ndash 369},
}

\bib{cmoz:thin}{inproceedings}{
      author={Manolescu, C.},
      author={Ozsv\'ath, P.},
       title={On the {K}hovanov and knot {F}loer homologies of
  quasi-alternating links},
        date={2007},
   booktitle={Proceedings of the 14th g\"okova geometry-topology conference},
       pages={60\ndash 81},
}

\bib{os:abs}{article}{
      author={Ozsv\'ath, P.},
      author={Szab\'o, Z.},
       title={Absolutely graded {F}loer homologies and intersection forms for
  four-manifolds with boundary},
        date={2003},
     journal={Advances in Math.},
      volume={173},
       pages={179\ndash 261},
}

\bib{os:disk2}{article}{
      author={Ozsv\'ath, P.},
      author={Szab\'o, Z.},
       title={Holomorphic disks and three-manifold invariants: {P}roperties and
  applications},
        date={2004},
     journal={Annals of Math.},
      volume={159},
       pages={1159\ndash 1245},
}

\bib{os:disk}{article}{
      author={Ozsv\'ath, P.},
      author={Szab\'o, Z.},
       title={Holomorphic disks and topological invariants for closed
  three-manifolds},
        date={2004},
     journal={Annals of Math.},
      volume={159},
       pages={1027\ndash 1158},
}

\bib{os:bc}{article}{
      author={Ozsv\'ath, P.},
      author={Szab\'o, Z.},
       title={On the {H}eegaard {F}loer homology of branched double-covers},
        date={2005},
     journal={Advances in Math.},
      volume={194},
       pages={1\ndash 33},
}

\bib{os:tri}{article}{
      author={Ozsv\'ath, P.},
      author={Szab\'o, Z.},
       title={Holomorphic triangles and invariants for smooth four-manifolds},
        date={2006},
     journal={Advances in Math.},
      volume={202},
       pages={326\ndash 400},
}

\bib{reza:ss}{article}{
      author={Rezazadegan, R.},
       title={A spectral sequence for {L}agrangian {F}loer homology},
        date={2012},
     journal={Electronic pre-print, arXiv:1201.4663v3},
}

\bib{rs:paths}{article}{
      author={Robbin, J.},
      author={Salamon, D.},
       title={The {M}aslov index for paths},
        date={1993},
     journal={Topology},
      volume={4},
       pages={827\ndash 844},
}

\bib{sarkar:tri}{article}{
      author={Sarkar, S.},
       title={Maslov index formulas for {W}hitney n-gons},
        date={2011},
     journal={J. {S}ymplectic {G}eom.},
      volume={9},
       pages={251\ndash 270},
}

\bib{sw:CHF}{article}{
      author={Sarkar, S.},
      author={Wang, J.},
       title={An algorithm for computing some {H}eegaard {F}loer homologies},
        date={2010},
     journal={Annals of Math.},
      volume={171},
       pages={1213\ndash 1236},
}

\bib{s:GL}{article}{
      author={Seidel, P.},
       title={Graded {L}agrangian submanifolds},
        date={2000},
     journal={Bull. Soc. Math. France},
      volume={128},
       pages={103\ndash 146},
}

\bib{ss:R1}{article}{
      author={Seidel, P.},
      author={Smith, I.},
       title={A link invariant from the symplectic geometry of nilpotent
  slices},
        date={2006},
     journal={Duke Math. J.},
      volume={134},
       pages={453\ndash 514},
}

\bib{ss:R2}{article}{
      author={Seidel, Paul},
      author={Smith, Ivan},
       title={Localization for involutions in {F}loer cohomology},
        date={2010},
     journal={Geom. Funct. Anal.},
      volume={20},
      number={6},
       pages={1464\ndash 1501},
}

\bib{szabo:kh}{article}{
      author={Szab\'o, Z.},
       title={A geometric spectral sequence in {K}hovanov homology.},
        date={2010},
     journal={Electronic pre-print, arXiv:1010.4252v1},
}

\bib{et:R2}{article}{
      author={Tweedy, E.},
       title={The anti-diagonal filtration: reduced theory and applications},
        date={2011},
     journal={Electronic pre-print, arXiv:1109.3425v1},
}

\bib{jw:plat}{article}{
      author={Waldron, J.},
       title={An invariant of link cobordisms from symplectic {K}hovanov
  homology},
        date={2009},
     journal={Electronic pre-print, arXiv:0912.5067},
}

\end{biblist}
\end{bibdiv}
\bibliographystyle{plain}
\end{document}